\documentclass[11pt]{amsart}

\usepackage[T1]{fontenc}
\usepackage[utf8]{inputenc}
\usepackage{textcomp}
\usepackage{lmodern}
\usepackage{microtype}

\usepackage{amssymb}
\usepackage{amsthm}
\usepackage{amscd}

\usepackage{mathscinet}
\usepackage{hyperref}

\numberwithin{equation}{section}

\makeatletter
\newtheorem*{rep@theorem}{\rep@title}
\newcommand{\newreptheorem}[2]{%
\newenvironment{rep#1}[1]{%
 \def\rep@title{#2 \ref{##1}}%
 \begin{rep@theorem}}%
 {\end{rep@theorem}}}
\makeatother
\newtheorem{theorem}{Theorem}[section]

\newtheorem{lemma}[theorem]{Lemma}
\newtheorem{prop}[theorem]{Proposition}
\newreptheorem{prop}{Proposition}
\newtheorem{assumption}[theorem]{Assumption}
\newtheorem{cor}[theorem]{Corollary}
\newtheorem{fac}[theorem]{Fact}

\newtheorem{problem}[theorem]{Problem}

\theoremstyle{definition}
\newtheorem{definition}[theorem]{Definition}

\newtheorem{remark}[theorem]{Remark}

\newtheorem{example}[theorem]{Example}

\newtheorem{claim}[theorem]{Claim}
\newtheorem*{claim*}{Claim}

 \newenvironment{claimproof}{\begin{proof}}{\end{proof}}

\newcommand{\norm}[1]{\left\lVert#1\right\rVert}

  \def \B{\operatorname{\mathcal{B}}}
    \def \E{\operatorname{\mathbb{E}}}

  \def \alg{\operatorname{alg}}

  \def \VC{\operatorname{VC}}

    \def \dist{\operatorname{dist}}
  
    \def \cU{\operatorname{\mathcal{U}}}
 \def \ttimes{\times \ldots \times}
   \newcommand{\opg}{\operatorname{opg}}
\newcommand{\pg}{\operatorname{pg}}
\newcommand{\ord}{\operatorname{ord}}
\newcommand{\range}{\operatorname{rng}}

\newcommand{\Uniform}{\operatorname{Uniform}[0,1]}
\newcommand{\qftp}{\operatorname{qftp}}
\newcommand{\tp}{\operatorname{tp}}
\newcommand{\monus}{\dot{-}}

\allowdisplaybreaks[2]

\begin{document}

\title{Hypergraph regularity and higher arity VC-dimension}
\author{Artem Chernikov and Henry Towsner}
\date{\today}

\begin{abstract}
  We generalize the fact that graphs with small VC-dimension can be approximated by rectangles, showing that hypergraphs with small VC$_k$-dimension (equivalently, omitting a fixed finite $(k+1)$-partite $(k+1)$-uniform hypergraph) can be approximated by $k$-ary cylinder sets.

  In the language of hypergraph regularity, this shows that when $H$ is a $k'$-uniform hypergraph with small VC$_k$-dimension for some $k<k'$, the decomposition of $H$ given by hypergraph regularity only needs the first $k$ levels---one can approximate $H$ using sets of vertices, sets of pairs, and so on up to sets of $k$-tuples---and that on most of the resulting $k$-ary cylinder sets, the density of $H$ is either close to $0$ or close to $1$.

  We also show a suitable converse: $k'$-uniform hypergraphs with large VC$_k$-dimension cannot have such approximations uniformly under all measures on the vertices.
\end{abstract}

\maketitle

\tableofcontents

\section{Introduction}

We generalize the fact that graphs with small VC-dimension can be approximated by rectangles  \cite{alon2007efficient, MR2815610}, showing that hypergraphs with small VC$_k$-dimension\footnote{See Definition \ref{def: VCk dim}.} (equivalently, hypergraphs omitting a fixed finite $(k+1)$-partite $(k+1)$-uniform hypergraph\footnote{See Remark \ref{rem: VC_k iff omits H}}) can be approximated by $k$-ary cylinder sets\footnote{See Definition \ref{def: smaller argebras in GPS}.}.

Our main result is:
\begin{theorem}\label{thm: intro main thm}
  For every $k<k'$, every $d$ and every $\varepsilon>0$, there is an $N$ such that whenever $H\subseteq{V\choose k'}$ has VC$_k$-dimension less than $d$, $H$ differs from a union of at most $N$ $k$-ary cylinder sets by at most $\varepsilon|V|^{k'}$ points.
\end{theorem}
Stated in a more general way, this is Corollary \ref{cor: main thm for hypergraphs}.  We also prove an appropriate converse: that if $H$ has this approximation property with a bound on $N$ which is uniform over all measures on $V$ then $H$ has small VC$_k$-dimension; this is Theorem \ref{thm:converse}.

To see why we should expect such a result, first recall the situation for graphs.  It is convenient to interpret the Szemer\'edi regularity lemma as saying that when $G=(V,E)$ is a large finite graph, we can  present the characteristic function $\chi_E$ of  the edge relation $E$ in the form
\[\chi_E=f^\top+f^\bot\]
where $f^\top$ is the ``structured'' portion of the form $$f^\top(x,y)=\sum_{i,j\leq n}\alpha_{i,j}\chi_{V_i}(x)\chi_{V_i}(y)$$
 where $V=\bigcup_{i\leq n}V_i$ is a partition and the $\alpha_{i,j}$ are real numbers, and $f^\bot$ is quasirandom.  That is, we can view $E$ as a finite partition with weights $\alpha_{i,j}$ indicating the density of edges between $V_i$ and $V_j$, with $f^\bot$ representing the random determination of which which edges are actually present.

When $G$ has small VC-dimension\footnote{That is, the family $\{E_x \mid x \in V \}$ of subsets of $V$ has small VC-dimension, where $E_x$ is the fiber $\{y \in V \mid (x,y) \in E \}$. Equivalently, there is a small bipartite graph which $G$ contains no induced copies of.}, the $f^\bot$ part is small \cite{alon2007efficient, MR2815610,2016arXiv160707701C}.  More precisely, for each $d$ and each $\varepsilon>0$, there is a bound $N$ so that whenever $G$ is a graph with VC-dimension at most $d$, there is a regularity partition into $N$ pieces so that the quasirandom part satisfies $\sum_{x,y\in V}|f^\bot(x,y)|^2<\varepsilon|V|^2$.  (Indeed, $N$ is polynomial in $\varepsilon$, with the degree of the polynomial depending on $d$.)

This means that the weights $\alpha_{i,j}$ are each either close to $1$ or close to $0$, so this is equivalent to saying that $G$ is approximately the union of those rectangles where $\alpha_{i,j}$ is close to $1$.

We cannot quite get a reverse implication, that $f^\bot$ being small implies small VC-dimension.  It cannot be exactly an equivalence because having small VC-dimension is a combinatorial property, while $f^\bot$ has a measure-theoretic character.  (For instance, if we take a very large graph of small VC-dimension, and then graft on it a small graph of large VC-dimension, say with size $o(|V|)$, the small graph cannot meaningfully change $f^\bot$.)  Instead, having small VC-dimension is equivalent to having $f^\bot$ be small \emph{uniformly} for all possible measures on $V$.

Now, consider what happens when we generalize to hypergraphs---that is, $H=(V,E)$ with $E\subseteq{V\choose k}$ for some $k\geq 2$.  Something similar holds if all slices of $E$ have small VC-dimension---that is, for every fixed $z_1,\ldots,z_{k-2}$ in $V$, the binary relation
\[E_{z_1,\ldots,z_{k-2}}=\left\{(x,y)\mid (x,y,z_1,\ldots,z_{k-1})\in E\right\}\subseteq{V\choose 2}\]
has small VC-dimension\footnote{It is more common to consider a stronger assumption, that all ways of viewing $E$ as a graph on $V\times V^{k-1}$ have small VC-dimension.  However the weaker slice-wise assumption here suffices, and is the notion for which we get a converse.  There are examples showing that the slice-wise assumption is strictly weaker.}.  When this holds, we have
\[\chi_E=f^\top+f^\bot\]
where the $f^\top$ portion has the form 
$$f^\top(x_1,\ldots,x_k)=\sum_{i_1,\ldots,i_k}\alpha_{i_1,\ldots,i_k}\prod_j\chi_{V_{i_j}}(x_j)$$
 and $\sum_{\bar x\in V^k}|f^\bot(\bar x)|^2<\varepsilon|V|^k$.  That is, $H$ is approximated by boxes (\cite{2016arXiv160707701C}, which corresponds to the case $k=1$ and $k'$ arbitrary of Theorem \ref{thm: intro main thm}).

This is a very strong conclusion, suggesting that small VC-dimension is a very restrictive condition for a hypergraph.  For a general regular hypergraph $H = (V,E)$, the characterization given by hypergraph regularity \cite{nagle:MR2198495,rodl:MR2069663,MR2373376} involves a more complicated decomposition
\[\chi_E=f_{k-1}+\cdots+f_1+f^\bot\]
where $f_1$ has the form $\sum_{i_1,\ldots,i_k}\alpha_{i_1,\ldots,i_k}\prod_j\chi_{V_{i_j}}(x_j)$ as above, but the $f_j$ in general are sums of $j$-ary cylinder sets.  (For instance, $f_2$ is, roughly speaking, the portion of $\chi_H$ which can be described using directed graphs.)

Small VC-dimension collapses not only the random part $f^\bot$, but also all the more complex parts $f_{k-1}+\cdots+f_2$.  There ought to be a weaker notion than small VC-dimension which corresponds to just $f^\bot$ being small; more generally, there ought to be notions which correspond to collapsing part of this sequence, so that $f^\bot+f_{k-1}+\cdots+f_{j+1}$ is small.

The natural candidate is the notion of VC$_k$-dimension\footnote{See Definitions \ref{def: VCk dim} and \ref{def: VCd for higher arity}.  VC$_1$ is ordinary VC-dimension.  A $(k+1)$-graph has small VC$_k$-dimension if it omits a small $(k+1)$-partite hypergraph.} implicit in Shelah's work in model theory \cite{MR3273451, MR3666349} and studied further in \cite{chernikov2014n}.

The proof of the aforementioned result for graphs of finite VC-dimension---which corresponds to the $k=1, k'=2$ case of Theorem \ref{thm: intro main thm}---is fairly short.  The key point is that if a graph $E$ has finite VC-dimension, so does the graph $E^*=\{(x,x',y)\mid y\in E_x\bigtriangleup E_{x'}\}$ on $V^2\times V$.  ($E_x$ is the fiber $\{y \in V \mid (x,y)\in E\}$.)  A graph with finite VC-dimension has small $\varepsilon$-nets \cite{haussler1987}: that is, there is a list of $y_1,\ldots,y_n \in V$ such that, for all pairs $(x,x')$, either the fiber $E^*_{x,x'}$ has density less than $\varepsilon$, or $E^*_{x,x'}\cap\{y_1,\ldots,y_n\}\neq\emptyset$.  That is, for any two points $E_x,E_{x'}$, either $E_x\bigtriangleup E_{x'}$ is small, or $E_x\bigtriangleup E_{x'}$ includes one of the points $y_1,\ldots,y_n$.  We call $\{y_1,\ldots,y_n\}$ an ``$\varepsilon$-net for differences'': the points $y_1,\ldots,y_n$ are a universal test for whether two fibers can be far apart.
There are only finitely many subsets of $\{y_1,\ldots,y_n\}$, so we can then approximate the graph as a union of rectangles of the form
\[\{x\mid (E_x\bigtriangleup E_{x_i})\cap\{y_1,\ldots,y_n\}=\emptyset\}\times E_{x_i}\]
for a short\footnote{In fact, using the bounds given by the VC theorem and Sauer-Shelah, of size polynomial in $\varepsilon$.} list of points $x_1,\ldots,x_m$.

A quick glance at this paper suggests that the proof of the generalization to hypergraphs will be slightly more complicated.

We carry out our argument in the setting of a Keisler graded probability space.  This is the natural infinitary setting for such arguments; in particular, it is the setting one obtains by considering a hypergraph $H\subseteq\prod_{i\leq k}V_i$ with $N\leq\min_i |V_i|$ and letting $N\rightarrow\infty$.  Many statements which would be approximate, or ``up to $o(N^k)$'', or something similar when considering large $N$ become exact in the infinitary setting.  Most importantly, in a probability space we can identify the ``lower dimensional information'' mentioned above with the projection onto a $\sigma$-algebra.  Additionally, this lets us speak of the distinction between finite and infinite $\VC_k$-dimension, rather than having to speak precisely of quantitative bounds for what it means to have a ``small'' $\VC_k$-dimension.

We further work in a compound multipartite setting, where we consider subsets of $\prod_{i\in[k]}V_i^{m_i}$---that is, we not only allow separate sets $V_i$ for each coordinate, we keep track of the possibility that we may have multiple coordinates coming from the same set.  (The graph $E^*$ above, which is naturally viewed as a subset of $V_1^2\times V_2$, suggests why this setting shows up in the course of the proof.)  For completeness, since it does not seem to have appeared in the literature, we write down the extension of the Keisler graded probability space to this setting in detail in Section \ref{sec:graded probability spaces}.  We need some results about the Gowers uniformity norms and their relationship to conditional expectation in this setting; these results are standard, but have also not been developed in the multipartite setting.  We include them for completeness as well, but postpone this discussion to Section \ref{sec:gowers}.

In Section \ref{sec:vc_dem_def} we define VC$_k$-dimension and recall some standard examples and facts.  However we will want to consider not just hypergraphs---that is, sets---but functions with range $[0,1]$.  We may think of these functions as weighted hypergraphs, with ordinary hypergraphs as the case where the functions are $\{0,1\}$-valued.  Such functions show up at intermediate steps anyway---for instance, in the decompositions above, the components $f^\top,f^\bot$ are naturally functions, not sets.  The extension of VC-dimension to functions has appeared in various places (e.g. \cite{MR893902,MR1385403,yaacov2009continuous}), and we give the analogous definition of VC$_k$-dimension in Definition \ref{def: VCk dimension of functions}.

We include some results showing that various operations preserve VC$_k$-dimension of functions; to avoid interrupting the main thread of the argument, we postpone this to Section \ref{sec:operations}.  The last and most difficult of these is Theorem \ref{prop: gen fib fin VCk dim}, showing that given a family of functions of low $\VC_k$-dimension, the ``average'' function still has low $\VC_k$-dimension (more precisely, the $\VC_k$-dimension of the function $f'(x_1,\ldots,x_{k+1}):=\int  f(x_1, \ldots, x_{k+2}) d \mu(x_{k+2})$ can be bounded in terms of the maximum of the $\VC_k$-dimensions of the functions $f(x_1,\ldots,x_{k+2})$ over all $x_{k+2}$). Our proof combines structural Ramsey theory with a variant of the Aldous-Hoover-Kallenberg theorem on exchangeable arrays of random variables. It provides a higher arity generalization of the main result of \cite{yaacov2009continuous} for $k=1$ using different methods.


Section \ref{sec:fibers} is devoted to proving the existence of ``$\varepsilon$-nets for differences'' for  hypergraphs of low VC$_k$-dimension.  It is a bit surprising that this is possible, because we do not have any analog of the existence of $\varepsilon$-nets; it is not even clear what the higher arity generalization of an $\varepsilon$-net would be.  Nonetheless, we \emph{do} have an analog of the $\varepsilon$-net for differences, in the following sense.

When $H\subseteq\prod_{i\leq k+1}V_i$ has small VC$_k$-dimension, it is no longer reasonable to expect that there is a short list $x_1,\ldots,x_n \in V$ so that every $k$-ary fiber $H_x$ with $x \in V$ is close (i.e.~has small symmetric difference) to one of the $H_{x_i}$.  Rather, we have to expect that each fiber $H_x$ is described by the $H_{x_i}$ \emph{together with lower dimensional information}.  This is the content of Proposition \ref{prop: finite VCk-dim implies approx}.  The remainder of Section \ref{sec:fibers} is devoted to further refinement of this result.

To prove Proposition \ref{prop: finite VCk-dim implies approx}, we suppose it fails and work with an infinite sequence of fibers which are all far from each other.  We then homogenize this sequence using many applications of Ramsey's Theorem and construct a counterexample to small VC$_k$-dimension from the resulting subsequence.  To manage the homogenization of the sequence, we pass to a sequence of indiscernibles in an ultrapower of the original graded probability space; this requires some model theoretic machinery.  We treat this machinery as a black box as much as possible, and isolate the model theoretic arguments to Section \ref{sec:indiscernibles}.

Having shown that there are finitely many $k$-ary fibers of $H$ which, up to lower dimensional information, approximate all the fibers, we are able to write down an approximation of $H$ using these fibers in Proposition \ref{prop: large proj k k plus 1}.  We then generalize this to the case where $H\subseteq\prod_{i\leq k'}V_i$ for any $k'>k$, concluding the main result of the paper, in Theorem \ref{thm: the very main thm soft} and then prove the quantitative Corollary \ref{thm: the very main thm} using one more detour through the model theoretic techniques of Section \ref{sec:indiscernibles}.

In Section \ref{sec:converse}, we prove the converse of the main theorem: if a function on $\prod_{i\in[k]}V_i$ has infinite VC$_k$-dimension, then there is some way to put a probability measure on the $V_i$ so that the function has no simple approximation using $(\leq k)$-ary sets.

Finally, in Section \ref{sec: final remarks} we discuss some questions and directions for future work that naturally arise given the results of the paper, along with some applications of our results in model theory.
\subsection{Acknowledgements}

Artem Chernikov was partially supported by the NSF CAREER grant DMS-1651321. He is grateful to Kota Takeuchi and Ita\"i Ben Yaacov for helpful discussions.  Henry Towsner was partially supported by NSF Grant DMS-1600263.  The authors thank the American Institute of Mathematics and the Institut Henri Poincar\'e for additional support.

\section{Preliminaries}\label{sec:prelim}

\subsection{Notation}\label{sec: notation}

We summarize the notation used throughout the article for a reference.

\begin{enumerate}
\item $\mathbb{N} = \{0,1, \ldots \}$. We write $\mathbb{R}_{>0}$ to denote the set of positive reals, $\mathbb{R}_{\neq 0}$ for the set of non-zero reals, $\mathbb{N}_{>0}$ for the set of positive integers, and $\mathbb{Q}^{[0,1]}$ for the set of rational numbers in the interval $[0,1]$.
\item For $i \in \mathbb{N}$, by a \emph{dyadic rational number of height $i$} we mean a rational number of the form $\frac{m}{n}$ with $m \in \mathbb{Z}$ and $n = 2^i$. We let $\mathbb{Q}_i$ be the set of all dyadic rationals of height $i$, and let $\mathbb{Q}_{\infty} = \bigcup_{i \in \mathbb{N}} \mathbb{Q}_i$ be the set of all dyadic rationals. We let $\mathbb{Q}_i^{[0,1]} := \mathbb{Q}_i \cap [0,1]$, note that it is a finite set of cardinality $2^i + 1$ for every $i \in \mathbb{N}$.
\item  For $k\in \mathbb{N}_{>0}$ we will denote by $[k]$ the set
  $\{1,\dotsc,k\}$, and $[0] := \emptyset$.
  \item For a set $V$ and $k \in \mathbb{N}$, $\binom{V}{k} = \left\{ W \subseteq V : |W| = k \right\} $ and $\binom{V}{\leq k} = \left\{ W \subseteq V : |W| \leq k \right\} $.
  \item For any $i,j \in \mathbb{N}$, $\delta_{i,j}$ is equal to $1$ if $i=j$ and $0$ otherwise. Given $k \in \mathbb{N}_{>0}$ and $i \in [k]$, $\bar{\delta}^k_i := (\delta_{i,j}: j \in [k])$. We let 
  $$\bar{0}^k := (\underbrace{0, \ldots, 0}_{k \textrm{ times}}), \bar{1}^k := (\underbrace{1, \ldots, 1}_{k \textrm{ times}}).$$
   We might omit $k$ if it is clear from the context, and simply write $\bar{\delta}_i, \bar{0}, \bar{1}$.
  \item Given $\bar{n} = (n_1, \ldots, n_k ), \bar{m} = (m_1, \ldots, m_k)$ in $\mathbb{N}^k$, we write $\bar{n} \leq \bar{m}$ if $n_i \leq m_i$ for every $i \in [k]$, and $\bar{n} < \bar{m}$ if $\bar{n} \leq \bar{m}$ and $n_i < m_i$ for at least one $i \in [k]$.
  \item Algebraic operations on tuples of numbers are always performed coordinate-wise. Given $\bar{n}, \bar{m} \in \mathbb{N}^k$ and $d \in \mathbb{N}$, we write $d \cdot \bar{n}$ to denote the tuple $(d n_1, \ldots, d n_k)$, $\bar{n} + \bar{m}$ to denote the tuple $(n_1 + m_1, \ldots, n_k + m_k)$, $\bar{n} \cdot \bar{m}$ for the tuple $(n_1 \cdot m_1, \ldots, n_k \cdot m_k)$, etc.
  \item Given two tuples $\bar{a} = (a_1, \ldots, a_m), \bar{b} = (b_1, \ldots, b_n)$, we write $\bar{a}^{\frown} \bar{b}$ for the concatenated tuple $(a_1, \ldots, a_m, b_1, \ldots, b_n)$.
\item Given a set $V$, $\mathcal{P}(V)$ denotes the set of its subsets.
\item For sets $V_1,\dotsc, V_k$ and $I\subseteq [k]$ we denote by $V_I$ the product 
$V_I=\prod_{i\in I} V_i$.  
\item Given $I \subseteq [k]$, a tuple $\bar{a} = (a_i : i \in I) \in V_I$ and a set $s \subseteq I$, we denote by $\bar{a}_s$ the subtuple $(a_i : i \in s) \in V_s$.
\item  Let  $R \subseteq V_1\ttimes V_k$ be a $k$-ary relation and $I\subseteq [k]$. Viewing
  $R$ as a binary relation on  $V_I\times V_{[k]\setminus I}$,
for  $b\in V_{[k]\setminus I}$
we denote by $R_{b}$ the fiber 
\[ R_{b} =\{ a \in V_I \colon (a,b) \in R\}. \]
\item If $\bar{a} \in V_1 \ttimes V_k$ and $\sigma: [k] \to [k]$ is a permutation, then $\sigma(\bar{a}) := \left( a_{\sigma(1)}, \ldots, a_{\sigma(k)} \right) \in V_{\sigma(1)} \ttimes V_{\sigma(k)}$.
\item Given a tuple $\bar{a} = (a_i : i \in I)$, $i \in I$ and $b \in V_i$, we let $\bar{a}_{b \to i}$ denote the tuple obtained from $\bar{a}$ by replacing $a_i$ by $b$.
\item Given a relation $R \subseteq V_1 \times \ldots \times V_k$, $I \subseteq [k]$ and $\bar{a} \in V_1 \times \ldots \times V_k$, we let 
$$R_{\bar{a} \to I} := \left\{ \bar{x} \in V_1 \times \ldots \times V_k  \mid \bar{x}_{a_i \to i, i \in I} \in R \right\}.$$
\item If $R \subseteq V_1 \ttimes V_k$ and $\sigma: [k] \to [k]$ is a permutation, then $R^{\sigma} := \{ \bar{a}_{\sigma} : \bar{a} \in R \} \subseteq V_{\sigma(1)} \ttimes V_{\sigma(k)}$.
\item For $R \subseteq V$, we write $\chi_R: V \to \{0,1\}$ to denote the characteristic function of $R$.
\item For $R \subseteq V$, we will use the notation $R^1 := R$ and $R^{-1} = \neg R := V \setminus R$.
\item If $X,Y$ are sets, then $X \triangle Y = (X\setminus Y) \cup (Y \setminus X)$ denotes their symmetric difference, and if $\varphi,\psi$ are first-order formulas, then $\phi \triangle \psi$ denotes the formula $\left( \varphi \land \neg \psi\right) \lor \left( \neg \varphi \land \psi \right)$.
\end{enumerate}
As usual, given a $\sigma$-algebra $\B$, a $\sigma$-subalgebra $\B_0 \subseteq \B$, and a $\B$-measurable function $f$, $\E(f \mid \B_0)$ denotes the conditional expectation. 
We will use freely that the conditional expectation corresponds to orthogonal projection in the corresponding Hilbert space of measurable functions---that is, for any $\B$-measurable functions $f,g$, $\int \E(f\mid\B_0)g d\mu=\int f\E(g\mid\B_0) d\mu$. As usual, the equality for functions in $L^2(\mu)$ is understood up to a measure $0$ set.
Given a set of $\B$-measurable functions $G$, for brevity $\E(f \mid \B_0 \cup G)$ will denote $\E \left(f \mid \sigma \left( \B_0 \cup G \right) \right)$. If $E \in \B$, we might write $\E(E \mid \B_0)$ to denote $\E \left( \chi_E \mid \B_0 \right)$.

\subsection{Graded probability spaces and cylinder sets}\label{sec:graded probability spaces}
We review and generalize to the partite setting the notion of graded probability spaces, which were introduced by Keisler in \cite{keisler1985chapter} and provide a natural setting for the analytic approach to the study of various hypergraph regularity phenomena.

%

We fix $k \in \mathbb{N}_{>0}$ and sets $(V_i)_{i \in [k]}$, and we are going to be considering the products $V^{n_1, \ldots, n_k} := \prod_{i \in [k]}V_i^{n_i}$ for arbitrary $n_1, \ldots, n_k \in \mathbb{N}$. An element of $\prod_{i \in [k]}V_i^{n_i}$ is a tuple
\[(v_{1,1},v_{1,2},\ldots,v_{1,n_1},v_{2,1},\ldots,v_{2,n_2},\ldots,v_{k,1},\ldots,v_{k,n_k}),\]
which we will usually abbreviate
\[(\bar v_1,\ldots,\bar v_k)\]
or just $\bar v$.
It is going to be convenient to define \emph{ordered concatenation}: if $\bar v=(\bar v_1,\ldots,\bar v_k)\in \prod_{i \in [k]}V_i^{n_i}$ and $\bar w=(\bar w_1,\ldots,\bar w_k)\in \prod_{i \in [k]}V_i^{m_i}$, we define
\[\bar v \oplus \bar w=(\bar v_1,\bar w_1,\bar v_2,\bar w_2,\ldots,\bar v_k,\bar w_k)\in \prod_{i\in [k]}V_i^{n_i+m_i}.\]

\begin{definition}\label{def: kGPS}
A \emph{$k$-partite graded probability space} $\mathfrak{P} = \left( V_{[k]}, \mathcal{B}_{\bar{n}}, \mu_{\bar{n}} \right)_{\bar{n} \in \mathbb{N}^k}$ consists of sets $(V_i)_{i \in [k]}$ and, for every $n_1,\ldots,n_k \in \mathbb{N}$, a $\sigma$-algebra $\mathcal{B}_{n_1,\ldots,n_k}\subseteq \mathcal{P} \left( \prod_{i \in  [k]}V_i^{n_i} \right)$ and a probability measure $\mu_{n_1,\ldots,n_k}$ on  $\mathcal{B}_{n_1,\ldots,n_k}$ satisfying the following axioms.

\begin{enumerate}
	\item (Symmetry) For every $n_1,\ldots,n_k \in \mathbb{N}$, $i \leq k$, permutation $\pi:[n_i]\rightarrow[n_i]$ and  $B\in\mathcal{B}_{n_1,\ldots,n_k}$, we let 
	\[B^\pi :=\{(\bar v_1, \ldots,\pi(\bar v_i),\ldots,\bar v_k)\mid (\bar v_1,\ldots,\bar v_i,\ldots,\bar v_k)\in B\}.\]
	Then 
	\begin{enumerate}
	\item $B^\pi  \in \mathcal{B}_{n_1,\ldots,n_k}$, and
	\item $\mu_{n_1,\ldots,n_k}(B^\pi)=\mu_{n_1,\ldots,n_k}(B)$.
	\end{enumerate}
	\item (Closure under products)
	If $B\in\mathcal{B}_{n_1,\ldots,n_k}$ and $C\in\mathcal{B}_{m_1,\ldots,m_k}$, then the reordered product
\[B\times C :=\{\bar v \oplus \bar w \mid \bar v\in B\text{ and }\bar w\in C\}\]
 belongs to $\mathcal{B}_{n_1+m_1,\ldots,n_k+m_k}$.
\item (Fubini property) 
Given $B\in\mathcal{B}_{n_1+m_1,\ldots,n_k+m_k}$ and $\bar w\in \prod_{i\in [k]}V_i^{m_i}$, write
\[B_{\bar w}=\{\bar v\mid \bar v \oplus \bar w\in B\}.\]
Then the \emph{Fubini property} holds for the algebras $\left(\B_{\bar{n}}, \B_{\bar{m}}, \B_{\bar{n} + \bar{m}} \right)$:
\begin{enumerate}
\item $B_{\bar w} \in \mathcal{B}_{n_1,\ldots,n_k}$ for all ${\bar w} \in \prod_{i \in [k]}V_i^{m_i}$;
\item the function ${\bar w} \mapsto \mu_{n_1,\ldots,n_k}(B_{\bar w})$ from $\prod_{i \in [k]}V_i^{m_i}$ to $[0,1]$ is $\mu_{m_1,\ldots,m_k}$-measurable; and 
\item $\mu_{n_1+m_1,\ldots,n_k+m_k}(B)=\int \mu_{n_1,\ldots,n_k}(B_{\bar w}) d\mu_{m_1,\ldots,m_k}(\bar w)$.
\end{enumerate}

\end{enumerate}

A \emph{graded probability space} $\mathfrak{P} = \left( V, \mathcal{B}_{n}, \mu_{n} \right)_{n \in \mathbb{N}}$ is just a $1$-partite graded probability space (equivalently, for any $k \in \mathbb{N}_{\geq 1}$, it can be identified with a $k$-partite graded probability space $V_1 = \ldots = V_k = V$, $\mathcal{B}_{n_1, \ldots, n_k} = \mathcal{B}_n$ and $\mu_{n_1, \ldots, n_k} = \mu_n$ for all $n_1, \ldots, n_{k}$ with $|n_1 + \ldots + n_k| = n$).
\end{definition}

\begin{remark}\label{rem: basic props of GPS}
	\begin{enumerate}
	\item A partite graded probability space canonically induces a $\sigma$-algebra and measure on any product $\prod_{j \in [d]}V_{i_j}$ with $i_j \in [k]$, by identifying elements of $\prod_{j \in [d]}V_{i_j}$ with elements of $\prod_{i \in [k]}V_i^{n_i}$ for any appropriate choice of $n_i$ and a  permutation of the coordinates (by symmetry, the choice of permutation does not matter).
	\item Let $\mathfrak{P} = \left( V_{[k]}, \mathcal{B}_{\bar{n}}, \mu_{\bar{n}} \right)_{\bar{n} \in \mathbb{N}^k}$ be a partite graded probability space. Recall that for any set $V$, $V^{0} = \{\emptyset\}$.
Then given $i \in [k]$, the set $\prod_{i \in [k]}V^{\delta_{i,j}}$ is naturally identified with the set $V_i$, the algebra $\mathcal{B}_{\bar{\delta}_i}$ is naturally identified with an algebra $\mathcal{B}^{\mathfrak{P}}_i$ of subsets of $V_i$, and the measure $\mu_{\bar{\delta}_i}$ with a measure $\mu^{\mathfrak{P}}_i$ on $\mathcal{B}^{\mathfrak{P}}_i$ (recall that $\bar{\delta}_i$ is the tuple with $1$ in the $i$th position and $0$ in the other positions, see Section \ref{sec: notation}).	
Then all of the measures $\mu_{\bar{n}}, \bar{n} \in \mathbb{N}^k$ are determined by the measures $(\mu^{\mathfrak{P}}_i : i \in [k])$ (by a straightforward induction on $|n_1 + \ldots + n_2|$ using symmetry and Fubini).
	
		\item For any $n_1, \ldots, n_k, m_1, \ldots, m_k$, we let
\[\mathcal{B}_{n_1,\ldots,n_k} \times \mathcal{B}_{m_1,\ldots,m_k} := \sigma \left( \{B \times C : B \in \mathcal{B}_{n_1,\ldots,n_k}, C \in \mathcal{B}_{m_1,\ldots,m_k}\} \right)\] be the product $\sigma$-algebra. Then 
$$\mathcal{B}_{n_1,\ldots,n_k} \times \mathcal{B}_{m_1,\ldots,m_k} \subseteq \mathcal{B}_{n_1+m_1,\ldots,n_k+m_k}$$
 (by the closure under products) and $\mu_{n_1+m_1,\ldots,n_k+m_k}$ extends the product measure $\mu_{n_1, \ldots, n_k} \times \mu_{m_1, \ldots, m_k}$ (by Fubini property).
 Note however that in a typical case of interest for us this inclusion of algebras is strict.
	\end{enumerate}
\end{remark}

\begin{remark}\label{rem: gen Fub}

Assume that the Fubini property as in Definition \ref{def: kGPS}(3) holds for $\left( \B_{\bar{n}}, \B_{\bar{m}}, \B_{\bar{n} + \bar{m}} \right)$. Then, via a straightforward approximation by simple functions argument, it also lifts from measures to general integrals. That is, for any $\B_{\bar{n} + \bar{m}}$-measurable function $f: V^{\bar{n} + \bar{m}} \to \mathbb{R}$ we have:
\begin{enumerate}
	\item the fiber $f_{\bar{w}}: V^{\bar{n}} \to [0,1], \bar{v} \mapsto f(\bar{v} \oplus \bar{w})$ is $\mu_{\bar{n}}$-measurable for all $\bar{w} \in V^{\bar{m}}$;
	\item the function $\bar{w} \mapsto \int f\left(\bar{v} \oplus \bar{w} \right) d\mu_{\bar{n}}\left( \bar{v} \right)$ is $\B_{\bar{m}}$-measurable;
	\item the function $\bar{v} \mapsto \int f\left(\bar{v} \oplus \bar{w} \right) d\mu_{\bar{m}}\left( \bar{w} \right)$ is $\B_{\bar{n}}$-measurable (using symmetry in Definition \ref{def: kGPS}(1));
	\item \begin{gather*}
		\int f \left( \bar{v} \oplus \bar{w} \right) d\mu_{\bar{n} + \bar{m}} \left( \bar{v} \oplus \bar{w} \right) = \int \left( \int f\left(\bar{v} \oplus \bar{w} \right) d\mu_{\bar{n}}\left( \bar{v} \right) \right) d\mu_{\bar{m}}\left( \bar{w} \right) = \\
		\int \left( \int f\left(\bar{v} \oplus \bar{w} \right) d\mu_{\bar{m}}\left( \bar{w} \right) \right) d\mu_{\bar{n}}\left( \bar{v} \right).
	\end{gather*}
\end{enumerate}

\end{remark}

\noindent We have the following natural way to form a new partite graded probability space from a given one.
\begin{remark}\label{rem: power graded prob space}
(``Gluing coordinates'') Assume $\left( V_{[k]}, \mathcal{B}_{\bar{n}}, \mu_{\bar{n}} \right)_{\bar{n} \in \mathbb{N}^k}$ is a $k$-partite graded probability space. Let $t \in \mathbb{N}$ and $\bar{n}_i = (n_{i,1}, \ldots, n_{i,k}) \in \mathbb{N}^k$ for $i \in [t]$ be arbitrary. We define $V'_i := V^{\bar{n}_i}$ for $i \in [t]$, and for $\bar{m} = (m_1, \ldots, m_t) \in \mathbb{N}^t$ we let $\bar{m}' := m_1 \bar{n}_1 + \ldots + m_t \bar{n}_t \in \mathbb{N}^{k}$, $\B'_{\bar{m}} := \B_{\bar{m}'}, \mu'_{\bar{m}} := \mu_{\bar{m}'}$.

Then $\mathcal{B}'_{\bar{m}}$ can be viewed as an algebra of subsets of $\prod_{i \in [t]}\left( V'_i \right)^{m_i}$ (identifying the product $\prod_{i \in [t]}\left( \prod_{j \in [k]} V_j^{n_{i,j}} \right)^{m_i}$ with $\prod_{j \in [k]} V_j^{\sum_{i \in [t]} m_i n_{i,j}}$ by Remark \ref{rem: basic props of GPS}(1)), and it is easy to see that $\left( V'_{[t]}, \mathcal{B}'_{\bar{n}}, \mu'_{\bar{n}} \right)_{\bar{n} \in \mathbb{N}^{t}}$ is a $t$-partite graded probability space.

\end{remark}

\begin{definition}\label{def: smaller argebras in GPS}
	Let $\left( V_{[k]}, \mathcal{B}_{\bar{n}}, \mu_{\bar{n}} \right)_{\bar{n} \in \mathbb{N}^k}$ be a partite graded probability space, and fix $\bar{n} = (n_1, \ldots, n_k) \in \mathbb{N}^k$. Let $n := \sum_{i \in [k]} n_i$.
	\begin{enumerate}
		\item For each $i \in [k]$, let $I_i \subseteq [n_i]$, and $\bar{I} := \left(I_1, \ldots, I_k \right)$. Then $\mathcal{B}_{\bar{n}, \bar{I}}$ is the $\sigma$-subalgebra of $\mathcal{B}_{\bar{n}}$ generated by all sets of the form
		$$\left\{ \bar{x} = (\bar{x}_1, \ldots, \bar{x}_k) \in \prod_{i \in [k]} V_i^{n_i} : \left( \left(\bar{x}_1\right)_{I_1}, \ldots, \left(\bar{x}_k\right)_{I_k} \right) \in X \right\},$$
		for $X \in \mathcal{B}_{|I_1|, \ldots, |I_k|}$.
		\item For $m < n$, we let $\B_{\bar{n},m}$ be the $\sigma$-subalgebra of $\mathcal{B}_{\bar{n}}$ generated by $\bigcup \left \{\mathcal{B}_{\bar{n},\bar{I}} : \sum_{i \in [k]} |I_i| \leq m  \right \}$.
		\item \label{def: smaller argebras in GPS 3} If $\bar{n} \in \mathbb{N}^r$ for some $r < k$, then $\B_{\bar{n}} := \B_{\bar{n}^{\frown} \bar{0}^{k-r}}$ --- a $\sigma$-algebra of subsets of $V^{\bar{n}} = \prod_{i \in [r]} V_i ^{n_i}$, and $\mu_{\bar{n}} := \mu_{\bar{n}^{\frown} \bar{0}^{k-r}}$ a measure on it.
		
		And if $m < \sum_{i \in [r]} n_i$, then $\B_{\bar{n}, m} := \B_{\bar{n}^{\frown}\bar{0}^{k-r},m}$.
              \end{enumerate}

            We refer to the sets in $\mathcal{B}_{\bar n,m}$ as the \emph{${\bar n\choose m}$-cylinder sets}, and to the sets in $\B_{\bar{1}^k,t}$ with $t \leq k$ as the \emph{$t$-ary cylinder sets}.
\end{definition}
\noindent In other words, $\mathcal{B}_{\bar{n}, m}$ is generated by those sets in $\mathcal{B}_{\bar{n}}$ that can be defined by measurable conditions each of which can involve at most $m$ out of $n$ variables. The inclusion $\B_{\bar{n},m} \subseteq \B_{\bar{n}}$ is strict in general.

\section{$\VC_k$-dimension}\label{sec:vc_dim}

\subsection{VC$_{k}$-dimension for relations}\label{sec:vc_dem_def}
We review the notion of VC$_k$-dimension, for $k \in \mathbb{N}$, generalizing the usual Vapnik-Chervonenkis dimension in the case $k=1$. It is implicit in Shelah's work on $k$-dependent theories in model theory \cite{MR3666349, MR3273451} and is studied in \cite{chernikov2014n}; and further in \cite{MR3509704, chernikov2017mekler, chernikov2019n} for model theory of  groups and fields, and in \cite{MR3787371} in connection to hypergraph growth rates.
\begin{definition}\label{def: VCk dim}
For $k \in \mathbb{N}$, let $V_1, \ldots, V_{k+1}$ be sets. We say that a $(k+1)$-ary relation $E\subseteq V_1 \ttimes V_{k+1}$ has \emph{$\VC_k$-dimension $\geq d$}, or $\VC_k(E) \geq d$, if there is a \emph{$k$-dimensional $d$-box} $A = A_1 \ttimes A_k$ with $A_i \subseteq V_i$ and $|A_i|=d$ for $i=1, \ldots, k$ \emph{shattered} by $E$. That is, for every $S \subseteq A$, there is some $b_S \in V_{k+1}$ such that $S = A \cap E_{b_S}$.
We say that $\VC_k(E) = d$ if $d$ is maximal such that there is a $d$-box shattered by $E$, and $\VC_k(E) = \infty$ if there are $d$-boxes shattered by $E$ for arbitrarily large $d$. 
\end{definition}

In the case $k=1$ and $E \subseteq V_1 \times V_2$, $\VC_1(E) = d$ simply means that the family $\mathcal{F} := \{ E_a : a \in V_2 \}$ of all subsets of $V_1$ given by the fibers of $E$ has $\VC$-dimension $d$. 

The following equivalence is straightforward (see \cite[Proposition 5.2]{chernikov2014n} for the details).
\begin{remark}\label{rem: VC_k iff omits H}
For $E\subseteq V_1 \times \ldots \times V_{k+1}$, $\VC_{k}(E) \leq d$ implies that $E$ omits some finite $(k+1)$-partite hypergraph as an induced partite hypergraph, with parts of size at most $d' := 2^{d^k}$. And if $E$ omits some finite $(k+1)$-partite hypergraph with all parts of size at most $d'$, then $\VC_k(E) \leq d'$.

In particular, $\VC_k(E) < \infty$ if and only if $E$ omits some finite $(k+1)$-partite hypergraph as an induced partite hypergraph.
\end{remark}

\begin{fac}\label{fac: basic prop of k-dep}
For every $d \in \mathbb{N}$ there exists some $D = D(d) \in \mathbb{N}$ such that: if $E,F \subseteq V_1 \ttimes V_{k+1}$ are two relations with $\VC_{k}(E), \VC_{k}(F) \leq d$, then:
\begin{itemize}
\item \cite[Corollary 3.15]{chernikov2014n} $\VC(\neg E), \VC(E \cap F), \VC(E \cup  F) \leq D$;
\item \cite[Corollary 5.3]{chernikov2014n} If $\sigma \in S_{k+1}$ is any permutation of the set \{1, \ldots, k+1\}, then $\VC_{k}(E^\sigma) \leq D$.
\end{itemize}
\end{fac}

We also extend the definition of $\VC_k$-dimension to relations of arity higher than $k+1$ as follows:
\begin{definition}\label{def: VCd for higher arity}
  Let $k<k'\in\mathbb{N}$ be arbtirary. We say that a $k'$-ary relation $E\subseteq V_1\ttimes V_{k'}$ has \emph{$\VC_k$-dimension $\leq d$} if for any $I\subseteq[k']$ with $|I|=k'-(k+1)$ and any $b\in V_I$, the relation $E_b$ (i.e.~the fiber of $E$ with the coordinates in $I$ fixed by the elements of the tuple $b$, viewed as a $(k+1)$-ary relation on $V_{[k']\setminus I}$) has $\VC_k$-dimension $\leq d$ (in the sense of Definition \ref{def: VCk dim}).

  We write $\VC_k(E)$ for the least $d$ such that $VC_k$-dimension of $E$ is $\leq d$, or $\infty$ if there is no such $d$.
\end{definition}
That is, when $E$ is a $k'$-ary relation with $k'>k+1$, the $\VC_k$-dimension of $E$ is the supremum of the $\VC_k$-dimension over all $(k+1)$-ary fibers $E_b$.

\begin{remark}\label{rem: props of VCk higher arity}
	It is easy to see that any fiber of a relation with finite $\VC_k$-dimension also has finite $\VC_k$-dimension; that finite $\VC_k$-dimension is preserved under Boolean combinations and permutations of variables (using Fact \ref{fac: basic prop of k-dep}); and that if $k' > k_2 \geq k_1 $ and $E$ is a $k'$-ary relation with $\VC_{k_1}(E) < \infty$, then also $\VC_{k_2}(E) < \infty$. 
\end{remark}


%

The natural examples of relations with finite $\VC_k$-dimension are those which are ``essentially $k$-ary''---that is, relations which are built from $k$-ary relations.

\begin{example}
Let $E \subseteq V_1 \ttimes V_{k+1}$ be a relation 	given by a finite Boolean combination of arbitrary relations $E_1, \ldots, E_m, m \in \mathbb{N}$, such that each $E_i$ is of the form $E'_i \times V_{[k+1] \setminus I_i}$ for some $I_i \subseteq [k+1]$ with $|I_i| \leq k$ and some $E'_i \subseteq V_{I_i}$. 
Then $\VC_k(E) < \infty$ by Fact \ref{fac: basic prop of k-dep}(1), since every relation of arity $\leq k$ trivially has finite $\VC_k$-dimension.

\end{example}
The main result of the paper essentially shows that, up to an error of arbitrarily small measure, every $k$-dependent relation is of this form.
\begin{example}
Assume $V_1 = V_2 = V_3 = V$, $ F, G, H \subseteq V^2$ are arbitrary (e.g.~quasi-random), and let $E \subseteq V^3$ consist of those triples $(x,y,z) \in V^3$ for which an odd number of the pairs $(x,y), (x,z), (y,z)$ belongs to $F,G,H$, respectively. We claim that VC$_2(E) \leq 65$. Consider any $\{y_1, \ldots, y_5 \} \subseteq V$ and $\{z_1, \ldots, z_{65} \} \subseteq V$. By Ramsey's theorem, possibly reordering the elements, we may assume that either $\{y_1, y_2, y_3 \} \times \{z_1, z_2, z_3 \} \subseteq H$ or $\{ y_1, y_2, y_3 \} \times \{ z_1, z_2, z_3 \} \cap H = \emptyset$. But then no $x \in V$ can satisfy $E_x \cap \{y_1, y_2, y_3 \} \times \{ z_1, z_2, z_3 \} = \{ (y_1,z_1), (y_2, z_2), (y_3, z_3) \}$, as this would imply that no two of the values $\chi_F(x,y_1), \chi_F(x,y_2), \chi_F(x,y_3)$ can be equal, which is impossible.
\end{example}

\begin{example}
Let $V$ be a $K$-vector space, where $K$ is one of the following fields: $\mathbb{F}_p$, $\mathbb{C}$, $\mathbb{F}_p^{\alg}$ or $\mathbb{R}$, where $p$ is a prime number. Let $f: V\times V \to K$ be a non-degenerate bilinear form. Then every relation definable in the structure $(V,K,f)$ (on tuples of any arity), in the sense of first order logic, has finite $\VC_2$-dimension. See \cite{chernikov2019n} for the details.
\end{example}

The following is a generalization of the Sauer-Shelah lemma from VC$_1$ to VC$_k$-dimension.

\begin{fac}\label{fac: n-dep Sauer-Shelah}\cite[Proposition 3.9]{chernikov2014n}
If $E \subseteq V_1 \ttimes V_{k+1}$ satisfies $\VC_{k}(E) < d$, then there is some $\varepsilon = \varepsilon(d) \in \mathbb{R}_{>0}$ such that: for any $A = A_1 \ttimes A_{k} \subseteq V_1 \ttimes V_k$ with $|A_1| = \ldots = |A_{k}| = m$, there are at most $2^{m^{k - \varepsilon}}$ different sets $S \subseteq A$ such that $S = A \cap E_b$ for some $b \in V_{k+1}$.
\end{fac}

\begin{remark}
More precisely, if $\VC_k \leq d$, then the upper bound above is actually given by $\sum_{i<z} {m^k \choose i} \leq 2^{m^{k-\varepsilon}}$ for $m \geq k$, where $z = z_k(m, d+1)$ is the Zarankiewicz number, i.e.~the minimal natural number $z$ satisfying: every $k$-partite $k$-hypergraph with parts of size $m$ and $\geq z$ edges contains the complete $k$-partite hypergraph with each part of size $d+1$. If $k=1$, then $z_1(m,d+1) = d+1$, hence the bound in Fact \ref{fac: n-dep Sauer-Shelah}  coincides with the Sauer-Shelah bound, and for a general $k$ the bound in Fact \ref{fac: n-dep Sauer-Shelah} appears close to optimal (see \cite[Proposition 3.9]{chernikov2014n} for the details).
\end{remark}

\subsection{$\VC_k$-dimension for real-valued functions}\label{sec:vc real}

We generalize the notion of $\VC_{k}$-dimension and some of its basic properties from relations to functions, generalizing \cite{MR893902,MR1385403} in the case $k=1$.

	\begin{definition}\label{def: VCk dimension of functions}
	Let $f: \prod_{i \in [k+1]} V_i \to [0, 1]$ be a function.
	\begin{enumerate}
		\item Given $r < s \in [0,1]$, we say that a box $A = A_1 \ttimes A_{k}$ with $A_i \subseteq V_i$ is \emph{$(r,s)$-shattered by $f$} if for every $S \subseteq A$ there exists some $c_S \in V_{k+1}$ so that $f(\bar{a},c_S) \leq r$ for every $\bar{a} \in S$ and $f(\bar{a},c_S) \geq s$ for every $\bar{a} \in A \setminus S$.
		\item Given $\bar{d} = (d_{r,s})_{r<s \in [0,1]}$ with each $d_{r,s} \in \mathbb{N}$, we will write $\VC_k(f) \leq  \bar{d}$ if for every $r<s \in [0,1]$, there is no box $A = \prod_{i \in [k]} A_i$ with $A_i \subseteq V_i$ and $|A_i| = d_{r,s}$ for each $i \in [k]$ which is $(r,s)$-shattered by $f$.
		\item We say that $f$ has \emph{finite $\VC_{k}$-dimension}, or $\VC_{k}(f) < \infty$, if there exists some sequence $\bar{d}$ with $d_{r,s} \in \mathbb{N}$ so that $\VC_{k}(f) \leq \bar{d}$; and that $f$ has infinite $\VC_{k}$-dimension or $\VC_k(f)=\infty$ otherwise.

		\item Given an arbitrary $k' \in \mathbb{N}$, we say that a function $f: \prod_{i \in [k']} V_i \to [0, 1]$ satisfies $\VC_k(f) \leq \bar{d}$ if either $k' \leq k$, or $k' > k$ and for any $I\subseteq[k']$ with $|I|=k'-(k+1)$ and any $\bar{b} \in V_I$, the function $f_{\bar{b}}: V_{[k'] \setminus I} \to [0,1], f_{\bar{b}}(\bar{x}) = f(\bar{x} \oplus b)$  has $\VC_k$-dimension $\leq \bar{d}$.
                  \end{enumerate}
	\end{definition}

	\begin{remark}
		Note that if $E \subseteq \prod_{i \in [k+1]} V_i$, then $\VC_{k}(E) \leq d$ if and only if $\VC_k(\chi_{E}) \leq \bar{d}$ with $d_{r,s} = d$ for all $r<s \in [0,1]$. 
	\end{remark}
	
	It is sometimes convenient to speak of the \emph{$\VC_k$-dimension of $f$ ``at $(r,s)$''}:
\begin{definition}
  Let $f: \prod_{i \in [k+1]} V_i \to [0, 1]$ be a function.  We will write $\VC_k^{r,s}(f)\leq d$ if there is no box $A = \prod_{i \in [k]} A_i$ with $A_i \subseteq V_i$ and $|A_i| = d$ for each $i \in [k]$ which is $(r,s)$-shattered by $f$.
\end{definition}
That is, $\VC_k(f)\leq\bar d$ is the same as $\VC_k^{r,s}(f)\leq d_{r,s}$ for all $r<s$.
	
	Finally, the following is a straightforward analog of Remark \ref{rem: VC_k iff omits H} for real-valued functions.
	\begin{remark}\label{rem: VCk for f iff omits}
For $f: \prod_{i \in [k+1]} V_i \to [0,1]$ and $r < s$ in $[0,1]$, $\VC^{r,s}_{k}(f) \leq d$ implies that $f$ omits some finite $(k+1)$-partite $k$-uniform hypergraph $H$ as an ``induced'' partite hypergraph with parts of size at most $d' := 2^{d^k}$, in the sense that there is no way to identify the $i$th part of $H$ to a subset of $V_i$ so that, restricting to these sets, $f$ takes values  $\leq r$ on the edges of $H$ and $\geq s$ on the non-edges of $H$).

And if $f$ omits some finite $(k+1)$-partite hypergraph with all parts of size at most $d'$ in this sense, then $\VC^{r,s}_k(E) \leq d'$.
	\end{remark}

        \section{Level sets and some lemmas about $L^2$-norm}\label{sec:level sets}
        Throughout this section, we fix $k \in \mathbb{N}_{\geq 1}$ and let $\mathfrak{P} = \left(V_{[k]},\B_{\bar{n}}, \mu_{\bar{n}} \right)_{\bar{n} \in \mathbb{N}^{k}}$ be a $k$-partite graded probability space. We fix $\bar{n} = (n_1, \ldots, n_k) \in \mathbb{N}^k$, $n = \sum_{i \in [k]} n_i$ and  $f: \prod_{i \in [k]} V_i^{n_i} \to [0,1]$ a $\B_{\bar{n}}$-measurable function.

        \subsection{Level Sets of functions}

        We will frequently need to consider the level sets of functions.

        \begin{definition}\label{def: level set}
        For $r,q \in \mathbb{R}$, we let 
        \begin{gather*}
        	f^{<r} := \left\{ \bar{x} \in V^{\bar{1}^{k}} : f(\bar{x}) < r \right\},\\
        	f^{\geq r} := V^{\bar{1}^{k} } \setminus f^{<r}, f^{[r,q)} := f^{<q} \cap f^{\geq r}.
        \end{gather*}
\end{definition}

        The next lemma captures the following idea: if $f$ is not $\mathcal{B}$-measurable then there should be points which are ``fuzzy'' with respect to $\mathcal{B}$, in the sense that there are an $r<s$ so that if we made a random choice of $x$ with respect to $\mathcal{B}$, there should be positive probability that $f(x)<r$ and positive probability that $f(x)>s$.  In the language of $\sigma$-subalgebras, this becomes the statement that both $\mathbb{E}(\chi_{f^{<r}}\mid\mathcal{B})(x)\geq\delta$ and $\mathbb{E}(\chi_{f^{\geq s}}\mid\mathcal{B})(x)\geq\delta$ for some $\delta>0$.

\begin{lemma} \label{lem: non-triv cond exp}
	Assume that  $f: V^{\bar{n}} \to [0,1]$ is a $\B_{\bar{n}}$-measurable  function, $\varepsilon \in \mathbb{R}_{>0}$ and $\B \subseteq \B_{\bar{n}}$ is a  $\sigma$-algebra such that $\norm{f - \E(f \mid \B)}_{L^2} \geq \varepsilon$.
Then there exist some $t = t(\varepsilon) \in \mathbb{N}$, $r<s \in \mathbb{Q}_t^{[0,1]}$ and $\delta = \delta(\varepsilon) \in \mathbb{R}_{>0}$ so that
$$\mu_{\bar{n}} \left(\left\{x \in V^{\bar{n}} : \E\left(\chi_{f^{<r}}\mid \B \right)(x) \geq \delta \land \E \left( \chi_{f^{\geq s}} \mid \B \right)(x) \geq  \delta \right\} \right) \geq \delta.$$
\end{lemma}

\begin{proof}

 
Let $\alpha \in \mathbb{R}_{>0}$ be arbitrary, and we fix a sufficiently large $t = t(\alpha) \in \mathbb{N}$ and an even $\ell = \ell(\alpha) \in \mathbb{N}$ and a partition $0 = q_0 < \ldots < q_{\ell} = 1$ of $[0,1]$ with $q_i \in \mathbb{Q}_t^{[0,1]}, q_{i} - q_{i-1} < \alpha$ for all $i \in [\ell]$. We let $U_{-1} := f^{ < q_1}$, $U_i := f^{\geq q_i} \cap f^{<q_{i+2}}$ for $i \in \{0, \ldots, \ell -2 \}$, $U_{i} := f^{\geq q_{i}}$ for $i \in \{\ell-1, \ell -2\}$. 

 Fix $\gamma \in \mathbb{R}_{>0}$, and let 
 \begin{gather*}
  Z := \left\{ x \in V^{\bar{n}} : \bigwedge_{i \in \{-1, \ldots, \ell\}} \E \left( \chi_{U_i} \mid \B \right) < 1 - \gamma \right\}, \textrm{ and for } i \in \{-1, \ldots, \ell \},\\
  V_i := \left\{ x \in V^{\bar{n}} \setminus Z : i = \min \left\{ j \in  \{-1, \ldots, \ell\} : \E(\chi_{U_j} \mid \B)(x) \geq 1 - \gamma \right\} \right\}.
  \end{gather*}  
Note that $\{ Z, V_{-1}, \ldots,  V_{\ell} \}$ is a partition of $V^{\bar{n}}$, and each of these sets is in $\B$. And for each $i \in \{-1, \ldots, \ell \}$ we have
 \begin{gather}\label{eq: exp of int2}
 	\mu_{\bar{n}} \left(V_i \cap U_i \right) = \int_{V_i} \chi_{U_i} d\mu_{\bar{n}} = \int_{V_i} \E \left( \chi_{U_i} \mid \B \right) d\mu_{\bar{n}}  \geq (1 - \gamma ) \mu_{\bar{n}} \left( V_i \right).
 \end{gather}
Consider the $\B$-measurable function $g := \sum_{i \in \{-1, \ldots, \ell\}} q_i \chi_{V_i}$. 

Fix  $i \in \{-1, \ldots, \ell\}$ and $x \in V_i \cap U_i$. Then $g(x) = q_i$, and by definition of the $U_i$'s: $f(x) \in [q_i, q_{i+2}]$ if $i \in \{0, \ldots, \ell-2 \}$,  $f(x) \in [q_i, 1]$ if $i \in \{\ell-1, \ell-2\}$, and $f(x) \in [0,q_1]$ if $i = -1$. In either case, we get $|f-g|(x) \leq 2 \alpha$.

Then, using the assumption on $f$, \eqref{eq: exp of int2} and that $f,g$ are $[0,1]$-valued, we have
\begin{gather*}
	\varepsilon \leq \norm{f - g}^2_{L^2} = \int (f-g)^2 d \mu_{\bar{n}} = \\
	\int_{Z}(f-g)^2 d \mu_{\bar{n}} + \sum_{i \in \{-1, \ldots, \ell\}} \int_{V_i \setminus U_i} (f-g)^2 d \mu_{\bar{n}} + \sum_{i \in \{-1, \ldots, \ell\}} \int_{V_i \cap U_i} (f-g)^2 d \mu_{\bar{n}}\\
	\leq \mu_{\bar{n}}(Z) + \sum_{i \in \{-1, \ldots, \ell\}} \mu_{\bar{n}} \left( V_i \setminus U_i \right) + \sum_{i \in \{-1, \ldots, \ell\}} (2\alpha)^2 \mu_{\bar{n}} (V_i \cap U_i)\\
	\leq \mu_{\bar{n}}(Z) + \gamma \sum_{i \in \{-1, \ldots, \ell\}} \mu_{\bar{n}} \left( V_i \right) + (2\alpha)^2 \sum_{i \in \{-1, \ldots, \ell\}} \mu_{\bar{n}} (V_i)\\
	\leq \mu_{\bar{n}}(Z)  + \gamma + (2\alpha)^2.
\end{gather*}
Assuming $\gamma + (2\alpha)^2 < \frac{\varepsilon}{2}$, we get $\mu_{\bar{n}}(Z) \geq \frac{\varepsilon}{2}$.

As $\left( U_{2i} : {i \in \{0, \ldots, \frac{\ell}{2}\}} \right)$ and $\left( U_{2i-1} : {i \in \{0, \ldots, \frac{\ell}{2}\}} \right)$ are both partitions of $V^{\bar{n}}$, we also have
  \begin{gather}\label{eq: exp of int1}
  	\textrm{for } \mu_{\bar{n}}\textrm{-almost every } x \in V^{\bar{n}}, \\ \sum_{i \in \{0, \ldots, \frac{\ell}{2}\}} \E(\chi_{U_{2i}} \mid \B)(x)=1 \textrm{ and } \sum_{i \in \{0, \ldots, \frac{\ell}{2}\}} \E(\chi_{U_{2i-1}} \mid \B)(x)=1. \nonumber
  \end{gather}

By definition, $x \in Z \implies \bigwedge_{i \in \{-1, \ldots, \ell\}} \E \left(\chi_{U_i} \mid \B \right)(x) \leq 1-\gamma$. In particular, 
taking $\delta_0 := \frac{ \gamma}{\ell} > 0$ and using \eqref{eq: exp of int1}, for each  $x \in Z$ there must exist some $i_0,i_1,j_0,j_1 \in \{-1, \ldots, \ell\}$ such that $i_0<i_1$ are both even, $j_0<j_1$ are both odd, and $ \E(\chi_{U_i} \mid \B)(x) \geq \delta$ for each $i\in\{i_0,i_1,j_0,j_1\}$. As there are at most $\ell^{4}$ possible choices for the quadruple $(i_0,i_1,j_0,j_1)$, by additivity of $\mu$ there is a set $Z' \subseteq Z, Z' \in \B$ with $\mu_{\bar{n}}(Z') \geq \delta_1 := \frac{\mu_{\bar{n}}(Z)}{\ell^4} \geq \frac{\varepsilon}{2 \ell^4} > 0$ and so that all $x \in Z'$ share the same values of $i_0,i_1,j_0,j_1$.  Then either $i_0+2<j_1$ (and so ) or $j_0+2<i_1$; we let $r := q_{i_0+2}, s := q_{j_1}$ in the former case, and $r := q_{j_0+2}, s := q_{j_1}$ in the latter case. Then $r<s$ and 
the conclusion of the lemma holds by monotonicity of conditional expectation, with $\delta := \min\{\delta_0, \delta_1\} > 0$ (note that the choice of $\delta$ and $t$ in the proof only depends on $\varepsilon$).
\end{proof}

\begin{lemma}\label{lem: comparing integrals}
	Let $(V, \B, \mu)$ be a probability space, and assume that $f_0, f_1: V \to [0,1]$ are $\B$-measurable functions so that $\int f_1 d\mu > \int f_0 d\mu$.
	Then there exist some $r < s \in \mathbb{Q}^{[0,1]}$ so that $\mu( f^{<r}_0 ) > \mu( f^{<s}_1 )$.
\end{lemma}
\begin{proof}
Without loss of generality we may replace $f_b$ by $f_b \circ \pi^{-1}_b : [0,1] \to [0,1]$, where $\pi_b: V \to [0,1] $ is a measure-preserving function (with respect to the Lebesgue measure on $[0,1]$)  so that $f_b \circ \pi_b^{-1}$ is monotone, for $b \in \{0,1\}$.  (We can take $\pi_b(x) := \mu \left(f^{<f_b(x)}_b \right)$ and make countably many tweaks for those $r \in [0,1]$ for which $\mu \left(\{x\mid f_b(x)=r\} \right)$ has positive measure.)

Now we almost have $\mu \left(f^{<f_b(x)}_b \right)=x$; the exception is if the left-handed derivative of $f_b$ at $x$ is equal to $0$ --- that is, if the set of $y$ such that $f_b(y)=f_b(x)$ has positive measure, and $x$ is in the middle or is the right endpoint of this constant interval.  But we at least have $\mu \left(f^{<f_b(x)}_b \right) \leq x$, and for all $r>f_b(x)$, we have $\mu(f^{<r}_b) \geq x$.

Let $\varepsilon>0$ be small enough and define $f'_0(x)=f_0(x+\varepsilon)$.  Then we have $\int_0^{1-\varepsilon}f'_0< \int_0^{1- \varepsilon}f_1$.  Then, since $f'_0,f_1$ are monotone, there must be an $x$ with $f'_0(x)<f_1(x)$.  Let $r\in (f'_0(x),f_1(x))$ and $s=f_1(x)$.  Then $\mu(f^{<r}_0) \geq x+\varepsilon>x \geq \mu(f^{<s}_1)$.
\end{proof}

\subsection{Lemmas about measure and $L^2$-norm}

In this section we collect some miscellaneous lemmas about measurability and the $L^2$-norm that will be needed later in the article.
\begin{remark} \label{rem : approx by indicator functions}
Let a $\sigma$-algebra $\B \subseteq \B_{\bar{n}}$, $\varepsilon \in \mathbb{R}_{>0}$ and a set $X \in \B_{\bar{n}}$ be given. If $\norm{X - \E \left( \chi_{X} \mid \B \right)}_{L^2} \leq \frac{\varepsilon^2}{2}$, then there exists some $Y \in \B$ such that 
	$\norm{\chi_X - \chi_Y}_{L^2} \leq 3 \varepsilon$ (and the converse implication obviously holds, with the same $\varepsilon$).
	
\end{remark}
\begin{proof}
	As $\E(\chi_X \mid \B)$ is $\B$-measurable, there must exist a $\B$-simple function $h=\sum_{i \in [m]} \alpha_i \chi_{C_i}$ for some $[m] \in \mathbb{N}$, $\alpha_i \in \mathbb{R}$ and pairwise disjoint sets $C_i \in \B$, such that $\norm{\E(\chi_X \mid \B) - h}_{L^2} < \frac{\varepsilon^2}{2}$, so $\norm{\chi_X-h}_{L^2}< \varepsilon^2$.  But then the measure of the union of those $C_i$ for which $\alpha_i \notin [0,\varepsilon)\cup(\varepsilon,1]$ must be at most $\varepsilon$ (as in Lemma \ref{lem: non-triv cond exp}).  So we may replace $\sum_i \alpha_i \chi_{C_i}$ by the union $C$ of those $C_i$ with $\alpha_i>\varepsilon$.  Then the $L^2$-distance of $\chi_C$ from $\sum_i \alpha_i\chi_{C_i}$ is at most $2\varepsilon$, so $||\chi_X-\chi_C||_{L^2}< \varepsilon^2+2\varepsilon  < 3 \varepsilon$.
\end{proof}

The following lemma is well known (see e.g.~\cite[Theorem 1.1]{bergelson1985sets}).
\begin{fac}\label{fac: Bergelson}
	For any $\varepsilon \in \mathbb{R}_{>0}$ and $m \in \mathbb{N}$ there exists some $N = N(\varepsilon, m) \in \mathbb{N}$ and $\xi = \xi(\varepsilon, m) \in \mathbb{R}_{>0}$ satisfying the following.
	Given any probability space $(V,\B,\mu)$ and any sequence $(X_i : i \in [N])$ of sets in $\B$ with $\mu(X_i) \geq \varepsilon$ for all $i \in [N]$, there exists some subsequence $(X_i : i \in I)$ with $I \subseteq [N]$, $|I| \geq m$ and such that $\mu (\bigcap_{i\in I} X_i) > \xi$.
\end{fac}

\begin{lemma}
\label{lem: Erdos-Stone} Let $R \in \mathcal{B}_{\bar{n}}$ be such that $\mu_{\bar{n}}\left(R\right)\geq\alpha>0$.
For $\bar{d} = (d_1, \ldots, d_k) \in \mathbb{N}^k_{\geq 1}$, let $\Sigma$ be the set 
\begin{gather*}
\bigg\{ \left(\bar{x}_{1},\ldots,\bar{x}_{k}\right)\in \prod_{i \in [k]} (V_i^{n_i})^{d_i} :\left(\bar{x}_{1,i_{1}},\ldots, \bar{x}_{k,i_{k}}\right)\in R \\
\text{ for all }i_{1} \in [d_1],\ldots,i_{k}\in\left[d_k\right]\bigg\} \text{.}
\end{gather*}
Then $\Sigma \in \mathcal{B}_{\bar{d} \cdot \bar{n}}$ and $\mu_{\bar{d} \cdot \bar{n}}\left(\Sigma \right)\geq\alpha^{d_1 \cdot \ldots \cdot d_k}>0$.
\end{lemma}

\begin{proof}
Let $R, \alpha$ and $d$ as above be fixed.
Let $i \in [k]$ be arbitrary, and let $R'$ be the set of all tuples
$\bar{x} = (\bar{x}_1, \ldots, \bar{x}_{i-1},(\bar{x}_{i,1}, \ldots, \bar{x}_{i,d_i}),\bar{x}_{i+1}, \ldots, \bar{x}_k )$ in $V_{1}^{n_1} \times \ldots \times V_{i-1}^{n_{i-1}} \times \left( V_i^{n_i} \right)^{d_i} \times V^{n_{i+1}}_{i+1} \times \ldots \times V^{n_k}_k$ so that 
$$(\bar{x}_1, \ldots, \bar{x}_{i-1}, \bar{x}_{i,j}, \bar{x}_{i+1}, \ldots,\bar{x}_k) \in R$$
 for every $i \in [d_i]$. Note that $R' \in \B_{\bar{n}_{d_i n_i \to i}}$ by closure under products.

Then, by Fubini property and H\"older inequality with $p=d_i,q=\frac{d_i}{d_i-1}$, we have
\begin{gather}
	\mu_{\bar{n}_{d_i n_i \to i}}\left(R'\right) = \int \mu_{\bar{0}_{d_i n_i \to i}} \left( R'_{\bar{x}_{[k] \setminus \{ i\}}} \right) d\mu_{\bar{n}_{0 \to i}}\left( \bar{x}_{[k] \setminus \{ i\}} \right) = \nonumber\\
	\int \mu_{\bar{0}_{n_i \to i} } \left( R_{\bar{x}_{[k] \setminus \{ i\}}} \right)^{d_i} d\mu_{\bar{n}_{0 \to i}}\left( \bar{x}_{[k] \setminus \{ i\}} \right)  = \nonumber \\
	\int \mu_{\bar{0}_{n_i \to i} } \left( R_{\bar{x}_{[k] \setminus \{ i\}}} \right)^{d_i} d\mu_{\bar{n}_{0 \to i}} \left( \bar{x}_{[k] \setminus \{ i\}} \right) \cdot \int 1^{\frac{d_i}{d_i - 1}} d\mu_{\bar{n}_{0 \to i}} \left( \bar{x}_{[k] \setminus \{ i\}} \right)  \nonumber \\
	\geq \left(\int \mu_{\bar{0}_{n_i \to i} } \left( R_{\bar{x}_{[k] \setminus \{ i\}}} \right) d\mu_{\bar{n}_{0 \to i}} \left( \bar{x}_{[k] \setminus \{ i\}} \right)\right)^{d_i} = \nonumber \\
	\left( \mu_{\bar{n}}(R) \right)^{d_i} \geq \alpha^{d_i} > 0. \nonumber
\end{gather}

Repeating the same argument for every coordinate $i \in [k]$ (using Remark \ref{rem: power graded prob space}(1)), we conclude 
$$\mu_{d_1 n_1, \ldots, d_k n_k} \left( \Sigma \right) \geq \alpha^{d_1 \cdot \ldots \cdot d_k} > 0.$$
\end{proof}

\begin{lemma}\label{lem: operations for k+1,k proof}
Assume that $\bar{n} = \bar{n}_1 + \bar{n}_2 \in \mathbb{N}^{k}$ and $f,g: V^{\bar{n}}\to [0,1]$ are $\B_{\bar{n}}$-measurable functions, and $\varepsilon \in (0,1)$.
	\begin{enumerate}
		
		\item \label{lemeq: small norm impl small norm for almost all fibers} The following implications hold:
\begin{gather*}
\norm{f-g}_{L^2} < \varepsilon \implies \\	
\mu_{\bar{n}_2} \left( \left\{ \bar{x}_2 \in V^{\bar{n}_2} : \norm{f(- \oplus \bar{x}_2) - g(- \oplus \bar{x}_2)}_{L^2 \left( \mu_{\bar{n}_1} \right)} > \varepsilon^{\frac{1}{2}} \right\} \right) \leq \varepsilon \\
\implies \norm{f-g}_{L^2} < \varepsilon^{\frac{3}{4}}.
\end{gather*}

		\item \label{lem: fibers approx impl av approx} More precisely, if $Y \in \B_{\bar{n}_2}$ with $\mu_{\bar{n}_2}(Y) > 0$ and 
 $$\norm{f(- \oplus \bar{x}_2)) - g(- \oplus \bar{x}_2)}_{L^2 \left(\mu_{\bar{n}_1} \right)} < \varepsilon$$
  for every $\bar{x}_2 \in Y$, then 
\begin{gather*}
	\Bigg\lVert \frac{1}{\mu_{\bar{n}_2}(Y)} \int f(\bar{x}_1 \oplus \bar{x}_2) \chi_{Y}(\bar{x}_2) d \mu_{\bar{n}_2}(\bar{x}_2) -\\
	-  \frac{1}{\mu_{\bar{n}_2}(Y)}   \int g(\bar{x}_1 \oplus \bar{x}_2) \chi_{Y}(\bar{x}_2) d \mu_{\bar{n}_2}(\bar{x}_2) \Bigg \rVert _{L^2\left(\mu_{\bar{n}_1}(\bar{x}_1) \right)} 
	 <  \varepsilon.
\end{gather*}

		\item \label{lem: sum prod of approx} If $f_i, g_i : V^{\bar{n}} \to [0,1]$ are $\B_{\bar{n}}$-measurable and $\norm{f_i - g_i}_{L^2} < \varepsilon$ for $i \in [\ell]$, then 
		\begin{gather*}
			\norm{\prod_{i \in [\ell]} f_i - \prod_{i \in [\ell]} g_i}_{L^2} < (2 \ell + 1) \varepsilon \textrm{ and } 
			\norm{\sum_{i \in [\ell]} f_i - \sum_{i \in [\ell]} g_i}_{L^2} < \ell \varepsilon.
		\end{gather*}
		
	
	\end{enumerate}
\end{lemma}
\begin{proof}
By the Fubini property in graded probability spaces (Remark \ref{rem: gen Fub}) and standard calculations. E.g., for \eqref{lem: fibers approx impl av approx}, taking $h := f - g$ we have $\norm{h(- \oplus \bar{x}_2)}_{L^2(\mu_{\bar{n}_1})} < \varepsilon$ for every $\bar{x}_2 \in Y$. By Jensen's inequality, for every fixed $\bar{x}_1$,
\begin{gather*}
\left( \frac{1}{\mu_{\bar{n}_2}(Y)} \int_{Y} h(\bar{x}_1 \oplus \bar{x}_2) d \mu_{\bar{n}_2}(\bar{x}_2)\right)^2 \leq \frac{1}{\mu_{\bar{n}_2}(Y)} \int_{Y} h(\bar{x}_1 \oplus \bar{x}_2)^2 d \mu_{\bar{n}_2}(\bar{x}_2).
\end{gather*}
Using this and Fubini, we have
\begin{gather*}
	\norm{\frac{1}{\mu_{\bar{n}_2}(Y)}\int h(\bar{x}_1 \oplus \bar{x}_2) \chi_{Y}(\bar{x}_2) d \mu_{\bar{n}_2}(\bar{x}_2)}_{L^2(\mu_{\bar{n}_1}(\bar{x}_1))}^2 =\\
	\int \left( \frac{1}{\mu_{\bar{n}_2}(Y)} \int_{Y} h(\bar{x}_1 \oplus \bar{x}_2) d \mu_{\bar{n}_2}(\bar{x}_2)\right)^2 d\mu_{\bar{n}_1}(\bar{x}_1) \leq \\
	\int \left( \frac{1}{\mu_{\bar{n}_2}(Y)} \int_{Y} h(\bar{x}_1 \oplus \bar{x}_2)^2 d \mu_{\bar{n}_2}(\bar{x}_2)\right) d\mu_{\bar{n}_1}(\bar{x}_1) = \\
	 \int_{Y} \left( \frac{1}{\mu_{\bar{n}_2}(Y)} \int h(\bar{x}_1 \oplus \bar{x}_2)^2 d \mu_{\bar{n}_1}(\bar{x}_1)\right) d\mu_{\bar{n}_2}(\bar{x}_2) \leq \\
	  \frac{1}{\mu_{\bar{n}_2}(Y)} \int_{Y} \varepsilon^2 d\mu_{\bar{n}_2}(\bar{x}_2) \leq \varepsilon^2.
\end{gather*}
\end{proof}

\begin{lemma}\label{lem: meas of av fib}
	Let $\B \subseteq \B_{\bar{n}}$ be an arbitrary $\sigma$-algebra. Let $\bar{m} \in \mathbb{N}^k$, and assume that $g: V^{\bar{n} + \bar{m}} \to \mathbb{R}$ is a $\B_{\bar{n} + \bar{m}}$-measurable function such that
	the set of  $\bar{y} \in V^{\bar{m}}$ for which the function $g(-,\bar{y}) : \bar{x} \mapsto  g \left( \bar{x} \oplus \bar{y} \right)$ is $\B$-measurable has $\mu_{\bar{m}}$-measure $1$.	 Then the ``average fiber'' function $g'(\bar{x}) := \int g\left(\bar{x} \oplus \bar{y} \right) d\mu_{\bar{m}}\left( \bar{y} \right)$ is also $\B$-measurable.
\end{lemma}
\begin{proof}

Let $h: V^{\bar{n}} \to \mathbb{R}$ be an arbitrary $\B_{\bar{n}}$-measurable function orthogonal to $L^2(\B)$ (in the space $L^2 \left( \B_{\bar{n}}\right)$). Then, for every fixed $\bar{y} \in V^{\bar{m}}$ outside of a $\mu_{\bar{m}}$-measure $0$ set, we have
\begin{gather*}	
\left\langle g(- \oplus \bar{y}),h \right\rangle_{L^2}	 = \int g \left(\bar{x} \oplus \bar{y} \right) \cdot h(\bar{x}) d \mu_{\bar{n}} \left( \bar{x} \right) = 0.
\end{gather*}

Hence, by Fubini,
\begin{gather*}
	\int g'(\bar{x}) \cdot h(\bar{x}) d\mu_{\bar{n}} (\bar{x}) = \int \left( \int g \left( \bar{x} \oplus \bar{y} \right) d\mu_{\bar{m}} (\bar{y})  \right) \cdot h(\bar{x}) d\mu_{\bar{n}} (\bar{x}) =\\
	\int \left( \int g(\bar{x} \oplus \bar{y}) \cdot h(\bar{x}) d\mu_{\bar{n}}(\bar{x}) \right) d\mu_{\bar{n}}(\bar{y}) = 0
\end{gather*}
(so $g'$ has no correlation with any function orthogonal to $L^2(\B)$). Now we can write
\begin{gather*}
	g' = \E \left( g' \mid \B \right) + g^{\perp},
\end{gather*}
where $\E \left( g' \mid \B \right)$ is the projection of $g'$ onto the closed subspace $L^2 \left( \B \right)$, and $g^{\perp}$ is orthogonal to it. Then 
 \begin{gather*}
 	 \norm{g'}_{L^2}^2 = \int g' \cdot \left( \E \left( g' \mid \B \right) + g^{\perp} \right) d\mu_{\bar{n}} = \\
 	\int g' \cdot \E \left( g' \mid \B \right)d\mu_{\bar{n}} + \int g' \cdot g^{\perp} d\mu_{\bar{n}} = \int g' \cdot \E \left( g' \mid \B \right)d\mu_{\bar{n}},
 \end{gather*}
which implies $\norm{g'}^2_{L^2} =  \norm{ \E \left( g' \mid \B \right) }^2_{L^2}$, and so $g' = \E \left( g' \mid \B \right)$  is $\B$-measurable.
\end{proof}

For $x,y \in [0,1], x \monus y = \max \{0, x-y\} \in [0,1]$, $x \dot{+} y = \min \{1, x + y\} \in [0,1]$, and for $p \in \mathbb{N}$, $p \dot{\times}x = \underbrace{x \dot{+} \ldots \dot{+} x}_{p \textrm{ times }}$.

\begin{lemma}\label{lem: smoothening char func}
Assume that $f,g: V^{\bar{n}}\to [0,1]$ are $\B_{\bar{n}}$-measurable functions and $\varepsilon \in (0,1)$.
We consider the $\B_{\bar{n}}$-measurable set 
		$$[f<g] := \left\{\bar{x} \in V^{\bar{n}} : f(\bar{x}) < g(\bar{x}) \right\}$$
		and, for $p \in \mathbb{N}$, the $\B_{\bar{n}}$-measurable function $[f<g]^{p}: V^{\bar{n}} \to [0,1]$ defined by 
		$$[f<g]^{p} := p \dot{\times} \left( g \monus  f \right).$$
		Then there exists some $p = p(f,g, \varepsilon) \in \mathbb{N}$ such that $\norm{\chi_{[f<g]} - [f<g]^{p}}_{L^2} < \varepsilon$.
\end{lemma}
\begin{proof}	
	As $[f<g] = \bigcup_{\gamma \in \mathbb{Q}_{>0}} [f < (g-\gamma)]$, by countable additivity of $\mu_{\bar{n}}$ we can choose $\gamma >0$ small enough so that $\mu_{\bar{n}} \left([f<g] \setminus [f < (g - \gamma)] \right) < \varepsilon^2$. Let $p \in \mathbb{N}$ satisfy $p\gamma \geq 1$. Then
\begin{gather*}
	\norm{\chi_{[f<g]} - [f<g]^{p}}_{L^2}^2 = \int_{ V^{\bar{n}} \setminus  [f<g] } 0 d\mu_{\bar{n}} + \int_{ [f<(g-\gamma)]} 0 d\mu_{\bar{n}} + \\ 
	 \int_{[f<g] \setminus [f < (g - \gamma)]} \left( \chi_{[f<g]} - [f<g]^{p} \right)^2 d\mu_{\bar{n}} \\
	\leq \mu_{\bar{n}} \left( [f<g] \setminus [f < (g - \gamma)] \right) \leq \varepsilon^2.
\end{gather*}
\end{proof}

\section{Approximation by finitely many fibers for functions of bounded $\VC_k$-dimension}\label{sec:fibers}
\subsection{Statement and some corollaries of the approximation result}
The aim of this section is to prove the following.

\begin{prop}\label{prop: finite VCk-dim implies approx}
Let $\left(V_{[k+1]}, \B_{\bar{n}}, \mu_{\bar{n}} \right)_{n \in \mathbb{N}^{k+1}}$ be a $(k+1)$-partite graded probability space. Suppose that $f: V^{\bar{1}^{k+1}} \to [0,1]$ is $\B_{\bar{1}^{k+1}}$-measurable and $\VC_k(f)$ is finite. Then for every $\varepsilon>0$, there exist some $x_1, \ldots, x_{N} \in V_{k+1}$ such that: for every $x \in V_{k+1}$ we have
$$\norm{f_{x} - \E\left(f_x \mid \mathcal{B}_{\bar{1}^k, k-1} \cup \{f_{x_1}, \ldots, f_{x_N} \} \right)}_{L^2} < \varepsilon.$$
\end{prop}

Recall that for $x \in V_{k+1}$, $f_x : V^{\bar{1}^k} \to [0,1]$ is the function $\bar{y}  \mapsto  f(\bar{y}^{\frown} (x))$ corresponding to the fiber of $f$ at $x$. By  Remark \ref{rem: gen Fub}, $f_x$ is  $\B_{\bar{1}^k} = \B_{\bar{1}^{k \frown} (0) }$-measurable (see Definition \ref{def: smaller argebras in GPS}\eqref{def: smaller argebras in GPS 3}) for every $x\in V_{k+1}$.

For relations (i.e.~$\{0,1\}$-valued functions) of finite $\VC_k$-dimension this immediately implies the following (using Remark \ref{rem : approx by indicator functions}).
\begin{cor}\label{cor: approx by fib for rels}

  

Let $(V_{[k+1]}, \B_{\bar{n}}, \mu_{\bar{n}})_{n \in \mathbb{N}^{k+1}}$ be a $(k+1)$-partite graded probability space. Suppose that $E \in \B_{\bar{1}^{k+1}}$ and $\VC_k(E)$ is finite. Then there exist some $x_1, \ldots, x_{N} \in V_{k+1}$ such that, for every $x\in V_{k+1}$, there is a set $D_x$ which is a Boolean combination of $E_{x_1},\ldots,E_{x_N}$ and sets from $\mathcal{B}_{\bar 1^k,k-1}$ such that
  $$\mu_{\bar{1}^k} \left( E_x \triangle D_x \right) < \varepsilon.$$
\end{cor}

\begin{remark}
When $k = 1$, Corollary \ref{cor: approx by fib for rels} corresponds to the familiar result for relations of finite VC-dimension discussed in the introduction.


Indeed, in this case the algebra $\B_{(1),0} = \left\{\emptyset, V_1  \right\}$ is trivial. Assume $E \in \B_{1,1}$. Then by Corollary \ref{cor: approx by fib for rels}, there exist finitely many fibers $E_{x_1}, \ldots, E_{x_N}$ of $E$ with $x_i \in V_2$ so that for every $x \in V_2$, $\mu_{1,0} \left(E_x \triangle D_x \right) < \frac{\varepsilon}{2}$ for some $D_x$ a Boolean combination of $E_{x_1}, \ldots, E_{x_N}$.

Let $D_{1}, \ldots, D_{N'}$ list all Boolean combinations of $E_{x_1}, \ldots, E_{x_N}$ that appear as $D_x$ for some $x \in V_{2}$.  Then, for each $D_i$, we may choose some $x'_i\in V_2$ with $\mu_{1,0}\left(D_i\triangle E_{x'_i}\right)<\frac{\varepsilon}{2}$.

Now for every $x \in V_2$ there exists some $i \in [N']$ so that $\mu_{1,0} \left( E_x \triangle E_{x'_i}\right) < \varepsilon$.
That is, up to symmetric difference $\varepsilon$, $E$ has at most $N'$ different fibers.

And using Sauer-Shelah, $N'$ can be bounded by a polynomial of degree $d$.
\end{remark}

%
%

\subsection{A quantitative statement of the approximation result}

In this section we restate Proposition \ref{prop: finite VCk-dim implies approx} in a more quantitative form.  This takes some work to state, because there should be quantitative bounds not only on the length of the sequence of fibers, but also on the complexity of the sets from $\mathcal{B}_{\bar 1^k,k-1}$ used in the approximations.

In fact, most of the extra work is formulating the statement: the quantitative strengthening follows from the qualitative form by a compactness argument.  We do not need this stronger form in what follows, so the reader can safely skip this subsection.  Nonetheless, we include this stronger version both because the potential for bounds is of independent interest, and because the quantitative form is the form that can be applied directly to large finite hypergraphs.

The main additional definition we need to state the quantitative version will be $\mathcal{F}^{f,n,\bar b}$, which will be the collection of sets formed by certain fibers of level sets (recall Definition \ref{def: level set}) of $f$.

\begin{definition} \label{def: fiber algebra} Let $k \in \mathbb{N}_{\geq 1}$, $\mathfrak{P} = \left(V_{[k+1]},\B_{\bar{n}}, \mu_{\bar{n}} \right)_{\bar{n} \in \mathbb{N}^{k+1}}$ be a $(k+1)$-partite graded probability space, and assume that $f: V^{\bar{1}^{k+1}} \to [0,1]$ is a $\B_{\bar{1}^{k+1}}$-measurable $(k+1)$-ary function. 
\begin{enumerate}

\item Let $\bar{b}$ be a tuple from $V_{k+1}$ (finite or infinite), $\bar{w} = (\bar{w}_1, \ldots, \bar{w}_k) \in V^{\bar{m}}$ for some $\bar{m} = (m_1, \ldots, m_k) \in \mathbb{N}^k$ and $m \in \mathbb{N}$. We let $\mathcal{F}^{f, n, \bar{b}}_{\bar{w}}$ be the family of all sets in $\B_{\bar{1}^k, k-1}$ of the form
	\begin{gather*}
		 \left( f^{<q}_b\right)_{a_i \to i, i \in I} = \left\{\bar{x} = (x_1, \ldots, x_k) \in V^{\bar{1}^k} : \bar{x}_{a_i \to i, i \in I} \ ^{\frown}(b) \in f^{<q}\right\}
	\end{gather*}
	for some $q \in \mathbb{Q}_{n}^{[0,1]}$, $\emptyset \neq I \in \binom{[k]}{\leq k-1} $, $a_i \in \bar{w}_i$ for $i\in I$ and $b \in \bar{b}$.
	
	\item Given a tuple $\bar{b}$ from $V_{k+1}$ (finite or infinite), let 
	$$\mathcal{F}^{f,\bar{b}} := \bigcup \left\{ \mathcal{F}^{f, n, \bar{b}}_{\bar{w}} : n \in \mathbb{N}, \bar{w} \in V^{\bar{m}}, \bar{m} \in \mathbb{N}^k \right\}.$$
	\item We let $\mathcal{F}^{f} := \bigcup \left\{ \mathcal{F}^{f,(b)} : b \in V_{k+1} \right\}$.
	\item We let $\mathcal{B}^{f,n,\bar{b}}_{\bar{w}}, \B^{f,\bar{b}}, \B^f_{\bar{1}^k, k-1}$ be the $\sigma$-subalgebras (and $\mathcal{B}^{f,n,\bar{b},0}_{\bar{w}}, \B^{f,\bar{b},0}, \B^{f,0}$ the Boolean subalgebras) of $\B_{\bar{1}^k,k-1}$ generated by $\mathcal{F}^{f,n,\bar{b}}_{\bar{w}}, \mathcal{F}^{f,\bar{b}}, \mathcal{F}^{f}$ respectively. Note that when $\bar{b}$ is finite, we have $\mathcal{B}^{f,n,\bar{b}}_{\bar{w}} = \mathcal{B}^{f,n,\bar{b},0}_{\bar{w}}$ are both finite.
\end{enumerate}

\end{definition}

Now we can state a quantitative refinement of Proposition \ref{prop: finite VCk-dim implies approx} (which says, among other things, that the only sets from $\mathcal{B}_{\bar{1}^k, k-1}$ needed to approximate $f_x$ are the fibers of the level sets of $f$).

\begin{prop}\label{prop: finite VCk-dim implies approx bounded}
 For every $k \in \mathbb{N}$, $\bar{d} = (d_{r,s})_{r<s \in \mathbb{Q} \cap [0,1]}$ with $d_{r,s} \in \mathbb{N}$ and $\varepsilon \in \mathbb{R}_{>0}$ there exist some $N,N_0\in \mathbb{N}$ satisfying the following.

Let $\left(V_{[k+1]}, \B_{\bar{n}}, \mu_{\bar{n}} \right)_{n \in \mathbb{N}^{k+1}}$ be a $(k+1)$-partite graded probability space. Suppose that $f: V^{\bar{1}^{k+1}} \to [0,1]$ is $\B_{\bar{1}^{k+1}}$-measurable and $\VC_k(f) \leq \bar{d}$. Then there exist some $x_1, \ldots, x_{N} \in V_{k+1}$ such that: for every $x \in V_{k+1}$, there exist some sets $D_1, \ldots, D_{N_{0}} \in \mathcal{F}^{f, N_0,  (x_1, \ldots, x_N,x)}$ and a $\left(\{ D_i \}_{i \in [N_0]} \cup \{f^{<q}_{x_i}\}_{i \in [N], q \in \mathbb{Q}_{N_0}^{[0,1]}} \right)$-simple\footnote{Recall that this means that $g_x$ is a finite linear combination of the characteristic functions of these sets.} function $g_x$ with coefficients in $\mathbb{Q}_{N_0}^{[0,1]}$
such that 
  $$\norm{f_x - g_x}_{L^2} < \varepsilon.$$
\end{prop}

This version of the proposition is non-trivial if we take the $V_i$ to be very large finite sets---$N,N_0$ depend only on $k,\bar d,\varepsilon$, so we can choose the $V_i$ to be much larger than $N,N_0$.  In that case the $\sigma$-algebras are trivialized---every set is $\mathcal{B}_{\bar 1^{k+1},1}$-measurable, since it can be written as a very large finite union of singleton sets.  But the collection $\mathcal{F}^{f,N_0,(x_1,\ldots,x_N,x)}$ is \emph{not} all sets, so the conclusion of the proposition is still useful.     

\subsection{Approximability by conditional expectations}
\

\begin{proof}[Proof of \ref{prop: finite VCk-dim implies approx} and \ref{prop: finite VCk-dim implies approx bounded}]
  To prove Proposition \ref{prop: finite VCk-dim implies approx}, assume towards a contradiction we are given a $(k+1)$-partite graded probability space $\mathfrak{P}_0=\left(V_{[k+1]}, \B_{\bar{n}}, \mu_{\bar{n}} \right)_{n \in \mathbb{N}^{k+1}}$ so that the conclusion of Proposition \ref{prop: finite VCk-dim implies approx} fails for some $\varepsilon \in \mathbb{R}_{>0}$.

  Then we may select an infinite sequence $x_1,x_2,\ldots$ of elements of $V_{k+1}$ by successively choosing $x_{i+1}$ so that
  $$\norm{f_{x_{i+1}} - \E\left(f_{x_{i+1}} \mid \mathcal{B}_{\bar{1}^k, k-1} \cup \{f_{x_1}, \ldots, f_{x_i} \} \right)}_{L^2} \geq \varepsilon.$$

  We would like to ``homogenize'' this sequence.   For instance, we would like to ensure that the measures of sets like $f^{<r}_{x_i}\cap f^{<r}_{x_j}$ do not depend on the particular elements $x_i,x_j$ in this sequence (as long as $x_i\neq x_j$).  Using Ramsey's Theorem, we can get part way there: we can find an infinite subsequence so that for any $i<j$, $\mu_{\bar{1}^k}(f^{<r}_{x_i}\cap f^{<r}_{x_j})$ belongs to some interval $(a,a+\delta)$ for some $a \in [0,1]$ and a small $\delta > 0$.  (We do this by partitioning $[0,1]$ into finitely many intervals $[0,1]=\bigcup_r I_r$ and coloring pairs $i,j$ by the $r$ such that $\mu_{\bar{1}^k}(f^{<r}_{x_i}\cap f^{<r}_{x_j})\in I_r$.)  However, it will be convenient to pin down $\mu_{\bar{1}^k}(f^{<r}_{x_i}\cap f^{<r}_{x_j})$ exactly so that we do not need to keep track of the extra bounds like $\delta$.  Furthermore (for instance, by Fact \ref{fac: Bergelson}), if the measure of this intersection is constant, it must be strictly positive, and similarly for intersections of any number of the sets $f^{<r}_{x_i}$.

  We will need to arrange that a sequence of intersections of this kind always has positive measure, not for the sets $f^{<r}_{x_i}$, but with a more complicated set we define below.

  Furthermore, we want to take into account an additional property.  Each $f_{x_i}$ is measurable with respect to some $\sigma$-algebra $\mathcal{B}_i\subseteq\mathcal{B}_{\bar 1^k}$, and we can consider the ``tail $\sigma$-algebra'' $\mathcal{B}=\bigcap_N\bigcup_{i\geq N}\mathcal{B}_i$.  It is convenient to take the fibers $f_{x_i}$ to be mutually independent over this tail $\sigma$-algebra, because then we can have $\mathbb{E}(f_{x_l}\mid\mathcal{B}\cup\{f_{x_i} : i<l\})=\mathbb{E}(f_{x_l}\mid\mathcal{B})$ for all $l$.  This, too, is essentially a kind of homogenization implied by a de Finetti-style argument.

  In order to fully homogenize, we may need to leave the original space $\mathfrak{P}_0$ for a different space $\mathfrak{P}$ (the \emph{ultrapower} of $\mathfrak{P}$) in which we can find a sequence similar to the one we began with, but which is fully homogeneous.  The details of this construction are given in Section \ref{sec:indiscernibles}.  For now we treat this as a black box and focus on the combinatorial portion of the proof.  We therefore have, by Theorem \ref{thm:pass to indiscernible}\footnote{The proof of the quantitative version, Proposition \ref{prop: finite VCk-dim implies approx bounded}, proceeds nearly identically: we assume the conclusion of Proposition \ref{prop: finite VCk-dim implies approx} fails for some fixed $k, \bar{d}$ and $\varepsilon \in \mathbb{R}_{>0}$. Without loss of generality $\varepsilon \in \mathbb{Q}_{>0}$. This means that for every $j \in \mathbb{N}$, there exists some $(k+1)$-partite graded probability space $\mathfrak{P}_j = (V^j_{[k+1]}, \B^j_{\bar{n}}, \mu^j_{\bar{n}})_{\bar{n} \in \mathbb{N}^{k+1}}$, some $\B^j_{\bar{1}^{k+1}}$-measurable function $f^j: \prod_{i \in [k+1]} V^j_i \to [0,1]$ with $\VC_{k}(f^j) \leq \bar{d}$ and some $x_1^j, \ldots, x_{j}^j \in V^j_{k+1}$ such that for every $t \leq j$ we have:
for any sets $D_1, \ldots, D_{j} \in \mathcal{F}^{f^j, j,  (x^j_1, \ldots, x^j_t)}$ and any $\left(\{ D_i \}_{i \in [j]} \cup \left\{f^{<q}_{x^j_i} : i \in [t-1], q \in \mathbb{Q}_{j}^{[0,1]}\right\} \right)$-simple function $g$ with coefficients in $\mathbb{Q}_{j}^{[0,1]}$, $\norm{f_{x_t^j} - g}_{L^2} \geq \varepsilon$.  Assumption \ref{ass: ineapproximable} still follows from Theorem \ref{thm:pass to indiscernible}, so the remainder of the proof is unchanged.}:
  \begin{assumption}\label{ass: ineapproximable}
There exists a $(k+1)$-partite graded probability space $\mathfrak{P} = (V_{[k+1]}, \B_{\bar{n}}, \mu_{\bar{n}})_{n \in \mathbb{N}^{k+1}}$, $\varepsilon \in \mathbb{R}_{>0}$, a $ \B_{\bar{1}^k}$-measurable function $f: V^{\bar{1}^{k+1}} \to [0,1]$
	and a sequence $(x_l)_{l \in \mathbb{Z}}$ in $V_{k+1}$ satisfying the following:
	\begin{enumerate}
	\item  $\VC_k(f) \leq \bar{d}$;
        \item whenever $0\leq r<r'<s'<s\leq s$ are in $\mathbb{Q}$, $\delta\in\mathbb{R}_{>0}$, and
          $$\mu_{\bar 1^k} \left(\left\{\bar{x} \in V^{\bar{1}^k} \mid \E\left(\chi_{f^{<r}_{x_0}}\mid \B \right)(x) \geq \delta \land \E \left( \chi_{f^{\geq s}_{x_0}} \mid \B \right)(x) \geq  \delta  \right\} \right)>0,$$
          then for any $l\in\mathbb{N}$,
          $$\mu_{\bar 1^k} \left(\bigcap_{i\in[l]}\left\{\bar{x} \in V^{\bar{1}^k} \mid \E\left(\chi_{f^{<r'}_{x_i}}\mid \B \right)(x) \geq \delta \land \E \left( \chi_{f^{\geq s'}_{x_i}} \mid \B \right)(x) \geq  \delta  \right\} \right)>0;$$     
	\item $ \norm{f_{x_l} - \mathbb{E} \left(f_{x_l}\mid\mathcal{B}_{\bar{1}^k,k-1}\cup\{f_{x_i} : i < l\} \right)}_{L^2} \geq\varepsilon$ for all $l \in \mathbb{Z}$;
		\item $\B_{\bar{1}^k,k-1} \subseteq \B \subseteq \B_{\bar{1}^k}$;
		\item for all $l \in \mathbb{N}$ we have
	\begin{gather*}
		\mathbb{E}\left(f_{x_l}\mid\mathcal{B}_{\bar{1}^k,k-1}\cup\{f_{x_i} : i < l\} \right)\\
		=\mathbb{E} \left(f_{x_l}\mid\mathcal{B}\cup\{f_{x_i} : i < l\} \right)\\
		 = \mathbb{E} \left(f_{x_l}\mid\mathcal{B} \right),
	\end{gather*}

	\end{enumerate}
where $\B := \sigma \left(\left\{f_{x_i} : i < 0\right\} \cup \mathcal{B}_{\bar{1}^k,k-1} \right)$.
\end{assumption}

We will now show that this leads to a contradiction.  The idea is that Assumption \ref{ass: ineapproximable}(3) implies that the fibers have some ``random behavior'' relative to each other, and with the help of Assumption \ref{ass: ineapproximable}(2), (4), and (5), this random behavior is consistent enough that we can find a large box $Y \subseteq \prod_{i \in [k]} V_i$ and an $r'<s'$ in $[0,1]$ so that $Y$ is $(r',s')$-shattered by $f$, contradicting Assumption \ref{ass: ineapproximable}(1).

By Assumption \ref{ass: ineapproximable}(3) and (5), we have 
 $$\norm{f_{x_0} - \E \left( f_{x_0} \mid \B \right)}_{L^2} \geq \varepsilon.$$

By Lemma \ref{lem: non-triv cond exp} there exist some $r < s \in \mathbb{Q}^{[0,1]}$ and $\delta \in \mathbb{R}_{>0}$ so that 
\begin{gather}\label{eq: pos meas Fdelta0}
\mu_{\bar 1^k} \left(\left\{\bar{y} \in V^{\bar{1}^{k}} : \E\left(\chi_{f^{<r}_{x_0}}\mid \B \right)(\bar{y}) \geq \delta \land \E \left( \chi_{f_{x_0}^{\geq s}} \mid \B \right)(\bar{y}) \geq  \delta \right\} \right) \geq \delta.
\end{gather}

Fix arbitrary $r',s' \in \mathbb{Q}^{[0,1]}$ so that $r<r'<s'<s$, and let $E^1 := f^{<r'}, E^{-1} := f^{\geq s'}$, and $E^0=V^{\bar 1^k}$.
 For $i \in \mathbb{N}$ let 
 $$F_{\delta}(x_i):=\left\{\bar{y} \in V^{\bar{1}^k} \mid \E\left(\chi_{E^{1}_{x_i}}\mid \B \right)(\bar{y}) \geq \delta \land \E \left( \chi_{E^{-1}_{x_i}} \mid \B \right)(\bar{y}) \geq  \delta \right\} \in \B.$$

 This is precisely the set to which Assumption \ref{ass: ineapproximable}(2) applies. We should think of $\bar y\in F_\delta(x_i)$ as the points where $f_{x_i}$ is ``ambiguous'' to $\mathcal{B}$ in the sense that---as far as $\mathcal{B}$ can tell---both $f_{x_i}(\bar y)<r'$ and $f_{x_i}(\bar y)\geq s'$ seem plausible.

\begin{definition}\label{def: Su}
  Given $l \in \mathbb{N}$, $u: [l] \rightarrow\{-1,0,1\}$ and $x_1, \ldots, x_l \in V_{k+1}$, let $S_u \in \B_{\bar{1}^k}$ be the subset of $V^{\bar{1}^k}$ given by 
\[ \bigcap_{i \in  [l]}E_{x_i}^{u(i)}.\]
\end{definition}
That is, $u$ specifies a configuration of the $x_i$---whether we want our points to be in $E^1_{x_i}=f^{<r'}_{x_i}$, in $E^{-1}_{x_i}=f^{\geq s'}_{x_i}$, or to ignore $x_i$.  $S_u$ is then all the points which satisfy this configuration.

Now we want to show that for any $m,l \in \mathbb{N}$ and any sequence of functions $\left( u_{\bar{i}} :  \bar{i} = (i_1, \ldots, i_k) \in [m]^k \right)$ with $u_{\bar{i}}: [l] \to \{-1, 0, 1 \}$, we have
\[\int \prod_{\bar{i} \in [m]^k}\chi_{S_{u_{\bar{i}}}}(y_{1,i_1}, \ldots, y_{k, i_k})d\mu_{m \cdot \bar{1}^k}(\bar{y}_1,\ldots,\bar{y}_k)>0,\]
where $\bar{y}_s = (y_{s,1}, \ldots, y_{s,m})$ for all $s \in [k]$.
In particular, suppose that we take $l=2^{m^k}$, let $\pi:[l]\rightarrow \mathcal{P}([m]^k)$ be a bijection, and for each $\bar{i} \in [m]^k$ we define 
\[u_{\bar{i}}(t)=\left\{\begin{array}{ll}
1&\text{if }\bar{i} \in \pi(t)\\
-1&\text{otherwise}\end{array}\right.\]
for all $t \in [l]$.
Since the integral is positive, we have $\mu_{m \cdot \bar{1}^k} \left( Z \right) >0$ for the set $Z \in \B_{m \cdot \bar{1}^k}$ defined by
$$\chi_{Z}\left((y_{s,t})_{s \in [k], t \in [m]} \right) := \prod_{\bar{i} \in [m]^k}\chi_{S_{u_{\bar{i}}}}\left(y_{1,i_1}, \ldots, y_{k, i_k} \right).$$
Then, taking any tuple $\left(y_{s,t} : s \in [k], t \in [m] \right) \in Z$, we have that for any $A \subseteq Y := \prod_{s \in [k]}\{y_{s,1}, \ldots, y_{s,m}\}$ there is some $i \in [l]$ so that 
\begin{gather*}
	A = Y \cap E^1 = Y \cap f^{<r'}_{x_i} \textrm{ and } Y \setminus A = Y \cap E^{-1} = Y \cap f^{\geq s'}_{x_i},
\end{gather*}
hence the box $Y$ is $(r',s')$-shattered by $f$. This would give a contradiction to Assumption \ref{ass: ineapproximable}(1) starting with some $m > d_{r',s'} $.

\medskip

We turn to showing that, for any choice of $m,l \in \mathbb{N}$ and functions $(u_{\bar{i}} : \bar{i} \in [m]^k)$,
\[\int \prod_{\bar{i} \in [m]^k}\chi_{S_{u_{\bar{i}}}}(y_{1,i_1}, \ldots, y_{k, i_k})d\mu_{m \cdot \bar{1}^k}(\bar{y}_1,\ldots,\bar{y}_k)>0.\]
Since the inside of this integral is always non-negative, it suffices to find some subset of positive measure on which it is strictly positive.

Let 
$$\tilde{F} := \bigcap_{i \in  [l]}F_{\delta}(x_i) \in \B.$$

By \eqref{eq: pos meas Fdelta0} and Assumption \ref{ass: ineapproximable}(2), $\beta := \mu_{\bar{1}^k} \left(\tilde{F} \right) >0$.  An element of $\tilde{F}$ is  ``ambiguous'' to $\mathcal{B}$ for all the $f_{x_i}$ at once; we should expect (and it follows from the work below) that for any $u$ and any positive measure $\mathcal{B}$-measurable set $D$, $\tilde F\cap D\cap S_u$ has positive measure.

By Lemma \ref{lem: Erdos-Stone}, the set
\[F :=\left\{(y_{s,t} : s \in [k], t \in [m]) \in V^{m \cdot \bar{1}^k} \mid \forall \bar{i} \in [m]^k \ (y_{1, i_1}, \ldots, y_{k,i_k}) \in \tilde{F} \right\}\]
is in $\B_{m \cdot \bar{1}^k}$ and $\mu_{m \cdot \bar{1}^k}(F) \geq \beta^{m^k} > 0 $.


We will show that
\[\int_F \prod_{\bar{i} \in [m]^k}\chi_{S_{u_{\bar{i}}}}(y_{1, i_1}, \ldots, y_{k, i_k})d\mu_{m \cdot \bar{1}^k}(\bar{y}_1,\ldots,\bar{y}_k)>0.\]
That is, we will show that we can find a positive measure set of matrices $\bar y=(\bar y_1,\ldots,\bar y_k)\in F$ so that each column traversal---that is, each sequence $(y_{1,i_1},\ldots,y_{k,i_k})$ consisting of one element from each column---belongs to $S_{u_{\bar i}}$.  If we select $\bar y$ randomly then, for each $\bar i\in[m]^k$, there is a positive probability that $(y_{1,i_1},\ldots,y_{k,i_k})$ belongs to $S_{u_{\bar i}}$.  The claim will then follow by showing that the behavior of each column traversal is sufficiently independent.  This is what we now show: that if we focus on one row $\bar{i}^0\in[m]^k$, the behavior of all the other column traversals is $\mathcal{B}$-measurable.

Pick any $\bar{i}^0 = (i_1^0, \ldots, i_k^0) \in [m]^k$. Let 
\begin{align*}
&W_{\bar{i}^0}' := \left\{ \bar{i} \in [m]^k : i_1 \in [m] \setminus \{ i_1^0 \}, \ldots, i_k \in [m] \setminus \{i^0_k \} \right\},\\
&W^*_{\bar{i}^0} := \left\{ \bar{i} \in [m]^k :  \bar{i} \neq \bar{i}^0 \land \left( \bigvee_{s \in [k]} i_s = i_s^0 \right)\right\}.	
\end{align*}
Note that $[m]^k$ is the disjoint union of $W'_{\bar{i}^0}, W^*_{\bar{i}^0}$ and $\{ \bar{i}^0\}$. Using the Fubini property we have 
\begin{gather}
\int_F \prod_{\bar{i} \in [m]^k}\chi_{S_{u_{\bar{i}}}}(y_{1, i_1}, \ldots, y_{k, i_k})d\mu_{m \cdot \bar{1}^k}(\bar{y}_1,\ldots,\bar{y}_k) = \label{eq5} \\
\int_{F'} \prod_{\bar{i} \in W'_{\bar{i}^0}}\chi_{S_{u_{\bar{i}}}}(y_{1, i_1}, \ldots, y_{k, i_k})\left(\int_{F^*(\bar{y}_1', \ldots, \bar{y}'_k)} \chi_{S_{u_{\bar{i}^0}}}(y_{1, i^0_1}, \ldots, y_{k, i^0_k}) \right. \label{eq1}\\
\left.\prod_{\bar{i} \in W^*_{\bar{i}^0}}\chi_{S_{u_{\bar{i}}}}(y_{1, i_1}, \ldots, y_{k, i_k}) d\mu_{\bar{1}^k}(y_{1, i^0_1}, \ldots, y_{k,i^0_k}) \right) d \mu_{(m-1) \cdot \bar{1}^k} (\bar{y}'_1, \ldots, \bar{y}'_k), \nonumber
\end{gather}
where $\bar{y}'_s := (y_{s,t} : t \in [m] \setminus \{ i^0_s \})$, and $F' \subseteq V^{(m-1) \cdot \bar{1}^k}$ and $F^*(\bar{y}'_1, \ldots, \bar{y}'_k) \subseteq V^{\bar{1}^k}$ are the analogs to $F$ on suitable coordinates, i.e.
\begin{align*}
&F' := \left\{(\bar{y}'_1, \ldots, \bar{y}'_k) \in V^{(m-1) \cdot \bar{1}^k} \mid \bigwedge_{\bar{i} \in W'_{\bar{i}^0}} (y_{1,i_1}, \ldots, y_{k, i_k}) \in \tilde{F} \right\},\\
&F^*(\bar{y}'_1, \ldots, \bar{y}'_k) := \Bigg\{ (y_{1,i^0_1}, \ldots, y_{k,i^0_k}) \in V^{\bar{1}^k} \mid  (y_{1, i^0_1}, \ldots, y_{k, i^0_k}) \in \tilde{F} \land \\
& \bigwedge_{\bar{i} \in W^*_{\bar{i}^0}}  (y_{1, i_1}, \ldots, y_{k, i_k}) \in \tilde{F} \Bigg\}.
\end{align*}

Obviously $F' \in \B_{(m-1) \cdot \bar{1}^k}$ since $\tilde{F} \in \B \subseteq \B_{\bar{1}^k}$. Note also that if $\bar{i} \in W^*_{\bar{i}^0}$, then by definition we must have $i_s \neq i^0_s$ for at least one $s \in [k]$, so $\bar{i} \land \bar{i}^0$ has length $\leq k-1$ and 
``$(y_{1, i_1}, \ldots, y_{k, i_k}) \in \tilde{F}$'' viewed as a condition on the tuple $(y_{1,i^0_1}, \ldots, y_{k,i^0_k})$ can involve at most $k-1$ coordinates (with all the other coordinates appearing in $\bar{y}'_1, \ldots, \bar{y}'_k$ fixed). Hence  
$$F^*(\bar{y}'_1, \ldots, \bar{y}'_k) \in \sigma \left(\left\{ \tilde{F} \right\} \cup \B_{\bar{1}^k,k-1} \right) \subseteq \B$$
 for any $(\bar{y}'_1, \ldots, \bar{y}'_k) \in V^{(m-1) \cdot \bar{1}^k}$.
The integral in (\ref{eq1}) can be rewritten as
\begin{gather*}
\int_{F'} \prod_{\bar{i} \in W'_{\bar{i}^0}}\chi_{S_{u_{\bar{i}}}}(y_{1, i_1}, \ldots, y_{k, i_k}) \left(\int_{V(\bar{y}_1', \ldots, \bar{y}'_k)} \chi_{S_{u_{\bar{i}^0}}} d\mu_{\bar{1}^k}(y_{1, i^0_1}, \ldots, y_{k,i^0_k}) \right)\\
 d \mu_{(m-1) \cdot \bar{1}^k} \left(\bar{y}'_1, \ldots, \bar{y}'_k \right),
\end{gather*}
where the set $V(\bar{y}_1', \ldots, \bar{y}'_k)$ is given by the intersection of $F^*(\bar{y}'_1, \ldots, \bar{y}'_k)$ with the set
\begin{gather*}
\Bigg\{ (y_{1,i^0_1}, \ldots, y_{k,i^0_k}) \in V^{\bar{1}^k} \mid  \bigwedge_{\bar{i} \in W^*_{\bar{i}^0}}  (y_{1, i_1}, \ldots, y_{k, i_k}) \in S_{u_{\bar{i}}} \Bigg\}.
\end{gather*}
As in the previous paragraph, each condition ``$(y_{1, i_1}, \ldots, y_{k, i_k}) \in S_{u_{\bar{i}}}$'' here, viewed as a condition on the tuple $(y_{1,i^0_1}, \ldots, y_{k,i^0_k})$,  can involve at most $k-1$ coordinates and is given by some fiber of $S_{u_{\bar{i}}} \in \B_{\bar{1}^k}$, hence we have 
$$V(\bar{y}_1', \ldots, \bar{y}'_k) \in \sigma \left(\left\{ \tilde{F} \right\}\cup \B_{\bar{1}^k,k-1} \right) \subseteq \B$$
 for any $(\bar{y}_1', \ldots, \bar{y}'_k) \in V^{(m-1) \cdot \bar{1}^k}$.  Thus, for any fixed $\bar{y}'_1, \ldots, \bar{y}'_k \in V^{(m-1) \cdot \bar{1}^k}$, we have:
\begin{align}
 & \int_{V(\bar{y}_1', \ldots, \bar{y}'_k)} \chi_{S_{u_{\bar{i}^0}}}\left(y_{1, i^0_1}, \ldots, y_{k,i^0_k} \right) d\mu_{\bar{1}^k}\left(y_{1, i^0_1}, \ldots, y_{k,i^0_k} \right) \nonumber\\
&=\int_{V(\bar{y}_1', \ldots, \bar{y}'_k)}\prod_{r \in [l]}\chi_{E_{x_r}^{u_{\bar{i}^0}(r)}} \left( y_{1,i^0_1}, \ldots, y_{k, i^0_k} \right) d\mu_{\bar{1}^k} \label{eq2}\\
&=\int_{\left(V(\bar{y}_1', \ldots, \bar{y}'_k) \cap \bigcap_{r \in  [l-1]}E_{x_r}^{u_{\bar{i}^0}(r)} \right)} \chi_{E_{x_l}^{u_{\bar{i}^0}(l)}} \left( y_{1,i^0_1}, \ldots, y_{k, i^0_k} \right) d\mu_{\bar{1}^k} \\
&=\int_{\left(V(\bar{y}_1', \ldots, \bar{y}'_k) \cap \bigcap_{r \in [l-1]}E_{x_r}^{u_{\bar{i}^0}(r)} \right)}  \E\left(\chi_{ E^{u_{\bar{i}^0}(l)}_{x_l} }\mid\mathcal{B} \cup \{E_{x_r} : r < l \} \right) d\mu_{\bar{1}^k} \\
&\textrm{(as the set over which we integrate is in } \sigma \left(\B \cup \{E_{x_r} : r < l \} \right) \textrm{)} \nonumber\\
&=\int_{ \left( V(\bar{y}_1', \ldots, \bar{y}'_k) \cap \bigcap_{r \in [l-1]}E_{x_r}^{u_{\bar{i}^0}(r)} \right)}  
\E \left(\chi_{ E^{u_{\bar{i}^0}(l)}_{x_l} }\mid\mathcal{B} \right)d\mu_{\bar{1}^k}\\
& \textrm{(by Assumption \ref{ass: ineapproximable}(5)}) \nonumber\\
&\geq \int_{\left( V(\bar{y}_1', \ldots, \bar{y}'_k) \cap \bigcap_{r \in [l-1]}E_{x_r}^{u_{\bar{i}^0}(r)} \right)}  \delta d\mu_{\bar{1}^k}\\
& \textrm{(by the definition of  } F_{\delta}(x_l) \textrm{, as } V(\bar{y}_1', \ldots, \bar{y}'_k) \subseteq F_{\delta}(x_l) \textrm{)} \nonumber\\
&= \delta \int_{V(\bar{y}_1', \ldots, \bar{y}'_k)} \prod_{r \in [l-1]}\chi_{E_{x_r}^{u_{\bar{i}^0}(r)}} \left(y_{1,i^0_1}, \ldots, y_{k, i^0_k} \right) d\mu_{\bar{1}^k} \label{eq3}\\
&\geq\cdots \textrm{(repeating steps (\ref{eq2}) through (\ref{eq3}) for }l-1, l-2, \ldots \textrm{)} \nonumber\\
&\geq \delta^l \int_{V(\bar{y}_1', \ldots, \bar{y}'_k)} 1 d\mu_{\bar{1}^k} \nonumber\\
& = \delta^l \int_{F^*(\bar{y}'_1, \ldots, \bar{y}'_k)} \prod_{\bar{i} \in W^*_{\bar{i}^0}}\chi_{S_{u_{\bar{i}}}} \left( y_{1, i_1}, \ldots, y_{k, i_k} \right) d\mu_{\bar{1}^k}.\nonumber
\end{align}
Hence for the original integral (\ref{eq5}) we have
\begin{align*}
&\int_F \prod_{\bar{i} \in [m]^k}\chi_{S_{u_{\bar{i}}}}(y_{1, i_1}, \ldots, y_{k, i_k})d\mu_{m \cdot \bar{1}^k}(\bar{y}_1,\ldots,\bar{y}_k) \\
&\geq \delta^l \int_{F} \prod_{\bar{i} \in [m]^k \setminus \{ \bar{i}^0\}} \chi_{S_{u_{\bar{i}}}}(y_{1, i_1}, \ldots, y_{k, i_k})d\mu_{m \cdot \bar{1}^k}(\bar{y}_1,\ldots,\bar{y}_k),
\end{align*}
and the tuple $\bar{i}^0$ no longer appears in the product. Iterating this process once for each tuple $\bar{i}^0 \in [m]^k$, we see that 
\[\int_F \prod_{\bar{i} \in [m]^k}\chi_{S_{u_{\bar{i}}}}d\mu_{m \cdot \bar{1}^k}(\bar{y}_1,\ldots,\bar{y}_k)\geq \delta^{lm^k}\mu_{m \cdot \bar{1}^k}(F)>0.\]

This concludes the proof of Propositions \ref{prop: finite VCk-dim implies approx} and \ref{prop: finite VCk-dim implies approx bounded}.
\end{proof}

\subsection{A positive measure set of approximations} 
Next we will strengthen the conclusion of Proposition \ref{prop: finite VCk-dim implies approx} from ``there exists an approximation'' to ``there exists a positive measure set of approximations'', in the following sense.

\begin{definition} \label{def: best approx of fiber}
Fix some $x_1, \ldots, x_l \in V_{k+1}$. 
\begin{enumerate}
 
\item Given $t \in \mathbb{N}$, $\bar{w} \in V^{\bar{m}}$ for some $\bar{m} \in \mathbb{N}^{k}$ and $x \in V_{k+1}$, let us denote by $f^t_{\bar{w},x}$ the best $\norm{\cdot}_{L^2}$-approximation to $f_x$ using a simple function relative to the Boolean algebra $\B^{t,x}_{\bar{w}}$  generated by 
$$\mathcal{F}^{f,t, (x_1, \ldots, x_l, x)}_{\bar{w}} \cup \left\{f^{<q}_{x_i} : i \in [l], q \in \mathbb{Q}^{[0,1]}_t \right\}$$
(see Definition \ref{def: fiber algebra}).

\item For $s \in \mathbb{N}$, we also denote by $f^{t,s}_{\bar{w},x}$  the best  $\norm{\cdot}_{L^2}$-approximation to $f_x$ by a simple function with respect to $\B^{t,x}_{\bar{w}}$ and with all coefficients in $\mathbb{Q}^{[0,1]}_s$.

	\item For $\varepsilon \in \mathbb{R}_{\geq 0}$, we say that $f_x$ is \emph{$\varepsilon$-nicely approximated} (with respect to $x_1, \ldots, x_l$) if there exist some $t \in \mathbb{N}, \bar{m} \in \mathbb{N}^k$ such that the set of tuples $\bar{w} \in V^{\bar{m}}$ with $||f_x-f^t_{\bar{ w},x}||_{L^2}\leq\varepsilon$ has positive $\mu_{\bar{m}}$-measure (this set is measurable by Fubini property in graded probability spaces, see the proof of Proposition \ref{prop: large proj k k plus 1} for the details).
\end{enumerate}
\end{definition}

\begin{lemma}\label{lem: pos meas approx}
  Suppose that $\mathfrak{P} = (V_{[k+1]}, \B_{\bar{n}}, \mu_{\bar{n}})_{n \in \mathbb{N}^{k+1}}$ is a $(k+1)$-partite graded probability space, $f: V^{\bar{1}^{k+1}} \to [0,1]$ is $\B_{\bar{1}^{k+1}}$-measurable, $\VC_k(f) = \bar{d} < \infty$ and $\varepsilon \in \mathbb{R}_{>0}$.
 Then there exist some $l = l(k, \bar{d},\varepsilon) \in \mathbb{N}$ and $x_1, \ldots, x_l \in V_{k+1}$ such that: for any $x \in V_{k
 +1}$, $f_x$ is $\varepsilon$-nicely approximated with respect to $x_1, \ldots, x_l$.
	
\end{lemma}
\begin{proof}
	Fix $\varepsilon > 0$. By Proposition \ref{prop: finite VCk-dim implies approx}, there exist some $l \in \mathbb{N}$ (we may assume $l = l(k, \bar{d}, \varepsilon)$ by Proposition \ref{prop: finite VCk-dim implies approx bounded}) and $x_1, \ldots, x_l \in V_{k+1}$ such that, for every $x \in V_{k+1}$,
\[\norm{f_x-\mathbb{E} \left( f_x \mid\mathcal{B}_{\bar{1}^k, k-1}\cup\{f_{x_i} : i \in [l]\} \right)}_{L^2} \leq  \varepsilon.\]

 Fix some $x \in V_{k+1}$. Note that $f_x$ is trivially $\norm{f_x}_{L^2 \left( \mu_{\bar{1}^k}\right)}$-nicely approximated.  Let 
 $$\delta := \inf \left\{ \delta' \in \mathbb{R}_{\geq 0} : f_x \textrm{ is } \delta'\textrm{-nicely approximated}\right\}.$$

 The function $\mathbb{E}\left( f_x \mid\mathcal{B}_{\bar{1}^k, k-1}\cup\{f_{x_i} : i \in [l]\} \right)$ is the best approximation to $f_x$ from all functions measurable with respect to the given $\sigma$-algebra.   We need to find an analogous function, which we will call $h$, which is the best approximation to $f_x$ with respect to the same $\sigma$-algebra among those approximations which can be obtained for positive measure of parameters $\bar w$.  This is not actually a projection on a $\sigma$-algebra, so we cannot use the standard result to show that $h$ exists, but the proof is essentially the same: first we show that any two near optimal approximations of positive measure must be close to each other, and then we use this to construct a Cauchy sequence converging to $h$.  This is the content of the two claims that follow.

\begin{claim}\label{cla: pairs approx}
	For every $\gamma>0$, there is a $0 < \theta = \theta(\gamma) < \gamma$ so that:  whenever $t_0,t_1 \in \mathbb{N}, \bar{m}_0, \bar{m}_1 \in \mathbb{N}^k$ and $\delta_0,\delta_1<\delta+\theta$, the set
	$$\left\{\bar{w}_0 \oplus \bar{w}_1 \in V^{\bar{m}_0 + \bar{m}_1} : \norm{f^{t_0}_{\bar{w}_0,x}-f^{t_1}_{\bar{w}_1,x}}_{L^2} > \gamma \land \bigwedge_{i=0,1} \norm{f_x-f^{t_i}_{ \bar{w}_i,x}}_{L^2} \leq \delta_i  \right\}$$
	
	is in $\B_{\bar{m}_0 + \bar{m}_1}$ and has $\mu_{\bar{m}_0 + \bar{m}_1}$-measure $0$.
      \end{claim}

\begin{claimproof}
Assume that this set has positive measure. For any $\bar{w}_0 \oplus \bar{w}_1$ in it, by the parallelogram rule for the $L^2$-norm we have
\begin{gather*}
	\norm{2 f_x - \left(f^{t_0}_{\bar{w}_0,x} + f^{t_1}_{\bar{w}_1,x} \right)}_{L^2}^2 + \norm{f^{t_1}_{\bar{w}_1,x} - f^{t_0}_{\bar{w}_0,x}}^2_{L^2} =\\
	2 \norm{f_x - f^{t_0}_{\bar{w}_0,x}}^2_{L^2} + 2 \norm{f_x - f^{t_1}_{\bar{w}_1,x} }^2_{L^2}, \textrm{ hence}\\
	 \norm{2 f_x - \left(f^{t_0}_{\bar{w}_0,x} + f^{t_1}_{\bar{w}_1,x} \right)}_{L^2}^2 \leq 2 (\delta_0^2 + \delta_1^2) - \gamma^2 < 4 (\delta + \theta)^2 - \gamma^2 \leq  4(\delta')^2
	\end{gather*}
	for some $\delta' < \delta$, assuming that $\theta$ is small enough with respect to $\gamma$ and $\delta$. Hence 
	\begin{align*}
	\norm{f_x - \frac{f^{t_0}_{\bar{w}_0,x} + f^{t_1}_{\bar{w}_1,x}}{2}}_{L^2} \leq \delta'.
\end{align*}

As $\frac{f^{t_0}_{\bar{w}_0,x} + f^{t_1}_{\bar{w}_1,x}}{2}$ is a $\B^{\max\{t_0, t_1\}, x}_{\bar{w}_0 \oplus \bar{w}_1}$-simple function, we have 
$$\norm{f_x-f^{\max\{t_0, t_1\}}_{\bar w_0 \oplus \bar w_1,x}}_{L^2} \leq \delta' < \delta$$
 for a positive $\mu_{\bar{m}_0 + \bar{m}_1}$-measure set of $\bar{w}_0 \oplus \bar{w}_1$, contradicting the choice of $\delta$. 
\end{claimproof}

This allows us to choose the ``best positive measure approximation'' of $f_x$, in the following sense.

\begin{claim}\label{cla: choice of h}
There exists a $\sigma \left(\mathcal{F}^{f, (x_1, \ldots, x_l, x)} \cup \{f^{<q}_{x_i} \}_{i \in [l], q \in \mathbb{Q}^{[0,1]}_{\infty}} \right)$-measurable function $h$ such that  $\norm{f_x - h}_{L^2} = \delta$ and for any $\sigma > 0$ there is some $t \in \mathbb{N}, \bar{m} \in \mathbb{N}^k$ so that the set $\left\{ \bar{w} \in V^{\bar{m}} : \norm{h - f^t_{\bar{w},x}}_{L^2} \leq \sigma \right\} \in \B_{\bar{m}}$ has positive $\mu_{\bar{m}}$-measure.
\end{claim}

\begin{claimproof}
	
Given $n \in \mathbb{N}$, let $\gamma_n := \frac{1}{n}$, and let $\theta_n > 0$ be given by Claim \ref{cla: pairs approx} for $\gamma_n$. By the choice of $\delta$, there exists some $t_n \in \mathbb{N}, \bar{m}_n \in \mathbb{N}^k$ such that the set 
\begin{align}
S_{n} := \{ \bar{w} \in V^{\bar{m}_n} : \norm{f_x - f^{t_n}_{\bar{w},x}}_{L^2} \leq \delta + \theta_n\} \in \B_{\bar{m}_n}\label{eq: Sn def}
\end{align}
has positive $\mu_{\bar{m}_n}$-measure.

By induction on $r \in \mathbb{N}$ we choose sets $S_{n}^{r} \in \B_{\bar{m}_n}, n \in \mathbb{N}$ and tuples $\bar{w}_r \in S_r^r$ satisfying the following:
\begin{align}
	&S_{n}^{r'} \subseteq S_{n}^{r} \subseteq S_n \textrm{ and } \mu_{\bar{m}_n}(S^{r}_n) = \mu_{\bar{m}_n}(S_n) > 0 \textrm{ for all } r' \geq r, n \in \mathbb{N};\label{eq: Sn decreasing}\\
	&S^{r'}_n = S^r_n \textrm{ for all } r,r' \leq n' \in \mathbb{N}; \nonumber\\
	&\norm{f^{t_r}_{\bar{w}_r,x} - f^{t_r}_{ \bar{w},x}}_{L^2} \leq \gamma_r \textrm{ for all } \bar{w} \in S_{n}^{r+1} \textrm{ with } n \geq r \in \mathbb{N}.\label{eq: Cauchy}
\end{align}
Let $S_n^1 := S_n$, then all the conditions are trivially satisfied. Now assume $\bar{w}_1, \ldots, \bar{w}_{r-1}$ and $(S^{r}_n)_{n \in \mathbb{N}}$ satisfying these conditions are given.

For each $n \in \mathbb{N}$, let 
$$T_{r}^{n} := \left\{\bar{w} \in S^r_r : \mu_{\bar{m}_{n}}\left( \left\{ \bar{w}' \in S_{n}^r : \norm{f^{t_r}_{\bar{w},x} -  f^{t_r}_{\bar{w}',x}}_{L^2} > \gamma_r \right\} \right) > 0\right\} \in \B_{\bar{m}_r}.$$

By Claim \ref{cla: pairs approx} and Fubini property (using (\ref{eq: Sn def}) and (\ref{eq: Sn decreasing})), $\mu_{\bar{m}_r}(T_r^{n}) = 0$ for any $n \geq r$.
Let $\tilde{S}_r^r := S_r^r \setminus \bigcup_{n \geq r} T_{r}^{n} $, then $\mu_{\bar{m}_r}(\tilde{S}_r^r) >0$. Let $\bar{w}_r$ be an arbitrary element in $\tilde{S}^r_r$.  For each $n > r$ let 
\begin{align*}
& S_{n}^{r+1} := S_{n}^{r} \cap \left\{ \bar{w}' \in S_{n}^r : \norm{f^{t_r}_{\bar{w}_r,x} -  f^{t_r}_{\bar{w}',x}}_{L^2} \leq  \gamma_r \right\}.
\end{align*}
Then $\mu_{\bar{m}_n} \left( S^{r+1}_n \right) = \mu_{\bar{m}_n} \left(S^{r}_n \right)$ for all 
$n \geq r$, by the choice of $t_r,\bar{w}_r$ and (\ref{eq: Cauchy}) is satisfied, concluding the construction.

The sequence $\left( f^{t_{r}}_{\bar{w}_r,x} \right)_{r \in \mathbb{N}}$ is Cauchy in the space 
$$L^2 \bigg(\sigma \left(\mathcal{F}^{f, (x_1, \ldots, x_l, x)} \cup \{f^{<q}_{x_i} \}_{i \in [l], q \in \mathbb{Q}^{[0,1]}_{\infty}} \right) \bigg)$$
 since for any $r_0, r_1 \geq r$ we have $\norm{f^{t_{r_0}}_{\bar{w}_{r_0},x} - f^{t_{r_1}}_{ \bar{w}_{r_1},x}}_{L^2} \leq \gamma_r$ by (\ref{eq: Sn decreasing}) and (\ref{eq: Cauchy}), and $\gamma_r \to 0$. By completeness of the $L^2$-space it has a limit which we denote by $h$.

For an arbitrary $\sigma >0$, let $r \in \mathbb{N}$ be such that $\gamma_r \leq \frac{\sigma}{2}$ and $\norm{h - f^{t_r}_{\bar{w}_{r},x}}_{L^2} \leq \frac{\sigma}{2}$. By (\ref{eq: Sn decreasing}) and (\ref{eq: Cauchy}), the set of $\bar{w} \in V^{\bar{m}_r}$ such that $\norm{ f^{t_r}_{ \bar{w}_{r},x} - f^{t_r}_{\bar{w},x}}_{L^2} \leq \frac{\sigma}{2}$ has positive $\mu_{\bar{m}_r}$-measure, and is contained in the set of $\bar{w} \in V^{\bar{m}_r}$ such that $\norm{ h - f^{t_r}_{\bar{w},x}}_{L^2} \leq \sigma$. Similarly we have $\norm{f_x - h}_{L^2} \leq \norm{f^{t_r}_{\bar{w}_r,x} - f_x}_{L^2} + \norm{h - f^{t_r}_{\bar{w}_r,x}}_{L^2} \leq (\delta + \theta_{r}) + \sigma$ by (\ref{eq: Sn def}), and as $\theta_r \to 0$ and $\sigma$ can be chosen arbitrarily small, we conclude that $\norm{f_x - h}_{L^2} \leq \delta$.
\end{claimproof}

 We will now show that $\delta\leq\varepsilon$.  Towards a contradiction, suppose that $\delta>\varepsilon$.  Then 
 \begin{gather*}
   \E \left(f_x \mid \B_{\bar{1}^k,k-1} \cup \{f_{x_i}\}_{i \in [l]} \right) \neq h \textrm{, hence}\\
 \norm{\E \left(f_x - h \mid \B_{\bar{1}^k,k-1} \cup \{f_{x_i}\}_{i \in [l]} \right)}_{L^2} = \\
  \norm{\E \left(f_x \mid \B_{\bar{1}^k,k-1} \cup \{f_{x_i}\}_{i \in [l]} \right) - h}_{L^2} > 0.
 \end{gather*}
 
So $f_x-h$ is non-orthogonal to $\B_{\bar{1}^k,k-1} \cup \{f_{x_i}\}_{i \in [l]}$, hence for some $\sigma \left( \left\{ f_{x_i} \right\}_{i \in [l]}\right)$-measurable function $g$ we have 
$$\norm{\E \left(g \cdot (f_x-h)\mid\mathcal{B}_{\bar{1}^k,k-1} \right)}_{L^2}>0.$$
Naturally, we will use this to show that we can find a positive measure set of parameters which give a strictly better approximation to $f_x$, contradicting the choice of $\delta$.

We know that there is some $\mathcal{B}_{\bar{1}^k,k-1}$-measurable function $u$ so that $\int u_0\cdot g\cdot(f_x-h)\,d\mu_{\bar{1}^k}>0$; in standard arguments about projections, we would then choose a $c$ so that $h-c\cdot u_0\cdot g$ would be a better approximation to $f_x$.  However, to contradict the definition of $\delta$, we cannot take an arbitrary $\mathcal{B}_{\bar{1}^k,k-1}\cup \{f_{x_i}\}_{i \in [l]}$-measurable function $u=u_0\cdot g$ to improve our approximation.

The \emph{Gowers uniformity norms} let us construct $u$ explicitly from $g\cdot(f_x-h)$: since $\norm{\E \left(g \cdot (f_x-h)\mid\mathcal{B}_{\bar{1}^k,k-1} \right)}_{L^2}>0$, the Gowers $U^k$-norm is positive.  This fact is by now standard (e.g. \cite{gowers2001new,MR3583029}), but for completeness, we develop it in the partite setting in Section \ref{sec:gowers}.

By Proposition \ref{prop: Gowers pos iff proj pos} we have
 \begin{gather*}
 \gamma := \norm{g \cdot (f_x-h)}^{2^k}_{U^{\bar{1}^k}} >0
\end{gather*}

The remainder of the proof consists of writing this integral out explicitly, approximating it with functions of the right kind, and doing the calculations to show that this gives us approximations of $f_x$ which contradict the definition of $\delta$.

 If we write out $\norm{g \cdot \left(f_x-h \right)}^{2^k}_{U^{\bar{1}^k}}$ (Definition \ref{def: gowers norms}), we get 
\begin{gather*}
\int \prod_{\alpha \in \{0,1\}^k} \Bigg( g \left( y_1^{\alpha(1)}, \ldots, y_{k}^{\alpha(k)} \right) \cdot \\
\cdot \left(f_x - h \right)\left( y_1^{\alpha(1)}, \ldots, y_{k}^{\alpha(k)} \right)  \Bigg) d\mu_{2 \cdot \bar{1}^k}\left(\bar{y}^0 \oplus \bar{y}^1\right) = \gamma.	
\end{gather*}

Let $S$ be the set of all $(y_1^1, \ldots, y_k^1) \in V^{\bar{1}^k}$ such that 
\begin{gather}
\int \prod_{\alpha \in \{0,1\}^k} \Bigg( g \left( y_1^{\alpha(1)}, \ldots, y_{k}^{\alpha(k)} \right) \cdot  \label{eq: inner prod h} \\
\cdot \left(f_x - h \right)\left( y_1^{\alpha(1)}, \ldots, y_{k}^{\alpha(k)} \right) \Bigg) d\mu_{\bar{1}^k}\left(\bar{y}^0\right) \geq  \gamma. \nonumber
\end{gather}

By the Fubini property $S \in \B_{\bar{1}^k}$ and $\mu_{\bar{1}^k}(S) > 0$.
%
%

Let $\sigma = \sigma(\gamma) > 0$ be sufficiently small (see below).
By Claim \ref{cla: choice of h} there exist some $t \in \mathbb{N}, \bar{m} \in \mathbb{N}^k$ and a set $T \in \B_{\bar{m}}$ with $\mu_{\bar{m}}(T) > 0$ so that $\norm{f^t_{\bar{w},x} - h }_{L^2} < \sigma$ for all $\bar{w} \in T$. 

We can also choose a sufficiently large $t_0 \in \mathbb{N}$ so that $\norm{f_x - f'_x}_{L^2} < \sigma$ and $\norm{g - g'}_{L^2} < \sigma$ for some function $f'_x$ that is simple with respect to the Boolean algebra generated by $\left\{ f^{<q}_x : q \in \mathbb{Q}^{[0,1]}_{t_0} \right\}$ and some function $g'$ that is simple with respect to the Boolean algebra generated by $\left\{ f^{<q}_{x_i} : i \in [l], q \in \mathbb{Q}^{[0,1]}_{t_0} \right\}$.

Then for any fixed $\bar{y}^1 \in S$ and any $\bar{w} \in T$, replacing $h$ by $f^t_{ \bar{w},x}$, $f_x$ by $f'_x$ and $g$ by $g'$ in the integral (\ref{eq: inner prod h}) we get (assuming $\sigma$ is small enough with respect to $\gamma$)
$$\int \prod_{\alpha \in \{0,1\}^k} g' \left( y_1^{\alpha(1)}, \ldots, y_{k}^{\alpha(k)} \right) \cdot (f'_x - f^t_{\bar{w},x})\left( y_1^{\alpha(1)}, \ldots, y_{k}^{\alpha(k)} \right) d\mu_{\bar{1}^k}\left(\bar{y}^0\right) \geq  \frac{\gamma}{2}.$$
For $\left( \bar{y}^1, \bar{w} \right) \in S \times T$, let $g'_{\bar{y}^1, \bar{w}}: V^{\bar{1}^k} \to \mathbb{R}$ be the function defined by 
\begin{gather*}
g'_{\bar{y}^1, \bar{w}} \left(y^0_1, \ldots, y^0_k \right) := 
\end{gather*}

$$\prod_{\alpha \in \{0,1\}^k} g' \left( y_1^{\alpha(1)}, \ldots, y_{k}^{\alpha(k)} \right) \cdot   \prod_{\alpha \in \{ 0,1\}^k \setminus \{ (0, \ldots, 0)\}}  (f'_x - f^t_{\bar{w},x})\left( y_1^{\alpha(1)}, \ldots, y_{k}^{\alpha(k)} \right).$$
Then
\begin{gather}
\int g'_{\bar{y}^1, \bar{w}}\left(\bar{y}^0 \right) \cdot \left( f'_x - f^t_{\bar{w},x} \right)\left(\bar{y}^0\right) d\mu_{\bar{1}^k}\left( \bar{y}^0 \right) \geq \frac{\gamma}{2} \textrm{, or in other words}\nonumber\\
\left \langle g'_{\bar{y}^1, \bar{w}}, f'_x - f^t_{\bar{w},x} \right \rangle_{L^2} \geq \frac{\gamma}{2}.\label{eq: inner prod}
\end{gather}
By the choice of $g'$ and $f'_x$, we have that $g'_{\bar{y}^1, \bar{w}}$ is a $\B^{t', x}_{\bar{y}^1 \oplus \bar{w}}$-simple function for $t' := \max\{t,t_0\}$, and also $\norm{g'_{\bar{y}^1, \bar{w}}}_{L^2} \leq 1$.

Taking the orthogonal projection of $f'_x - f^t_{\bar{w},x}$ in the Hilbert space $L^2 \left(\B_{\bar{1}^k,k-1} \right)$ onto the closed subspace generated by $g'_{\bar{y}^1, \bar{w}}$, we can write 
$$f'_x - f^t_{\bar{w},x} = u + v$$
for some $\B^{t',x}_{\bar{y}^1 \oplus \bar{w}}$-simple function $u$ and some  $\B_{\bar{1}^k,k-1}$-measurable function $v$ orthogonal to this subspace.
Note that 
\begin{gather*}
\norm{f'_x - f^t_{x, \bar{w}}}_{L_2} \leq \norm{f'_x - h}_{L^2} + \norm{h - f^t_{\bar{w},x}}_{L^2} \leq \delta + \sigma \textrm{ and }\\
 u = \left\langle f'_x - f^t_{\bar{w},x}, \frac{g'_{\bar{y}^1, \bar{w}}}{\norm{g'_{\bar{y}^1, \bar{w}}}_{L^2}} \right\rangle_{L^2} \cdot \frac{g'_{\bar{y}^1, \bar{w}}}{\norm{g'_{\bar{y}^1, \bar{w}}}_{L^2}}.
 \end{gather*}
 Hence, using \eqref{eq: inner prod},
 \begin{gather*}
\norm{v}^2_{L^2} = \norm{f'_x - f^t_{\bar{w},x}}^2_{L^2} - \frac{\left\langle f'_x - f^t_{\bar{w},x}, g'_{\bar{y}^1, \bar{w}} \right\rangle^2_{L^2}}{\norm{g'_{\bar{y}^1, \bar{w}}}^2_{L^2}} \leq \\
(\delta + \sigma)^2 - \frac{\gamma^2}{4 \norm{g'_{\bar{y}^1, \bar{w}}}^2_{L^2}} \leq (\delta + \sigma)^2 - \frac{\gamma^2}{4}.
\end{gather*}
Hence, assuming $\sigma$ is small enough with respect to $\gamma$, there is some $\delta' = \delta'(\gamma) < \delta - \sigma$ so that $\norm{v}_{L^2} \leq \delta'$. Observe that $f'_x - (f^t_{\bar{w},x} + u) = v$ and $f^t_{\bar{w},x} + u$ is a  $\B^{t',x}_{\bar{y}^1 \oplus \bar{w}}$-simple function.
Thus for any $(\bar{y}^1, \bar{w}) \in S \times T$ we have 
$$\norm{f_x - f^{t'}_{\bar{y}^1 \oplus \bar{w},x}}_{L^2}  \leq \norm{f'_x - f^{t'}_{\bar{y}^1 \oplus \bar{w},x}}_{L^2} + \sigma \leq \delta' + \sigma < \delta,$$
and $\mu_{\bar{1}^k+\bar{m}}(S \times T) > 0$. This contradicts the choice of $\delta$.
\end{proof}

\section{Main Theorem}

\subsection{Proof of the $(k+1,k)$-case}

\begin{prop}\label{prop: large proj k k plus 1}
  Suppose that $\mathfrak{P} = (V_{[k+1]}, \B_{\bar{n}}, \mu_{\bar{n}})_{n \in \mathbb{N}^{k+1}}$ is a $(k+1)$-partite graded probability space, $f: V^{\bar{1}^{k+1}} \to [0,1]$ is a $(k+1)$-ary  $\B_{\bar{1}^{k+1}}$-measurable function and $\VC_k(f)  < \infty$.
 Then $f$ is $\mathcal{B}_{\bar{1}^{k+1},k}$-measurable. 

  More precisely, for every $\varepsilon > 0$ there exist some 
  $N \in \mathbb{N}$, $\gamma_i \in \mathbb{Q}^{[0,1]}$ for $i \in [N]$ and $\B_{I}(f)$-measurable (see Definition \ref{def: B(f) all smaller fibers})  $(\leq k)$-ary functions
  $f^i_I: \prod_{j \in I} V_{j} \to [0,1]$ for $i \in [N], I \in  \binom{[k+1]}{ \leq k} $ so that, defining $g: V^{\bar{1}^{k+1}} \to [0,1]$ via
  \begin{gather*}
 	g(\bar{x}) :=  \sum_{i \in [N] } \gamma_i \cdot \prod_{I \in \binom{[k+1]}{\leq k}} f^i_I(\bar{x}_I),
 \end{gather*}
 we have $\norm{f - g}_{L^2} < \varepsilon$.
\end{prop}

\begin{remark}\label{rem: k+1,k explicit}
Furthermore, $N$ can be bounded depending only on $\VC_k(f)$ and $\varepsilon$ (this will be established as part of the more general Corollary \ref{thm: the very main thm}).
\end{remark}

The main idea of the proof is not so complicated.  By Proposition \ref{prop: finite VCk-dim implies approx}, there are $x_1,\ldots,x_N$ so that, for every $x \in V_{k+1}$,
$$\norm{f_{x} - \E\left(f_x \mid \mathcal{B}_{\bar{1}^k, k-1} \cup \{f_{x_1}, \ldots, f_{x_N} \} \right)}_{L^2} < \varepsilon.$$
Now $\E\left(f_x \mid \mathcal{B}_{\bar{1}^k, k-1} \cup \{f_{x_1}, \ldots, f_{x_N} \}\right)$ can be approximated by a finite sum of the form
\[\E\left(f_x \mid \mathcal{B}_{\bar{1}^k, k-1} \cup \{f_{x_1}, \ldots, f_{x_N} \}\right)(\bar x)\approx \sum_{i\leq N}\gamma_i f_{x_i}(\bar x)\prod_{I\in{[k]\choose \leq k-1}}\chi_{C_{i,I,x}}(\bar x_I),\]
where  each $C_{i,I,x}$ is some set from $\mathcal{B}_{\bar{1}^k, I}$.
By countable additivity, outside of a set of measure $<\varepsilon$, $N$ can be bounded uniformly in $x$.  (In fact, by Proposition \ref{prop: finite VCk-dim implies approx bounded}, $N$ can be bounded uniformly in all $x$.)

We can combine these representations for different $x \in V_{k+1}$ by replacing the sets $C_{i,I,x}$ with the set $C_{i,I}=\{(\bar x,x)\mid \bar x_{I} \in C_{i,I,x}\}$, obtaining (with some rearranging of terms) a single function
\[g=\sum_{i\leq N}\gamma_i f_{x_i}(\bar x)\prod_{I\in{[k]\choose \leq k-1}}\chi_{C_{i,I}}(\bar x_I,x).\]
By its form, $g$ is $\mathcal{B}_{\bar{1}^{k+1},k}$-measurable, and for almost every $x \in V_{k+1}$, $\norm{f_x-g_x}_{L^2}$ is small, so $\norm{f-g}_{L^2}$ is small as well.

There is one complication: just because each of the sets $C_{i,I,x}$ are measurable, it does not follow that the set $C_{i,I}$ is also measurable.  Therefore to carry this argument out correctly, we need to write out this cylinder sets in a way that is sufficiently uniform in $x$ (as, more or less, combinations of level sets of fibers of $f$) to guarantee that $C_{i,I}$ is measurable, and use some averaging arguments relying on Lemma \ref{lem: pos meas approx}.

To make this explicit, we define a slight variant of the algebras associated to fibers of a function considered earlier.
\begin{definition}\label{def: B(f) all smaller fibers}
	Let $r \in \mathbb{N}$, $(V_{[r]},\B_{\bar{n}}, \mu_{\bar{n}})_{\bar{n} \in \mathbb{N}^{r}}$ be an $r$-partite graded probability space, $f: V^{\bar{1}^{r}}\to [0,1]$ a $\B_{\bar{1}^{r}}$-measurable function and $\bar{w}_1, \ldots, \bar{w}_{\ell} \in V^{\bar{1}^r}$. Let $t \in \mathbb{N}$, $I \subseteq [r]$, and let $\bar{n}_{I} \in \mathbb{N}^{r}$ be defined by $\bar{n}_{I}  := \sum_{i \in I} \bar{\delta}_i$.
	We let $\B^t_{I, (\bar{w}_1, \ldots, \bar{w}_{\ell})}(f)$ be the finite Boolean subalgebra of $\B_{\bar{n}_I}$ generated by all subsets of $V^{\bar{n}_I} = \prod_{i \in I} V_i$ of the form
	\begin{gather*}
		\left\{\bar{x} \in V^{\bar{n}_I} :  \bar{x} \oplus \left(\bar{w}_i \right)_{[r] \setminus I} \in f^{<q}\right\}
	\end{gather*}
	for some $i \in [\ell]$ and $q \in \mathbb{Q}^{[0,1]}_t$.  
		We let $\B^t_{I}(f)$ be the $\sigma$-subalgebra of $\B_{\bar{n}_I}$ generated by $\bigcup_{\bar{w} \in V^{\bar{1}^r}}B^t_{I,\bar{w}}$, and $\B_{I}(f)$ the $\sigma$-subalgebra of $\B_{\bar{n}_I}$ generated by $\bigcup_{t \in \mathbb{N}}\B^t_{I}(f)$.
\end{definition}
We are ready to prove Proposition \ref{prop: large proj k k plus 1}, in the explicit form stated in Remark \ref{rem: k+1,k explicit}.
\begin{proof}[Proof of Proposition \ref{prop: large proj k k plus 1}]
  Let $\varepsilon \in \mathbb{R}_{>0}$ be given.
  
We fix $\delta \in \mathbb{Q}_{>0}, t \in \mathbb{N}, \bar{m} = (m_1, \ldots, m_k) \in \mathbb{N}^{k}$, to be determined later. 

By Lemma \ref{lem: pos meas approx} there exist some $x_1, \ldots, x_l \in V_{k+1}$ such that for any $x \in V_{k+1}$, $f_x$ is $\delta$-nicely approximated with respect to $x_1, \ldots, x_l$. 

Let $0=r_1 < \ldots < r_{L}=1$ list all elements of $\mathbb{Q}^{[0,1]}_t$ in the increasing order. As usual, for $q,q' \in \mathbb{Q}^{[0,1]}$ and $\bowtie \in \{<, \geq, = \}$ we let
	\begin{gather*}
		f^{\bowtie q} := \{\bar{y} \in V^{\bar{1}^{k+1}} : f(\bar{y}) \bowtie q\},\\
		f^{[q,q')} = f^{<q'} \cap f^{\geq q}.
	\end{gather*}	
	Let 
\begin{gather*}
	S := \left \{ s \mid s: [m_1] \times \ldots \times [m_k] \times [l+1] \times  \binom{[k]}{\leq k-1} \to [L] \right \}.
\end{gather*}	
Let $\bar{m}' := \bar{1}^{k \frown }(0) + \bar{m}^{\frown}(0) + \bar{\delta}_{k+1}  \in \mathbb{N}^{k+1}$. For $s \in S$ let
\begin{gather*}
	A^{s}: V^{\bar{m}'} \to \{0,1\},\\
	A^{s} \left( \bar{y},\bar{w},x \right) := \prod_{\substack{(i_1, \ldots, i_k,j) \in [m_1] \times \ldots \times [m_k] \times [l]\\ I \in \binom{[k]}{\leq k-1}}}  \chi_{f^{ \left[r_{s(\bar{i},j,I)}, r_{s(\bar{i},j,I)+1} \right)}_{x_j} }\left( \bar{y}_{w_{t,i_t} \to t, t \in I} \right) \cdot \\
 \prod_{\substack{(i_1, \ldots, i_k) \in [m_1] \times \ldots \times [m_k] \\ I \in \binom{[k]}{\leq k-1} \setminus \{ \emptyset \} }} \chi_{f^{\left[r_{s(\bar{i},l+1,I)}, r_{s(\bar{i},l+1,I)+1} \right)}} \left(\bar{y}_{w_{t,i_t} \to t, t \in I}, x\right).
\end{gather*}

	By definition (see Definition \ref{def: best approx of fiber}), for every $\bar{w} \in V^{\bar{m}}$ and $x \in V_{k+1}$, every atom of the algebra $\B^{t,x}_{\bar{w}}$  has characteristic function of the form $A^{s}(-,\bar{w},x)$
for some $s \in S$ (some of the atoms may be repeated in this presentation).

For $\bar{\alpha} = (\alpha_{s} \in \mathbb{Q}^{[0,1]}_t : s \in S) \in Q := \left( \mathbb{Q}^{[0,1]}_t  \right)^S$, we consider the function
\begin{gather*}
f^{t}_{\bar{\alpha}}\to [0,1],\\
	f^{t}_{\bar{\alpha}}(\bar{y}, \bar{w},x) := \sum_{s \in S} \alpha_{s} A^{s}(\bar{y},\bar{w},x).
\end{gather*}

Then $f^{t}_{\bar{\alpha}}$ is $\B_{\bar{m}'}$-measurable, and every $\B^{t,x}_{\bar{w}}$-simple function with coefficients in $\mathbb{Q}_t^{[0,1]}$  is of the form $f^{t}_{\bar{\alpha}}(-,\bar{w},x)$ for some $\bar{\alpha} \in Q$.

Recall (Definition \ref{def: best approx of fiber}(2)) that  $f^{t,t}_{\bar{w},x}$ denotes the best $L^2$-approximation to $f_x$ using a $\B^{t,x}_{\bar{w}}$-simple function with coefficients in $\mathbb{Q}_t^{[0,1]}$. We can define it explicitly as follows. Let $\triangleleft$ be an arbitrary well order on $Q$. For $\bar{\alpha}, \bar{\beta} \in Q$, let
\begin{gather*}
	C_{\bar{\alpha}, \bar{\beta}} := 
\left\{ (\bar{w},x) \in V^{\bar{m}^{\frown}(1)} : \norm{f_x - f^t_{\bar{\beta}}(-,\bar{w},x)}_{L^2}  < \norm{f_x - f^t_{\bar{\alpha}}(-,\bar{w},x)}_{L^2} \right\},\\
C_{\bar{\alpha}} := \bigcap_{\bar{\beta} \in Q} \left( V^{\bar{m}^{\frown}(1)} \setminus C_{\bar{\alpha}, \bar{\beta}}\right),\\
D_{\bar{\alpha}} := C_{\bar{\alpha}} \setminus \left( \bigcup_{\bar{\beta} \triangleleft \bar{\alpha}}  C_{\bar{\beta}}\right).
\end{gather*}
Note that $C_{\bar{\alpha}, \bar{\beta}} \in \B_{\bar{m}^{\frown}(1)}$ (as $\norm{x}_{L^2} = \left(\int{x^2}\right)^{\frac{1}{2}}$ is a composition of functions preserving measurability using Fubini). So $(\bar{w},x) \in C_{\bar{\alpha}}$ if and only if $\norm{f_x - f^{t}_{\bar{\alpha}}(-,\bar{w},x)}_{L^2}$ is minimal among all $\bar{\beta} \in Q$. As there can be multiple $\bar{\alpha} \in Q$ that give equally good approximations, we let $D_{\bar{\alpha}}$ consist of those $(\bar{w},x)$ for which $\bar{\alpha}$ is $\triangleleft$-minimal giving the best approximation. Then $\left\{ D_{\bar{\alpha}} : \bar{\alpha} \in Q \right\}$ forms a partition of $V^{\bar{m}^{\frown}(1)}$, and we define a function $h: V^{\bar{m}'} \to [0,1]$ via
\begin{gather*}
	h(\bar{y}, \bar{w}, x) := \sum_{\bar{\alpha} \in Q} \chi_{D_{\bar{\alpha}}}(\bar{w},x) \cdot  f^{t}_{\bar{\alpha}}(\bar{y}, \bar{w},x).
\end{gather*}

From the definition we see that $h$ is $\B_{\bar{m}'}$-measurable and for every fixed $(\bar{w},x)$, $h(-,\bar{w},x) = f^{t,t}_{\bar{w},x}$.

For $\bar{m} \in \mathbb{N}^k$ and $t,s \in \mathbb{N}$, let 
\begin{gather*}
G_{t,\bar{m}} := \left\{ (\bar{w},x) \in V^{\bar{m}^{\frown}(1)} : \norm{f_x - f^t_{\bar{w},x}}_{L^2} < \delta \right\},\\
G_{t,s, \bar{m}} := \left\{ (\bar{w},x) \in V^{\bar{m}^{\frown}(1)} : \norm{f_x - f^{t,s}_{\bar{w},x}}_{L^2} < 2 \delta \right\}.
\end{gather*}
As every $\B^{t,x}_{\bar{w}}$-simple function can be approximated up to $L^2$-distance $\delta$ by some $\B^{t,x}_{\bar{w}}$-simple function with coefficients in $\mathbb{Q}^{[0,1]}_s$ assuming $s \in \mathbb{N}$ is large enough, we have 
\begin{gather}\label{eq: approx for func 10}
	G_{t,\bar{m}} = \bigcup_{s \in \mathbb{N}} G_{t,s, \bar{m}}.
\end{gather}

Also, for $\rho \in \mathbb{Q}_{>0}$, let 
\begin{gather*}
		G_{t, \bar{m}, \rho} := \left\{x \in V_{k+1} : \mu_{\bar{m}} \left( \left(G_{t,\bar{m}}\right)_x \right) \geq \rho \right\},\\
		G_{t,s, \bar{m}, \rho} := \left\{x \in V_{k+1} : \mu_{\bar{m}} \left( \left(G_{t,s,\bar{m}}\right)_x \right) \geq \rho \right\}.
\end{gather*}
Then $G_{t, \bar{m}, \rho}, G_{t,s, \bar{m}, \rho} \in \B_{\bar{\delta}_{k+1}}$ by Fubini. And by the choice of $x_1, \ldots, x_l$ we have that $V_{k+1}$ is covered by the sets $\left\{G_{t,\bar{m}, \rho} : t \in \mathbb{N}, \bar{m} \in \mathbb{N}^k, \rho \in \mathbb{Q}_{>0} \right\}$, hence also covered by 
the sets $\left\{G_{t,s,\bar{m}, \rho} : t,s \in \mathbb{N}, \bar{m} \in \mathbb{N}^k, \rho \in \mathbb{Q}_{>0} \right\}$ by \eqref{eq: approx for func 10}.

Hence, by countable additivity of the measure (noting that $t \leq t' \land s \leq s' \land \bar{m} \leq \bar{m}' \land \rho  \geq \rho'$ implies $G_{t,s,\bar{m}, \gamma} \subseteq G_{t',s',\bar{m}', \gamma'}$), we can choose some $t \in \mathbb{N}, \bar{m} \in \mathbb{N}^k$ and $\rho \in \mathbb{Q}_{> 0}$ so that 
\begin{gather}\label{eq: approx for func 9.5}
	\mu_{\bar{\delta}_{k+1}}(G_{t,t,\bar{m}, \rho}) \geq 1 - \delta.
\end{gather}

We define 
\begin{gather*}
H := \left\{ (\bar{w},x) \in V^{\bar{m}^{\frown}(1)} : \norm{f_x - h(-, \bar{w},x)}_{L^2} < 2\delta \right\}.
\end{gather*}

We also define a $\B_{\bar{1}^{k+1}}$-measurable (by Fubini) function $g': V^{\bar{1}^{k+1}} \to [0,1]$ via
\begin{gather*}
	g'(\bar{y}, x) := \frac{1}{\max \left\{ \rho, \mu_{\bar{m}}(H_x) \right\}} \int h(\bar{y},\bar{w},x) \cdot \chi_{H}(\bar{w},x) d\mu_{\bar{m}}(\bar{w}).
\end{gather*}

As $h(-,\bar{w},x) = f^{t,t}_{\bar{w},x}$ for every fixed $(\bar{w},x) \in V^{\bar{m}^{\frown}(1)}$, we have $\mu_{\bar{\delta}_{k+1}} \left( Z \right) \geq 1 - \delta$ for 
$$Z := \left\{ x \in V_{k+1} : \mu_{\bar{m}} \left(H_x\right) \geq \rho  \right\}.$$
Note that $Z \in \B_{\bar{\delta}_{k+1}}$ by Fubini. Now, for any $x \in Z$ and $\bar{w} \in H_x$,  $\norm{f_x - h(-, \bar{w},x)}_{L^2} < 2\delta$ by definition of $H$. And for every fixed $x \in V_{k+1}$ with $\mu_{\bar{m}}(H_x) > 0$ we have 
$$f_x(\bar{y}) = \frac{1}{\mu_{\bar{m}}(H_x)}\int f_x(\bar{y}) \cdot \chi_{H}(\bar{w},x) d\mu_{\bar{m}} (\bar{w}) $$ 
for all $\bar{y} \in V^{\bar{1}^k}$. Then, by Lemma \ref{lem: operations for k+1,k proof}\eqref{lem: fibers approx impl av approx}, averaging over $\bar{w} \in H_x$, we get
\begin{gather}\label{eq: approx for func 10.5}
	\norm{f_x - g'(-,x)}_{L^2 \left( \mu_{\bar{1}^k} \right)} \leq 2\delta \textrm{ for every fixed } x \in Z.
\end{gather}

But then, as $\mu \left( Z \right) \geq 1 - \delta$ by \eqref{eq: approx for func 9.5}, using the second implication in Lemma \ref{lem: operations for k+1,k proof}\eqref{lemeq: small norm impl small norm for almost all fibers} we get
\begin{gather}\label{eq: measurability expl -1}
	\norm{f - g'}_{L^2} \leq (4 \delta^2)^{\frac{3}{4}}.
\end{gather}

Next we will approximate $g'$  by a function of the required form.

\begin{claim}\label{cla: meas of unary func expl}
	The following functions are $\B_{\{k+1\}}(f)$-measurable.
	\begin{enumerate}
	\item $x \in V_{k+1} \mapsto \chi_{D_{\bar{\alpha}}}({\bar{w}},x)$ for every fixed $\bar{\alpha} \in Q$ and $\bar{w} \in V^{\bar{m}}$;
	\item $x \in V_{k+1} \mapsto  \chi_{H}(\bar{w},x) $ for every fixed $\bar{w} \in V^{\bar{m}}$;
	\item $x \in V_{k+1} \mapsto  \frac{1}{\max \{\rho, \mu_{\bar{m}}(H_x) \} }$.
	\end{enumerate}
\end{claim}
\begin{claimproof}

(1) Let  $\bar{\alpha}$ and $\bar{w}$ be fixed. If we also fix $\bar{y}$, then the function $x \mapsto f_x(\bar{y}) - f^t_{\bar{\alpha}}(\bar{y}, \bar{w},x)$ is clearly $\B_{\{k+1\}}(f)$-measurable. Then $x \mapsto \left(f_x(\bar{y}) - f^t_{\bar{\alpha}}(\bar{y}, \bar{w},x) \right)^2$ is also $\B_{\{k+1\}}(f)$-measurable. Applying Lemma \ref{lem: meas of av fib}, the function 
$$x \mapsto \int \left(f_x(\bar{y}) - f^t_{\bar{\alpha}}(\bar{y}, \bar{w},x)\right)^2 d \mu_{\bar{1}^k}(\bar{y})$$
 is also $\B_{\{k+1\}}(f)$-measurable, and using uniform continuity of $x \mapsto x^{\frac{1}{2}}$ on $[0,1]$, $x \mapsto \norm{ f_x(\bar{y}) - f^t_{\bar{\alpha}}(\bar{y}, \bar{w},x)}_{L^2 \left(\mu_{\bar{1}^k} \right)}$ is also $\B_{\{k+1\}}(f)$-measurable. Following the definition of $D_{\bar{\alpha}}$ and standard arguments, we see that  $x \in V_{k+1} \mapsto \chi_{D_{\bar{\alpha}}}(\bar{w},x)$ is also $\B_{\{k+1\}}(f)$-measurable.

(2),(3) similar unwinding the definitions and using Lemma \ref{lem: meas of av fib} every time integration is applied.
\end{claimproof}

Let $J \in  \binom{[k+1]}{\leq k}$ be arbitrary, and let $\bar{m}_J \in \mathbb{N}^{k+1}$ be given by $\bar{m}_J := \sum_{i \in J}\bar{\delta}_i$. If $J \subseteq [k]$ and $s \in S$, we define $B^{s}_{J}: V^{\bar{m}_J + \bar{m}} \to [0,1]$ via
\begin{gather*}
B^{s}_J(\bar{z} \oplus \bar{w}) := \\
\prod_{\substack{(i_1, \ldots, i_k,j) \in [m_1] \times \ldots \times [m_k] \times [l]}}  \chi_{f^{ \left[r_{s(\bar{i},j,[k]\setminus J)}, r_{s(\bar{i},j,[k]\setminus J)+1} \right)}_{x_j} }\left( \bar{z} \oplus \bar{0}^{k+1}_{w_{t,i_t} \to t, t \in [k] \setminus J} \right).
\end{gather*}

Otherwise, $J = I \sqcup \{k+1\}$ for some $I \in \binom{[k]}{\leq k -1 }$, in particular $[k] \setminus I \neq \emptyset$,  and we define $B^{s}_{J}: V^{\bar{m}_J + \bar{m}} \to [0,1]$ via 
\begin{gather*}
B^{s}_J(\bar{z} \oplus \bar{w}) := \\
 \prod_{\substack{(i_1, \ldots, i_k) \in [m_1] \times \ldots \times [m_k] }} \chi_{f^{\left[r_{s(\bar{i},l+1,[k]\setminus I)}, r_{s(\bar{i},l+1,[k] \setminus I)+1} \right)}} \left(\bar{z} \oplus \bar{0}^{k+1}_{w_{t,i_t} \to t, t \in [k] \setminus I}\right).
\end{gather*}
Comparing to the definition of $A^s$, we see that for every $s \in S$ and $\bar{y} \in V^{\bar{1}^{k}},\bar{w} \in V^{\bar{m}}, x \in V_{k+1}$, taking $\bar{z} :=  \bar{y}^{\frown}(0) + \bar{0}^{k \frown} (x) + \bar{w}$ we have
\begin{gather*}
	A^s(\bar{y}, \bar{w},x) = \prod_{J \in \binom{[k+1]}{\leq k}} B^s_J(\bar{z}_J \oplus \bar{w}).
\end{gather*}

And from the definition, for every $J \in \binom{[k+1]}{\leq k}$ and every fixed $\bar{w} \in V^{\bar{m}}$,
\begin{gather}\label{eq: measurability expl 1}
	\textrm{the function } \bar{z} \in V^{\bar{m}_J} \mapsto B^s_J(\bar{z} \oplus \bar{w}) \in [0,1] \textrm{ is }\B^t_J(f) \textrm{-measurable}
\end{gather}
(see Definition \ref{def: B(f) all smaller fibers}).

Consider the $\sigma$-algebra $\mathcal{C}_1 \subseteq \B_{\bar{1}^{k+1}}$ generated by the collection of sets 
$$\left\{ h^{[q,r)}_{\bar{w}} \cap \chi_{H_{\bar{w}}}^{=t}: \bar{w} \in V^{\bar{m}}, q < r \in \mathbb{Q}^{[0,1]}, t \in \{0,1\} \right\}.$$
 
 For every fixed $\bar{w} \in V^{\bar{m}}$, the function $(\bar{y},x) \mapsto h(\bar{y},\bar{w},x) \cdot \chi_{H}(\bar{w},x)$ is clearly $\mathcal{C}_1$-measurable. Hence, by Lemma \ref{lem: meas of av fib}, the function 
 \begin{gather*}
	h_1: (\bar{y},x) \mapsto \int h(\bar{y},\bar{w},x) \cdot \chi_{H}(\bar{w},x) d\mu_{\bar{m}}(\bar{w})
 \end{gather*}
is $\mathcal{C}_1$-measurable. Then we can approximate it up to $L^2 \left( \mu_{\bar{1}^{k+1}} \right)$-distance $\delta$ by a $\mathcal{C}_1$-simple function 
\begin{gather}\label{eq: measurability expl 2}
\sum_{i } \beta_i \cdot \chi_{h_{\bar{w}_i}^{[q_i, r_i)}}(\bar{y},x) \cdot \chi_{H_{\bar{w}_i}^{=t_i}}(x)
\end{gather}
for some finitely many $\bar{w}_i \in V^{\bar{m}}$, $\beta_i, q_i, r_i \in \mathbb{Q}_{\infty}^{[0,1]}$ and $t_i \in \{0,1\}$.

We consider a single summand, so we fix $\bar{w}$ and $q < r$ and  $t$. By definition of $h$ and $f^t_{\bar{\alpha}}$'s, \begin{gather}\label{eq: measurability expl 3}
	\chi_{h_{\bar{w}}^{[q, r)}}(\bar{y},x) = \sum_{\bar{\alpha} \in Q} \chi_{ \left(D_{\bar{\alpha}} \right)_{\bar{w}} }(x) \cdot  \chi_{\left(f^{t}_{\bar{\alpha}} \right)_{\bar{w}}^{[q,r)} }(\bar{y},x),\\
	\chi_{\left(f^{t}_{\bar{\alpha}} \right)_{\bar{w}}^{[q,r)} }(\bar{y},x) = \sum_{s \in S \land \alpha_s \in [r,s)} \alpha_{s} A^{s}_{\bar{w}}(\bar{y},x), \nonumber\\
	A^{s}_{\bar{w}}(\bar{y},x) = \prod_{J \in \binom{[k+1]}{\leq k}} \left(B^s_J \right)_{\bar{w}}(\left(\bar{y},x \right)_J). \nonumber
\end{gather}
Note that each $(B^s_J)_{\bar{w}}$ is a $\B_J(f)$-measurable $k$-ary function by \eqref{eq: measurability expl 1}, each $\chi_{ \left(D_{\bar{\alpha}} \right)_{\bar{w}} }$ is $\B_{\{k+1\}}(f)$-measurable by Claim \ref{cla: meas of unary func expl}(1) and $\chi_{H^{=t}_{\bar{w}}}$ is $\B_{\{k+1\}}(f)$-measurable by Claim \ref{cla: meas of unary func expl}(2).
Then, replacing each summand in \eqref{eq: measurability expl 2} by a corresponding expression from \eqref{eq: measurability expl 3} and regrouping the sum, we conclude that $h_1$ can be approximated up to $L^2 \left( \mu_{\bar{1}^{k+1}} \right)$-distance $\delta$  by a finite sum of the form
\begin{gather}\label{eq: measurability expl 4}
	\hat{h}_1: \bar{x} \in V^{\bar{1}^{k+1}} \mapsto \sum_{i } \hat{\beta}_i \cdot \prod_{J \in \binom{[k+1]}{\leq k}}\hat{f}^i_J(\bar{x}_J)
\end{gather}
with $\hat{\beta}_i \in \mathbb{Q}^{[0,1]}$ and $\hat{f}^i_J : \prod_{i \in J} V_{i}\to[0,1]$ a $\B_{J}(f)$-measurable $k$-ary function. But then, considering the function $h_2: (\bar{y},x) \in V^{\bar{1}^{k+1}} \mapsto \frac{1}{\max \{\rho, \mu_{\bar{m}}(H_x) \} }$ and applying Lemma \ref{lem: operations for k+1,k proof}\eqref{lem: sum prod of approx}, $\norm{g' - h_2 \cdot \hat{h}_1}_{L_2} = \norm{h_2 \cdot h_1 -  h_2 \cdot \hat{h}_1}_{L_2} \leq 3 \delta$. The map $x \in V_{k+1} \mapsto  \frac{1}{\max \{\rho, \mu_{\bar{m}}(H_x) \} }$ is $\B_{\{k+1\}}(f)$-measurable by Claim \ref{cla: meas of unary func expl}(3), hence multiplying the sum in \eqref{eq: measurability expl 4} by it and regrouping, the product $h_2 \cdot \hat{h}_1$ is of the form
\begin{gather*}
	g: \bar{x} \in V^{\bar{1}^{k+1}} \mapsto \sum_{i } \gamma_i \cdot \prod_{J \in \binom{[k+1]}{\leq k}} f^i_J(\bar{x}_J)
\end{gather*}
for some finitely many $\gamma_i \in \mathbb{Q}^{[0,1]}$ and $\B_{J}(f)$-measurable $k$-ary functions
 $f^i_J : \prod_{i \in J} V_{i}\to[0,1]$. So $g$ is of the required form, and using \eqref{eq: measurability expl -1}
 \begin{gather*}
 	\norm{f - g}_{L^2} \leq \norm{f - g'}_{L^2} + \norm{g' - g}_{L^2} \leq  (4 \delta^2)^{\frac{3}{4}} + 3\delta < \varepsilon
 \end{gather*}
 assuming we started with $\delta$ sufficiently small with respect to $\varepsilon$.  
 \end{proof}

This argument actually gives us an additional uniformity we will need in the next subsection.
 \begin{cor}\label{prop: large proj k k plus 1 uniform}
Suppose that $k' > k \in \mathbb{N}$ and $(V_{[k']}, \B_{\bar{n}}, \mu_{\bar{n}})_{n \in \mathbb{N}^{k'}}$ is a $k'$-partite graded probability space. Let $I' := [k'] \setminus [k+1]$ and $\bar{n}' := \sum_{i \in I'} \bar{\delta}_i \in \mathbb{N}^{k'}$. Suppose $f:V^{\bar{1}^{k'}}\to[0,1]$ is $\B_{\bar{1}^{k'}}$-measurable and, for every $\bar z \in V^{\bar{n}'}$, $\VC_k(f_{\bar z})<\infty$. Then, for every $\varepsilon>0$, there exist some $N \in \mathbb{N}$, $\gamma_i \in \mathbb{Q}^{[0,1]}$ for $i \in [N]$, and $\B_{I \cup I'}(f)$-measurable functions
 $f^i_I:  \left( \prod_{j \in I} V_{j} \right) \times V^{\bar{n}'} \to[0,1] $ for $i \in [N], I \in  \binom{[k+1]}{ \leq k} $ so that, defining $g: V^{\bar{1}^{k'}} \to [0,1]$ via
 \begin{gather*}
	g: \bar{x} \mapsto   \sum_{i \in [N] } \gamma_i \cdot \prod_{I \in \binom{[k+1]}{\leq k}} f^i_I(\bar{x}_I, \bar{x}_{I'}),
\end{gather*}
we have $\norm{f_{\bar{z}} - g_{\bar{z}}}_{L^2 \left( \mu_{\bar{1}^{k+1}} \right)} < \varepsilon$ for all $\bar{z} \in V^{\bar{n}'}$ except for a set of $\mu_{\bar{n}'}$-measure $\varepsilon$.
 \end{cor}
 \begin{proof}
The proof of Proposition \ref{prop: large proj k k plus 1} can be carried out uniformly in all those $\bar z$ such that $(m_1,\ldots,m_k)$ are large enough relative to $\VC_k(f_{\bar z})$.  In particular, by countable additivity, we can choose $(m_1,\ldots,m_k)$ large enough to work except for a set of $\bar z$ of measure $<\varepsilon$.
\end{proof}

\subsection{Proof of the general case}

We are now ready to prove the main theorem.

\begin{theorem}\label{thm: the very main thm soft}

 Suppose that $\left(V_{[k']}, \B_{\bar{n}}, \mu_{\bar{n}} \right)_{\bar{n} \in \mathbb{N}^{k'}}$ is a $k'$-partite graded probability space and $f: V^{\bar{1}^{k'}} \to [0,1]$ is a $k'$-ary $\B_{\bar{1}^{k'}}$-measurable function with $\VC_k(f)  < \infty$ (see Definition \ref{def: VCk dimension of functions}(4)) for some $k<k'$.  Then $f$ is $\mathcal{B}_{\bar 1^{k'},k}$-measurable.
 
 More precisely, for every $\varepsilon > 0$ there exist some 
  $N \in \mathbb{N}$, $\gamma_i \in \mathbb{Q}^{[0,1]}$ for $i \in [N]$ and $\B_{I}(f)$-measurable $(\leq k)$-ary functions
  $f^i_I: \prod_{j \in I} V_{j} \to [0,1]$ for $i \in [N], I \in  \binom{[k']}{ \leq k} $ so that, defining $g: V^{\bar{1}^{k'}} \to [0,1]$ via
  \begin{gather*}
 	g(\bar{x}) :=  \sum_{i \in [N] } \gamma_i \cdot \prod_{I \in \binom{[k']}{\leq k}} f^i_I(\bar{x}_I),
 \end{gather*}
 we have $\norm{f - g}_{L^2(\mu_{\bar{1}^{k'}})} < \varepsilon$.
 
 \end{theorem}

\begin{remark}
	Consider the simplest case, where $k'=3$ and $k=1$.  Corollary \ref{prop: large proj k k plus 1 uniform} says
\[f(x_1,x_2,x_3)\approx \sum_{i\in[N]}\gamma_i f^i_1(x_1,x_3)f^i_2(x_2,x_3)\]
for almost all fixed $x_3 \in V_3$, where $f^i_1,f^i_2$ are $\mathcal{B}_{I \cup \{3\}}(f)$-measurable for suitable $I \in \binom{[2]}{\leq 1}$.  The elements of $\mathcal{B}_I(f)$ are built from the levels sets of $f$.  While this does not ensure that they are themselves of finite $\VC_1$-dimension, we will show that they are closely approximated by sets of finite $\VC_1$-dimension.  This implies that we can apply Proposition \ref{prop: large proj k k plus 1} to the approximations of the functions $f^i_1,f^i_2$ approximating them by unary functions, and putting it together we obtain the desired representation of $f$.
\end{remark}

%
%

\begin{proof}[Proof of Theorem \ref{thm: the very main thm soft}]
  We prove the proposition by induction on $k' - k$. The base case $k' - k = 1$ is given by Proposition \ref{prop: large proj k k plus 1}.

			So let $1 \leq k < k'$ with $k'-k \geq 2$ be fixed, and assume that the claim holds for all pairs $k_0 < k'_0$ with $k'_0 - k_0 < k' - k$.  Let $\varepsilon \in \mathbb{R}_{>0}$ be given, and fix $\delta \in \mathbb{R}_{>0}$ sufficiently small with respect to $\varepsilon$, to be determined later.	
					 
			Assume that $\left(V_{[k']}, \B_{\bar{n}}, \mu_{\bar{n}} \right)_{\bar{n} \in \mathbb{N}^{k'}}$ is a $k'$-partite graded probability space, and $f: V^{\bar{1}^{k'}} \to [0,1]$ is  a $k'$-ary $\B_{\bar{1}^{k'}}$-measurable function with $\VC_k(f) \leq \bar{d} < \infty$. 
			
			As $k \leq k'-2 < k'$, the latter implies that also $\VC_{k'-2}(f) < \infty$, in particular $\VC_{k'-2}(f_{x_{k'}}) < \infty$ for every $x_{k'} \in V_{k'}$. Applying Corollary \ref{prop: large proj k k plus 1 uniform} with $k_0 := k'-2, k'_0 := k', \delta$ in place of $k,k',\varepsilon$,  there exist some $N' \in \mathbb{N}$, $\gamma_i \in \mathbb{Q}^{[0,1]}$ for $i \in [N']$, and $\B_{I \cup \{ k' \}}(f)$-measurable $(\leq k'-1)$-ary functions
 $f^i_I:  \left( \prod_{j \in I} V_{j} \right) \times V_{k'} \to[0,1] $ for $i \in [N'], I \in  \binom{[k'-1]}{ \leq k'-2} $ so that, taking $g: V^{\bar{1}^{k'}} \to [0,1]$ to be
 \begin{gather}\label{eq: gen case of main thm def g}
	g: \bar{x} \in V^{\bar{1}^{k'}} \mapsto  \sum_{i \in [N'] } \gamma_i \cdot \prod_{I \in  \binom{[k'-1]}{ \leq k'-2}} f^i_I(\bar{x}_I, \bar{x}_{k'}),
\end{gather}
we have 
\begin{gather}\label{eq: g approx f almost everywhere}
	\norm{f_{x_{k'}} - g_{x_{k'}}}_{L^2 \left( \mu_{\bar{1}^{(k'-1)}} \right)} < \delta \textrm{ for all } x_{k'} \in V_{k'} \setminus X'_{k'},\\
	\textrm{ for some } X'_{k'} \in  \B_{\bar{\delta}_{k'}} \textrm{ with } \mu_{\bar{\delta}_{k'}}(X'_{k'}) < \delta. \nonumber
\end{gather}
At this point we would like to apply the inductive hypothesis to the $(\leq k'-1)$-ary functions $f^i_I$, however a priori there is no reason for them to be of finite $\VC_k$-dimension: if $\VC_k(f) < \infty$, then we might still have $\VC_k(\chi_{f^{<r}}) = \infty$ for a fixed $r \in \mathbb{Q}^{[0,1]}$. We show that at least these functions can be approximated arbitrarily well in $L^2$-norm by functions of finite $\VC_k$-dimension.

%

So fix some $i \in [N']$ and $I \in \binom{[k'-1]}{ \leq k'-2}$, and let $J := I \cup \{k'\} \in \binom{[k']}{ \leq k'-1}$. 
Let $\delta'>0$ be arbitrary. As $f^i_I$ is $\B_{J}(f)$-measurable, by definition of $\B_{J}(f)$ (see Definition \ref{def: B(f) all smaller fibers}) we can choose a sufficiently large $t \in \mathbb{N}$ and some $\bar{w}_1, \ldots, \bar{w}_t \in V^{\bar{1}^{k'}}$ so that $\norm{f^{i}_{I} - h}_{L^2 \left(\mu_{ \bar{m}_J} \right)} < \delta'$, where $\bar{m}_J := \sum_{j \in J} \bar{\delta}_j \in \mathbb{N}^{k'}$ and $h$ is a function of the form 
\begin{gather*}
	h: \bar{x} \in \prod_{j \in J} V_j \mapsto \sum_{u \in [t]} \alpha_u \cdot \chi_{f^{[r_u, s_u)}_{\left(\bar{w}_u \right)_{[k'] \setminus J}}}(\bar{x})
\end{gather*}
for some $\alpha_u, r_u, s_u \in \mathbb{Q}^{[0,1]}_t$. Let $\delta'' >0$ be arbitrary. By Lemma \ref{lem: smoothening char func} we can choose a sufficiently large $p \in \mathbb{N}$ so that, for every $u \in [t]$ and $q \in \{s_u, r_u\}$, taking $\hat{f}^{<q}_u := p \dot{\times} \left(  q \monus  f_{\left(\bar{w}_u \right)_{[k'] \setminus J}} \right) $, we have
\begin{gather*}
	\norm{\chi_{f^{<q}_{\left(\bar{w}_u \right)_{[k'] \setminus J}}} - \hat{f}^{<q}_u}_{L^2 \left( \mu_{\bar{m}_J} \right) } < \delta''.
\end{gather*}
Letting $\hat{f}_u^{[s_u, r_u)} := \hat{f}_u^{<r_u} \cdot (1 - \hat{f}_u^{<s_u})$ and using Lemma \ref{lem: operations for k+1,k proof}\eqref{lem: sum prod of approx}, we thus have 
\begin{gather}\label{eq: main thm induction 3}
	\norm{\chi_{f^{[r_u,s_u)}_{\left(\bar{w}_u \right)_{[k'] \setminus J}}} - \hat{f}^{[r_u, s_u)}_u}_{L^2 \left( \mu_{\bar{m}_J} \right) } < 2 \delta''.
\end{gather}
for every $u \in [t]$. Let 
\begin{gather*}
	h': \bar{x} \in \prod_{j \in J} V_j \mapsto \sum_{u \in [t]} \alpha_i \cdot \hat{f}_u^{[r_u, s_u)}(\bar{x}),
\end{gather*}
then, by \eqref{eq: main thm induction 3} and Lemma \ref{lem: operations for k+1,k proof}\eqref{lem: sum prod of approx} again, we have
\begin{gather}\label{eq: main thm induction 4}
	\norm{f^{i}_{I} - h'}_{L^2 \left(\mu_{ \bar{m}_J} \right)} \leq \norm{f^{i}_{I} - h}_{L^2 \left(\mu_{ \bar{m}_J} \right)} + \norm{h - h'}_{L^2 \left( \mu_{\bar{m}_J} \right) } \\
	\leq \delta' + (2t+1)2\delta'' < 2 \delta' \nonumber
\end{gather}
assuming we took $\delta''$ small enough with respect to $\delta'$ and $t$.

Note that, for every $u \in [t]$, $\hat{f}_u^{<q}$ is clearly $\B_{J}(f)$-measurable from the definition, hence also $h'$ is $\B_{J} (f)$-measurable. Also, since fixing some of the coordinates or permuting the coordinates preserves finiteness of the $\VC_k$-dimension of a function (Proposition \ref{prop: perm of vars fin VCk dim}) and $\VC_{k}(f) < \infty$, it follows that $\VC_k \left( f_{\left( \bar{w}_u \right)_{[k'] \setminus J}  } \right) < \infty$ for every $u \in [t]$. By several applications of Proposition \ref{prop: VCk comp with cont funct} we then have $\VC_k \left( \hat{f}_u^{[r_u, s_u)} \right) < \infty$, and hence $\VC_k \left( h' \right) < \infty$.

We enumerate $J$ as $j_1<\cdots<j_{\ell}\in[k']$ for some $\ell\leq k'-1$, where $j_{\ell}=k'$.  We let $V'_i:=V_{j_i}$ for $i\in[\ell]$ and, for all $\bar m=(m_1,\ldots,m_{\ell})\in\mathbb{N}^\ell$ we let $\bar m':=\sum_{i\in[\ell]}m_i\bar\delta_{j_i}\in\mathbb{N}^{k'}$, $\mathcal{B}'_{\bar m}:=\mathcal{B}_{\bar m'}$, $\mu'_{\bar m}:=\mu_{\bar m'}$.  By Remark \ref{rem: power graded prob space}, $\left(V'_{[\ell]}, \bar{B}'_{\bar{n}}, \mu'_{\bar{n}} \right)_{\bar{n} \in \mathbb{N}^{\ell}}$ is an $\ell$-partite graded probability space and the $\ell$-ary function $h': \prod_{j \in [\ell]} V'_j \to [0,1]$ is $\B'_{\bar{1}^{\ell}}$-measurable with $\VC_k(h') < \infty$. As $\ell \leq k'-1$ (hence $\ell - k < k'-k$), applying the inductive hypothesis and unwinding the conclusion in terms of the original graded probability space we thus have $\norm{h' - g^i_I}_{L^2 \left( \mu_{\bar{m}_J} \right) } < \delta'$ for a function $g^i_I: \prod_{j \in J} V_j \to [0,1]$ of the form
 \begin{gather*}
 	g^i_I : \bar{x} \in \prod_{j \in J} V_j \mapsto  \sum_{u \in [N_{i,I}] } \beta^{i,I}_{u} \cdot \prod_{K \in \binom{J}{\leq k}} f^{i,I,u}_{K} (\bar{x}_K)
 \end{gather*}
 for some $N_{i,I} \in \mathbb{N}$, some $\beta^{i,I}_{u} \in \mathbb{Q}^{[0,1]}$ and some $(\leq k)$-ary $\B_{K}(h')$-measurable (and hence $\B_{K}(f)$-measurable) functions $f^{i,I,u}_{K}: \prod_{j \in K} V_j \to [0,1]$. Combining with \eqref{eq: main thm induction 4}, we have 
 \begin{gather*}
 	\norm{f^i_I - g^i_I}_{L^2 \left( \mu_{\bar{m}_J} \right) } < 3 \delta' \textrm{ for every } i \in [N'] \textrm{ and } I \in \binom{[k'-1]}{ \leq k'-2}.
 \end{gather*}

 Let $g': V^{\bar{1}^{k'}} \to [0,1]$ be obtained from $g$ by replacing $f^i_I$ with $g^i_I$  in \eqref{eq: gen case of main thm def g} for every $i \in [N'], I \in \binom{[k'-1]}{ \leq k'-2}$. Using Lemma \ref{lem: operations for k+1,k proof}\eqref{lem: sum prod of approx} this implies
 \begin{gather}\label{eq: g approx g'}
 \norm{g - g'}_{L^2 \left( \mu_{\bar{1}^{k'}} \right)} \leq   N' \cdot \left( 2 \left \lvert \binom{[k'-1]}{ \leq k'-2}  \right \rvert + 1\right) \cdot 3 \delta' < \delta,
 \end{gather}
assuming that we took $\delta' = \delta'(k', \delta, N')$ sufficiently small.

Regrouping the elements of the expression for $g'$, we see that it is of the form 
	$$g'(\bar{x}) :=  \sum_{i \in [N] } \gamma'_i \cdot \prod_{I \in \binom{[k']}{\leq k}} h^i_I(\bar{x}_I)$$
	 for some $N \in \mathbb{N}, \gamma'_i \in \mathbb{Q}^{[0,1]}$ and $h^i_I$ a $(\leq k)$-ary $\B_{I}(f)$-measurable functions $h^i_I: \prod_{j \in I} V_j \to [0,1]$.	 Hence $g'$ has the required form, and it remains to show that $g'$ approximates $f$ in $L^2$-norm. 
	 
	 By \eqref{eq: g approx g'} and the first implication in Lemma \ref{lem: operations for k+1,k proof}\eqref{lemeq: small norm impl small norm for almost all fibers}, there exists some set $X''_{k'} \in \B_{\bar{\delta}_{k'}}$ with $\mu_{\bar{\delta}_{k'}} \left( X''_{k'} \right) \leq \delta$ such that $\norm{g_{x_{k'}} - g'_{x_{k'}}}_{L^2 \left( \mu_{\bar{1}^{(k'-1)}} \right)} \leq \delta^{\frac{1}{2}}$ for all $x_{k'} \in V_{k'} \setminus X''_{k'}$. Combining this with \eqref{eq: g approx f almost everywhere} and taking $X_{k'} := X'_{k'} \cup X''_{k'}$, we thus have $\mu_{\bar{\delta}_{k'}}(X_{k'}) \leq 2 \delta$ and $\norm{f_{x_{k'}} - g'_{x_{k'}}}_{L^2 \left( \mu_{\bar{1}^{(k'-1)}} \right)} \leq \delta + \delta^{\frac{1}{2}}$ for every $x_{k'} \in V_{k'} \setminus X_{k'}$. Hence, by the second implication in Lemma \ref{lem: operations for k+1,k proof}\eqref{lemeq: small norm impl small norm for almost all fibers}, we have $\norm{f - g'}_{L^2 \left( \mu_{\bar{1}^{k'}} \right)} \leq \left(\max \{2 \delta, (\delta + \delta^{\frac{1}{2}})^2 \} \right)^{\frac{3}{4}} < \varepsilon$ assuming $\delta$ was chosen small enough with respect to $\varepsilon$.
\end{proof}

\begin{remark}
  We note that there is an alternate approach which avoids the careful analysis of the sets in $\mathcal{B}_{I}(f)$, at the price of using additional machinery about $\sigma$-subalgebras.  We illustrate the idea in the simplest case, where $k'=3$ and $k=1$.  Given $f(x_1,x_2,x_3)$, two applications of Corollary \ref{prop: large proj k k plus 1 uniform}---once with $x_1$ as the parameter and once with $x_2$ as the parameter---tells us that
\[f(x_1,x_2,x_3)\approx \sum_{i\in[N]}\beta_i \chi_{U^i_1}(x_1,x_2)\chi_{U^i_2}(x_1,x_3)\]
and also
\[f(x_1,x_2,x_3)\approx \sum_{j\in[N]}\gamma_j \chi_{W^j_1}(x_1,x_2)\chi_{W^j_2}(x_2,x_3),\]
for an appropriate choice of the coefficients $\beta_i, \gamma_j$ and sets 
$$U_1^i, W_1^j \in \B_{\{1,2\}}(f), U^i_2 \in \B_{\{1,3\}}(f), W^{j}_2 \in  \B_{\{2,3\}}(f).$$
By rearranging the sums to be over intersections $U^i_1\cap W^j_1$, we may assume the sums are over the same collection of sets---that is,
\[f(x_1,x_2,x_3)\approx\sum_{i\in[N]}\beta_i \chi_{U^i_1}(x_1,x_2)\chi_{U^i_2}(x_1,x_3)\approx \sum_{i\in[N]}\chi_{U^i_1}(x_1,x_2)\chi_{W^i_1}(x_2,x_3).\]
But then on each of the sets $U^i_1$, we have $\chi_{U^i_2} \approx \chi_{W^i_2}$, which means the sets $U^i_2,W^i_2$ must be close to not depending on $x_1$ or $x_2$, respectively: that is, we could replace $U^i_2$ with $u^i_2(x_3)=\int \chi_{U^i_2}(x_1,x_3)\,d\mu_{\bar{\delta}_1}(x_1)$.

So, after rearranging, we get
\[f(x_1,x_2,x_3)\approx \sum_{i\in[N]}\gamma'_i \chi_{U^i_1}(x_1,x_2)\chi_{U^i_2}(x_3).\]
That is, $f$ is measurable with respect to the $\sigma$-subalgebra of $\mathcal{B}_{\bar 1^3}$ generated by sets of the form $A(x_1,x_2)\times B(x_3)$.  (In the notation of \cite{MR3583029}, this $\sigma$-subalgebra is called $\mathcal{B}_{\bar 1^3,\{\{0,1\},\{2\}\}}$.)

This argument is symmetric, so $f$ also has approximations using sets of the form $A(x_1,x_3)\times B(x_2)$ and $A(x_2,x_3)\times B(x_1)$.  One can show (for instance, using the generalized Gowers uniformity norms) that a function which has several different representations in terms of restricted kinds of sets also has a simultaneous representation respecting all restrictions at once.  In a slightly different setting, this is \cite[Lemma 8.23]{Towsner2018}.


\end{remark}

Finally, we derive a more quantitative version of Theorem \ref{thm: the very main thm soft}.

\begin{cor}\label{thm: the very main thm}
  For every $k < k' \in \mathbb{N}, \bar{d} < \infty$ and $\varepsilon \in \mathbb{R}_{>0}$ there exists some $N = N(k,k',\bar{d}, \varepsilon) \in \mathbb{N}$ satisfying the following.

  Suppose that $\left(V_{[k']}, \B_{\bar{n}}, \mu_{\bar{n}} \right)_{\bar{n} \in \mathbb{N}^{k'}}$ is a $k'$-partite graded probability space and $f: V^{\bar{1}^{k'}} \to [0,1]$ is a $k'$-ary $\B_{\bar{1}^{k'}}$-measurable function with $\VC_k(f)  < \bar{d}$ (see Definition \ref{def: VCk dimension of functions}(4)). 
 
 Then for $i \in [N]$ there exist some $\gamma_i \in \mathbb{Q}^{[0,1]}_N$, $\bar{w}_i \in V^{\bar{1}^{k'}}$  and, for each $I \in \binom{[k']}{\leq k}$, a $(\leq k)$-ary function $f^i_{I}: \prod_{i \in I} V_i \to [0,1]$ simple with respect to the algebra $\B^N_{I, \bar{w}_1, \ldots, \bar{w}_N}(f)$ (see Definition \ref{def: B(f) all smaller fibers}) and with all of its coefficients in $\mathbb{Q}^{[0,1]}_N$ so that, defining a $\B_{\bar{1}^{k'},k}$-measurable function $g: V^{\bar{1}^{k+1}} \to [0,1]$ via
 \begin{gather*}
	g(\bar{x}) :=  \sum_{i \in [N] } \gamma_i \cdot \prod_{I \in \binom{[k']}{\leq k}} f^i_I(\bar{x}_I),
\end{gather*}
 we have $\norm{f - g}_{L^2} < \varepsilon$.


\end{cor}
\begin{proof}
  This follows from Theorem \ref{thm: the very main thm soft} 
 via a compactness argument relying on the techniques of Section \ref{sec:indiscernibles}, as we explain below. 
 
 Assume first that $\mathfrak{P} = \left(V_{[k']}, \B_{\bar{n}}, \mu_{\bar{n}} \right)_{\bar{n} \in \mathbb{N}^{k'}}$ is an arbitrary $k'$-partite graded probability space and $g: V^{\bar{1}^{k'}} \to [0,1]$ is as in the conclusion of Theorem \ref{thm: the very main thm soft}. Approximating each  $f^i_I$ by a $\B^{t}_{I, \bar{w}_1, \ldots, \bar{w}_{t}}(f)$-simple functions for a sufficiently large $t$ and some  $\bar{w}_1, \ldots, \bar{w}_{t} \in V^{\bar{1}^{k'}}$, we may assume that $g$ is of the form 
$g(\bar{x}) :=  \sum_{i \in [N] } \gamma_i \cdot \prod_{I \in \binom{[k']}{\leq k}} f^i_I(\bar{x}_I)$ for some $N \in \mathbb{N}$, $\gamma_i \in \mathbb{Q}^{[0,1]}_N$ and 
$$f^i_{I} = \sum_{j \in [t_i]} \alpha^{i,I}_j \cdot \chi_{f^{[r^{i,I}_j, s^{i,I}_j)}_{\left(\bar{w}^{i,I}_j \right)_{[k'] \setminus I}}}
$$ 
for some $t_i \in \mathbb{N}$, $\bar{w}^{i,I}_j \in V^{\bar{1}^{k'}}$ and $\gamma_i, \alpha^{i,I}_j, r^{i,I}_j, s^{i,I}_j \in \mathbb{Q}^{[0,1]}_N$. Substituting these expressions for $f^i_I$'s into $g$ and rearranging, we may thus assume that $g$ is of the form 
$$h_{N, \bar{\alpha}, \bar{r}, \bar{s}, \bar{w}}(\bar{x}) :=  \sum_{i \in [N] } \alpha_i \cdot \prod_{I \in \binom{[k']}{\leq k}}  \chi_{f^{[r^{i,I}, s^{i,I})}_{\left(\bar{w}^{i,I} \right)_{[k'] \setminus I}} }(\bar{x}_{I})$$
for some bigger $N \in \mathbb{N}$ and some
$\bar{\alpha} = (\alpha_i \in \mathbb{Q}^{[0,1]}_{N} : i \in [N])$, $\bar{r} = \left( r^{i,I} \in \mathbb{Q}^{[0,1]}_N : i \in [N], I \in \binom{[k']}{\leq k} \right)$ and $\bar{s} = \left( s^{i,I} \in \mathbb{Q}^{[0,1]}_N : i \in [N], I \in \binom{[k']}{\leq k} \right)$ with $r^{i,I} < s^{i,I}$, and $\bar{w} = \left( \bar{w}^{i,I} \in V^{\bar{1}^{k'}} : i \in [N], I \in  \binom{[k']}{\leq k} \right)$. Following the proof of Lemma \ref{lem: type-def of norm}(3) with straightforward modifications, we see that for every fixed $N, \bar{\alpha}, \bar{r}, \bar{s}$ and $\varepsilon \in \mathbb{R}_{>0}$ there exists a countable collection of $\mathcal{L}_{\infty}$-sentences $\Theta^{N, \bar{\alpha}, \bar{r}, \bar{s}}_{\varepsilon}$ so that: for any $k'$-partite graded probability space $\mathfrak{P} = \left(V_{[k']}, \B_{\bar{n}}, \mu_{\bar{n}} \right)_{\bar{n} \in \mathbb{N}^{k'}}$, a $\B_{\bar{1}^{k'}}$-measurable $f$ and any $\mathcal{L}_{\infty}$-structure $\mathcal{M}' \propto \mathcal{M}_{\mathfrak{P},f}$, 
\begin{gather}\label{eq: main thm induction 6}
	\mathcal{M}' \models \Theta^{N, \bar{\alpha}, \bar{r}, \bar{s}}_{\varepsilon} \iff \\
	\textrm{for all tuples } \bar{w} = \left( \bar{w}^{i,I} \in V^{\bar{1}^{k'}} : i \in [N], I \in  \binom{[k']}{\leq k}  \right), \nonumber \\ 
	 \norm{f - h_{N, \bar{\alpha}, \bar{r}, \bar{s}, \bar{w}}}_{L^2} \geq \varepsilon. \nonumber
\end{gather}

Now assume towards a contradiction that the conclusion of the theorem fails for some $k,k',\bar{d}, \varepsilon$. 
This means that for every $j \in \mathbb{N}$, there exists some $k'$-partite graded probability space $\mathfrak{P}_j = (V^j_{[k']}, \B^j_{\bar{n}}, \mu^j_{\bar{n}})_{\bar{n} \in \mathbb{N}^{k'}}$ and  some $\B^j_{\bar{1}^{k'}}$-measurable function $f^j: \prod_{i \in [k']} V^j_i \to [0,1]$ with $\VC_{k}(f^j) \leq \bar{d}$ such that, in view of the previous paragraph  \eqref{eq: main thm induction 6} and that $\mathcal{M}_{\mathfrak{P}^j, f^j} \propto \mathcal{M}_{\mathfrak{P}^j, f^j}$ trivially, 
\begin{gather*}
	\mathcal{M}_{\mathfrak{P}^j, f^j} \models  \bigwedge_{\bar{\alpha} \in \left( \mathbb{Q}^{[0,1]}_j\right)^{[j]} } \bigwedge_{\bar{r}, \bar{s} \in \left( \mathbb{Q}^{[0,1]}_j\right)^{[j] \times \binom{[k']}{\leq k} }} \Theta^{j, \bar{\alpha}, \bar{r}, \bar{s}}_{\varepsilon}.
\end{gather*}
Let $\cU$ be a non-principal ultrafilter on $\mathbb{N}$. Let $\tilde{\mathfrak{P}} :=\left(\tilde{V}_{[k']}, \tilde{\B}_{\bar{n}}, \tilde{\mu}_{\bar{n}} \right)_{\bar{n} \in \mathbb{N}^{k'}}$ be the $k'$-partite graded probability space, the $\tilde{\B}_{\bar{1}^{k'}}$-measurable function $\tilde{f}: \tilde{V}^{\bar{1}^{k'}} \to [0,1]$ and $\mathcal{\tilde{M}}$ the $\mathcal{L}_{\infty}$-structure defined by the corresponding ultraproduct in  Section \ref{sec: ultraproducts of k-GPS} (Fact \ref{fac: ultraproduct props}).
By \L os' theorem we then have
\begin{gather*}
	\tilde{\mathcal{M}} \models \bigwedge_{j \in \mathbb{N} } \bigwedge_{\bar{\alpha} \in \left( \mathbb{Q}^{[0,1]}_j\right)^{[j]} } \bigwedge_{\bar{r}, \bar{s} \in \left( \mathbb{Q}^{[0,1]}_j\right)^{[j] \times \binom{[k']}{\leq k} }} \Theta^{j, \bar{\alpha}, \bar{r}, \bar{s}}_{\varepsilon}.
\end{gather*}

As $\tilde{\mathcal{M}} \propto \mathcal{M}_{\tilde{\mathfrak{P}}, \tilde{f}}$, using \eqref{eq: main thm induction 6} this implies that $\tilde{f}$ does not satisfy the conclusion of Theorem \ref{thm: the very main thm soft} for any $N \in \mathbb{N}$ --- a contradiction.
\end{proof}

Specializing to the case of hypergraphs instead of arbitrary functions, we immediately get the following corollary. 

\begin{cor}\label{cor: main thm for hypergraphs}
For every $k < k' \in \mathbb{N}, d \in \mathbb{N}$ and $\varepsilon \in \mathbb{R}_{>0}$ there exists some $N = N(k,k',d, \varepsilon) \in \mathbb{N}$ satisfying the following.

 Suppose that $\left(V_{[k']}, \B_{\bar{n}}, \mu_{\bar{n}} \right)_{\bar{n} \in \mathbb{N}^{k'}}$ is a $k'$-partite graded probability space and $E \in \B_{\bar{1}^{k'}}$ is a $k'$-ary relation with $\VC_k(E)  \leq d$.	
 
 Then there exists some $(\leq k)$-ary fibers $F_{1}, \ldots, F_N$ of $E$ (so each $F_i$ is obtained from $E$ by fixing all but at most $k$ coordinates by some parameters from the corresponding $V_i$'s) and $F$ a Boolean combination of $F'_1, \ldots, F'_N$ so that $\mu_{\bar{1}^{k'}} \left( E \triangle F \right) < \varepsilon$.
 
 (Where for $I \in \binom{[k']}{k}$ and an $|I|$-ary fiber $F \subseteq \prod_{i \in I} V_i$, $F'$ is the $k'$-ary relation $\left\{ \bar{x} \in \prod_{i \in [k']} V_i : \bar{x}_I \in F \right\}$.)
\end{cor}
\begin{proof}
	Applying Corollary \ref{thm: the very main thm} to $\chi_E$, we get that $\norm{\chi_E - g}_{L^2} < \varepsilon$ for some $g$ of the form
	$$g(\bar{x}) = \sum_{i \in [N] } \alpha_i \cdot \prod_{I \in \binom{[k']}{\leq k}}  \chi_{E^{=t^{i,I}}_{\left(\bar{w}^{i,I} \right)_{[k'] \setminus I}} }(\bar{x}_{I})$$
for some $\alpha_i \in \mathbb{Q}^{[0,1]}_{N}$, $t^{i,I} \in \{0,1\}$ and $\bar{w}^{i,I} \in V^{\bar{1}^{k'}}$ for $i \in [N], I \in \binom{[k']}{\leq k} $. As in Remark \ref{rem : approx by indicator functions}, replacing $N$ by some larger $N' = N'(N,\varepsilon)$, we may assume that $\alpha_i \in \{0,1\}$ for all $i \in N'$ --- which gives the required presentation.
\end{proof}

\section{High $\VC_k$-dimension implies inapproximability}\label{sec:converse}

We now consider the converse to the results of the previous section.  As pointed out in the introduction, we cannot expect that every $\mathcal{B}_{\bar 1^{k'},k}$-measurable $k'$-ary function has finite $\VC_k$-dimension, because the infinite shattered set could have measure $0$.  To find the right converse, we should notice that the conclusion of Corollary \ref{thm: the very main thm} depends only on the $\VC_k$-dimension of $f$; this means that we would have approximations with the same bound on their complexity if we replaced the measures $\mu_{\bar n}$ with different measures.  That is, Corollary \ref{thm: the very main thm} holds \emph{uniformly under all measures}.\footnote{Compare the distinction between sets with the Glivenko-Cantelli property, the universal Glivenko-Cantelli property, and the uniform Glivenko-Cantelli property.  It is only the last which equivalent to having finite VC dimension \cite{MR893902,MR1385403,MR1115159}.}

So the expected converse is that $f$ should have finite $\VC_k$-dimension if $f$ has the property that for every $\varepsilon>0$ there is an $N$ so that, for all choices of measures on the $V_i$, $f$ can be approximated to within $\varepsilon$ in $L^2$-norm with respect to those measures by a function of the form $g(\bar x)=\sum_{j\in[N]}\gamma_j\cdot\prod_{I\in{[k']\choose\leq k}}f^j_I(\bar x_I)$ as in Theorem \ref{thm: the very main thm soft}.

\begin{theorem}\label{thm:converse}
  Let $k'>k$ and $f:\prod_{i\in[k']}V_i \to[0,1]$ be given such that, for every $\varepsilon>0$ there is an $N \in \mathbb{N}$ such that: for any $k'$-partite graded probability space $\left(V_{[k']}, \B_{\bar{n}}, \mu_{\bar{n}} \right)_{\bar{n} \in \mathbb{N}^{k'}}$ such that $f$ is $\B_{\bar{1}^{k'}}$-measurable, there is a function  $g: \prod_{i \in [k']}V_i \to [0,1]$ of the form
  \[g(\bar x)=\sum_{j\in[N]}\gamma_j\cdot\prod_{I\in{[k']\choose\leq k}}f^j_I(\bar x_I),\]
  with some coefficients $\gamma_j$ and each $f^j_I$ a $\B_{\sum_{i \in I} \bar{\delta}_i}$-measurable $(\leq k)$-ary function, and $\norm{f-g}_{L^2 \left(\mu_{\bar{1}^{k'}} \right)}<\varepsilon$.  Then $\VC_k(f)<\infty$.
\end{theorem}
\begin{proof}

  Let $k'>k$ and $f:\prod_{i\in[k']}V_i \to[0,1]$ satisfy the assumption of the theorem, and towards a contradiction suppose that $\VC_k(f)=\infty$.  By Definition \ref{def: VCk dimension of functions}(4) this means that there exist some $I\subseteq [k']$ with $|I|=|k'-(k+1)|$ and some $b = (b_{i} : i \in I) \in V_I$ such that the $k+1$-ary fiber of $f$ at $b$, $f_{b}: \prod_{j \in [k'] \setminus I} \to [0,1]$ has $\VC_k(f) = \infty$. Fix $r,s$ so that $\VC_k^{r,s}(f_b)=\infty$.  Then, by Remark \ref{rem: VCk for f iff omits}, for every finite $(k+1)$-partite hypergraph $H$ there is an induced copy of $H$ in $f_b$, in the sense that $f_b$ is $\leq r$ on edges of $H$, and $\geq s$ on non-edges of $H$. Permuting the coordinates if necessary (see Remark \ref{rem: basic props of GPS}(1)), we may assume that $I = [k+1]$.

 For each $d \in \mathbb{N}$, we choose uniformly at random a finite $(k+1)$-partite $(k+1)$-uniform hypergraph $H_d\subseteq[d]^{k+1}$.  With probability $1$, $\lim_{d\rightarrow\infty}\frac{|H_d|}{d^{k+1}}=1/2$ and $\lim_{d\rightarrow\infty}||\chi_{H_d}-1/2||_{U^{\bar 1^{k+1}}}=0$ (for $U^{\bar 1^{k+1}}$ with respect to the uniform measure; see the proof of \cite[Theorem 9.2]{MR3583029}, for instance, for the second calculation).
  For each $d$  we define probability measures $\mu_{\bar{\delta}_i}^d$ which concentrate on the single element $b_i$ if $i\in I$ and concentrate uniformly on the vertices of the $i$th part in a chosen copy of $H_d$ contained in $V_i$ otherwise.  Note that these are atomic measures, so the extension of the $\mu_{\bar{\delta}_i}^d$ to a Keisler graded probability space on all subsets of the products of the $V_i$ is immediate: there is a unique extension to all subsets depending on the intersection of a set with the finitely many atoms of the measure (see also Remark \ref{rem: basic props of GPS}).

  This gives us $k'$-partite graded probability spaces $\mathfrak{P}_d= \left(V_{[k']},\mathcal{B}_{\bar n},\mu^d_{\bar n} \right)_{\bar n\in\mathbb{N}^{k'}}$ (where $\mathcal{B}_{\bar n}$ is the algebra of all subsets of  $V^{\bar{n}}$).  Fix an arbitrary $E \in \mathbb{N}_{>0}$. Using the assumption we may choose some $N = N_{E} \in \mathbb{N}$ and approximations $g^{d,E}(\bar x)=\sum_{j\in[N]}\gamma^{d,E}_j\cdot\prod_{I\in{[k']\choose\leq k}}f^{d,E,j}_I(\bar x_I)$ of $f$ to within $\frac{1}{E}$ with respect to $L^2 \left( \mu_{\bar{1}^{k'}}^d \right)$.

  We fix some non-principal ultrafilter $\mathcal{U}$ on $\mathbb{N}$ and consider the ultraproduct $\tilde{\mathfrak{P}} :=\left(\tilde{V}_{[k']}, \tilde{\B}_{\bar{n}}, \tilde{\mu}_{\bar{n}} \right)_{\bar{n} \in \mathbb{N}^{k'}}$ of the $\mathfrak{P}_d$'s, $\tilde{f}: \tilde{V}^{\bar{1}^{k'}} \to [0,1]$ of the functions $f_d = f$, $\tilde f^{E,j}_I$  of the functions $f^{d,E,j}_{I}$ and $\tilde g^E: \tilde{V}^{\bar{1}^{k'}} \to [0,1]$ of the functions $g^{d,E}$ as in Section \ref{sec: ultraproducts of k-GPS} (namely, $\tilde{\mathfrak{P}}$ is defined with respect to the ultraproduct of the structures $\mathcal{M}_d := \mathcal{M}_{\mathfrak{P}_d, f, \left(f^{d,E,j}_I : i \in \binom{[k']}{\leq k} \right), g^{d,E}}$ for $d \in \mathbb{N}$ in the notation there). Let $\tilde{b} = (\tilde{b}_i : i \in I)$ with $\tilde{b}_i = (b_i, b_i, \ldots) / \mathcal{U} \in \tilde{V}_i$.

By the choice of $H_d$ and $\mu^d_{\bar{\delta}_i}, i \in [k']$, we have 
  \begin{gather}\label{eq: transf U norm to UP}
  	\lim_{d \to \infty} \mu^d_{\bar{1}^{k'}}\left(f^{\leq r} \right) = \lim_{d \to \infty} \mu^d_{\bar{1}^{k+1}}\left(f^{\leq r}_b \right) = \frac{1}{2}, \textrm{ and }\\
  	\lim_{d\rightarrow\infty} \norm{\chi_{f^{\leq r}}-1/2}_{U^{\bar 1^{k'}} \left( \mu^d_{\bar{1}^{k'}} \right)}= \lim_{d\rightarrow\infty} \norm{\chi_{f^{\leq r}_b}-1/2}_{U^{\bar 1^{k+1}} \left( \mu^d_{\bar{1}^{k+1}} \right)} = 0. \nonumber
  \end{gather}
 This implies that in the ultraproduct we get the exact equalities. Indeed, as in the proof of Lemma \ref{lem: type-def of norm}(2), for any $\alpha \in \mathbb{Q}_{>0}$ there exist countable collections of $\mathcal{L}_{\infty}$-sentences $\Theta_{\alpha}, \Theta'_{\alpha}$ so that: for any $k'$-partite graded probability space $\mathfrak{P} = \left(V_{[k']}, \B_{\bar{n}}, \mu_{\bar{n}} \right)_{\bar{n} \in \mathbb{N}^{k'}}$, a $\B_{\bar{1}^{k'}}$-measurable function $f$ and any $\mathcal{L}_{\infty}$-structure $\mathcal{M}' \propto \mathcal{M}_{\mathfrak{P},f}$, 
\begin{gather*}\label{eq: main thm induction 6}
	\mathcal{M}' \models \Theta_{\alpha} \iff  \mu_{\bar{1}^{k'}} \left(\chi_{f^{\leq r}} \right) \in \left[\frac{1}{2} - \alpha, \frac{1}{2} + \alpha \right],\\
	\mathcal{M}' \models \Theta'_{\alpha} \iff  ||\chi_{f^{\leq r}}-1/2||_{U^{\bar 1^{k'}} \left( \mu_{\bar{1}^{k'}} \right)} \leq \alpha.
\end{gather*}
    As trivially $\mathcal{M}_d \propto \mathcal{M}_d$ for every $d \in \mathbb{N}$, using \L os' theorem and  \eqref{eq: transf U norm to UP} we have that 
   \begin{gather*}
   	\tilde{\mathcal{M}} \models \bigwedge_{\alpha \in \mathbb{Q}_{>0}} \Theta_{\alpha} \land \Theta'_{\alpha},
   \end{gather*}
which together with $\tilde{\mathcal{M}}\propto \mathcal{M}_{\tilde{\mathfrak{P}}, \tilde{f}} $ (Remark \ref{rem : UP is prop to MP 2}) implies (using that $\tilde{\mu}_i$ is concentrated on the single element $\tilde{b}_i$ for all $i \in [k'] \setminus [k+1]$) that
 \begin{gather}
  \tilde{\mu}_{\bar{1}^{k'}}\left(\tilde{f}^{\leq r} \right) =   \tilde{\mu}_{\bar{1}^{k+1}}\left(\tilde{f}^{\leq r}_{\tilde{b}} \right) = \frac{1}{2}, \textrm{ and } \label{eq: highVC impl inapprox 1}\\
 \norm{\chi_{\tilde{f}^{\leq r}}-1/2}_{U^{\bar 1^{k'}} \left( \tilde{\mu}_{\bar{1}^{k'}} \right)}= \norm{\chi_{\tilde{f}^{\leq r}_{\tilde{b}}}-1/2}_{U^{\bar 1^{k+1}} \left( \tilde{\mu}_{\bar{1}^{k+1}} \right)} = 0. \label{eq: highVC impl inapprox 2}
  \end{gather}

 By Lemma \ref{prop: Gowers pos iff proj pos} (applied to the $(k+1)$-partite graded probability space obtained from $\tilde{\mathfrak{P}}$ by forgetting all but the first $k+1$ coordinates and the $(k+1)$-ary function $\tilde{f}_{\tilde{b}}$ on it, see Remark \ref{rem: power graded prob space}), \eqref{eq: highVC impl inapprox 2}  implies 
 $$\norm{\mathbb{E} \left(\chi_{\tilde{f}^{\leq r}_{\tilde{b}}}- 1/2 \mid \tilde{\mathcal{B}}_{\bar 1^{k+1},k} \right)}_{L^2 \left( \tilde{\mu}_{\bar{1}^{k+1}} \right)}=0.$$

 On the other hand, \eqref{eq: highVC impl inapprox 1}  implies  $\norm{\chi_{\tilde{f}^{\leq r}_{\tilde{b}}}-1/2}_{L^2 \left(\tilde{\mu}_{\bar{1}^{k+1}} \right)}=1/4$, hence in particular $\tilde{f}^{\leq r}_{\tilde{b}}$ cannot be $\tilde{\mathcal{B}}_{\bar 1^{k+1},k}$-measurable.

But each of the functions $\tilde f^{E,j}_I$ is $\tilde{B}_{\sum_{i \in I} \bar{\delta}_i}$-measurable (by definition and Fact \ref{fac: ultraproduct props}(7)), hence each of the functions $\tilde g^E, E \in \mathbb{N}_{>0}$ is  $\tilde{\mathcal{B}}_{\bar 1^{k'},k}$-measurable, and so each of their fibers $\tilde g^E_{\tilde{b}}, E \in \mathbb{N}_{>0}$ is $\tilde{\mathcal{B}}_{\bar 1^{k+1},k}$-measurable.

Using type-definability of $L^2$-norm and \L os' theorem as above, the assumption that $\norm{g^{d,E} - f}_{L^2 \left( \mu^d_{\bar{1}^{k'}} \right)} < \frac{1}{E}$ for all $d \in \mathbb{N}$ implies $\norm{\tilde{g}^{E} - \tilde{f}}_{L^2 \left( \tilde{\mu}_{\bar{1}^{k'}} \right)} < \frac{1}{E}$, which implies $\norm{\tilde{g}^{E}_{\tilde{b}} - \tilde{f}_{\tilde{b}}}_{L^2 \left( \tilde{\mu}_{\bar{1}^{k+1}} \right)} < \frac{1}{E}$
 by the choice of the measures. As $E \in \mathbb{N}_{>0}$ was arbitrary, this implies that $\tilde{f}_{\tilde{b}}$ is $\tilde{\mathcal{B}}_{\bar 1^{k+1},k}$-measurable, a contradiction.
 \end{proof}

\begin{remark}\label{rem: converse fin supp is suff}
	As the proof of Theorem \ref{thm:converse} shows, in order to conclude that $\VC_{k}(f) < \infty$ it is enough that the stated approximation by functions of arity  $\leq k$ holds for all $k'$-partite graded probability spaces on $V^{\bar{1}^{k'}}$ with \emph{finitely supported} measures $\mu_{\bar{\delta}_i}, i \in [k']$.
\end{remark}

When the sets $V_i$ are finite, all functions have finite $\VC_k$-dimension, so Theorem \ref{thm:converse} is not directly applicable.  To make sense of this result in the finite setting, we have to consider a ``modulus of uniform approximability''.  Given a function $\tilde N$, we could say $f:\prod_{i\in[k']}V_i\to [0,1]$ has ``$\tilde N$-uniform approximations'' if, for all graded probability spaces on the $V_i$ and all $\varepsilon$, $f$ has an approximation to within $\varepsilon$ in the form $g(\bar x)=\sum_{j\in[\tilde N(\varepsilon)]}\gamma_j\cdot\prod_{I\in{[k']\choose\leq k}}f^j_I(\bar x_I)$.  (To avoid notational issues, it is more convenient to think of $\tilde N$ as a function whose input is the integer $\lceil 1/\varepsilon\rceil$, as we do below.)

What we will show is that for any function $\tilde N$, there is a specific $\bar d$ so that any $f$ with $\tilde N$-uniform approximations must satisfy $\VC_k(f)\leq \bar d$.

\begin{cor}
  Let $k'>k$ be given.  For any function $\tilde N:\mathbb{N}\rightarrow\mathbb{N}$ and any $r<s$ in $[0,1]$ there is a $d \in \mathbb{N}$ so that whenever $V_i$ are finite sets and $f:\prod_{i\in[k']}V_i\to [0,1]$ and $\VC_k^{r,s}(f)\geq d$, there is some $E \in \mathbb{N}$ and some probability measures $\mu_{i}$ on the $V_i$ (uniquely determining a $k'$-partite graded probability space on the algebra of all subsets of $\prod_{i \in [k']}V_i$, see Remark \ref{rem: basic props of GPS}(2)) such that for every function of the form
  \[g(\bar x)=\sum_{j\in[\tilde N(E)]}\gamma_j\cdot\prod_{I\in{[k']\choose\leq k}}f^j_I(\bar x_I),\]
  we have $||f-g||_{L^2}\geq 1/E$.
\end{cor}
\begin{proof}
  Towards a contradiction, suppose this failed, and let $k'>k$, $\tilde N$, and $r<s$ be a counterexample.  That is, for each $d \in \mathbb{N}$, we have some finite sets $(V_i^d)_{i \in [k']}$ and a function $f^d:\prod_{d\in[k']}V_i^d\to[0,1]$ satisfying $\VC_k^{r,s}(f^d)\geq d$, but such that for any probability measures $\mu_i$ on $V_i$, $f^d$ can be approximated in $L^2$-norm on the corresponding graded probability space up to $\frac{1}{E}$ by some function $g$ of the above form given by a sum of size $\tilde{N}(E)$.  
  
  Taking a non-principal ultraproduct of these examples (see Section \ref{sec: ultraproducts of k-GPS}), we obtain a $k'$-ary function $\tilde{f}:\prod_{i\in[k']} \tilde{V}_i\to[0,1]$ with $\VC_k(\tilde{f})=\infty$ (by Lemma \ref{lem: shattering is definable}).  Then Theorem \ref{thm:converse} gives us measures $\mu'_i$ on the $\tilde{V}_i$ with finite support, $\varepsilon>0$ and a corresponding $k'$-partite graded probability space $\left(\tilde{V}_{[k']}, \B'_{\bar{n}}, \mu'_{\bar{n}} \right)_{\bar{n} \in \mathbb{N}^{k'}}$ uniquely determined by setting $\B'_{\bar{n}}$ to be the algebra of all internal subsets of $V^{\bar{n}}$, and $\mu'_{\bar{\delta}_i} = \mu'_i$, such that $\tilde{f}$ is $\B'_{\bar{1}^{k'}}$-measurable, but no function  $g: \prod_{i \in [k']}\tilde{V}_i \to [0,1]$ of the form
  \[g(\bar x)=\sum_{j\in[N]}\gamma_j\cdot\prod_{I\in{[k']\choose\leq k}}f^j_I(\bar x_I),\]
  with some coefficients $\gamma_j$ and each $f^j_I$ a $\B'_{\sum_{i \in I} \bar{\delta}_i}$-measurable $(\leq k)$-ary function can satisfy $\norm{\tilde{f}-g}_{L^2 \left(\mu'_{\bar{1}^{k'}} \right)}<\varepsilon$. 

Replacing $\varepsilon$ with $\frac{1}{\lceil 1/\varepsilon\rceil}$, we may assume $\varepsilon=1/E$ for some $E$.  Since the measures $\mu'_i, i \in [k']$ have finite support, for each $i \in [k']$ and $d \in \mathbb{N}$ there exist probability measures $\mu^d_i$ on the $V^d_i$ so that the ultraproduct of $\left( \mu^d_i : d \in \mathbb{N} \right)$ (in the sense of Section \ref{sec: ultraproducts of k-GPS}) is the measure $\mu'_i$.  But then, by assumption, for each $d$ there also exists an approximation $g^{d}=\sum_{j\in[\tilde N(E)]}\gamma^d_j\cdot\prod_{I\in{[k']\choose\leq k}}f^{d,j}_I(\bar x_I)$ with $||f^d-g^d||_{L^2 \left(\mu^d_{\bar{1}^{k'}} \right)}<1/E$ and each $f^{d,j}_{I}$ is $\B^d_{\sum_{i \in I} \bar{\delta}_i}$-measurable, where $\left(V^d_{[k']}, \B^d_{\bar{n}}, \mu^d_{\bar{n}}\right)_{\bar{n} \in \mathbb{N}^{k'}}$ is the $k'$-partite graded probability space with $\B^d_{\bar{n}}$ the algebra of all subsets of $\prod_{i \in [k']} \left(V^d \right)^{n_i}$ and $\mu^d_{\bar{\delta}_i} := \mu^d_i$ for $i \in [k']$.  But then their ultraproduct $\tilde g=\sum_{j\in[\tilde N(E)]}\gamma_j\cdot\prod_{I\in{[k']\choose\leq k}}\tilde f^j_I(\bar x_I)$ satisfies $\norm{ \tilde{f}-\tilde g}_{L^2 \left(\mu'_{\bar{1}^{k'}} \right)}<1/E$, and each $\tilde{f}^j_I$ is $\B'_{\sum_{i \in I} \bar{\delta}_i}$-measurable --- which is a contradiction.
\end{proof}

\section{Correlation and measurability with respect to subalgebras}\label{sec:gowers}
In this section, we develop some aspects of the theory of Gowers' uniformity norms in the context of partite graded probability spaces used throughout the article. Throughout this section, we let $\left( V_{[k]}, \mathcal{B}_{\bar{n}}, \mu_{\bar{n}} \right)_{\bar{n} \in \mathbb{N}^k}$ be a $k$-partite graded probability space. We fix $\bar{n} = (n_1, \ldots, n_k) \in \mathbb{N}^k$, $n := \sum_{i \in [k]} n_i$ and a bounded $\B_{\bar{n}}$-measurable function $f: \prod_{i\in [k]} V^{n_i} \to \mathbb{R}$.
\subsection{Gowers uniformity norms}
Gowers' uniformity norms were introduced in \cite{gowers2001new}. The crucial property is Proposition \ref{prop: Gowers pos iff proj pos} below, which says that they exactly measure correlation with the $\sigma$-algebra $\mathcal{B}_{\bar n,n-1}$; the useful feature is that it lets us test whether $f$ has any correlation with $\mathcal{B}_{\bar n,n-1}$ by evaluating a single integral which only involves $f$.

The material in this section is standard, and the presentation in this subsection closely follows \cite[Section 7.4]{goldbring2014approximate}, however we work in the partite setting and include the details for the sake of completeness.

 \begin{definition}\label{def: gowers norms}
 
 	We define the (partite) \emph{Gowers uniformity seminorm}  of $f$ by 
 	\begin{gather}
 		 \norm{f}_{U^{\bar{n}}}  = \Bigg[\int_{V^{2\bar{n}}} \prod_{\substack{\bar{\alpha}_1 \in \{0,1\}^{n_1},\\ \vdots \\ \bar{\alpha}_n \in \{0,1\}^{n_k}}} f\left( x_{1,1}^{\alpha_{1,1}}, \ldots, x_{1,n_1}^{\alpha_{1,n_1}}, \ldots, x_{k,1}^{\alpha_{k,1}}, \ldots, x_{k,n_k}^{\alpha_{k,n_k}} \right) \nonumber\\
 		 d\mu_{2\bar{n}}\left( x_{1,1}^0, \ldots, x_{1,n_1}^0, x_{1,1}^1, \ldots, x_{1,n_1}^1; \ldots; x_{k,1}^0, \ldots, x_{k,n_k}^0, x^1_{k,1}, \ldots, x^1_{k,n_k} \right) \Bigg]^{\frac{1}{2^{n}}}. \nonumber
 	\end{gather}
      \end{definition}
      The usual Gowers $U^k$-norm is the case where $\vec n=(\underbrace{1,\ldots,1}_{n\text{ times}})$.  More generally, the integral is taken over two copies of $V^{\bar n}$, and given two elements $\bar x^0,\bar x^1\in V^{\bar n}$, the product is taken over all the $2^n$ possible ways to select an element of $V^{\bar n}$ by choosing, separately for each coordinate, whether to take it from the corresponding component of $\bar x^0$ or $\bar x^1$.
 
 Given a tuple $\bar{\alpha} = (\bar{\alpha}_1, \ldots, \bar{\alpha}_k) \in \prod_{i \in [k]} \{0,1 \}^{n_i}$ and $\bar{x} = (\bar{x}_1, \ldots, \bar{x}_k) \in V^{\bar{n}}$, we write $\bar{x}^{\bar{\alpha}} = (\bar{x}_1^{\bar{\alpha}_1}, \ldots, \bar{x}_k^{\bar{\alpha}_k})$, with $\bar{x}_i^{\bar{\alpha}_i} = (x_{i,j}^{\alpha_{i,j}} : j \in [n_i])$ for $i \in [k]$.
 
 \begin{lemma}\label{lem: Gowers norms 1}
 $$\left \lvert \int f d\mu_{\bar{n}} \right \rvert \leq \norm{f}_{U^{\bar{n}}}.$$
 \end{lemma}
 \begin{proof} Let $i \in [k]$ and $j \in [n_i]$ be arbitrary. By Fubini property,
 \begin{gather}
 \left \lvert \int  f (\bar{x}) d \mu_{\bar{n}}(\bar{x}) \right \rvert^{2^n} = \bigg \lvert \int \left(  \int f(\bar{x}) d \mu_{\bar{\delta}_i}(x_{i,j})\right) d \mu_{\bar{n}_{(n_i - 1) \to i}} \nonumber\\ 
 	(\bar{x}_1, \ldots, \bar{x}_{i-1}, (x_{i,1}, \ldots, x_{i,j-1}, x_{i,j+1}, \ldots, x_{i,n_i}), \bar{x}_{i+1}, \ldots, \bar{x}_k) \bigg \rvert^{2^n} \nonumber \\
  \textrm{(by Cauchy-Schwarz and Fubini again)} \nonumber\\
 \leq \Bigg ( \int f\left( \bar{x}_1, \ldots, \bar{x}_{i-1},  (x_{i,1}, \ldots, x_{i,j-1}, x^0_{i,j}, x_{i,j+1}, \ldots, x_{i,n_i}), \bar{x}_{i+1}, \ldots, \bar{x}_k \right) \cdot \nonumber \\
 	\cdot f\left( \bar{x}_1, \ldots, \bar{x}_{i-1},  (x_{i,1}, \ldots, x_{i,j-1}, x^1_{i,j}, x_{i,j+1}, \ldots, x_{i,n_i}), \bar{x}_{i+1}, \ldots, \bar{x}_k \right) \nonumber \\
 	d \mu_{\bar{n}_{(n_i + 1) \to i}}\Bigg)^{2^{n-1}} \nonumber
 \end{gather}
 
 Repeating this process for every pair $i \in [k]$ and $j \in [n_i]$, we arrive at 
\begin{gather}
 \left \lvert \int  f (\bar{x}) d \mu_{\bar{n}}(\bar{x}) \right \rvert^{2^n} \leq \left( \int \prod_{\bar{\alpha} \in \prod_{i \in [k]} \{0,1 \}^{n_i}} f(\bar{x}^{\bar{\alpha}}) d \mu_{2 \bar{n}} \left( \bar{x}^{\bar{0}} \oplus \bar{x}^{\bar 1} \right)\right) =  \norm{f}_{U^{\bar{n}}}. \nonumber
 \end{gather}

 \end{proof}
 
 \begin{lemma} \label{lem: Gowers norms 2}
 	For each $\bar{I} = (I_1, \ldots, I_k)$ with $I_i \subseteq [n_i]$ and $\sum_{i \in [k]} |I_i| = n-1$, let $B_{\bar{I}}$ be a set in $\mathcal{B}_{\bar{n},\bar{I}}$. Then $0 \leq \norm{f \cdot \prod_{\bar{I}} \chi_{B_{\bar{I}}}}_{U^{\bar{n}}} \leq \norm{f}_{U^{\bar{n}}}$. 
 \end{lemma}
 
 \begin{proof}
 	It suffices to show that $0 \leq \norm{f \cdot \chi_{B_{\bar{I}}}}_{U^{\bar{n}}} \leq \norm{f}_{U^{\bar{n}}}$ for a single $\bar{I}$ (as each $\chi_{B_{\bar{I}}}$ takes values in $[0,1]$). We consider $\bar{I}$ with $I_i := [n_i]$ for $i \in [k-1]$ and $I_k := [n_k - 1]$. As $f = f \cdot \chi_{B_{\bar{I}}} + f \cdot \chi_{\neg B_{\bar{I}}}$, we have
 	$$\norm{f}_{U^{\bar{n}}}^{2^n} = \norm{f \cdot \chi_{B_{\bar{I}}} + f \cdot \chi_{\neg B_{\bar{I}}}}_{U^{\bar{n}}}^{2^n},$$
 	which in turn expands into a sum of $2^{2^n}$ terms of the form 
 	\begin{gather}
 	\int \prod_{\bar{\alpha}_1 \in \{0,1 \}^{n_1}, \ldots, \bar{\alpha}_k \in \{0,1 \}^{n_k}} \left(f \cdot \chi_{S_{\bar{\alpha}_1, \ldots, \bar{\alpha}_k}} \right) \left( \bar{x}_1^{\bar{\alpha}_1}, \ldots, \bar{x}_k^{\bar{\alpha}_k} \right) d \mu_{2 \bar{n}} (\bar{x}^{\bar{0}} \oplus \bar{x}^{\bar{1}}), \label{eq: Gowers lemma 1}	
 	\end{gather}
 	where each $S_{\bar{\alpha}_1, \ldots, \bar{\alpha}_k}$ is either $B_{\bar{I}}$ or its complement $\neg B_{\bar{I}}$. Note that $\norm{f \cdot \chi_{B_{\bar{I}}}}_{U^{\bar{n}}}$ is equal to such a term with each $S_{\bar{\alpha}_1, \ldots, \bar{\alpha}_k} = B_{\bar{I}}$. Thus it suffices to show that all of the $2^{2^n}$ terms are non-negative.
 	
 	Assume that $\bar{\alpha}, \bar{\alpha}' \in \prod_{i \in [k]} \{0,1 \}^{n_i}$ are such that $\alpha_{i,j} = \alpha'_{i,j}$ for all $i \in [k]$ and $j \in I_k$, but $S_{\bar{\alpha}} \neq S_{\bar{\alpha}'}$. As $B_{\bar{I}} \in \mathcal{B}_{\bar{n}, \bar{I}}$ (so whether a tuple $\bar{x} \in V^{\bar{n}}$ belongs to it or not does not depend on the coordinate $x_{k, n_k}$ by the choice of $\bar{I}$), for every tuple $\bar{v} = \bar{v}^{\bar{0}} \oplus \bar{v}^{\bar{1}} \in V^{2 \bar{n}}$, we have $\chi_{S_{\bar{\alpha}}} \cdot \chi_{S_{\bar{\alpha}'}}(\bar{v}) = \chi_{S_{\bar{\alpha}}} \cdot \chi_{\neg S_{\bar{\alpha}}}(\bar{v}) = 0$ --- hence the corresponding integral in  \eqref{eq: Gowers lemma 1} is $0$. We thus only need to consider the case where, whenever $\alpha_{i,j} = \alpha'_{i,j}$ for all $i \in [k], j \in I_k$, then $S_{\bar{\alpha}} = S_{\bar{\alpha}'}$. In this case, using Fubini, we have
 		\begin{gather}
 	\int \prod_{\bar{\alpha}_1 \in \{0,1 \}^{n_1}, \ldots, \bar{\alpha}_k \in \{0,1 \}^{n_k}} \left(f \cdot \chi_{S_{\bar{\alpha}}} \right) \left( \bar{x}_1^{\bar{\alpha}_1}, \ldots, \bar{x}_k^{\bar{\alpha}_k} \right) d \mu_{2 \bar{n}} (\bar{x}^{\bar{0}} \oplus \bar{x}^{\bar{1}}) = \nonumber \\
 	\int  \left(\int\prod_{\bar{\alpha}_1 \in \{0,1 \}^{n_1}, \ldots, \bar{\alpha}_k \in \{0,1 \}^{n_k}} \left(f \cdot \chi_{S_{\bar{\alpha}}} \right) \left( \bar{x}_1^{\bar{\alpha}_1}, \ldots, \bar{x}_k^{\bar{\alpha}_k} \right) d \mu_{0, \ldots, 0, 2}\left(x^0_{k,n_k}, x^1_{k,n_k} \right) \right) \nonumber \\
 	 d \mu_{2n_1, \ldots, 2 n_{k-1}, 2n_k - 2} \left(\bar{x}^0_1, \bar{x}^1_1; \ldots; \bar{x}^0_{k-1}, \bar{x}^1_{k-1}; (x^0_{k,1}, \ldots, x^0_{k, n_k -1}), (x^1_{k,1}, \ldots, x^1_{k, n_k -1}) \right) \nonumber\\
 	 = \int \Bigg( \int \prod_{\substack{\bar{\alpha}_1 \in \{0,1 \}^{n_1},\\ \ldots,\\ \bar{\alpha}_{k-1} \in \{0,1 \}^{n_{k-1}},\\ \bar{\alpha}_k \in \{0,1 \}^{n_k - 1}}} \left(f \cdot \chi_{S_{\bar{\alpha}}} \right) \left( \bar{x}_1^{\bar{\alpha}_1}, \ldots, \bar{x}_{k-1}^{\bar{\alpha}_{k-1}}, (x^{\alpha_{k,1}}_{k,1}, \ldots, x^{\alpha_{k, n_{k-1}}}_{k,n_k -1}, x_{k,n_k}) \right) \nonumber\\
 	    d \mu_{0, \ldots, 0, 1}\left(x_{k,n_k}\right)\Bigg)^2  d \mu_{2n_1, \ldots, 2 n_{k-1}, 2n_k - 2}. \nonumber 
 		\end{gather}
 	Since the inside of the integral is always non-negative, this term is non-negative.
 \end{proof}

\begin{definition}
	 We let the function $D(f) : V^{\bar{n}} \to \mathbb{R}$ be defined by 
 $$D(f) \left( \bar{x}^{\bar{0}} \right) := \int \prod_{\bar{\alpha} \in \prod_{i \in [k]} \{0,1\}^{n_i},\bar{\alpha} \neq (\bar{0}, \ldots, \bar{0})} f(\bar{x}_1^{\bar{\alpha}_1}, \ldots, \bar{x}_k^{\bar{\alpha}_k}) d\mu_{\bar{n}}\left(\bar{x}^1 \right).$$
\end{definition}

\begin{remark}\label{rem: Gowers norms 1}
Observe that, by Fubini, $\norm{f}_{U^{\bar{n}}}^{2^n} = \int f \cdot D(f) d\mu_{\bar{n}}\left(\bar{x}^{\bar{0}} \right)$.
\end{remark}
 
 \begin{lemma}\label{lem: Gowers norms 3}
 	The function $D(f)$ is measurable with respect to $\mathcal{B}_{\bar{n}, n-1}$.
 \end{lemma}
 \begin{proof}
 
 Note that, for a fixed $\bar{x}^{\bar{1}} \in V^{\bar{n}}$, the function
 	\begin{gather}\label{eq: Gowers 5}
 	\bar{x}^0 \mapsto \prod_{\substack{\bar{\alpha} \in \prod_{i \in [k]} \{0,1\}^{n_i}\\ \bar{\alpha} \neq (\bar{0}, \ldots, \bar{0})}} f\left( \bar{x}^{\bar{\alpha}} \right)
 	\end{gather} 
 is $\B_{\bar{n}, n-1}$-measurable (as for every such $\bar{\alpha}$, at least one of the coordinates in $\bar{x}^{\bar{\alpha}}$ is then fixed).
 Then $D(f)$ is also $\B_{\bar{n}, n-1}$-measurable by Lemma \ref{lem: meas of av fib}.

 \end{proof}

 \begin{prop}\label{prop: Gowers pos iff proj pos}
 	$\norm{f}_{U^{\bar{n}}} > 0$ if and only if $\norm{\E\left( f \mid \B_{\bar{n},n-1} \right)}_{L^2} > 0$.
 \end{prop}
 \begin{proof}
 If $\norm{f}_{U^{\bar{n}}} >0$, then $\int f \cdot D(f) d\mu_{\bar{n}} (\bar{x}^{\bar{0}}) > 0$ (by Remark \ref{rem: Gowers norms 1}). That is, $f$ is not orthogonal to $D(f)$ in the space $L^2(\B_{\bar{n}})$. As $D(f)$ is $\B_{\bar{n},n-1}$-measurable by Lemma \ref{lem: Gowers norms 3}, we conclude $\norm{\E \left( f \mid \B_{\bar{n}, n-1} \right)} >0$.
 	
 	For the other direction, assume that $\norm{\E\left(f \mid \mathcal{B}_{\bar{n}, n-1} \right)} > 0$. The there exist some sets $B_{\bar{I}} \in \mathcal{B}_{\bar{n},\bar{I}}$, for $\bar{I} = (I_1, \ldots, I_k)$ with $I_i \subseteq [n_i]$ and $\sum_{i \in [k]}|I_i| \leq n-1$, so that $\int f \cdot \prod_{\bar{I}} \chi_{B_{\bar{I}}} d\mu_{\bar{n}} \neq 0$ (as $f$ and its projection onto the subspace of $\mathcal{B}_{\bar{n}, n-1}$-measurable functions are non-orthogonal). Then, by Lemmas \ref{lem: Gowers norms 1} and \ref{lem: Gowers norms 2},
 	$$0 < \left \lvert \int f \prod_{\bar{I}}\chi_{B_{\bar{I}}} d\mu_{\bar{n}} \right\rvert \leq \norm{f \prod_{\bar{I}} \chi_{B_{\bar{I}}}}_{U^{\bar{n}}} \leq \norm{f}_{U^{\bar{n}}}.$$
 	 \end{proof}
 
 \subsection{Subalgebras of fibers}

 We will later need to know when $\mathcal{D}\subseteq\mathcal{B}_{\bar n,n-1}$ is large enough that $\mathbb{E}(f\mid\mathcal{B}_{\bar n,n-1})=\mathbb{E}(f\mid\mathcal{D})$ and, slightly more generally, when $\mathbb{E}(f\mid\mathcal{B}_{\bar n,n-1}\cup G)=\mathbb{E}(f\mid\mathcal{D}\cup G)$ for some set $G$.

 We can determine this by examining the previous subsection more carefully: if $||\mathbb{E}(f\mid\mathcal{B}_{\bar n,n-1})||_{L^2}>0$, we know that it is because $\int f \cdot  D(f) d\mu_{2 \bar{n}}>0$, so it suffices to investigate exactly which sets are needed to ensure that $D(f)$ is $\mathcal{D}$-measurable.  To deal with the more general case, we need to consider not just when $D(f)$ is measurable, but when functions of the form $D(f)\cdot g$ are measurable for a certain class of functions $g$.
 
\begin{definition}\label{def: containing fibers}
Let $\mathcal{D}$ be a $\sigma$-subalgebra of $\B_{\bar{n}}$. 
\begin{enumerate}
\item  Let $\bar{a} \in V^{\bar{n}}$.
We say that $\mathcal{D}$ \emph{contains $\bar{a}$-fibers of $f$} if, for each interval $I \subseteq \mathbb{R}$ and each $i\leq k$ and $j \in [n_i]$, 
$$\left\{ \bar{x} = (x_{i,j} : i \in [k], j \in [n_i]) \in V^{\bar{n}} : f \left( \bar{x}_{a_{i,j} \to (i,j)} \right) \in I \right\} \in \mathcal{D}.$$
Recall that $ \bar{x}_{a_{i,j} \to (i,j)}$ is the tuple obtained from $\bar{x}$ by substituting $a_{i,j}$ into position $(i,j)$ (see Section \ref{sec: notation}).
\item We say that $\mathcal{D}$ \emph{contains $(n-1)$-ary fibers of $f$} if the set of $\bar{a} \in V^{\bar{n}}$ such that $\mathcal{D}$ contains $\bar{a}$-fibers of $f$ has $\mu_{\bar{n}}$-measure $1$.
\item  We say that $\mathcal{D}$ is \emph{closed under fibers} if for every set $B \in \mathcal{D}$, $\mathcal{D}$ contains $(n-1)$-ary fibers of $\chi_{B}$.

	\item   Let $G$ be a set of $\mathcal{B}_{\bar{n}}$-measurable functions.  We say that $\mathcal{D}$ \emph{contains $(n-1)$-ary fibers of $f$ with products from $G$} if, for every function $g$ which is a finite product of functions from $G$, $\mathcal{D}$ contains $(n-1)$-ary fibers of $g$ and $f \cdot g$.
\end{enumerate}
 \end{definition}

\begin{remark}\label{rem: fib alg 1}
\begin{enumerate}
\item If $\mathcal{D}$ is closed under fibers, then for any $\mathcal{D}$-measurable function $f$, $\mathcal{D}$ contains $(n-1)$-ary fibers of $f$ (by assumption this holds for the indicator functions of sets in $\mathcal{D}$, and follows for an arbitrary $\mathcal{D}$-measurable function approximating it by $\mathcal{D}$-simple functions)
	\item The algebra $\B_{\bar{n}, n-1}$ is both closed under fibers (by Fubini and closure under products) and contains $(n-1)$-ary fibers of any $\B_{\bar{n}}$-measurable function (by Fubini property, see Remark \ref{rem: gen Fub}).
\end{enumerate}
\end{remark}

The following is immediate from the definitions (see Definition \ref{def: containing fibers}).
\begin{remark}\label{rem: algebra of fibers}
\begin{enumerate}
\item Each of the algebras $\mathcal{B}^{f, n, \bar{b}}_{\bar{w}}, \B^{f,\bar{b}}, \B^f_{\bar{1}^k, k-1}$ is closed under fibers.

	\item For every $b \in V_{k+1}$ and a tuple $\bar{b}$ in $V_{k+1}$, the algebra $\B^{f,(b)^{\frown}\bar{b}}$ contains $(k-1)$-ary fibers of $f_{b}$ with products from $\left\{ f_{b'} : b' \in \bar{b} \right\}$.
	\item For every  $b$ and  $(b_i)_{i \in I}$ in  $V_{k+1}$, where $I$ is an arbitrary index set, the algebra $\mathcal{B}^{f}_{\bar{1}^k, k-1}$  contains $(k-1)$-ary fibers of $f_b$ with products from $\left\{f_{b_i} : i \in I \right\}$.
\end{enumerate}

\end{remark}

\begin{lemma}\label{lem: fib alg 1}
	  If $||\mathbb{E}(f\mid\mathcal{B}_{\bar{n},n-1})||_{L^2}>0$ and $\mathcal{D}$ contains $(n-1)$-ary fibers of $f$, then $||\mathbb{E}(f\mid\mathcal{D})||_{L^2}>0$.
\end{lemma}
\begin{proof}
  If $||\mathbb{E}(f\mid\mathcal{B}_{\bar{n},n-1})||_{L^2}>0$, then $||f||_{U^{\bar{n}}}>0$ (by Proposition \ref{prop: Gowers pos iff proj pos}).
  If $||f||_{U^{\bar{n}}}>0$ then, by Remark \ref{rem: Gowers norms 1}, we have
  \[0<||f||_{U^{\bar{n}}}^{2^n}= \int f \cdot  D(f) d\mu_{2 \bar{n}}.\]
  
  As $\mathcal{D}$ contains $(n-1)$-ary fibers of $f$, the function in \eqref{eq: Gowers 5} is $\mathcal{D}$-measurable for a measure $1$ set of $\bar{x}^{\bar{1}} \in V^{\bar{n}}$. Hence $D(f)$ is $\mathcal{D}$-measurable by Lemma \ref{lem: meas of av fib}.
Thus $f$ is not orthogonal to $L^2\left( \mathcal{D}  \right)$.
\end{proof}

\begin{lemma}\label{lem: fib alg 3}
	  If $||\mathbb{E}(f\mid\mathcal{B}_{\bar{n},n-1}\cup  G )||_{L^2}>0$ and $\mathcal{D}$ contains $(n-1)$-ary fibers of $f$ with products from $G$, then $||\mathbb{E}(f\mid\mathcal{D} \cup G)||_{L^2}>0$.
\end{lemma}
\begin{proof}
    Suppose $||\mathbb{E}(f\mid\mathcal{B}_{\bar{n},n-1}\cup G)||_{L^2}>0$.  Then there must exist some $g$, a product of finitely many functions from $G$, so that $||\mathbb{E}(f \cdot g\mid\mathcal{B}_{\bar{n},n-1})||_{L^2}>0$, and therefore $||\mathbb{E}(f \cdot g\mid\mathcal{D})||_{L^2}>0$ by Lemma \ref{lem: fib alg 1}, hence $||\mathbb{E}(f\mid\mathcal{D} \cup G)||_{L^2}>0$.
\end{proof}

\begin{lemma}\label{lem: min alg with add functions}
	 If $\mathcal{D}$ is closed under fibers and contains $(n-1)$-ary fibers of $f$ with products from $G$, then
  \[\mathbb{E}(f\mid\mathcal{B}_{\bar{n},n-1}\cup G)=\mathbb{E}(f\mid\mathcal{D} \cup G).\]\end{lemma}
\begin{proof}
  Let $f^- :=f-\mathbb{E}(f\mid\mathcal{D})$.  Consider any $\bar{a} \in V^{\bar{n}}$ such that $\mathcal{D}$ contains $\bar{a}$-fibers of $f$.  Let $g$ be a finite product of functions from $G$. Then, for every $i \in [k], j \in [n_i]$ and $\bar{x} \in V^{\bar{n}}$, we have 
  $$f^- \cdot g \left(\bar{x}_{a_{i,j} \to (i,j)} \right)= f \cdot g \left(\bar{x}_{a_{i,j} \to (i,j)} \right)-\mathbb{E} \left(f\mid\mathcal{D} \right) \cdot g \left(\bar{x}_{a_{i,j} \to (i,j)} \right).$$
    For any interval $I$, the sets 
    $$\left\{\bar x\mid f \cdot g\left(\bar{x}_{a_{i,j} \to (i,j)}\right)\in I\right\}, \left\{\bar x\mid g\left(\bar{x}_{a_{i,j} \to (i,j)}\right)\in I\right\}$$
     are both in $\mathcal{D}$, as $\mathcal{D}$ contains $(n-1)$-ary fibers of $f$ with products from $G$. And 
    $\{\bar x\mid\mathbb{E}(f\mid\mathcal{D})(\bar{x}_{a_{i,j} \to (i,j)})\in I\}$ also belongs to $\mathcal{D}$ by Remark \ref{rem: fib alg 1}(1), as $\mathcal{D}$ is closed under fibers. So, by taking unions and intersections of such sets, $\{\bar{x} \mid f^-(\bar{x}_{a_{i,j} \to (i,j)})\in I\}$ belongs to $\mathcal{D}$ as well, hence $\mathcal{D}$ contains $(n-1)$-ary fibers of $f^-$ with products from $G$.

 If $\mathbb{E}(f\mid\mathcal{D} \cup G)\neq \mathbb{E}(f\mid\mathcal{B}_{\bar{n},n-1} \cup G)$, then 
  $$\norm{\E \left( f^- \mid \B_{\bar{n},n-1} \cup G \right)}_{L^2} = \norm{\E (\left(f \mid \B_{\bar{n},n-1} \cup G \right)}_{L^2} - \norm{\E(f \mid \mathcal{D} \cup G)} > 0.$$
   Hence $||\mathbb{E}(f^-\mid\mathcal{D} \cup G)||_{L^2}>0$ by Lemma \ref{lem: fib alg 1}, which is a contradiction to the choice of $f^{-}$.
\end{proof}

\section{Indiscernible sequences of random variables}\label{sec:indiscernibles}

In this section we gather the model theoretic compactness arguments we need and providing the necessary background on ultraproducts and indiscernible sequences. We also prove a couple of de Finetti-style results that are used in the proof of the main theorem.
  
\subsection{Generic $k$-partite $k$-uniform hypergraphs}
We define some classes of ordered partite hypergraphs and related structures, and discuss their basic model-theoretic properties (see \cite{chernikov2014n} for further discussion).

\begin{definition}\label{def: generic hypergraph}
For $k \in \mathbb{N}_{\geq 1}$, let $G_{k,p}$ denote the countable \emph{generic $k$-partite $k$-uniform ordered hypergraph}, viewed as the unique countable first-order structure in the language $\mathcal{L}_{\opg}^k = (R_k,P_1, \ldots, P_k, <)$ with the underlying set $G$ satisfying the following first-order $\mathcal{L}_{\opg}^k$-theory $T^k_{\opg}$:
\begin{enumerate}
\item $P_1, \ldots, P_k$ are unary predicates giving a partition of $G$;
\item $R_k \subseteq \prod_{i \in [k]} P_i$;
\item $<$ is a total linear order on $G$ and $P_1 < \ldots < P_k$;
\item $(P_i, <\restriction_{P_i})$ is a dense linear ordering for each $i \in  [k]$;
\item
for every $j \in [k]$, any finite disjoint sets 
$A_0,A_1\subset {\prod_{i \in [k] \setminus \{j\}}P_i}$
 and
$b_0<b_1\in P_j$, there exists some $b_0<b<b_1$ such that 
\begin{gather*}
	G_{k,p} \models R_k(a_1, \ldots , a_{j-1}, b, a_{j+1}, \ldots, a_{k}) \iff\\ \bar{a} = (a_1, \ldots, a_{j-1}, a_{j+1}, \ldots, a_k) \in A_0
\end{gather*}
 for all $\bar{a} \in A_0 \cup A_1$.

\end{enumerate}

We also let $\mathcal{G}_{k,p}$ be the class of all finite $k$-partite $k$-uniform ordered hypergraphs (i.e.~$\mathcal{G}_{k,p}$ is the class of all finite $\mathcal{L}_{\opg}^k$-structures satisfying axioms (1)--(3) in Definition \ref{def: generic hypergraph}).

\end{definition}

\begin{definition}\label{def: reducts of Gkp}
\begin{enumerate}
	\item We denote by $O_{k,p}$ the reduct of $G_{k,p}$ to the language $\mathcal{L}^k_{\ord} = (P_1, \ldots, P_k, <)$ (i.e.~the structure obtained from $G_{k,p}$ by forgetting the edge relation).
	We let $T^k_{\ord}$ be the $\mathcal{L}^k_{\ord}$-theory consisting of (1),(3) and (4) in Definition \ref{def: generic hypergraph}; and $\mathcal{O}_{k,p}$ be the class of all finite $\mathcal{L}^k_{\ord}$-structures satisfying (1) and (3).
	\item We let $G'_{k,p}$ the reduct of $G_{k,p}$ to the language  $\mathcal{L}^k_{\pg} = (R_k, P_1, \ldots, P_k)$ (i.e.~the structure obtained from $G_{k,p}$ by forgetting the ordering).	 We let $T^{k}_{\pg}$ be the $\mathcal{L}^k_{\pg} = (R_k, P_1, \ldots, P_k)$-theory consisting of (1), (2) and the infinite set of sentences expressing the following:
	\begin{enumerate}
	\item[$(5)'$] for every $j \in [k]$ and any finite disjoint sets 
$A_0,A_1\subset {\prod_{i \in [k] \setminus \{j\}}P_i}$
 there exists some $b \in P_j$ such that 
\begin{gather*}
	G_{k,p} \models R_k(a_1, \ldots , a_{j-1}, b, a_{j+1}, \ldots, a_{k}) \iff\\ \bar{a} = (a_1, \ldots, a_{j-1}, a_{j+1}, \ldots, a_k) \in A_0
\end{gather*}
 for all $\bar{a} \in A_0 \cup A_1$.
	\end{enumerate}
	Finally, we let $\mathcal{G}'_{k,p}$ be the class of all finite $\mathcal{L}^k_{\pg} $-structures satisfying (1) and (2) in Definition \ref{def: generic hypergraph}.
\end{enumerate}

\end{definition}

\begin{definition}\label{def: induced substr}
	Given a structure $\mathcal{M} = (M, (R_i)_{i \in I})$ in a relational language $\mathcal{L} = (R_i : i \in I)$, with $R_i$ a relational symbol of arity $n_i$, and $A \subseteq M$, we let $\mathcal{M}|_A := \left( A, (R_i \cap A^{n_i})_{i \in I} \right)$ be the \emph{substructure induced on $A$}.
\end{definition}

The following is well-known (we refer to e.g.~\cite[Chapter 7.1]{hodges1993model} for the details).
\begin{fac}\label{fac: Ramsey classes}
\begin{enumerate}
\item Each of the theories $T^k_{\opg}, T^{k}_{\pg}$ and $T^k_{\ord}$ is complete, has quantifier elimination, and is $\aleph_0$-categorical (i.e.~ there exists a unique, up to isomorphism, countable structure satisfying the corresponding theory).
	\item $G_{k,p}$ ($G'_{k,p}$, $O_{k,p}$) is the Fra\"iss\'e limit of 	$\mathcal{G}_{k,p}$ ($\mathcal{G}'_{k,p}$, $\mathcal{O}_{k,p}$, respectively).
	\item In particular, $G_{k,p}$ embeds every countable $k$-partite $k$-uniform ordered hypergraphs as an induced substructure; and its finite induced substructures, up to isomorphism, are precisely the structures in $\mathcal{G}_{k,p}$. Analogous statements hold for $G'_{k,p}$, $O_{k,p}$.
	\item Each of the structures $G_{k,p}, G'_{k,p}$, $O_{k,p}$ is \emph{ultrahomogeneous}, i.e.~every isomorphism between two finite induced substructures extends to an isomorphism of the whole structure.
	\end{enumerate}
\end{fac}

The following property will be important in Section \ref{sec: Aldous-Hoover}. 
\begin{definition}\cite[Definition 2.17]{crane2018relatively}\label{def: nDAP}
	Let $\mathcal{K}$ be a collection of finite structures in a relational language $\mathcal{L}$. For $n \in \mathbb{N}_{\geq 1}$, we say that $\mathcal{K}$ satisfies the \emph{$n$-disjoint amalgamation property} (\emph{$n$-DAP}) if for every collection of $\mathcal{L}$-structures $\left( \mathcal{M}_i = (M_i, \ldots)  : i \in [n] \right)$ so that each $\mathcal{M}_i$ is isomorphic to some structure in $\mathcal{K}$, $M_i = [n] \setminus \{ i \}$ and $\mathcal{M}_i|_{[n] \setminus \{i,j\}} = \mathcal{M}_j|_{[n] \setminus \{i,j\}}$ for all $i \neq j \in [n]$, there exists an $\mathcal{L}$-structure $\mathcal{M} = (M, \ldots)$ isomorphic to some structure in $\mathcal{K}$, and such that $M = [n]$ and $\mathcal{M}|_{[n] \setminus \{ i \}} = \mathcal{M}_i$ for every $1 \leq i \leq n$.
	
	We say that an $\mathcal{L}$-structure $\mathcal{M}$ satisfies $n$-DAP if the collection of its finite induced substructures does.
\end{definition}

\begin{prop}\label{prop: Gkp has nDAP}
	$G'_{k,p}$ satisfies $n$-DAP for all $n \in \mathbb{N}_{\geq 1}$.
\end{prop}
\begin{proof}
	Fix $k \geq 2$. By Fact \ref{fac: Ramsey classes}, we need to show that the class of finite structures $\mathcal{G}'_{k,p}$ satisfies $n$-DAP. Let $n \in \mathbb{N}$ and $\left( \mathcal{M}_i : i \in [n] \right)$ with $\mathcal{M}_i \in G'_{k,p}$ as in Definition \ref{def: nDAP} be given. In particular, each $\mathcal{M}_i$ satisfies (1) and (2) in Definition \ref{def: generic hypergraph}.
	Then 
	\begin{gather}\label{eq: nDAP1}
		P_i^{\mathcal{M}_j} \cap P_{i'}^{\mathcal{M}_{j'}} = \emptyset \textrm{ for every } i \neq i' \in [k] \textrm{ and } j,j' \in [n]. 
	\end{gather}
	
	Indeed, assume $\ell \in [n]$ is such that $\ell \in P_i^{\mathcal{M}_j} \cap P_{i'}^{\mathcal{M}_{j'}}$. If $j \neq j'$, then necessarily $\ell \in [n] \setminus \{j,j'\}$. By assumption $\mathcal{M}_j|_{[n] \setminus \{j,j'\}} = \mathcal{M}_{j'}|_{[n] \setminus \{j,j'\}}$, hence $\ell \in P_i^{\mathcal{M}_j} \cap P_{i'}^{\mathcal{M}_{j}}$. But this is impossible as $\mathcal{M}_j$ satisfies (1) of Definition \ref{def: generic hypergraph}.
	Also 
	\begin{gather}\label{eq: nDAP2}
		\ell \in [n] \implies \ell \in P_i^{\mathcal{M}_j} \textrm{ for some } i \in [k], j \in [n].
	\end{gather}
	Indeed, if $\ell \in [n]$, then $\ell \in M_j$ for any $j \in [n] \setminus \{ \ell \}$, hence belongs to $P_i^{\mathcal{M}_j}$ for some $i \in [k]$ as $\left( P_i^{\mathcal{M}_j} \right)_{i \in [k]}$ is a partition of $M_j$ by assumption.
	
	For $i \in [k]$, we let $P_i^{\mathcal{M}} := \bigcup_{j \in [n]} P_i^{\mathcal{M}_j}$. Then the sets $P_1^{\mathcal{M}}, \ldots, P_k^{\mathcal{M}}$ give a partition of $M = [n]$ by \eqref{eq: nDAP1} and \eqref{eq: nDAP2}.
	
	We let $R_k^{\mathcal{M}} := \bigcup_{j \in [n]} R_k^{\mathcal{M}_j}$. As $R^{\mathcal{M}_j}_k \subseteq \prod_{i \in [k]} P_i^{\mathcal{M}_j}$ for every $j \in [n]$ by assumption, it follows that $R_k^{\mathcal{M}} \subseteq \prod_{i \in [k]} P_i^{\mathcal{M}}$.
	Hence the structure $\mathcal{M} := \left(M, \left(P_i^{\mathcal{M}} \right)_{i \in [k]}, R_k^{\mathcal{M}}  \right)$ satisfies (1) and (2) of Definition \ref{def: generic hypergraph}, hence $\mathcal{M} \in G'_{k,p}$.
\end{proof}

\begin{remark}
$O_{k,p}$ (and hence $G_{k,p}$) do not satisfy $3$-DAP.
\end{remark}

\subsection{Generalized indiscernibles}

Many combinatorial arguments around $\VC_k$-dimension can be considerably  simplified using a combination of structural Ramsey theory and logical compactness,  encapsulated in the model-theoretic notion of \emph{generalized indiscernible sequences} (this method does not typically provide strong bounds however).

\begin{definition} Let $\mathcal{M}$ be a first-order structure in a language $\mathcal{L}$.
\begin{enumerate}
\item Let $I$ be a structure in a language $\mathcal{L}_0$.
We say that a collection $\left(a_{i}\right)_{i\in I}$ of tuples in $\mathcal{M}$ is \emph{$I$-indiscernible over a set of parameters $C \subseteq \mathcal{M}$} if for
all $n\in\mathbb{N}$ and all $i_{0},\ldots,i_{n}$ and $j_{0},\ldots,j_{n}$
from $I$ we have:
$$
\qftp_{\mathcal{L}_0}\left(i_{0}, \ldots, i_{n}\right)=\qftp_{\mathcal{L}_0}\left(j_{0},\ldots,j_{n}\right)
\Rightarrow $$
$$\tp_{\mathcal{L}}\left(a_{i_{0}}, \ldots, a_{i_{n}} / C\right)=\tp_{\mathcal{L}}\left(a_{j_{0}}, \ldots, a_{j_{n}} / C \right)
.$$

	\item For two $\mathcal{L}_0$-structures $I$ and $J$, we say that a collection of tuples $\left(b_{i}\right)_{i\in J}$ in $\mathcal{M}$
is \emph{based on} a collection of tuples $\left(a_{i}\right)_{i\in I}$ in $\mathcal{M}$ over a set of parameters $C \subseteq \mathcal{M}$ if for any finite
set $\Delta$ of $\mathcal{L}(C)$-formulas,  and for any finite tuple $\left(j_{0},\ldots,j_{n}\right)$
from $J$ there is a tuple $\left(i_{0},\ldots,i_{n}\right)$ from
$I$ such that:

\begin{itemize}
\item $\qftp_{\mathcal{L}_0}\left(j_{0},\ldots,j_{n}\right)=\qftp_{\mathcal{L}_0}\left(i_{0},\ldots,i_{n}\right)$ and
\item $\tp_{\Delta}\left(b_{j_{0}},\ldots,b_{j_{n}}\right)=\tp_{\Delta}\left(a_{i_{0}},\ldots,a_{i_{n}}\right)$.
\end{itemize}

\end{enumerate}

\end{definition}

\begin{definition}
	When $(I,<)$ is an arbitrary linear order and $(a_i)_{i \in I}$ is a sequence of finite tuples in $\mathcal{M}$, we say that the sequence $(a_i)_{i \in I}$ is \emph{indiscernible} (indiscernible over $C$) if $(a_i)_{i \in I}$ is $(I,<)$-indiscernible over $\emptyset$ (over $C$).
\end{definition}

%
%

The following is standard, relying on the usual Ramsey theorem for (1), and on the fact that finite ordered partitioned hypergraphs form a \emph{Ramsey class} \cite{nevsetvril1977partitions, abramson1978models, nevsetvril1983ramsey} for (2).

\begin{fac}\label{fac: existence of indiscernibles}
	Let $\mathcal{L}$ be a countable language, $\mathcal{M}$ an $\aleph_1$-saturated $\mathcal{L}$-structure and $C \subseteq \mathcal{M}$ a countable subset. 
	
\begin{enumerate}
\item (see e.g.~\cite[Lemma 5.1.3]{tent2012course}) For every countable infinite linear orders $I$ and $J$ and a sequence $(a_i)_{i \in I}$ of finite tuples in $\mathcal{M}$, there exists some sequence $(b_i)_{i \in J}$ of tuples in $\mathcal{M}$ indiscernible over $C$ and based on $(a_i)_{i \in I}$.
\item \cite[Corollary 4.8]{chernikov2014n} For any $k \in \mathbb{N}_{\geq 1}$ and a collection of finite tuples $\left(a_{g}\right)_{g\in G_{k,p}}$ in $\mathcal{M}$, there is some collection of finite tuples $\left(b_{g}\right)_{g\in G_{k,p}}$ in $\mathcal{M}$ which is $G_{k,p}$-indiscernible over $C$ and is based on $\left(a_{g}\right)_{g\in G_{k,p}}$ over $C$. The same holds with $O_{k,p}$ instead of $G_{k,p}$ everywhere.
\end{enumerate}
\end{fac}

%

\subsection{Ultraproducts of functions on partite graded probability spaces}\label{sec: ultraproducts of k-GPS}

We assume familiarity with ultraproducts of first-order structures and the construction of Loeb's measure. There are multiple ways to make sense of ultraproducts and compactness of measure spaces and measurable functions (Keisler's \emph{probability logic} \cite{keisler1985chapter} and its variants, \emph{continuous logic} \cite{yaacov2008model}, \emph{AML logic} \cite{goldbring2014approximate}, etc.), but here we use the most basic approach relying on the familiar ultraproduct construction for first-order logic (and similar to the one used e.g.~by Hrushovski in \cite{hrushovski2012stable}).

\begin{definition}\label{def: k-GPS as a FO struct}
		Assume that $\mathfrak{P} = \left( V_{[k]}, \B_{\bar{n}}, \mu_{\bar{n}} \right)_{\bar{n} \in \mathbb{N}^k}$ is a $k$-partite graded probability space, $I$ is a \emph{countable} set and  $\bar{f} = \left(f_{\alpha} : \alpha \in I \right), f_{\alpha}: \prod_{i \in [k]} V_i \to [0, 1]$ is a collection of $\B_{\bar{1}^k}$-measurable functions. We associate to it a $k$-sorted first-order structure $\mathcal{M}_{\mathfrak{P},\bar{f}}$ in a language $\mathcal{L}_{\infty,I}$ (or just $\mathcal{L}_{\infty}$ when $I$ is clear from the context) with sorts $V_1, \ldots, V_k$ in the following way.
		
	For every $q \in \mathbb{Q}^{[0,1]}$ and $\alpha \in I$, $\mathcal{L}_0$ contains a $k$-ary relational symbol 
	$$F^{< q}_{\alpha}(x_1, \ldots, x_k)$$
	 with the variable $x_i$ of sort $V_i$, interpreted in $\mathcal{M}_{\mathfrak{P},\bar{f}}$ via
	$$\mathcal{M}_{\mathfrak{P},f} \models F^{<q}_{\alpha}(x_1, \ldots, x_k) ~ :\iff ~ f_{\alpha}(x_1, \ldots, x_k) < q$$
	for any $(x_1, \ldots, x_k) \in V^{\bar{1}^k}$. We write $F^{\geq  q}_{\alpha}$ as an abbreviation for $\neg F^{< q}$. Note that for every $q,\alpha$, the set $\left\{\bar{b} \in V^{\bar{1}^k} : \mathcal{M}_{\mathfrak{P},\bar{f}} \models F^{<q}_{\alpha}(\bar{b}) \right\}$ is in $\B_{\bar{1}^k}$ by measurability of $f_{\alpha}$.
	
By induction on $i \in \mathbb{N}$, we define a \emph{countable} language $\mathcal{L}_i$ as follows. In addition to all the symbols in $\mathcal{L}_i$, for every \emph{quantifier-free} $\mathcal{L}
_i$-formula $\varphi(\bar{x}, \bar{y})$ such that the tuple $\bar{x}$ corresponds to $V^{\bar{n}}, \bar{n} \in \mathbb{N}^k$ and $r \in \mathbb{Q}$, we add to $\mathcal{L}_{i+1}$ a new relational symbol $m_{\bar{x}} < r. \varphi(\bar{x}, \bar{y})$ with free variables $\bar{y}$, interpreted by: for every tuple $\bar{b}$ corresponding to $\bar{y}$,
	$$\mathcal{M}_{\mathfrak{P},\bar{f}} \models m_{\bar{x}} < r. \varphi(\bar{x}, \bar{b}) ~ :\iff \mu_{\bar{n}} \left( \varphi(\bar{x}, \bar{b}) \right) < r,$$
	where as usual $\varphi(\bar{x},\bar{b}) = \{ \bar{a} \in V^{\bar{n}} \mid  \mathcal{M}_{\mathfrak{P},\bar{f}} \models \varphi(\bar{a}, \bar{b}) \}$ is the set defined by the corresponding instance of $\varphi$
	(note that this set is $\mu_{\bar{n}}$-measurable by Fubini property in $\mathfrak{P}$ and induction). Let $\mathcal{L}_{\infty} := \bigcup_{i \in \mathbb{N}} \mathcal{L}_i$.
	We will write $m_{\bar{x}} \geq r$ as an abbreviation for $\neg m_{\bar{x}} < r$.
	
	We also write $\mathcal{M}_{\bar{f}}$ to denote the $\mathcal{L}_0$-reduct of $\mathcal{M}_{\mathfrak{P},\bar{f}}$.
\end{definition}

Now assume that for each $j \in \mathbb{N}$, $\mathfrak{P}^j = \left( V^j_{[k]}, \B^j_{\bar{n}}, \mu^j_{\bar{n}} \right)_{\bar{n} \in \mathbb{N}^k}$ is a $k$-partite graded probability space and $f^j_{\alpha}: \prod_{i \in [k]}V^j_i \to [0,1]$ is a  $\B^j_{\bar{1}^{k}}$-measurable function for $\alpha \in I$.
Let $\mathcal{U}$ be a non-principal ultrafilter on $\mathbb{N}$.

For $i \in [k]$, we let $\tilde{V}_i := \prod_{j \in \mathbb{N} } V^j_i / \cU$. Then for any $\bar{n} \in \mathbb{N}^k$, $\tilde{V}^{\bar{n}}$ is naturally identified with $\prod_{j \in \mathbb{N}}\left( \prod_{i \in [k]} \left(V_i^j\right)^{n_i} \right) /\cU$.

We let $\tilde{\mathcal{M}} := \prod_{j \in \mathbb{N}} \mathcal{M}_{\mathfrak{P}_j, \bar{f}^j} / \cU$ (i.e., the usual ultraproduct of $\mathcal{L}_{\infty}$-structures). 

For $\alpha \in I$, we define a function $\tilde{f}_{\alpha}: \tilde{V}^{\bar{1}^k} \to [0,1]$  via $\tilde{f}_{\alpha}(\bar{x}) := \inf\{q \in \mathbb{Q}^{[0,1] }: \tilde{M
} \models F^{<q}_{\alpha}(\bar{x}) \}$ (and refer to it as the ultraproduct of $f^j_{\alpha}$'s with respect to $\mathcal{U}$). 

For $\bar{n} \in \mathbb{N}^k$, we let $\tilde{\B}_{\bar{n}}^0$ consist of all subsets of $\tilde{V}^{[\bar{n}]}$ of the form $X = \prod_{j \in \mathbb{N}} X_j / \cU$ for some $X_j \in \B^j_{\bar{n}}$.

For such a set $X$, we define $\tilde{\mu}^0_{\bar{n}} (X) := \lim_{j \to \cU} \mu^j_{\bar{n}}(X_j) \in [0,1]$.

We let $\tilde{\B}_{\bar{n}}$ be the $\sigma$-algebra of subsets of $\tilde{V}^{\bar{n}}$ generated by $\tilde{\B}_{\bar{n}}^0$. 

As in the standard construction of Loeb's measure, we have the following fact.
\begin{fac}\label{fac: ultraproduct props}
	\begin{enumerate}
	\item  For every $\bar{n} \in \mathbb{N}^k$, $\tilde{\mu}^0_{\bar{n}}$ is a finitely-additive probability measure on the Boolean algebra $\tilde{\B}^0_{\bar{n}}$.		\item $\tilde{\mathcal{M}}$ is an \emph{$\aleph_1$-saturated} $\mathcal{L}_{\infty}$-structure (in particular, for every finite tuple of variables $\bar{x}$ and a countable collection of $\mathcal{L}_{\infty}$-formulas $\varphi_i(\bar{x},\bar{b}_i)$, with $\bar{b}_i$ an arbitrary tuple of parameters from $\tilde{\mathcal{M}}$, if every finite subset of $\{\varphi_i(\bar{x},\bar{b}_i) : i \in \mathbb{N} \}$ is realized by some tuple in $\tilde{\mathcal{M}}$, then the whole set is realized by some tuple in $\tilde{\mathcal{M}}$).
		\item For every $\bar{n} \in \mathbb{N}^k$, there exists a unique countably-additive probability measure $\tilde{\mu}_{\bar{n}}$ on $\tilde{\B}_{\bar{n}}$ extending $\tilde{\mu}^0_{\bar{n}}$.
		\item $\tilde{\mathfrak{P}} :=\left(\tilde{V}_{[k]}, \tilde{\B}_{\bar{n}}, \tilde{\mu}_{\bar{n}} \right)_{\bar{n} \in \mathbb{N}^{k}}$ is a $k$-partite graded probability space.
		\item Fot every $r \in \mathbb{Q}^{[0,1]}, \alpha \in I$ and $\bar{x} \in \tilde{V}^{\bar{1}^k}$, we have
		$$\tilde{f}_{\alpha}(\bar{x}) < r \implies \tilde{\mathcal{M}} \models F^{<r}_{\alpha}(\bar{x}) \implies \tilde{f}_{\alpha}(\bar{x}) \leq r.$$
		\item For every $r \in \mathbb{Q}^{[0,1]}$, $\bar{n} \in \mathbb{N}^k$, $\varphi(\bar{x}, \bar{y})$ a quantifier-free $\mathcal{L}_\infty$-formula with $\bar{x}$ corresponding to $\tilde{V}^{\bar{n}}$ and $\bar{b}$ a tuple from $\tilde{\mathcal{M}}$, we have 
		$$\tilde{\mu}_{\bar{n}}\left( \varphi(\bar{x}, \bar{b}) \right) < r \implies \tilde{\mathcal{M}} \models m_{\bar{x}} < r . \varphi(\bar{x}, \bar{b}) \implies \tilde{\mu}_{\bar{n}} \left( \varphi(\bar{x}, \bar{b}) \right) \leq r.$$
		\item The functions $\tilde{f}_{\alpha}$ are $\tilde{\B}_{\bar{1}^k}$-measurable.
	\end{enumerate}
\end{fac}
Here (1), (5)  and (6) hold by \L os' theorem and basic properties of ultralimits; (2) is a standard model-theoretic fact; (3) follows from $\aleph_1$-saturation restricting to any countable sublanguage and Carath\'eodory's extention theorem; (4) is a routine verification, e.g.~to check that Fubini property holds in the ultraproduct, one approximates the integral by a sum of $\tilde{\B}^0_{\bar{n}}$-simple functions, and these are arbitrary close to satisfying Fubini by \L os and the assumption that each $\mathfrak{P}_j$ satisfies Fubini; (7) holds as $\left\{ \bar{x} \in \tilde{V}^{\bar{1}^k} : \tilde{f}(\bar{x}) < r \right\} = \bigcup_{\varepsilon \in \mathbb{Q}_{>0}} \left\{ \bar{x} \in \tilde{V}^{\bar{1}^k} : \tilde{M} \models F^{< r- \varepsilon}(\bar{x}) \right\}$ by (5), and every set on the right is in $\tilde{\B}^0_{\bar{1}^k}$. 

The following subtle point can be  mostly  ignored in the conclusions, but we will have to keep track of it in the proofs.
\begin{remark}\label{rem : UP is prop to MP 2}
	Note that the interpretation of the $F^{<r}_{\alpha}$ and $m_{\bar{x}} < r$ predicates may differ in $\mathcal{M}_{\tilde{\mathfrak{P}}, \bar{\tilde{f}}}$ and $\tilde{\mathcal{M}}$, but not by much: due to Fact \ref{fac: ultraproduct props}(5) and (6), we have  $\tilde{\mathcal{M}} \propto \mathcal{M}_{\tilde{\mathfrak{P}}, \bar{\tilde{f}}}$ in the sense of the following definition.
\end{remark}

\begin{definition}
	Let $\mathcal{M}, \mathcal{M}'$ be two $\mathcal{L}_{\infty}$-structures. We write $\mathcal{M} \propto \mathcal{M}'$ if the structures $\mathcal{M}, \mathcal{M}'$ have the same underlying sorts $V_1, \ldots, V_k$, and 
	for every $\alpha \in I$, $r \in \mathbb{Q}^{[0,1]}$ and $\varepsilon \in \mathbb{Q}_{>0}$ so that $r+\varepsilon \leq 1$ we have
\begin{gather*}
	\mathcal{M}' \models F^{<r}_{\alpha}(\bar{b}) \Rightarrow  \mathcal{M} \models F^{<r}_{\alpha}(\bar{b}) \Rightarrow 	\mathcal{M}' \models F^{<(r+\varepsilon)}_{\alpha}(\bar{b}) \textrm{ and}\\
		\mathcal{M}' \models m_{\bar{x}} < r . \varphi(\bar{x}, \bar{b}) \Rightarrow \mathcal{M} \models m_{\bar{x}} < r . \varphi(\bar{x}, \bar{b}) \Rightarrow
		\mathcal{M}' \models m_{\bar{x}} < \left(r + \varepsilon\right) . \varphi(\bar{x}, \bar{b})
\end{gather*}
for every  quantifier-free $\mathcal{L}_{\infty}$-formula $\varphi(\bar{x}, \bar{y})$ and a tuple $\bar{b}$ from $\mathcal{M}$ of appropriate length.

If $\mathcal{M}, \mathcal{M}'$ are just $\mathcal{L}_{0}$-structures, we write $\mathcal{M} \propto \mathcal{M}'$ when the first of these two conditions is satisfied.
\end{definition}

\subsection{Lemmas on indiscernible sequences}
Throughout this section, $k \in \mathbb{N}_{\geq 1}$, $\mathfrak{P} = \left(V_{[k]}, \B_{\bar{n}}, \mu_{\bar{n}} \right)_{\bar{n} \in \mathbb{N}^k}$ is a $(k+1)$-partite graded probability space and $f: V^{\bar{1}^{k+1}} \to [0,1]$ is a $ \B_{\bar{1}^{k+1}}$-measurable function. We let $\mathcal{M}_{\mathfrak{P},f}$ be the associated $\mathcal{L}_{\infty}$-structure and let $\mathcal{M}'$ be some $\mathcal{L}_{\infty}$-structure satisfying $\mathcal{M}' \propto \mathcal{M}_{\mathfrak{P},f}$ (Definition \ref{def: k-GPS as a FO struct}). We verify that various probabilistic conditions on the fibers of $f$ are \emph{type-definable} in $\mathcal{M}'$, via appropriate finitary approximations, and prove some lemmas on indiscernible sequences in the spirit of the classical de Finetti's theorem on exchangeable sequences of random variables.

\begin{definition}

A set $X \subseteq V^{\bar{n}}$ is \emph{type-definable} in an $\mathcal{L}_{\infty}$-structure $\mathcal{M}'$ if there exists a countable set $\left\{\varphi_i(\bar{x}, \bar{b}_i) : i \in \mathbb{N} \right\}$ of $\mathcal{L}_{\infty}$-formulas with the tuple of variables $\bar{x}$ corresponding to  $V^{\bar{n}}$ and parameters in $\mathcal{M}'$ so that $X = \left\{\bar{a} \in  V^{\bar{n}} : \mathcal{M}' \models \varphi_i(\bar{a}, \bar{b}_i) \textrm{ for all } i \in \mathbb{N} \right\}$.
%
%
\end{definition}

\begin{remark}\label{rem: fib alg unif def}
	The $\sigma$-algebra $\mathcal{B}^{f}_{\bar{1}^k,k-1}$ (recall Definition \ref{def: fiber algebra}(5)) has a generating set  that is uniformly definable in $\mathcal{M}'$. Namely, given $q \in \mathbb{Q}^{[0,1]}$, we consider the $\mathcal{M}'$-definable set 
	$$F^{<q} := \left\{\bar{x} \in V^{\bar{1}^{k+1}} : \mathcal{M}' \models F^{<q}(\bar{x}) \right\}.$$	
	Using $\mathcal{M}' \propto \mathcal{M}_{\mathfrak{P}, f}$ we have $f^{<r} = \bigcup_{\varepsilon \in \mathbb{Q}_{>0}} F^{< r - \varepsilon}$ and $F^{< r} = \bigcup_{\varepsilon \in \mathbb{Q}_{>0}} f^{r-\varepsilon}$, hence $\{F^{<q} : q \in \mathbb{Q}^{[0,1]} \}$ is a generating set for $\sigma(f)$.
	
Now, for each $\emptyset \neq I \in \binom{[k]}{\leq k-1}$ and $q \in \mathbb{Q}^{[0,1]}$, we consider the quantifier-free $\mathcal{L}_0$-formula 
$$\varphi_{I,q} \left( \bar{x}, \bar{y} \right) := F^{<q} \left(  \bar{x}_{y_i \to x_i, i \in I} \ ^{\frown}  ( y_{k+1} )  \right),$$
where $\bar{x}$ is a tuple of variables corresponding to $V^{\bar{1}^k}$ and $\bar{y}$ is a tuple of variables corresponding to $V^{\bar{1}^{k+1}}$.

\noindent Then, for any $a \in V_{k+1}$, every set in $\B^{f,a}$ (see Definition \ref{def: fiber algebra}) is in the $\sigma$-algebra generated  by the sets of solutions of $\varphi_{I,q} \left( \bar{x}, \bar{b}^{\frown} (a) \right)$ in $\mathcal{M}_{\mathfrak{P}}$ for some $I  \in  K := \binom{[k]}{\leq k-1} \setminus \{\emptyset \}$, $q \in \mathbb{Q}^{[0,1]}$ and $\bar{b} \in V^{\bar{1}^k}$.	
\end{remark}

This allows us to uniformly define various other algebras and their generating sets.

%
\begin{definition}\label{def: various formulas}
Given $n \in \mathbb{N}$, let 
	\begin{gather*}
		S_{n} := \left\{ s \mid s: [n] \times K \times \mathbb{Q}^{[0,1]}_n \to \{-1, 1\} \right\}, \\
		U_{n} := \left\{ u \mid u: [n] \times \mathbb{Q}_n^{[0,1]} \to \{-1, 1\} \right\}, \textrm{ and}\\
		Q_{n} := S_{n} \times S_n \times U_{n}.
	\end{gather*}

 Given $n$ and $(s,t,u) \in Q_n$, we consider the quantifier-free $\mathcal{L}_0$-formula 
\begin{gather*}
	\varphi^{s,t,u}(\bar{x}; \bar{y}_1, \ldots, \bar{y}_n; z_1, \ldots, z_n) := \\
	 \bigwedge_{(i,I,q) \in [n] \times K \times \mathbb{Q}_n^{[0,1]}, s(i,I,q)=1 } \varphi_{I,q} \left(\bar{x},{\bar{y}_i}\right) \land  \\\bigwedge_{(i,I,q) \in [n] \times K \times \mathbb{Q}_n^{[0,1]}, s(i,I,q)=-1 } \neg \varphi_{I,q} \left(\bar{x},\bar{y}_i \right) \land 
\end{gather*}

	$$\bigwedge_{(i,q) \in [n] \times \mathbb{Q}_n^{[0,1]}, t(i,q)=1 } F^{<q}(\bar{x},z_i) \land \bigwedge_{(i,q) \in [n] \times \mathbb{Q}_n^{[0,1]} \land t(i,q)= -1 } \neg F^{<q}(\bar{x},z_i). $$
	\end{definition}

\begin{remark}\label{rem: meaning of phi}
	\begin{enumerate}
		\item Every subset of $V^{\bar{1}^k}$ defined by an instance of $\varphi^{s,t,u}$ in $\mathcal{M}'$ is in $\B_{\bar{1}^k}$. 
		
		\item For any $\bar{a} \in V_k^n$ and $\bar{b}_1, \ldots, \bar{b}_n \in V^{\bar{1}^{k+1}}$, the sets 
	$$\left\{ \varphi^{\bar{v}} \left(\bar{x}; \bar{b}_1, \ldots, \bar{b}_n; \bar{a} \right) \mid  \bar{v} \in Q_n \right\}$$ 
	are precisely the atoms of the Boolean algebra generated by 
	$$\left\{\varphi_{I,q}\left(\bar{x}, \bar{b}_i \right) : I \in K, i \in [n], q \in \mathbb{Q}_n^{[0,1]}  \right\} \cup \left\{F^{<q}_{a_i} : i \in [n], q \in \mathbb{Q}_n^{[0,1]} \right\}.$$
	\end{enumerate}
\end{remark}

\begin{lemma}\label{lem: lin ineq type def}
	For any $n \in \mathbb{N}_{\geq 1}$, any quantifier-free $\mathcal{L}_{\infty}$-formulas $\varphi_i(\bar{x}, \bar{y}_1, \bar{y}_2)$, $1 \leq i \leq n$ with $\bar{x}$ corresponding to $V^{\bar{n}}$ and $\bar{y}_i$ to $V^{\bar{m}_i}$, $\varepsilon \in \mathbb{R}_{>0}$ and $ \beta_1, \ldots, \beta_n  \in \mathbb{R}$, there exists countable partial $\mathcal{L}_{\infty}$-types $\Gamma^{\bar{\varphi}}_{\leq \varepsilon, \bar{\beta}}(\bar{y}_1)$, $\Gamma^{\bar{\varphi}}_{\geq \varepsilon, \bar{\beta}}(\bar{y}_1)$ satisfying the following.

	For every $\mathfrak{P}, f, \mathcal{M}' \propto \mathcal{M}_{\frak{P},f}$ and $\bar{b} \in V^{\bar{m}_1}$, 
		\begin{gather*}
		\mathcal{M}' \models \Gamma^{\bar{\varphi}}_{\leq \varepsilon, \bar{\beta}}(\bar{b}) \iff \textrm{ for every }\bar{c} \in V^{\bar{m}_2}, \sum_{i=1}^{n} \beta_i \cdot \mu_{\bar{n}}\left( \varphi_i(\bar{x}, \bar{b}, \bar{c}) \right) \leq \varepsilon.
	\end{gather*}
	
	And similarly for ``$\leq$'' replaced by ``$\geq$''.
	\end{lemma}
	\begin{proof}
		Fix some $\mathfrak{P}, f, \mathcal{M}' \propto \mathcal{M}_{\frak{P},f}$. Without loss of generality $\beta_i \neq 0$ for all $i \in [n]$. Then for any $\bar{b} \in V^{\bar{m}_1}$ we have
		\begin{gather*}
			\forall \bar{c} \in V^{\bar{m}_2}, \sum_{i=1}^{n} \beta_i \cdot \mu_{\bar{n}}\left( \varphi_i(\bar{x}, \bar{b}, \bar{c}) \right) \leq \varepsilon \iff \\
			\bigwedge_{\substack{r_1, \ldots, r_n \in \mathbb{Q}_{\geq 0}\\\sum_{i=1}^{n} \beta_i r_i > \varepsilon } } \bigwedge_{i \in [n]} \forall \bar{c} \in V^{\bar{m}_2}, \mu_{\bar{n}}\left( \varphi_i(\bar{x}, \bar{b}, \bar{c}) \right) \bowtie_i r_i ,
		\end{gather*}
		where $\bowtie_i$ is ``$<$'' if $\beta_i > 0$, and $\bowtie_i$ is ``$\geq$'' if $\beta_i < 0$, for every $i \in [n]$.
		
		As $\mathcal{M}' \propto \mathcal{M}_{\mathfrak{P}}$, for every $r \in \mathbb{Q}$, $i \in [n]$ and $\bar{b} \in V^{\bar{m}_1}, \bar{c} \in V^{\bar{m}_2}$ we have 
		$$\mu_{\bar{n}}\left( \varphi_i(\bar{x}, \bar{b}, \bar{c}) \right) < r \Rightarrow \mathcal{M}' \models m_{\bar{x}} < r . \varphi(\bar{x}, \bar{b}, \bar{c}) \Rightarrow \mu_{\bar{n}} \left( \varphi(\bar{x}, \bar{b}, \bar{c}) \right) \leq r.$$
		Hence, for any $\bar{b} \in V^{\bar{m}_1}$,
			\begin{gather*}
			\forall \bar{c} \in V^{\bar{m}_2}, \sum_{i=1}^{n} \beta_i \cdot \mu_{\bar{n}}\left( \varphi_i(\bar{x}, \bar{b}, \bar{c}) \right) \leq \varepsilon \iff \\
			 \mathcal{M}' \models \Gamma^{\bar{\varphi}}_{\leq \varepsilon, \bar{\beta}}(\bar{b}) :=  \bigwedge_{t \in \mathbb{N}}\bigwedge_{\substack{r_1, \ldots, r_n \in (\mathbb{Q}_t)_{\geq 0}\\\sum_{i=1}^{n} \beta_i r_i > \varepsilon } } \bigwedge_{i \in [n]} \forall \bar{y}_2\left(m_{\bar{x}} \bowtie_i r_i. \varphi_i(\bar{x}, \bar{b}, \bar{y}_2) \right).
		\end{gather*}
		Note that the definition of $\Gamma^{\bar{\varphi}}_{\leq \varepsilon, \bar{\beta}}$ does not depend on  $\mathfrak{P}, f, \mathcal{M}' \propto \mathcal{M}_{\frak{P},f}$. The argument for ``$\geq$'' is symmetric.
	\end{proof}
	
	\begin{definition}\label{def: simple func h}
	Given some $n \in \mathbb{N}$, $r \in \mathbb{Q}$, $\bar{\alpha} = \left( \alpha_{\bar{v}} \in \mathbb{Q} : \bar{v} \in Q_n \right)$, tuples $\left( \bar{b}_j \in V^{\bar{1}^{k+1}} : j \in [n] \right)$ and $\bar{a} = \left(  a_j \in V_{k+1} : j \in [n] \right)$, we define the function
	
		$$h_{n, \bar{\alpha}, \bar{a}, \bar{b}_1, \ldots, \bar{b}_n }\left( \bar{x} \right) = \sum_{\bar{v} \in Q_n } \alpha_{\bar{v}}\chi_{\varphi^{\bar{v}}(\bar{x}, \bar{b}_1, \ldots, \bar{b}_n, a_{1}, \ldots, a_{n})}$$
		from $V^{\bar{1}^k}$ to $[0,1]$ (where, as usual, $\varphi^{\bar{v}}(\bar{x}, \bar{b}_1, \ldots, \bar{b}_n, a_{1}, \ldots, a_{n})$ represents the set of solutions of this formula evaluated in $\mathcal{M}'$).
	\end{definition}
	
	\begin{lemma}\label{lem: type-def of norm}
	 For any fixed $n \in \mathbb{N}$, $r \in \mathbb{Q}$, $\bar{\alpha}= \left( \alpha_{\bar{v}} \in \mathbb{Q} : \bar{v} \in Q_n \right)$ there exist quantifier-free $\mathcal{L}_{\infty}$-formulas $\Theta_{<r}^{n, \bar{\alpha}}, \Theta^{n, \bar{\alpha}}_{\geq r}$ and countable partial $\mathcal{L}_{\infty}$-types $\Lambda_{\leq r}^{n, \bar{\alpha}}$, $\Lambda_{\geq r}^{n, \bar{\alpha}}$, $\tilde{\Lambda}_{\leq r}^{n, \bar{\alpha}}$ and $\tilde{\Lambda}_{\geq r}^{n, \bar{\alpha}}$  satisfying the following for any $\mathfrak{P}, f, \mathcal{M}' \propto \mathcal{M}_{\frak{P},f}$:
	 \begin{enumerate}
	 \item for any $(c, \bar{a}, \bar{b}_1, \ldots, \bar{b}_n) \in V^{\bar{1}^k} \times \times V_{k+1}^n \times \left( V^{\bar{1}^{k+1}} \right)^n$, 
	 $$  \mathcal{M}' \models \Theta_{ < r}^{n, \bar{\alpha}}(c, \bar{a}, \bar{b}_1, \ldots, \bar{b}_n) \iff h_{n, \bar{\alpha}, \bar{a}, \bar{b}_1, \ldots, \bar{b}_n }(c) < r;$$
	 	\item for any $(a, \bar{a}, \bar{b}_1, \ldots, \bar{b}_n) \in V_{k+1} \times V_{k+1}^n \times \left( V^{\bar{1}^{k+1}} \right)^n$, 
	 $$  \mathcal{M}' \models \Lambda_{\leq r}^{n, \bar{\alpha}}(a, \bar{a}, \bar{b}_1, \ldots, \bar{b}_n) \iff \norm{f_{a} - h_{n, \bar{\alpha}, \bar{a}, \bar{b}_1, \ldots, \bar{b}_n }}_{L^2} \leq r;$$
	 \item for any $(a,\bar{a}) \in V_{k+1} \times V^n_{k+1}$,
	 \begin{gather*}
	 	\mathcal{M}' \models \tilde{\Lambda}_{\leq r}^{n, \bar{\alpha}}(a, \bar{a}) \iff \\
	 	 \forall (\bar{b}_1, \ldots, \bar{b}_n) \in  \left( V^{\bar{1}^{k+1}} \right)^n, \norm{f_{a} - h_{n, \bar{\alpha}, \bar{a}, \bar{b}_1, \ldots, \bar{b}_n }}_{L^2} \leq r.
	 \end{gather*}
	 \end{enumerate}
		And the same for ``$\geq$''.
		 
	\end{lemma}
	\begin{proof}
	
	(1) Let 
$W := \left\{ w \subseteq Q_n \mid  \sum_{\bar{v} \in w} \alpha_{\bar{v}} < r \right\}$.
Then clearly
\begin{gather*}
h_{n, \bar{\alpha}, \bar{a}, \bar{b}_1, \ldots, \bar{b}_n }(c) < r \iff 
\mathcal{M}' \models 
	\Theta_{ < r}^{n, \bar{\alpha}}(c, \bar{a}, \bar{b}_1, \ldots, \bar{b}_n) :=\\
	 \bigvee_{w \in W} \Bigg(  \bigwedge_{\bar{v} \in w} \varphi^{\bar{v}}(c, \bar{b}_1, \ldots, \bar{b}_n, a_{i_1}, \ldots, a_{i_n}) \land \\
	 \bigwedge_{\bar{v} \in Q_n \setminus w} \neg \varphi^{\bar{v}}(c, \bar{b}_1, \ldots, \bar{b}_n, a_{i_1}, \ldots, a_{i_n}) \Bigg).
\end{gather*}

 (2) and (3) Note that, using $\mathcal{M}' \propto \mathcal{M}_{\frak{P},f}$, for any $r < s \in \mathbb{Q} \cap [0,1]$ and $a \in V_{k+1}$, if $\bar{x} \in F^{\geq r}_{a} \cap F^{<s}_{a}$, then $|f_{a}(\bar{x}) - r| \leq r-s$. Then for any $\varepsilon > 0$ we can choose $m_{\varepsilon}, \ell_{\varepsilon} \in \mathbb{N}$ large enough and a partition $(q^{\varepsilon}_i : i \in [\ell_{\varepsilon}] )$ of $[0,1]$ with $q^{\varepsilon}_i \in \mathbb{Q}_{m_{\varepsilon}}^{[0,1]}, q^{\varepsilon}_i < q^{\varepsilon}_{i+1}, q^{\varepsilon}_1 = 0, q^{\varepsilon}_{\ell} = 1$ so that for any $a \in V_{k+1}$, any quantifier-free $\mathcal{L}_{\infty}$-formula $\psi(\bar{x}, \bar{y})$ with $\bar{x}$ corresponding to $V^{\bar{1}^k}$ and any tuple $\bar{c}$ corresponding to $\bar{y}$ we have:
	\begin{gather}\label{eq: type-def integrals1}
	\int f_a \cdot \chi_{\psi(\bar{x},\bar{c})} d\mu_{\bar{1}^k} 	\approx^{\varepsilon}  A_{\varepsilon}^{a, \psi(\bar{x}, \bar{c})} := \\
 \sum_{i = 1}^{\ell_{\varepsilon} - 2} q^{\varepsilon}_i \cdot \mu_{\bar{1}^k} \left( F^{\geq q^{\varepsilon}_i}_a \cap  F^{<q^{\varepsilon}_{i+1}}_a \cap \psi(\bar{x}, \bar{c})\right) + q^{\varepsilon}_{\ell - 1} \cdot \mu_{\bar{1}^k} \left( F^{\geq q^{\varepsilon}_{\ell-1}}_a \cap \psi(\bar{x}, \bar{c}) \right).\nonumber
	\end{gather}

	For any tuple $(a, \bar{a}, \bar{b}_1, \ldots, \bar{b}_n)$, using \eqref{eq: type-def integrals1} we have 
	\begin{gather*}
\norm{f_{a} - h_{n, \bar{\alpha}, \bar{a}, \bar{b}_1, \ldots, \bar{b}_n }}_{L^2}^2 = \\
\int \left(f_a -  \sum_{\bar{v} \in Q_n} \alpha_{\bar{v}}\chi_{\varphi^{\bar{v}}(\bar{x}, \bar{b}_1, \ldots, \bar{b}_n, a_{1}, \ldots, a_{n})} \right)^2 d\mu_{\bar{1}^k} \\
= \int  f_{a}^2 d \mu_{\bar{1}^k} + \sum_{\bar{v} \in Q_n} (-2 \alpha_{\bar{v}}) \int f_{a} \cdot  \chi_{\varphi^{\bar{v}}(\bar{x}, \bar{b}_1, \ldots, \bar{b}_n, a_{1}, \ldots, a_{n}) } d\mu_{\bar{1}^k}
	+ \\
	 \sum_{\bar{v},\bar{v}' \in Q_n} \alpha_{\bar{v}}  \alpha_{\bar{v}'}\int \chi_{\varphi^{\bar{v}}(\bar{x}, \bar{b}_1, \ldots, \bar{b}_n, a_{1}, \ldots, a_{n})} 
	\chi_{\varphi^{\bar{v}'}(\bar{x}, \bar{b}_1, \ldots, \bar{b}_n, a_{1}, \ldots, a_{n})}  
	d\mu_{\bar{1}^k} 
\end{gather*}

As $f_a$ takes values in $[0,1]$, as in \eqref{eq: type-def integrals1} for the first integral we have 
$$\int  f_{a}^2 d \mu_{\bar{1}^k} \approx^{2 \varepsilon} B_{\varepsilon}^{a} := \sum_{i = 1}^{\ell_{\varepsilon} - 2} \left(q_i^{\varepsilon} \right)^2 \cdot \mu_{\bar{1}^k} \left( F^{\geq q^{\varepsilon}_i}_a \cap  F^{<q^{\varepsilon}_{i+1}}_a \right) + \left(q^{\varepsilon}_{\ell - 1} \right)^2 \cdot \mu_{\bar{1}^k} \left( F^{\geq q^{\varepsilon}_{\ell-1}}_a \right).$$
	Using \eqref{eq: type-def integrals1} for the second integral we have 
	\begin{gather*}
		\sum_{\bar{v} \in Q_n} (-2 \alpha_{\bar{v}}) \int f_{a} \cdot  \chi_{\varphi^{\bar{v}}(\bar{x}, \bar{b}_1, \ldots, \bar{b}_n, a_{1}, \ldots, a_{n}) } d\mu_{\bar{1}^k} \approx^{\varepsilon \cdot |Q_n|} C_{\varepsilon}^{a, \bar{b}_1, \ldots, \bar{b}_n, \bar{a}} :=\\
		\sum_{\bar{v} \in Q_n} (-2 \alpha_{\bar{v}}) \cdot A_{\varepsilon}^{a, \varphi^{\bar{v}}(\bar{x}, \bar{b}_1, \ldots, \bar{b}_n, a_{1}, \ldots, a_{n})} = \\
		\sum_{\bar{v} \in Q_n} \sum_{i = 1}^{\ell_{\varepsilon}-2} (-2\alpha_{\bar{v}}\cdot q^{\varepsilon}_i) \cdot \mu_{\bar{1}^k} \left( F^{\geq q^{\varepsilon}_i}_a \cap  F^{<q^{\varepsilon}_{i+1}}_a \cap \varphi^{\bar{v}}(\bar{x}, \bar{b}_1, \ldots, \bar{b}_n, a_{1}, \ldots, a_{n})\right) \\
		+ \sum_{\bar{v} \in Q_n}(-2\alpha_{\bar{v}} \cdot q^{\varepsilon}_{\ell - 1}) \cdot \mu_{\bar{1}^k} \left( F^{\geq q^{\varepsilon}_{\ell-1}}_a \cap \varphi^{\bar{v}}(\bar{x}, \bar{b}_1, \ldots, \bar{b}_n, a_{1}, \ldots, a_{n}) \right).
	\end{gather*}
And the third integral is equal to
\begin{gather*}
	D^{\bar{b}_1, \ldots, \bar{b}_n, \bar{a}}_{\varepsilon} := \sum_{\substack{\bar{v},\bar{v}' \in Q_n}} (\alpha_{\bar{v}} \cdot \alpha_{\bar{v}'} ) \cdot \mu_{\bar{1}^k} \left( \left(\varphi^{\bar{v}} \land \varphi^{\bar{v}'} \right)\left(\bar{x}, \bar{b}_1, \ldots, \bar{b}_n, a_{1}, \ldots, a_{n}\right) \right).
\end{gather*}
Combining, we get 
\begin{gather*}
	\norm{f_{a} - h_{n, \bar{\alpha}, \bar{a}, \bar{b}_1, \ldots, \bar{b}_n }}^2_{L^2} \approx^{\left(|Q_n| + 2 \right) \varepsilon} E_{\varepsilon}^{a,\bar{a}, \bar{b}_1, \ldots, \bar{b}_n} :=
	\\  B_{\varepsilon}^{a} + C_{\varepsilon}^{a, \bar{b}_1, \ldots, \bar{b}_n, \bar{a}} + D^{\bar{b}_1, \ldots, \bar{b}_n, \bar{a}}_{\varepsilon}.
\end{gather*}

By definition of $h_{n, \bar{\alpha}, \bar{a}, \bar{b}_1, \ldots, \bar{b}_n }$ and assumption on $f$, $\norm{f_{a} - h_{n, \bar{\alpha}, \bar{a}, \bar{b}_1, \ldots, \bar{b}_n }}$ takes values in $[-c,c]$ for some $c = c(\bar{\alpha}) \in \mathbb{R}_{>0}$, hence

\begin{gather*}
	\norm{f_{a} - h_{n, \bar{\alpha}, \bar{a}, \bar{b}_1, \ldots, \bar{b}_n }}_{L^2} \approx^{ c \sqrt{\left(|Q_n| + 2 \right) \varepsilon}} E_{\varepsilon}^{a,\bar{a}, \bar{b}_1, \ldots, \bar{b}_n} :=
	\\  B_{\varepsilon}^{a} + C_{\varepsilon}^{a, \bar{b}_1, \ldots, \bar{b}_n, \bar{a}} + D^{\bar{b}_1, \ldots, \bar{b}_n, \bar{a}}_{\varepsilon}.
\end{gather*}

By Lemma \ref{lem: lin ineq type def} and the definition of $E_{\varepsilon}^{a,\bar{a}, \bar{b}_1, \ldots, \bar{b}_n}$, for any $r \in \mathbb{Q}$ there exist some countable partial $\mathcal{L}_{\infty}$-types $\Gamma_{\varepsilon,r}, \tilde{\Gamma}_{\varepsilon,r}$ over $\emptyset$
	such that 
	\begin{gather*}
	\mathcal{M}' \models \Gamma_{\varepsilon,r}(a, \bar{a}, \bar{b}_1, \ldots, \bar{b}_n) \iff E_{\varepsilon}^{a,\bar{a}, \bar{b}_1, \ldots, \bar{b}_n} \leq r,\\
	\mathcal{M}' \models \tilde{\Gamma}_{\varepsilon,r}(a,\bar{a}) \iff \forall(\bar{b}_1, \ldots, \bar{b}_n), E_{\varepsilon}^{a,\bar{a}, \bar{b}_1, \ldots, \bar{b}_n} \leq r.
	\end{gather*}
	
	For each $\varepsilon \in \mathbb{Q}_{>0}$, pick some $\varepsilon' \in \mathbb{Q}_{>0}$ such that $c \sqrt{\left(|Q_n| + 2 \right) \cdot \varepsilon'} < \varepsilon$. Then 
	\begin{gather*}
		\norm{f_{a} - h_{n, \bar{\alpha}, \bar{a}, \bar{b}_1, \ldots, \bar{b}_n }}_{L^2} \leq r \iff \\
		\mathcal{M}' \models \Lambda_{\leq r}^{n, \bar{\alpha}}(a, \bar{a}, \bar{b}_1, \ldots, \bar{b}_n) :=  \bigwedge_{\varepsilon \in \mathbb{Q}_{>0}} \Gamma_{\varepsilon', r + \varepsilon}(a, \bar{a}, \bar{b}_1, \ldots, \bar{b}_n),\\
		\forall(\bar{b}_1, \ldots, \bar{b}_n), \norm{f_{a} - h_{n, \bar{\alpha}, \bar{a}, \bar{b}_1, \ldots, \bar{b}_n }}_{L^2} \leq r \iff \\
		\mathcal{M}' \models \tilde{\Lambda}_{\leq r}^{n, \bar{\alpha}}(a, \bar{a}) :=  \bigwedge_{\varepsilon \in \mathbb{Q}_{>0}} \tilde{\Gamma}_{\varepsilon', r + \varepsilon}(a, \bar{a}).
	\end{gather*}
	Note that the definitions of $\Lambda_{\leq r}^{n, \bar{\alpha}}, \tilde{\Lambda}_{\leq r}^{n, \bar{\alpha}}$ do not depend on $\mathfrak{P}, f, \mathcal{M}' \propto \mathcal{M}_{\frak{P},f}$. The argument for ``$\geq$'' is analogous.
	\end{proof}
	
	\begin{lemma}\label{lem: meaning of formulas}

	 Given an arbitrary countable linear order $I$ and $\varepsilon \in \mathbb{Q}_{>0}$, there exists a countable partial $\mathcal{L}_{\infty}$-type $\pi_{I, \varepsilon}\left( (z_i)_{i \in I} \right)$ such that the following holds.
	 
	 For any $\mathfrak{P}, f, \mathcal{M}' \propto \mathcal{M}_{\frak{P},f}$  and sequence $(a_i)_{i \in I}$ in $V_{k+1}$,
	 \begin{gather*}
	 	\mathcal{M}' \models \pi_{I, \varepsilon} \left( (a_i)_{i \in I} \right) \iff \\
	 	\norm{f_{a_i} - \E \left(f_{a_i} \mid \B_{\bar{1}^k, k-1} \cup \{ f_{a_j} : j \in I \land j <i \} \right)}_{L^2} \geq \varepsilon \textrm{ for every } i \in I.
	 \end{gather*}
	\end{lemma}
	\begin{proof}
 Fix $\mathfrak{P}, f, \mathcal{M}' \propto \mathcal{M}_{\frak{P},f}, (a_i)_{i \in I}$ and $i \in I$.
 
  By Lemma \ref{lem: min alg with add functions} (which can be applied here in view of Remarks \ref{rem: algebra of fibers} and \ref{rem: fib alg unif def}), we have 
		 \begin{gather*}
		 	\E \left( f_{a_i} \vert \B_{\bar{1}^k,k-1} \cup \left\{ f_{a_j} : j < i\right\} \right) = \E \left( f_{a_i} \vert \B^{f, (a_j : j \leq i)} \cup \left\{ F^{<q}_{a_j} : j < i, q \in \mathbb{Q} \right\} \right).
		 \end{gather*}

		 Approximating by a simple function, for any $\delta \in \mathbb{R}_{>0}$ there exist some $n \in \mathbb{N}, \bar{\alpha} = (\alpha_{\bar{v}} \in \mathbb{Q}_n^{[0,1]} : \bar{v} \in Q_n)$, some tuples   $\bar{b}_j \in V^{\bar{1}^{k+1}}$  and some $i_1 < \ldots < i_n < i$ in $I$ so that the $\sigma\left(\B^{f, (a_j : j \leq i)} \cup \left\{F^{<q}_{a_j} : j < i, q \in \mathbb{Q}_n^{[0,1]} \right\} \right)$-simple function  
$h_{n, \bar{\alpha}, (a_{i_1}, \ldots, a_{i_n}), \bar{b}_1, \ldots, \bar{b}_n }$ (all of them are of these form, see Definition \ref{def: simple func h} and Remark \ref{rem: meaning of phi}(2)) satisfies 
	$$\norm{\E \left( f_{a_i} \vert \B^{f, (a_j : j \leq i)} \cup \left\{ F^{<q}_{a_j} : j < i, q \in \mathbb{Q} \right\} \right) - h_{n, \bar{\alpha}, (a_{i_1}, \ldots, a_{i_n}), \bar{b}_1, \ldots, \bar{b}_n }}_{L^2} \leq \delta.$$
	
	Hence, for a fixed $i \in I$,
	\begin{gather*}
		\norm{f_{a_i} - \E \left(f_{a_i} \mid \B_{\bar{1}^k, k-1} \cup \{ f_{a_j} :  j <i \} \right)}_{L^2} \geq  \varepsilon  \iff \\
		 \bigwedge_{n \in \mathbb{N}} \bigwedge_{i_1< \ldots < i_n < i \in I} \bigwedge_{ \substack{\bar{\alpha} = ( \alpha_{\bar{v}} \in \mathbb{Q}_n^{[0,1]} : \\ \bar{v} \in Q_n )}}
		 \forall (\bar{b}_1, \ldots, \bar{b}_n) \norm{f_{a_i} - h_{n, \bar{\alpha}, (a_{i_1}, \ldots, a_{i_n}), \bar{b}_1, \ldots, \bar{b}_n }}_{L^2} \geq \varepsilon.
	\end{gather*}
	
	By Lemma \ref{lem: type-def of norm} we thus have
		\begin{gather*}
		\norm{f_{a_i} - \E \left(f_{a_i} \mid \B_{\bar{1}^k, k-1} \cup \{ f_{a_j} :  j <i \} \right)}_{L^2} \geq  \varepsilon \textrm{ for all } i \in I  \iff \\
		\mathcal{M}' \models \pi_{I, \varepsilon} ((a_i)_{i \in I}) := \\\bigwedge_{i \in I} \bigwedge_{n \in \mathbb{N}} \bigwedge_{i_1< \ldots < i_n < i \in I} \bigwedge_{\bar{\alpha} = ( \alpha_{\bar{v}} \in \mathbb{Q}_n^{[0,1]} : \bar{v} \in Q_n )}
		 \tilde{\Lambda}_{\geq \varepsilon}^{n, \bar{\alpha}}(a_i, a_{i_1}, \ldots, a_{i_n}).
	\end{gather*}
	
	Note that the definition of the partial type $\pi_{I, \varepsilon}$ does not depend on $\mathfrak{P}, f, \mathcal{M}' \propto \mathcal{M}_{\frak{P},f}, (a_i)_{i \in I}$, since neither did $ \tilde{\Lambda}_{\geq \varepsilon}^{n, \bar{\alpha}}$.
		\end{proof}

\begin{remark}\label{rem : star subsequence}
It is easy to see from the definition that $(a_i)_{i \in I} \models \pi_{I, \varepsilon}$ $\iff$ $(a_i)_{i \in I'} \models \pi_{I', \varepsilon}$ for every finite $I' \subseteq I$.
\end{remark}

The following is a version of de Finetti's theorem suitable for our context (in particular we observe that $\mathcal{L}_{\infty}$-indiscernibility implies \emph{exchangeability} in the probabilistic sense).

\begin{prop}\label{prop: de Finetti}
Assume that $\mathcal{M}' \propto \mathcal{M}_{\mathfrak{P},f}$, $\mathcal{M}'$ is an $\aleph_1$-saturated $\mathcal{L}_{\infty}$-structure, 
	 $I = \mathbb{Z}$ and $(a_i : i \in I)$ in $V_{k+1}$ is an $\mathcal{L}_{\infty}$-indiscernible sequence in the sense of $\mathcal{M}'$. Let $ \mathcal{B} := \sigma \left(\left\{f_{a_j} : j < 0 \right\} \cup \mathcal{B}_{\bar{1}^k,k-1} \right)$. Then:
	\begin{enumerate}
		\item $\B_{\bar{1}^k,k-1} \subseteq \B \subseteq \B_{\bar{1}^k}$;
		\item for all $i \in \mathbb{N}$ we have
 	\[\mathbb{E} \left(f_{a_i}\mid\mathcal{B}_{\bar{1}^k,k-1}\cup\left\{f_{a_j} : j < i\right\} \right) = \mathbb{E} \left(f_{a_i}\mid\mathcal{B}\cup\left\{f_{a_j} : j < i\right\} \right)=\mathbb{E}\left(f_{a_i} \mid\mathcal{B} \right).\]
	\end{enumerate}
	\end{prop}
\begin{proof}
It is obvious that (1) holds for $\B$. In (2), it is enough to show the equality of the first and the last expressions. As in the proof of Lemma \ref{lem: meaning of formulas}, by Lemma \ref{lem: min alg with add functions} and Remark \ref{rem: algebra of fibers}, we have 
\begin{gather*}
		 	\E \left( f_{a_i} \vert \B_{\bar{1}^k,k-1} \cup \left\{ f_{a_j} : j < i\right\} \right) = \E \left( \chi_{E_{a_i}} \vert \B^f_{\bar{1}^k,k-1} \cup \left\{ F^{<q}_{a_j} : j < i, q \in \mathbb{Q} \right\} \right).
		 \end{gather*}
	Fix $i \geq 0$ arbitrary. Let $\varepsilon \in \mathbb{Q}_{\geq 0}$ be arbitrary, and assume that 
	$$\norm{f_{a_i} - \mathbb{E} \left(f_{a_i} \mid\mathcal{B}^f_{\bar{1}^k,k-1}\cup\left\{F^{<q}_{a_j} : j < i, q \in \mathbb{Q}\right\} \right)}_{L^2} \leq \varepsilon.$$
	
	By definition of $\mathbb{E}$, for any $\delta \in \mathbb{R}_{>0}$ there exist some $n \in \mathbb{N}$, $\alpha_{s,t,u} \in \mathbb{Q}$, tuples $\bar{b}_j \in V^{\bar{1}^{k+1}}$  and $i_1 < \ldots < i_n < i$ in $I$, such that taking $\bar{a} = (a_{i_1}, \ldots, a_{i_n})$, the $\sigma\left(\B_{\bar{1}^k,k-1}^f \cup \left\{F^{<q}_{a_j} : j < i, q \in \mathbb{Q}_n^{[0,1]} \right\} \right)$-simple function  
	$h_{n, \bar{\alpha}, \bar{a}, \bar{b}_1, \ldots, \bar{b}_n }$ satisfies
	
	$$\norm{\E \left( f_{a_i} \vert \B^f_{\bar{1}^k,k-1} \cup \left\{ F^{<q}_{a_j} : j < i, q \in \mathbb{Q} \right\} \right) - h_{n, \bar{\alpha}, \bar{a}, \bar{b}_1, \ldots, \bar{b}_n }}_{L^2} \leq \delta,$$
	hence	
	$$\norm{f_{a_i} - h_{n, \bar{\alpha}, \bar{a}, \bar{b}_1, \ldots, \bar{b}_n }}_{L^2} \leq \varepsilon + \delta.$$

By Lemma \ref{lem: type-def of norm}, there is a countable partial $\mathcal{L}_{\infty}$-type $\Lambda_{\leq \varepsilon + \delta}^{n, \bar{\alpha}}$ so that for any $\bar{b}'_1, \ldots, \bar{b}'_n \in V^{\bar{1}^{k+1}}, a'_1, \ldots, a'_n, a' \in V_{k+1}$ we have
\begin{gather*}
	\norm{f_{a'} - h_{n, \bar{\alpha}, (a'_1, \ldots, a'_n), \bar{b}'_1, \ldots, \bar{b}'_n }}_{L^2} \leq \varepsilon + \delta \iff\\
	\mathcal{M}' \models\Lambda_{\leq \varepsilon + \delta}^{n, \bar{\alpha}} \left(a', a'_1, \ldots, a'_n, \bar{b}'_1, \ldots, \bar{b}'_n \right).
\end{gather*}

Then, by $\aleph_1$-saturation of $\mathcal{M}'$, the condition 
	$$\exists \bar{y}_1 \ldots \exists \bar{y}_n \norm{f_{a'} - h_{n, \bar{\alpha}, \bar{a}', \bar{y}_1, \ldots, \bar{y}_n }}_{L^2} \leq \varepsilon + \delta$$
	on the tuple $(a'_1, \ldots, a'_n, a')$ is also $\mathcal{L}_{\infty}$-type-definable in $\mathcal{M}'$, and is satisfied on $(a_{i_1}, \ldots, a_{i_n}, a_i)$ in $\mathcal{M}'$ by assumption. Since the sequence $(a_i)_{i \in \mathbb{Z}}$ is $\mathcal{L}_{\infty}$-indiscernible in $\mathcal{M}'$, it follows that it is also satisfied by the tuple $(a_{j_1}, \ldots, a_{j_n}, a_j)$ for any $j_1 < \ldots < j_n < j$ in $\mathbb{Z}$.
	
	In particular, taking arbitrary $j_1 < \ldots < j_n < 0$ and $j=i$, 
we have 	
	\begin{gather}
\norm{f_{a_i} - h_{n, \bar{\alpha}, (a_{j_1}, \ldots, a_{j_n}), \bar{b}'_1, \ldots, \bar{b}'_n }}_{L^2} \leq \varepsilon + \delta	\label{eq: indisc bounds exp}
	\end{gather}
	for some $\bar{b}'_1, \ldots, \bar{b}'_n \in V^{\bar{1}^{k+1}}$. Note that $h_{n, \bar{\alpha}, (a_{j_1}, \ldots, a_{j_n}), \bar{b}'_1, \ldots, \bar{b}'_n }$ is a $\B$-simple function.
	As $\varepsilon, \delta > 0$ were arbitrary, we thus conclude that 
	$$\norm{f_{a_i} - \mathbb{E}\left(f_{a_i} \mid \B \right)}_{L^2} \leq \norm{f_{a_i} - \mathbb{E}\left(f_{a_i}\mid\mathcal{B}_{\bar{1}^k,k-1}\cup \left\{f_{a_j} : j < i\right\}\right)}_{L^2}.$$
	
	But since conditional expectation corresponds to orthogonal projection in the Hilbert space $L^2(\B_{\bar{1}^k})$, and $L^2(\B)$ is a closed subspace of $L^2(\sigma(\mathcal{B}_{\bar{1}^k,k-1}\cup\{f_{a_j} : j < i\}))$, this last inequality implies that 
	$$\mathbb{E}\left(f_{a_i} \mid\mathcal{B}_{\bar{1}^k,k-1}\cup\left\{f_{a_j} : j < i\right\} \right) = \mathbb{E}\left(f_{a_i} \mid \B \right).$$
\end{proof}

\begin{lemma}\label{lem: indisc Bergelson}
Let $\mathcal{M}' \propto \mathcal{M}_{\mathfrak{P},f}$, $\mathcal{M}'$ an $\aleph_1$-saturated $\mathcal{L}_{\infty}$-structure, 
	 $I = \mathbb{Z}$ and $(a_i : i \in I)$ in $V_{k+1}$ is an $\mathcal{L}$-indiscernible sequence in the sense of $\mathcal{M}$.
	 Let $\B = \sigma \left(\left\{f_{a_j} : j < 0 \right\} \cup \mathcal{B}_{\bar{1}^k,k-1} \right)$, $\delta \in \mathbb{R}_{>0} $ and $r< s \in \mathbb{Q}^{[0,1]}$. Let 
	 $$G^{r,s}_\delta(a_i):=\left\{\bar{x} \in V^{\bar{1}^k} \mid \E\left(\chi_{f^{<r}_{a_i}}\mid \B \right)(x) \geq \delta \land \E \left( \chi_{f^{\geq s}_{a_i}} \mid \B \right)(x) \geq  \delta  \right\} \in \B.$$
	  Assume that $\mu_{\bar{1}^k} \left( G^{r,s}_\delta(a_0) \right) > 0$. Then $\mu_{\bar{1}^k} \left( \bigcap_{i \in [l]} G^{r',s'}_\delta(a_i) \right) > 0$ for any $l \in \mathbb{N}$ and $r < r' < s' < s$.
\end{lemma}
\begin{proof}
Fix some $l \in \mathbb{N}$ and $r < r' < s' < s$. Let
$$F^{r,s}_\delta(a_i):=\left\{\bar{x} \in V^{\bar{1}^k} \mid \E\left(\chi_{F^{<r}_{a_i}}\mid \B \right)(x) \geq \delta \land \E \left( \chi_{F^{\geq s}_{a_i}} \mid \B \right)(x) \geq  \delta  \right\} \in \B.$$
As $\mathcal{M}' \propto \mathcal{M}_{\mathfrak{P},f}$, by monotonicity of conditional expectation we have $q := \mu_{\bar{1}^k}\left(F^{r,s'}_{\delta}(a_0) \right) \geq  \mu_{\bar{1}^k}\left(G^{r,s}_{\delta}(a_0) \right) > 0$.
 Let $\xi = \xi(\frac{q}{2},l) > 0$ be as given by Fact \ref{fac: Bergelson}. Fix some $0 < \varepsilon < \min \left\{\frac{q}{2},  \frac{\xi}{l}\right\}$.

Fix $i \in \mathbb{N}$. Note that $F^{r,s'}_{\delta}(a_i) = \bigcup_{\gamma \in \mathbb{Q}_{>0}} F^{r,s'}_{\delta+\gamma}(a_i)$, and by countable additivity 
\begin{align}
& \mu_{\bar{1}^k} \left( F^{r,s'}_\delta(a_i) \setminus F^{r,s'}_{\delta + \gamma}(a_i) \right) \to 0 \textrm{ as } \gamma \to 0. \label{eqmu}
\end{align}

For arbitrary $\B_{\bar{1}^k}$-measurable functions $h_1, h_2$ and $\gamma \in \mathbb{R}_{>0}$, we define the set 
$$F_{h_1, h_2, \gamma} := \left\{\bar{x} \in V^{\bar{1}^k} \mid h_1 (\bar{x}) \geq \delta + \gamma \land h_2 (\bar{x}) \geq \delta + \gamma\right\}.$$ 

\begin{claim}
There exists some $ \gamma >0 $ such that for arbitrary $h_1, h_2$ we have: 
\begin{gather*}
	\norm{\E \left(\chi_{F^{<r}_{a_i}} \mid \B \right) - h_1}_{L^2} < \gamma^{\frac{3}{2}}  \land \norm{\E \left(\chi_{F^{\geq s'}_{a_i}} \mid \B \right) - h_2}_{L^2}  < \gamma^{\frac{3}{2}} \implies \\
	 \mu_{\bar{1}^k}\left(F^{r,s'}_\delta(a_i) \triangle F_{h_1, h_2,\gamma} \right) < \varepsilon.
\end{gather*}

\end{claim}
\begin{claimproof}
	Let 
	\begin{gather*}
		D_1 := \left\{ \bar{x} \in V^{\bar{1}^k}: \left \lvert\E 
	\left(\chi_{F^{<r}_{a_i}} \mid \B \right)(\bar{x}) - h_1(\bar{x}) \right \rvert \geq \gamma \right\} \in \B_{\bar{1}^k},\\
	D_2 := \left\{ \bar{x} \in V^{\bar{1}^k}: \left \lvert\E 
	\left(\chi_{F^{\geq s'}_{a_i}} \mid \B \right)(\bar{x}) - h_2(\bar{x}) \right \rvert \geq \gamma \right\} \in \B_{\bar{1}^k}.
	\end{gather*}
	 Then $\mu_{\bar{1}^k}(D_t) < \gamma$ by assumption on $h_t$ for $t \in \{1,2\}$, and $F_{h_1, h_2, \gamma} \setminus (D_1 \cup D_2) \subseteq F^{r,s'}_\delta(a_i)$.
	And similarly $F^{r,s'}_{\delta + 2 \gamma}(a_i) \setminus (D_1 \cup D_2) \subseteq F_{h_1, h_2,\gamma}$.  Taking $0 < \gamma < \frac{\varepsilon}{4}$ small enough, by (\ref{eqmu}) we have
	$\mu_{\bar{1}^k} \left(F^{r,s'}_\delta(a_i) \setminus F^{r,s'}_{\delta + 2\gamma}(a_i) \right) < \frac{\varepsilon}{2}$. But 
	$$F^{r,s'}_\delta(a_i) \triangle F_{h_1, h_2,\gamma} \subseteq D_1 \cup D_2 \cup \left( F^{r,s'}_\delta(a_i) \setminus F^{r,s'}_{\delta + 2\gamma}(a_i) \right),$$
hence $\mu_{\bar{1}^k} \left(F^{r,s'}_\delta(a_i) \triangle F^{r,s'}_{h,\gamma} \right) < 2\gamma +  \frac{\varepsilon}{2} < \varepsilon$.
\end{claimproof}

From now on, fix some $\gamma >0$ satisfying the conclusion of Claim 1. By definition of $\B$, for every $i \in \mathbb{N}$, the function $\mathbb{E} \left(\chi_{F^{<r}_{a_i}} \mid \B \right)$ can be approximated  arbitrarily well in $L^2$-norm by functions of the form 
$$h_{n, \bar{\alpha}, (a_{j_1}, \ldots, a_{j_n}), \bar{b}'_1, \ldots, \bar{b}'_n }$$
 with $n \in \mathbb{N}$, $\bar{\alpha} = \left(\alpha_{\bar{v}} \in \mathbb{Q} : \bar{v} \in Q_n \right)$, $j_1, \ldots, j_n < 0$, $\bar{b}'_1, \ldots, \bar{b}'_n \in V^{\bar{1}^{k+1}}$ (and all such functions are $\B$-simple).
As in the proof of Proposition \ref{prop: de Finetti},  using $\mathcal{L}_{\infty}$-indiscernibility of the sequence $\left( a_i \right)_{i \in \mathbb{Z}}$ and type-definability of the corresponding condition, for every such function, $\beta \in \mathbb{R}_{>0}$ and $i,i' \in \mathbb{N}$, we have 
\begin{align*}
	&\exists\bar{b}'_1 \ldots \exists \bar{b}'_n \in V^{\bar{1}^{k+1}} \norm{\chi_{F^{<r}_{a_i}} - h_{n, \bar{\alpha}, (a_{j_1}, \ldots, a_{j_n}), \bar{b}'_1, \ldots, \bar{b}'_n }}_{L^2} \leq \beta \iff \\
	&\exists\bar{b}'_1 \ldots \exists \bar{b}'_n \in V^{\bar{1}^{k+1}} \norm{\chi_{F^{<r}_{a_{i'}}} - h_{n, \bar{\alpha}, (a_{j_1}, \ldots, a_{j_n}), \bar{b}'_1, \ldots, \bar{b}'_n }}_{L^2} \leq \beta.
\end{align*}
It follows that $\norm{\chi_{F^{<r}_{a_i}} - \mathbb{E} \left(\chi_{F^{<r}_{a_i}} \mid \B \right)}_{L^2}$ does not depend on $i \in \mathbb{N}$, and we denote its value by $\beta_1 \in \mathbb{R}_{\geq 0}$. Similarly $\beta_2 := \norm{\chi_{F^{\geq s'}_{a_i}} - \mathbb{E} \left(\chi_{F^{\geq s'}_{a_i}} \mid \B \right)}_{L^2}$ does not depend on $i \in \mathbb{N}$.

\begin{claim}
For every $\gamma > 0$ there exists $\gamma' = \gamma'(\gamma, \beta_1, \beta_2)>0$ such that: for every $i \in \mathbb{N}$ and a $\B$-measurable function $h$, 
\begin{gather*}
\norm{\chi_{F^{<r}_{a_i}} - h}_{L^2} \leq \beta_1 + \gamma' \implies \norm{\mathbb{E} \left(\chi_{F^{<r}_{a_i}}\mid \B \right) - h}_{L^2} \leq  \gamma, \textrm { and }\\	
\norm{\chi_{F^{\geq s'}_{a_i}} - h}_{L^2} \leq \beta_2 + \gamma' \implies \norm{\mathbb{E} \left(\chi_{F^{\geq s'}_{a_i}} \mid \B \right) - h}_{L^2} \leq  \gamma.
\end{gather*}
\end{claim}

\begin{claimproof}
	Assume that $\norm{\chi_{F^{<r}_{a_i}} - h}_{L^2} \leq \beta_1 + \gamma' $. By the parallelogram rule for the $L^2$-norm, as the function $\frac{1}{2} \left(\mathbb{E} \left(\chi_{F^{<r}_{a_i}} \mid \B \right)+h \right)$ is $\B$-measurable, we have
	\begin{gather*}
		\norm{\mathbb{E} \left(\chi_{F^{<r}_{a_i}}\mid \B \right) - h}_{L^2}^2 = 2 \norm{\chi_{F^{<r}_{a_i}} - \mathbb{E} \left(\chi_{F^{<r}_{a_i}}\mid \B \right)}^2_{L^2} + 2 \norm{\chi_{F^{<r}_{a_i}}-h}_{L^2}^2\\ 
		- 4\norm{\chi_{F^{<r}_{a_i}} - \frac{1}{2} \left(\mathbb{E} \left(\chi_{F^{<r}_{a_i}}\mid \B \right)+h \right)}_{L^2}^2\\
	\leq 	2 \beta_1^2 + 2 \left( \beta_1 + \gamma' \right)^2 - 4 \beta_1^2 = 2 \beta_1 \gamma' + \left( \gamma' \right)^2 \leq \gamma
	\end{gather*}
	assuming that $\gamma'$ is sufficiently small with respect to $\beta_1$ and $\gamma$. The argument for $\chi_{F^{\geq s'}_{a_i}}$ is similar.
\end{claimproof}
From now on, fix some $\gamma' > 0$ satisfying the conclusion of Claim 2 with respect to $\gamma^{\frac{3}{2}}$ instead of $\gamma$. By the choice of $\beta_1, \beta_2$,  we can choose $n, \bar{\alpha}^1, \bar{\alpha}^2, \bar{b}_1, \ldots, \bar{b}_n$ and $i_1, \ldots, i_n < 0$ in $\mathbb{Z}$ so that, writing $\bar{a} := (a_{i_1}, \ldots, a_{i_n})$,  
\begin{gather}\label{eq: choice of approx Beil}	
\norm{\chi_{F^{<r}_{a_0}} - h_{n, \bar{\alpha}^1, \bar{a},\bar{b}_1, \ldots, \bar{b}_n}}_{L^2} \leq \beta_1 + \gamma', \\
\norm{\chi_{F^{\geq s'}_{a_0}} - h_{n, \bar{\alpha}^2, \bar{a},\bar{b}_1, \ldots, \bar{b}_n}}_{L^2} \leq \beta_2 + \gamma'. \nonumber
\end{gather}

The set $F_{h_{n, \bar{\alpha}^1, \bar{a}, \bar{b}_1, \ldots, \bar{b}_n }, h_{n, \bar{\alpha}^2, \bar{a}, \bar{b}_1, \ldots, \bar{b}_n }, \gamma}$ is definable in $\mathcal{M}'$ by a quantifier-free $\mathcal{L}_{\infty}$-formula by Lemma \ref{lem: type-def of norm}(1).  Hence the 
condition 
$$\mu_{\bar{1}^k}\left(F_{h_{n, \bar{\alpha}^1, \bar{a}, \bar{b}_1, \ldots, \bar{b}_n }, h_{n, \bar{\alpha}^2, \bar{a}, \bar{b}_1, \ldots, \bar{b}_n }, \gamma} \right) \geq r -\varepsilon$$
 on the tuple $\left(\bar{b}_1, \ldots, \bar{b}_n, a_{i_1}, \ldots, a_{i_n} \right)$ is $\mathcal{L}_{\infty}$-type-definable in $\mathcal{M}'$ by Lemma \ref{lem: lin ineq type def}. Then the following condition on the tuple $(a_{i_1}, \ldots, a_{i_n}, a_i)$ is also $\mathcal{L}_{\infty}$-type-definable in $\mathcal{M}'$:
\begin{align*}
	\exists \bar{b}_1 \ldots \exists \bar{b}_n \Bigg(  & \norm{\chi_{F^{<r}_{a_i}} - h_{n, \bar{\alpha}^1, \bar{a}, \bar{b}_1, \ldots, \bar{b}_n }}_{L^2} \leq \beta_1  + \gamma' \ \land \\
	& \norm{\chi_{F^{\geq s'}_{a_i}} - h_{n, \bar{\alpha}^2, \bar{a}, \bar{b}_1, \ldots, \bar{b}_n }}_{L^2} \leq \beta_2  + \gamma' \ \land \\
	&  \mu_{\bar{1}^k} \left(F_{h_{n, \bar{\alpha}^1, \bar{a}, \bar{b}_1, \ldots, \bar{b}_n }, h_{n, \bar{\alpha}^2, \bar{a}, \bar{b}_1, \ldots, \bar{b}_n },  \gamma} \right) \geq q - \varepsilon \Bigg).
\end{align*}

It holds for the tuple $\left( a_{i_1}, \ldots, a_{i_n}, a_0 \right)$ by \eqref{eq: choice of approx Beil}, the choice of $\gamma'$ and Claims 1 and 2. Hence, by $\mathcal{L}_{\infty}$-indiscernibility of the sequence $(a_i)_{i \in \mathbb{Z}}$, it holds for any tuple $\left( a_{i_1}, \ldots, a_{i_n}, a_i \right)$ with $i \in \mathbb{N}$; let $\vec{b}_i = \left(\bar{b}_1^i, \ldots, \bar{b}^i_n \right)$ be some tuple witnessing that. 

In particular, for every $i \in \mathbb{N}$, we have $\mu_{\bar{1}^k} \left(F_{h_{n, \bar{\alpha}^1, \bar{a}, \vec{b}_i }, h_{n, \bar{\alpha}^2, \bar{a}, \vec{b}_i }, \gamma} \right) \geq q - \varepsilon \geq \frac{q}{2}$ by the choice of $\varepsilon$. Then, by the choice of $\xi$ and Fact \ref{fac: Bergelson}, there exist some $j_1 < \ldots < j_l \in \mathbb{N}$ such that  
\begin{gather*}
	\mu_{\bar{1}^k} \left(\bigcap_{p \in [l]} F_{h_{n, \bar{\alpha}^1, \bar{a}, \vec{b}_{j_p} }, h_{n, \bar{\alpha}^2, \bar{a}, \vec{b}_{j_p} }, \gamma}  \right) \geq \xi.
\end{gather*}

Then, by Lemma \ref{lem: lin ineq type def} and $\mathcal{L}_{\infty}$-indiscernibility of $\left(a_i\right)_{i \in \mathbb{Z}}$ again,
there exist some $\vec{b}'_1, \ldots, \vec{b}'_l$ so that: 
\begin{enumerate}
	\item[(a)] $\norm{\chi_{F^{<r}_{a_i}} - h_{n, \bar{\alpha}^1, \bar{a}, \vec{b}'_i}}_{L^2} \leq \beta_1  + \gamma'$ for every $i \in [l]$;
	\item[(b)] $\norm{\chi_{F^{\geq s'}_{a_i}} - h_{n, \bar{\alpha}^2, \bar{a}, \vec{b}'_i}}_{L^2} \leq \beta_2  + \gamma'$ for every $i \in [l]$;
	\item[(c)] $\mu_{\bar{1}^k} \left(\bigcap_{i \in [l]} F_{h_{n, \bar{\alpha}^1, \bar{a}, \vec{b}'_i }, h_{n, \bar{\alpha}^2, \bar{a}, \vec{b}'_i }, \gamma}  \right) \geq \xi$.
\end{enumerate}
By (a), (b), Claim 2 and the choice of $\gamma'$, for every $i \in [l]$ we have 
$$\norm{ \E \left( \chi_{F^{<r}_{a_i}} \mid \B  \right) - h_{n, \bar{\alpha}^1, \bar{a}, \vec{b}'_i}}_{L^2} \leq \gamma^{\frac{3}{2}} \ \land \ \norm{ \E \left( \chi_{F^{\geq s'}_{a_i}} \mid \B  \right) - h_{n, \bar{\alpha}^2, \bar{a}, \vec{b}'_i}}_{L^2} \leq \gamma^{\frac{3}{2}}.$$
By Claim 1 this implies that $\mu_{\bar{1}^k} \left( F^{r,s'}_{\delta}(a_i) \triangle F_{h_{n, \bar{\alpha}^1, \bar{a}, \vec{b}'_i }, h_{n, \bar{\alpha}^2, \bar{a}, \vec{b}'_i }, \gamma}  \right) < \varepsilon$ for every $i \in [l]$.
But then from (c),  $\mathcal{M}' \propto \mathcal{M}_{\mathfrak{P},f}$ and  monotonicity of conditional expectation, we have
\begin{gather*}
	\mu_{\bar{1}^k} \left(  \bigcap_{i \in [l]} G^{r',s'}_{\delta}(a_i) \right)  \geq \mu_{\bar{1}^k} \left(  \bigcap_{i \in [l]} F^{r,s'}_{\delta}(a_i) \right) \geq \\
	\mu_{\bar{1}^k} \left(  \bigcap_{i \in [l]} F_{h_{n, \bar{\alpha}^1, \bar{a}, \vec{b}'_i}, h_{n, \bar{\alpha}^2, \bar{a}, \vec{b}'_i}, \gamma} \right) - l\varepsilon \geq \xi - l \varepsilon > 0
\end{gather*}

by the choice of $\varepsilon$.
\end{proof}

\subsection{Passing to an indiscernible counterexample}

Finally, we use the results developed in this section to show how to achieve the Assumption \ref{ass: ineapproximable} in the proof of Proposition \ref{prop: finite VCk-dim implies approx}.

\begin{theorem}\label{thm:pass to indiscernible}
  Let $\bar d$ be fixed and suppose that for each $j$ there is a $(k+1)$-partite graded probability space $\mathfrak{P}_j=(V^j_{[k+1]},\mathcal{B}^j_{\bar n},\mu^j_{\bar n})_{\bar n\in\mathbb{N}^{k+1}}$, a $\mathcal{B}^j_{\bar 1^{k+1}}$-measurable function $f^j:\prod_{i\in[k+1]}V^j_i\to[0,1]$ with $\VC_k(f^j)\leq\bar d$ and some $x_1^j, \ldots, x_{j}^j \in V^j_{k+1}$ such that for every $t \leq j$ we have:
  for any sets $D_1, \ldots, D_{j} \in \mathcal{F}^{f^j, j,  (x^j_1, \ldots, x^j_t)}$ and any $\left(\{ D_i \}_{i \in [j]} \cup \{f^{<q}_{x^j_i}\}_{i \in [t-1], q \in \mathbb{Q}_{j}^{[0,1]}} \right)$-simple function $g$ with coefficients in $\mathbb{Q}_{j}^{[0,1]}$, $\norm{f_{x_t^j} - g}_{L^2} \geq \varepsilon$.

  Then there exists a $(k+1)$-partite graded probability space 
  $$\mathfrak{P} = (V_{[k+1]}, \B_{\bar{n}}, \mu_{\bar{n}})_{n \in \mathbb{N}^{k+1}},$$
   $\varepsilon \in \mathbb{R}_{>0}$, a $ \B_{\bar{1}^k}$-measurable function $f: V^{\bar{1}^{k+1}} \to [0,1]$
	and a sequence $(x_l)_{l \in \mathbb{Z}}$ in $V_{k+1}$ satisfying the following:
	\begin{enumerate}
	\item  $\VC_k(f) \leq \bar{d}$;
        \item whenever $0\leq r<r'<s'<s\leq s$ are in $\mathbb{Q}$, $\delta\in\mathbb{R}_{>0}$, and
          $$\mu_{\bar 1^k} \left(\left\{\bar{x} \in V^{\bar{1}^k} \mid \E\left(\chi_{f^{<r}_{x_0}}\mid \B \right)(x) \geq \delta \land \E \left( \chi_{f^{\geq s}_{x_0}} \mid \B \right)(x) \geq  \delta  \right\} \right)>0,$$
          then for any $l\in\mathbb{N}$,
          $$\mu_{\bar 1^k} \left(\bigcap_{i\in[l]}\left\{\bar{x} \in V^{\bar{1}^k} \mid \E\left(\chi_{f^{<r'}_{x_i}}\mid \B \right)(x) \geq \delta \land \E \left( \chi_{f^{\geq s'}_{x_i}} \mid \B \right)(x) \geq  \delta  \right\} \right)>0;$$     
	\item $ \norm{f_{x_l} - \mathbb{E} \left(f_{x_l}\mid\mathcal{B}_{\bar{1}^k,k-1}\cup\{f_{x_i} : i < l\} \right)}_{L^2} \geq\varepsilon$ for all $l \in \mathbb{Z}$;
		\item $\B_{\bar{1}^k,k-1} \subseteq \B \subseteq \B_{\bar{1}^k}$;
		\item for all $l \in \mathbb{N}$ we have 
		\begin{gather*}
			\mathbb{E}\left(f_{x_l}\mid\mathcal{B}_{\bar{1}^k,k-1}\cup\{f_{x_i} : i < l\} \right)\\
			=\mathbb{E} \left(f_{x_l}\mid\mathcal{B}\cup\{f_{x_i} : i < l\} \right)\\
			= \mathbb{E} \left(f_{x_l}\mid\mathcal{B} \right),
		\end{gather*}
	\end{enumerate}
where $\B := \sigma \left(\left\{f_{x_i} : i < 0\right\} \cup \mathcal{B}_{\bar{1}^k,k-1} \right)$.
\end{theorem}
\begin{proof}
Let $\cU$ be a non-principal ultrafilter on $\mathbb{N}$. Let $\tilde{\mathfrak{P}} :=\left(\tilde{V}_{[k]}, \tilde{\B}_{\bar{n}}, \tilde{\mu}_{\bar{n}} \right)_{\bar{n} \in \mathbb{N}^{k}}$ be the $(k+1)$-partite graded probability space, the $\tilde{\B}_{\bar{1}^{k+1}}$-measurable function $\tilde{f}: \tilde{V}^{\bar{1}^{k+1}} \to [0,1]$ and $\mathcal{\tilde{M}}$ the $\mathcal{L}_{\infty}$-structure defined by the corresponding ultraproduct in  Section \ref{sec: ultraproducts of k-GPS} (Fact \ref{fac: ultraproduct props}).

\begin{claim}\label{lem: taking UP of counterex}
	\begin{enumerate}
		\item $\VC_k(\tilde{f}) \leq \bar{d}'$ for some $\bar{d}' = \bar{d}'(\bar{d})$.
		\item There exists an infinite sequence $(x_i : i \in \mathbb{Z})$ in $\tilde{V}_{k+1}$ such that $\tilde{\mathcal{M}} \models \pi_{\mathbb{Z}, \varepsilon}\left( (x_i)_{i \in \mathbb{Z}} \right)$ and $(x_i : i \in \mathbb{Z})$ is $\mathcal{L}_{\infty}$-indiscernible in $\tilde{\mathcal{M}}$.
	\end{enumerate}
\end{claim}
\begin{claimproof}

	(1) By Lemma \ref{lem: shattering is definable}.	
	
	(2) 	
	For $i \in \mathbb{N}$, let $\tilde{x}_i := \left(x^j_i : j \in \mathbb{N} \right) / \cU \in \tilde{V}_{k+1}$.
	
	Let $\pi_0$ be an arbitrary finite set of formulas from $\pi_{\mathbb{N}, \varepsilon}$, all formulas in $\pi_{0}$ only involve the variables $ z_{i_1}, \ldots, z_{i_n}$ for some $n \in \mathbb{N}$ and $i_1 \leq \ldots \leq i_{n} \in \mathbb{N}$. From the definition of $\pi_{\mathbb{N}, \varepsilon}$ and Lemma \ref{lem: meaning of formulas} (as obviously $\mathcal{M}_{\mathfrak{P}^j, f^j} \propto \mathcal{M}_{\mathfrak{P}^j, f^j}$ for every $j \in \mathbb{N}$), it is not hard to see that for all sufficiently large	$j \in \mathbb{N}$ (so that $j > n$ and all the rational coefficients appearing among the formulas in $\pi_0$ are in $\mathbb{Q}_j^{[0,1]}$) we have 
	$$\mathcal{M}_{\mathfrak{P}_j,f^j} \models \pi_{0}\left(x^j_{i_1}, \ldots, x^j_{i_n} \right),$$
	hence by \L os' theorem
	$$\tilde{\mathcal{M}}  \models \pi_{0}\left(\tilde{x}_{i_1}, \ldots, \tilde{x}_{i_n} \right),$$
	and so $$\tilde{\mathcal{M}} \models \pi_{\mathbb{N}, \varepsilon} \left( \left(\tilde{x}_i \right)_{i \in \mathbb{N}} \right).$$
	
	As $\tilde{\mathcal{M}}$ is an $\aleph_1$-saturated $\mathcal{L}_{\infty}$-structure, by Fact \ref{fac: existence of indiscernibles}(1) we can find an infinite $\mathcal{L}_{\infty}$-indiscernible sequence $(x_i : i \in \mathbb{Z})$ in $\tilde{V}_k$ based on $(\tilde{x}_i)_{i \in \mathbb{N}}$. In particular, using Remark \ref{rem : star subsequence}, $\tilde{\mathcal{M}} \models \pi_{\mathbb{Z}, \varepsilon} \left( \left(x_i \right)_{i \in \mathbb{Z}} \right)$.
\end{claimproof}

As $\tilde{\mathcal{M}} \propto \mathcal{M}_{\tilde{\mathfrak{P}}, \tilde{f}}$ by Remark \ref{rem : UP is prop to MP 2}, by Lemma \ref{lem: meaning of formulas} for every $i \in \mathbb{Z}$ we have
\begin{gather*}
		 	\norm{\tilde{f}_{a_i} - \E \left(\tilde{f}_{a_i} \mid \tilde{\B}_{\bar{1}^k, k-1} \cup \{ \tilde{f}_{a_j} : j \in I \land j <i \} \right)}_{L^2} \geq \varepsilon.
\end{gather*}
	
Taking $\mathfrak{P} := \tilde{\mathfrak{P}}, \mathcal{M}' := \tilde{\mathcal{M}}, f := \tilde{f}$, replacing $\bar{d}$ by $\bar{d}'$ and applying Proposition \ref{prop: de Finetti}, we have thus arrived at the desired situation.
\end{proof}

\section{Operations on functions preserving finite $\VC_k$-dimension}\label{sec:operations}

\subsection{Basic operations }

In this section we demonstrate that finiteness of the $\VC_k$-dimension is preserved under various natural operations on real-valued functions, obtaining a generalization of Fact \ref{fac: basic prop of k-dep}. These results are used in the proof of the main Theorem \ref{thm: the very main thm soft} in particular.


\begin{lemma}\label{lem: shattering is definable}
	Assume that, in the notation of Section \ref{sec: ultraproducts of k-GPS}, for some $\bar{d}$ we have $\VC_{k} \left( f^j \right) \leq \bar{d}$ for all $j \in \mathbb{N}$. Then $\VC_{k}\left( \tilde{f} \right) \leq \bar{d}'$, where we can take $d'_{r,s} := d_{r', s'}$ for any $r < r' < s' < s$ in $\mathbb{Q}^{[0,1]}$.
\end{lemma}
\begin{proof}
	Fix arbitrary $r< r' < s' < s \in \mathbb{Q} \cap [0,1]$. By assumption, for any $j \in \mathbb{N}$, no $d_{r',s'}$-box can be $(r',s')$-shattered by $f^j$, hence 
	\begin{gather*}
		\mathcal{M}_{f^j} \models \neg \exists \left(x^{s}_{t} :  s \in [k], t \in [d_{r',s'}]\right), \left(y_{u} : u \subseteq [d_{r',s'}]^k \right) \bigwedge_{u \subseteq [d_{r',s'}]^k}\\
		\bigwedge_{(t_1, \ldots, t_k) \in u} F^{<r'}(x^1_{t_1}, \ldots, x^k_{t_k}) \land \bigwedge_{(t_1, \ldots, t_k) \in [d_{r',s'}]^k \setminus u} F^{\geq s'}(x^1_{t_1}, \ldots, x^k_{t_k}).
	\end{gather*}
	
	By \L os' theorem, the same $\mathcal{L}_0$-sentence holds in the ultraproduct $\tilde{\mathcal{M}}$ as well. As $\tilde{\mathcal{M}} \propto \mathcal{M}_{\tilde{f}}$, this implies that no $d_{r',s'}$-box is $(r,s)$-shattered by $\tilde{f}$.
\end{proof}
The following characterization of finiteness of $\VC_k$-dimension in terms of generalized indiscernibles was observed in \cite[Lemma 6.2]{chernikov2014n} for relations, and we generalize it to real-valued functions.
\begin{lemma}\label{lem: full arity indisc witness to IPn} 

Let $\mathcal{M}$ be an $\aleph_1$-saturated $\mathcal{L}$-structure in a language $\mathcal{L} \supseteq \mathcal{L}_0$, $k < r \in \mathbb{N}$, $\mathcal{M}\restriction_{\mathcal{L}_0} \propto \mathcal{M}_{f} = \left(V_1, \ldots, V_r, \ldots \right)$, $f: \prod_{i \in [r]}V_i \to [0,1]$ and $\bar{b} \in \prod_{i \in [r] \setminus [k+1]} V_i$ (could be an empty tuple when $r = k+1$). Then the following are equivalent for $f_{\bar{b}}:  \prod_{i \in [k+1]} V_i \to [0,1], (a_1, \ldots, a_{k+1}) \mapsto f(a_1, \ldots, a_{k+1}, \bar{b})$.
\begin{enumerate}
	\item $\VC_k \left( f_{\bar{b}} \right) = \infty$.
	\item There exist some $r < s \in \mathbb{Q} \cap [0,1]$ and elements $(a_g)_{g \in G_{k+1,p}}$ in $\mathcal{M}$ such that:
	\begin{enumerate}
	\item $g \in P_i \implies a_{g} \in V_{i}$;
	\item $(a_g)_{g \in G_{k+1,p}}$ is $G_{k+1,p}$-indiscernible over $\emptyset$ (in $\mathcal{M}$);
	\item  For all $(g_1, \ldots, g_{k+1}) 	\in \prod_{i \in [k+1]} P_i$ we have:
	\begin{itemize}
		\item $G_{k+1,p} \models R_{k+1}(g_1, \ldots, g_{k+1}) \Rightarrow f(a_{g_1}, \ldots, a_{g_{k+1}}, \bar{b}) \leq r$;
		\item $G_{k+1,p} \models \neg R_{k+1}(g_1, \ldots, g_{k+1}) \Rightarrow f(a_{g_1}, \ldots, a_{g_{k+1}}, \bar{b}) \geq s$.
	\end{itemize}
	\end{enumerate}
	\end{enumerate}	
\end{lemma}
\begin{proof}
	
	(2) $\Rightarrow$ (1). Assume that (2) holds, and let $Q_i \subseteq P_i, i \in [k]$ be arbitrary finite sets and $Q := \prod_{i \in [k]} Q_i$. By the definition of $G_{k+1,p}$ (Definition \ref{def: generic hypergraph}), for every subset $S \subseteq Q$ there exists some $g_{S} \in P_{i+1}$ so that for every $(g_1, \ldots, g_k)\in Q$ we have $G_{k+1,p} \models R_{k+1}(g_1, \ldots, g_{k}, g_S) \iff (g_1, \ldots, g_k) \in S$. By (c) this implies that, taking $A_i := \{a_g : g \in Q_i\} \subseteq V_i$, the box $A := \prod_{i \in [k]} A_i$ is $(r,s)$-shattered by $f_{\bar{b}}$.
	
%
	
	(1) $\Rightarrow$ (2). Assume that $r<s \in \mathbb{Q} \cap [0,1]$ are such that there for every $d \in \mathbb{N}$ there exists a finite box $A = \prod_{i \in [k]} A_i \subseteq \prod_{i \in [k]} V_i$ with $|A_i| \geq d$ for each $i \in [k]$  which is  $(r,s)$-shattered by $f_{\bar{b}}$. 
	In particular, for any finite $(k+1)$-partite hypergraph $(R; D_1, \ldots, D_{k+1})$ with $R \subseteq \prod_{i \in [k+1]}D_i$ we can choose some sets $A_{i} \subseteq V_{i}$ and bijections $\alpha_i: D_i \to A_i$ so that for every $(b_1, \ldots, b_{k+1}) \in \prod_{i \in [k+1]}D_i$, 
	\begin{gather}\label{eq: perm of vars pres VCk dim1}
		(b_1, \ldots, b_{k+1}) \in R \implies f \left(\alpha_1(b_1), \ldots, \alpha_{k+1}(b_{k+1}), \bar{b} \right) \leq r;\\
		(b_1, \ldots, b_{k+1}) \notin R \implies f \left(\alpha_1(b_1), \ldots, \alpha_{k+1}(b_{k+1}), \bar{b} \right) \geq s.\nonumber
	\end{gather}
	Fix arbitrary $r',s' \in \mathbb{Q}^{[0,1]}$ with $r < r' < s' < s$ and consider the countable partial $\mathcal{L}_0$-type $\pi \left( (x_g)_{g \in G_{k+1,p}} \right)$ with a finite tuple of parameters $\bar{b}$ given by 
		\begin{gather*}
		\bigwedge_{(g_1, \ldots, g_{k+1}) \in R_{k+1}} F^{<r'}(x_{g_1}, \ldots, x_{g_k}, \bar{b}) \land  \\
		\bigwedge_{(g_1, \ldots, g_{k}) \in \prod_{i \in [k+1] P_i \setminus R_{k+1}}} F^{\geq s'}(x_{g_1}, \ldots, x_{g_k}, \bar{b}).
	\end{gather*}
	By \eqref{eq: perm of vars pres VCk dim1} and using $\mathcal{M}\restriction_{\mathcal{L}_0} \propto \mathcal{M}_{f}$, every finite set of formulas from $\pi$ is realized in $\mathcal{M}$.
Then, by $\aleph_1$-saturation of $\mathcal{M}$, we can find some tuples $(a_g)_{g \in G_{k+1,p}}$ (with $ a_g \in V_i$ for $g \in P_i$) so that $\mathcal{M} \models \pi \left((a_g)_{g \in G_{k+1,p}} \right)$.

By Fact \ref{fac: existence of indiscernibles}(2), let $(a'_g )_{g \in G_{k+1,p}}$ be \emph{$G_{k+1,p}$-indiscernible over $\bar{b}$} in $\mathcal{M}$ based on $(a_g )_{g \in G_{k+1,p}}$. Then we still have $\mathcal{M} \models \pi \left((a'_g)_{g \in G_{k+1,p}} \right)$. In particular, using $\mathcal{M}\restriction_{\mathcal{L}_0} \propto \mathcal{M}_{f}$ again, we get that $(a'_g )_{g \in G_{k+1,p}}$ satisfies (c) with respect to $r', s'$.
\end{proof}

	        Next we show an analog of Fact \ref{fac: basic prop of k-dep} for real valued functions (generalizing \cite[Proposition 3.7]{yaacov2009continuous} in the case $k=1$). We will use the following variant of the Stone-Weierstrass theorem.

\begin{fac}\cite[Proposition 1.14]{ben2010continuous}\label{fac: dense system of conn} Let $X$ be a compact Hausdorff space. Assume that $\mathfrak{B} \subseteq C(X, [0,1])$ satisfies the following:
\begin{enumerate}
	\item if $f \in \mathfrak{B}$, then $1-f \in \mathfrak{B}$;
	\item if $f,q \in \mathfrak{B}$, then $f \monus g \in \mathfrak{B}$ (where for any $x,y \in [0,1]$, $x \monus y := \max \{x-y, 0 \}$;
	\item if $f \in \mathfrak{B}$, then $\frac{f}{2} \in \mathfrak{B}$;
	\item if $x \neq y \in X$, then $f(x) \neq f(y)$ for some $f \in \mathfrak{B}$.
\end{enumerate}
Then $\mathfrak{B}$ is dense in $C(X, [0,1])$ with respect to the uniform convergence topology.
\end{fac}

\begin{lemma}\label{lem: monus pres VCk}
	\begin{enumerate}
		\item Assume that a sequence of functions $f_i: \prod_{i \in [k+1]}V_i \to [0,1], i \in \mathbb{N}$ converges uniformly to $g:  \prod_{i \in [k+1]}V_i  \to [0,1]$, and $\VC_{k}(f_i) < \infty$ for every $i \in \mathbb{N}$. Then also $\VC_{k}(g) < \infty$.
		\item For every $\bar{d}$ there exists some $\bar{d}'$ such that if $f,g: \prod_{i \in [k+1]}V_i \to [0,1]$ and $\VC_{k}(f), \VC_{k}(g) \leq \bar{d}$, then $\VC_{k}\left(\frac{f}{2} \right), \VC_{k}(1 - f), \VC_{k}\left(f \monus g \right) \leq \bar{d}'$.
\end{enumerate}
\end{lemma}
\begin{proof}
	(1) Let $r<s \in [0,1]$ be arbitrary, and assume that some box $A = \prod_{i \in [k]}A_i$ with each $A_i$ infinite is $(r,s)$-shattered by $g$. Let $\varepsilon := \frac{s-r}{3} > 0$. By assumption there exists some $n \in \mathbb{N}$ such that $|f_n(\bar{x}) - g(\bar{x})| < \varepsilon$ for every $\bar{x} \in \prod_{i \in [k+1]}V_i$. But then $A$ is $(r+\varepsilon, s - \varepsilon)$-shattered by $f_n$.
	
	(2) It is clear that if a box $A \subseteq \prod_{i \in [k]} V_i$ is $(r,s)$-shattered by $\frac{f}{2}$ then it is $(2r,2s)$-shattered by $f$, and if $A$ is $(r,s)$-shattered by $1-f$ then it is $(1-s,1-r)$-shattered by $f$.

        Suppose $A$ is $(r,s)$-shattered by $f\monus g$ where $A=\prod_{i\in[k]}A_i$ with $|A_i|$ sufficiently large (as determined later).  Let $\epsilon=s-r$.  By Ramsey's Theorem, we may choose a box $A'=\prod_{i\in[k]}A'_i$ with

	Assume towards a contradiction that there exist some $\bar{d}$ and $r < s \in \mathbb{Q}^{[0,1]}$ such that for any $j \in \mathbb{N}$ there exist some functions $f^j_1,f^j_2: \prod_{i \in [k+1]} V_i \to [0,1]$ such that $\VC_{k}(f^j_1), \VC_{k}(f^j_2) < \bar{d}$ but $f^j_3 := f^j_1 \monus f^j_2$ $(r,s)$-shatters some box $\prod_{i \in [k]} A^j_i$ with $|A^j_i| \geq j$. Let $\tilde{\mathcal{M}} := \prod_{j \in \mathbb{N}} \mathcal{M}_{f^j_1, f^j_2, f^j_3} / \mathcal{U}$ and $\tilde{A}_i := \prod_{j \in \mathcal{U}} A^j_i$ for $i \in [k]$. Then we have:
	\begin{itemize}
	\item $\tilde{\mathcal{M}} \propto \mathcal{M}_{\tilde{f}_1, \tilde{f_2}, \tilde{f}_3}$ (by Fact \ref{fac: ultraproduct props});
	\item $\tilde{f_3} = \tilde{f_1} \monus \tilde{f_2}$ (easy as  $\tilde{\mathcal{M}} \propto \mathcal{M}_{\tilde{f}_1, \tilde{f_2}, \tilde{f}_3}$);	
	\item $\VC_{k}(\tilde{f}_1), \VC_k(\tilde{f}_2) \leq \bar{d}'$ for some $\bar{d}' < \infty$ (by Lemma \ref{lem: shattering is definable});
	\item for any $r',s' \in \mathbb{Q}^{[0,1]}$ with $r < r' < s' < s$, $\tilde{f_3}$ $(r',s')$-shatters the box $\prod_{i \in [k]} \tilde{A}_i$ and each $\tilde{A}_i$ is infinite (by Lemma \ref{lem: shattering is definable}).
	\end{itemize}

By Lemma \ref{lem: full arity indisc witness to IPn}, (1)$\Rightarrow$(2) there exist some $r < r' < s' < s \in \mathbb{Q}^{[0,1]}$ and $(a_g)_{g \in G_{k+1,p}}$ (with $g \in P_i \Rightarrow a_g \in \tilde{V}_{i}$ for $i \in [k+1]$) so that $(a_g)_{g \in G_{k+1,p}}$ is $G_{k+1,p}$-indiscernible (in $\tilde{\mathcal{M}}$),  and for all $(g_1, \ldots, g_{k+1}) 	\in \prod_{i \in [k+1]} P_i$ we have:
	\begin{itemize}
		\item if $G_{k+1,p} \models R_{k+1}(g_1, \ldots, g_{k+1})$ then $\tilde{f_1} \monus \tilde{f_2} (a_{g_1}, \ldots, a_{g_{k+1}}) \leq r'$;
		\item if $G_{k+1,p} \models \neg R_{k+1}(g_1, \ldots, g_{k+1})$ then $\tilde{f_1} \monus \tilde{f_2}(a_{g_1}, \ldots, a_{g_{k+1}}) \geq s'$.
	\end{itemize}
	
	Fix some $(g_1, \ldots, g_{k+1}) \in R_{k+1}$ and $(h_1, \ldots, h_k) \in \prod_{i \in [k+1]} P_i \setminus R_{k+1}$. By definition of $\monus$, one of the following two cases must occur:
	\begin{itemize}
		\item $\tilde{f}_1(a_{h_1}, \ldots, a_{h_{k+1}}) - \tilde{f}_1(a_{g_1}, \ldots, a_{g_{k+1}})  \geq \frac{s' - r'}{2}$;
		\item $\tilde{f}_2(a_{g_1}, \ldots, a_{g_{k+1}}) - \tilde{f}_2(a_{h_1}, \ldots, a_{h_{k+1}}) \geq \frac{s' - r'}{2}$.
	\end{itemize}
	
	In the first case, let $r'' := \tilde{f}_1(a_{g_1}, \ldots, a_{g_{k+1}})$ and $s'' := \tilde{f}_1(a_{h_1}, \ldots, a_{h_{k+1}})$, then $r'' < s'' \in [0,1]$ and we have 
		\begin{itemize}
		\item if $G_{k+1,p} \models R_{k+1}(g_1, \ldots, g_{k+1})$ then $\tilde{f_1} (a_{g_1}, \ldots, a_{g_{k+1}}) = r''$;
		\item if $G_{k+1,p} \models \neg R_{k+1}(g_1, \ldots, g_{k+1})$ then $\tilde{f_1} (a_{g_1}, \ldots, a_{g_{k+1}}) = s''$.
	\end{itemize}
	Indeed, for any $(g'_1, \ldots, g'_{k+1}) \in R_{k+1}$ we obviously have 
	$$\qftp_{\mathcal{L}^{k+1}_{\opg}}(g'_1, \ldots, g'_{k+1}) = \qftp_{\mathcal{L}^{k+1}_{\opg}}(g_1, \ldots, g_{k+1}).$$
	 Hence by $G_{k+1,p}$-indiscernibility of $(a_g)_{g \in G_{k+1,p}}$, for any $q \in \mathbb{Q}^{[0,1]}$ we have $$\tilde{\mathcal{M}} \models F_1^{<q}\left(a_{g_1}, \ldots, a_{g_{k+1}} \right) \iff \tilde{\mathcal{M}} \models F_1^{<q} \left(a_{g'_1}, \ldots, a_{g'_{k+1}} \right),$$
which using $\tilde{\mathcal{M}} \propto \mathcal{M}_{\tilde{f}_1}$ implies $\tilde{f_1} \left(a_{g_1}, \ldots, a_{g_{k+1}} \right) = \tilde{f_1} \left( a_{g'_1}, \ldots, a_{g'_{k+1}} \right)$ (and the second bullet is similar). By Lemma \ref{lem: full arity indisc witness to IPn}, (2)$\Rightarrow$(1) this implies $\VC_k(1-\tilde{f}_1) = \infty$, hence $\VC_k(\tilde{f}_1) = \infty$ --- a contradiction.

In the second case, a similar argument shows that $\VC_k(\tilde{f}_2) = \infty$.
	\end{proof}

\begin{prop}\label{prop: VCk comp with cont funct}
	Assume that $n \in \mathbb{N}$ and $g: [0,1]^n \to [0,1]$ is an arbitrary continuous function. Then for any $\bar{d} < \infty$ there exists some $\bar{D} = \bar{D} \left(\bar{d},g\right) < \infty$ satisfying the following. Let $f_1, \ldots, f_n : V^{\bar{1}^{k+1}} \to [0,1]$ satisfy $\VC_k(f_i) < \bar{d}$ for $i \in [n]$. Then $h := g(f_1, \ldots, f_n) : V^{\bar{1}^{k+1}} \to [0,1]$ satisfies $\VC_{k}(h) < \bar{D}$. 
\end{prop}
\begin{proof}
	By Fact \ref{fac: dense system of conn}, $g$ can be uniformly approximated by finite compositions of the functions $\frac{x}{2}, 1-x, x \monus y$. Plugging $f_1, \ldots, f_n$ into the arguments of such a composition gives a function of finite $\VC_{k}$-dimension by Lemma \ref{lem: monus pres VCk}(2), hence $g$ has finite $\VC_k$-dimension by Lemma \ref{lem: monus pres VCk}(1).
\end{proof}

Permutation of the variables of function also preserves finiteness of the $\VC_k$-dimension.
\begin{prop}\label{prop: perm of vars fin VCk dim}
	Given $\bar{d}=(d)_{r<s \in \mathbb{Q}^{[0,1]}} < \infty$, let $\bar{D} = (D_{r,s})_{r<s \in \mathbb{Q}^{[0,1]}} < \infty$ be given by $D_{r,s} := 2^{d^k_{r,s}}$. Assume $f: \prod_{i \in [k+1]} V_i \to [0,1]$ satisfies $\VC_{k}(f) \leq \bar{d}$ and $\sigma: [k+1] \to [k+1]$ is an arbitrary permutation. Let $f^{\sigma}: \prod_{i \in [k+1]} V_{\sigma(i)} \to [0,1]$ be given by $f^{\sigma}(x_1, \ldots, x_{k+1}) := f \left(x_{\sigma(1)}, \ldots, x_{\sigma(k+1)} \right)$. Then $\VC_k(f^{\sigma}) \leq \bar{D}$.
\end{prop}
\begin{proof}
	If some box $A_1 \times \ldots \times A_k$ with $A_i \subseteq V_i, |A_i| = D$ is $(r,s)$-shattered by a function $f: \prod_{i \in [k+1]} V_i \to [0,1]$, then for any $(k+1)$-partite hypergraph $(R; [D], \ldots, [D])$ with $R \subseteq [D]^{k+1}$ we can choose some set $A_{k+1} \subseteq V_{k+1}$ and bijections $\alpha_i: [D] \to A_i$ so that for every $(b_1, \ldots, b_{k+1}) \in [D]^{k+1}$, 
	\begin{gather}\label{eq: perm of vars pres VCk dim}
		(b_1, \ldots, b_{k+1}) \in R \implies f \left(\alpha_1(b_1), \ldots, \alpha_{k+1}(b_{k+1}) \right) \leq r;\\
		(b_1, \ldots, b_{k+1}) \notin R \implies f \left(\alpha_1(b_1), \ldots, \alpha_{k+1}(b_{k+1}) \right) \geq s.\nonumber
	\end{gather}

Assume $D = 2^{d^{k}}$. Given a permutation $\sigma: [k+1] \to [k+1]$ with $\sigma(k+1) = i^*$, consider the $(k+1)$-partite hypergraph $(R; [D], \ldots, [D])$ so that for every $B \subseteq [d]^{k}$ there exists some $b_{B} \in [D]$ satisfying: for every $(b_{1}, \ldots, b_{i^* -1},  b_{i^* +1}, b_{k+1}) \in [d]^k$,
\begin{gather*}
	(b_{1}, \ldots, b_{i^* -1}, b_B, b_{i^* +1}, \ldots, b_{k+1}) \in R \iff (b_{1}, \ldots, b_{i^* -1},  b_{i^* +1}, b_{k+1}) \in B.
\end{gather*}

Combined with \eqref{eq: perm of vars pres VCk dim} and taking $A'_i := \{ \alpha_{\sigma(i)}(j) : j \in [d] \}$, this implies that the box $A'_1 \times \ldots \times A'_k$ with $|A'_i| = d, A'_i \subseteq V_{\sigma(i)}$  is  $(r,s)$-shattered by $f_{\sigma}$.

	\end{proof}
\subsection{Integration preserves finite $\VC_{k}$-dimension}

The aim of this subsection is to prove the following theorem, after developing some tools for it.

\begin{theorem}\label{prop: gen fib fin VCk dim}
For every $k \in \mathbb{N}_{\geq 1}$ and $\bar{d} = (d_{r,s})_{r < s \in \mathbb{Q}^{[0,1]}}$ with $d_{r,s} \in \mathbb{N}$ there exists some $\bar{D}= (D_{r,s})_{r < s \in \mathbb{Q}^{[0,1]}}$ with $D_{r,s} \in \mathbb{N}$ satisfying the following.

Assume that $(V_{[k+2]}, \B_{\bar{n}}, \mu_{\bar{n}})_{n \in \mathbb{N}^{k+2}}$ is a $(k+2)$-partite graded probability space, $f: V^{\bar{1}^{k+2}} \to [0,1]$ is a $\B_{\bar{1}^{k+2}}$-measurable function and $\VC_k(f) \leq \bar{d}$. Then the $(k+1)$-ary ``average'' function $f': V^{\bar{1}^{k+1}} \to [0,1]$ defined by
$$f'(x_1, \ldots, x_{k+1}) := \int  f(x_1, \ldots, x_{k+2}) d \mu_{\bar{\delta}_{k+2}}(x_{k+2})$$
satisfies $\VC_k(f') \leq \bar{D}$.
\end{theorem}

\begin{remark}
Theorem \ref{prop: gen fib fin VCk dim} generalizes \cite[Corollary 4.2]{yaacov2009continuous} in the case $k=1$.	
\end{remark}

\begin{cor}
	For every $k \in \mathbb{N}_{\geq 1}$ there exists some $\bar{D}= (D_{r,s})_{r < s \in \mathbb{Q}^{[0,1]}} < \infty$ satisfying the following.
	
Assume that $(V_{[k+2]}, \B_{\bar{n}}, \mu_{\bar{n}})_{n \in \mathbb{N}^{k+2}}$ is a $(k+2)$-partite graded probability space, and for each $I \in \binom{[k+1]}{\leq k}$ let $\bar{n}_I := \sum_{i \in I} \bar{\delta}_i + \bar{\delta}_{k+2}$ and $E^I \in \B_{\bar{n}_I}$ arbitrary.

Then the $(k+1)$-ary function $f': V^{\bar{1}^{k+1}} \to [0,1]$ defined by
$$f'(\bar{x}) \mapsto \mu_{\bar{\delta}_{k+2}} \left( \bigcap_{I \in \binom{[k+1]}{\leq k}} E^I_{\bar{x}_{I}} \right)$$
satisfies $\VC_k(f') \leq \bar{D}$.
\end{cor}
\begin{proof}
Consider the relation $F \in \B_{\bar{1}^{k+2}}$ defined by 
$$(x_1, \ldots, x_{k+2}) \in F :\iff \bigwedge_{I \in \binom{[k+1]}{\leq k}} \left( \bar{x}_I^{\frown}(x_2) \in  E^I \right).$$

	Then for any fixed $b \in V_{k+2}$, the $(k+1)$-ary relation $F_b$ is a conjunction of the $\leq k$-ary relations $E^I_{b}, I \in \binom{[k+1]}{\leq k}$, hence trivially $\VC_{k}\left(F_b \right) \leq \bar{d}$ with $d_{r,s} := 1$ for all $r,s \in \mathbb{Q}^{[0,1]}$. Applying Proposition \ref{prop: gen fib fin VCk dim} to $F$ and noting that $\mu_{\bar{\delta}_{k+2}} \left( \bigcap_{I \in \binom{[k+1]}{\leq k}} E^I_{\bar{x}_{I}} \right) = \int \chi_{F}(x_1, \ldots, x_{k+2})d \mu_{\bar{\delta}_{k+2}}(x_{k+2}) $, we can conclude.
\end{proof}
The same holds with any fixed Boolean combination instead of a conjunction.

\subsubsection{Intersections of measurable sets indexed by generic hypergraphs and exchangeability}\label{sec: Aldous-Hoover}

In this section we let $\nu_n$ denote the Lebesgue probability measure on $[0,1]^n$.

Given two collections of random variables $(\xi_i : i \in I)$ on a probability space $(V,\B,\mu)$ and $(\xi'_i : i \in I)$ on a probability space $(V',\B',\mu')$ indexed by the same ordered set $I$ and taking values in $[0,1]$, we write $(\xi_i : i \in I) =^{\dist} (\xi'_i : i \in I)$ to denote that they have the same joint distribution (that is, for every finite set $J \subseteq I$ and any $p_i \in [0,1]$ for $i \in J$, $\mu \left(\{x \in V : \bigwedge_{i \in J} \xi_i(x) < p_i\} \right) = \mu' \left( \{x \in V' : \bigwedge_{i \in J} \xi'_i(x) < p_i\}\right)$).

We will need a generalization of the Aldous-Hoover-Kallenberg theorem on exchangeable arrays of random variables \cite{aldous1981representations, hoover1979relations, kallenberg2006probabilistic} for a restricted form of exchangeability with respect to $k$-partite generic hypergraphs. We will rely on the setting of \cite{crane2018relatively}.

\begin{definition}
\begin{enumerate}
	\item Let $\mathcal{L}' = \{R'_1, \ldots, R'_{k'}\}$ be a finite relational language, with each $R'_i$ a relation symbol of arity $r'_i$. By a \emph{random $\mathcal{L}'$-structure} we mean a collection of random variables 
	$$ \left( \xi^{i}_{\bar{n}} : i \in [k'], \bar{n} \in \mathbb{N}^{r'_i} \right)$$
	 on some probability space $\left(V, \B, \mu \right)$ with $\xi^i_{\bar{n}}: V \to \{0,1\}$. (Equivalently, we can think of this as equipping the space of all countable $\mathcal{L}'$-structures with a measure, and picking a random $\mathcal{L}'$-structure according to it.)
	\item Let now $\mathcal{L} = \{R_1, \ldots, R_k\}$ be another relational language, with $R_i$ a relation symbol of arity $r_i$, and let $\mathcal{M} = (\mathbb{N}, \ldots)$ be a countable $\mathcal{L}$-structure with domain $\mathbb{N}$. We say that a random $\mathcal{L}'$-structure $ \left( \xi^{i}_{\bar{n}} : i \in [k'], \bar{n} \in \mathbb{N}^{r'_i} \right)$ is \emph{$\mathcal{M}$-exchangeable} if for any two finite subsets $A =\{a_1, \ldots, a_\ell\}, A' = \{a'_1, \ldots, a'_\ell\} \subseteq \mathbb{N}$ 
	\begin{gather*}
	\qftp_{\mathcal{L}}\left(a_1, \ldots, a_{\ell} \right) = \qftp_{\mathcal{L}}\left(a'_1, \ldots, a'_{\ell} \right) \implies \\
		\left( \xi^i_{\bar{n}} : i \in [k'], \bar{n} \in A^{r'_i} \right) =^{\dist} \left( \xi^i_{\bar{n}} : i \in [k'], \bar{n} \in (A')^{r'_i} \right).
	\end{gather*}
\end{enumerate} 
\end{definition}

Given a tuple $\bar{n} = (n_1, \ldots, n_r)$ we let $\range \bar{n}$ denote the set of distinct elements in $\bar{n}$, and write $\bar{m} \subseteq \bar{n}$ if $\bar{m} = (n_{p_1}, \ldots, n_{p_r'})$ for an increasing sequence $p_1 < \ldots < p_{r'} \in [r]$.

\begin{fac}\label{fac: Ald-Hoov}\cite[Theorem 3.2]{crane2018relatively} Let $\mathcal{L}' = \{R'_i : i \in [k']\}, \mathcal{L} = \{R_i : i \in [k]\}$ be finite relational languages with all $R'_i$ of arity at most $r'$, and $\mathcal{M} = (\mathbb{N}, \ldots)$ a countable ultrahomogeneous $\mathcal{L}$-structure that has $n$-DAP for all $n \geq 1$ (see Definition \ref{def: nDAP}). Suppose that $ \left( \xi^{i}_{\bar{n}} : i \in [k'], \bar{n} \in \mathbb{N}^{r'_i} \right)$ is a random $\mathcal{L}'$-structure that is $\mathcal{M}$-exchangeable such that the relations $R'_i$ are symmetric with probability $1$.

Then there exists a probability space  $(V',\B',\mu')$, $\{0,1\}$-valued Borel functions $f_1, \ldots, f_{r'}$ and a collection of $\Uniform$ i.i.d.~random variables $\left( \zeta_s : s \subseteq \mathbb{N}, |s| \leq r' \right)$ on $V'$ so that
\begin{gather*}
	\left( \xi^i_{\bar{n}} : i \in [k'], \bar{n} \in \mathbb{N}^{r'_i} \right)=^{\dist} \\
	\left( f_i \left( \mathcal{M}|_{\range \bar{n}}, \left( \zeta_s \right)_{s \subseteq \range \bar{n}}\right) : i \in [k'], \bar{n} \in \mathbb{N}^{r'_i} \right).
\end{gather*}
\end{fac}

\begin{remark}\label{rem: measure comparison}
Given $n \in \mathbb{N}$, let $(\zeta_i : i < n)$ be uniformly distributed $[0,1]$-valued independent random variables on a probability space $(V, \mathcal{B}, \mu)$. Let $A \subseteq [0,1]^n$ be a Borel set. Then 
$$\nu_n(A) = \mu \left(\left\{x \in V : \left(\zeta_1(x), \ldots, \zeta_n(x) \right) \in A \right\} \right).$$
\end{remark}

\begin{proof}
Assume $\nu_n(A) = r$, and let $\varepsilon > 0$ be arbitrary. As $A$ is measurable with respect to $\nu_n$, we can find some Borel sets $A_{1,j}, \ldots, A_{n,j} \subseteq [0,1]$ for $j \in \mathbb{N}$ such that $A \subseteq A' := \bigsqcup_{j \in \mathbb{N}} \prod_{1 \leq i \leq n}A_{i,j}$ and $\nu_n(A') \leq r + \varepsilon$. Then we have:
\begin{gather*}
	\mu \left(\left\{x \in V : (\zeta_1(x), \ldots, \zeta_n(x)) \in A \right\} \right) \leq \\
	\mu(\left\{x \in V : (\zeta_1(x), \ldots, \zeta_n(x)) \in A' \right\}) =\\
	\sum_{j \in \mathbb{N}}\mu\left( \left\{ x \in V : (\zeta_1(x), \ldots, \zeta_n(x)) \in \prod_{1 \leq i \leq n}A_{i,j} \right\}\right)\\
	\textrm{(by countable additivity and disjointness of the boxes)}\\
	= \sum_{j \in \mathbb{N}} \mu \left(\left\{ x \in V :\zeta_1 (x) \in A_{1,j} \right\} \right) \cdot \ldots \cdot \mu \left(\left\{ x \in V :\zeta_n (x) \in A_{n,j} \right\} \right)\\
	\textrm{(as the random variables $\zeta_1, \ldots, \zeta_n$ are independent)}\\
	= \sum_{j \in \mathbb{N}} \nu_1 \left(A_{1,j}\right) \cdot \ldots \cdot \nu_1 \left( A_{n,j}\right)\\
	\textrm{(as each of the random variables $\zeta_1, \ldots, \zeta_n$ is uniformly distributed on $[0,1]$)}\\
	= \nu_n(A') \leq r + \varepsilon.
\end{gather*}
Applying the same argument to the complement of $A$ we get that also 
$$\mu(\left\{x \in V : (\zeta_1(x), \ldots, \zeta_n(x)) \in A \right\}) \geq r - \varepsilon,$$
and, since $\varepsilon > 0$ was arbitrary, the claim follows.
\end{proof}

The following can be viewed as an analog of Lemmas \ref{lem: indisc Bergelson} (which in turn is an ``indiscernible'' version of Fact \ref{fac: Bergelson}), where instead of indexing by a sequence we are indexing by a generic partite hypergraph.

\begin{lemma} \label{lem: random graph Bergelson}
Let $(V,\mathcal{B}, \mu)$ be a probability space, and $k \in \mathbb{N}$ and $r \in [0,1]$ arbitrary.
	Let $G'_{k,p} = (P_1, \ldots, P_k, R_k)$ be the generic $k$-partite hypergraph (see Definition \ref{def: reducts of Gkp}). Assume that for each tuple $\bar{a}=(a_i)_{i \in [k]} \in \prod_{i \in [k]} P_i$ we have some sets $E^0_{\bar{a}}, E^1_{\bar{a}}  \in \mathcal{B}$ satisfying the following:
	\begin{enumerate}
		\item $\mu(E^0_{\bar{a}}) > \mu(E^1_{\bar{a}'})$ for some $\bar{a} \in R_k, \bar{a}' \notin R_k$;
		\item for any for any $m \in \mathbb{N}$, $\bar{a}^i = \left(a^i_1, \ldots, a^i_m \right) \in P_i$ and $\bar{b}^i = \left(b^i_1, \ldots, b^i_m \right) \in P_i$ for $i \in \{1, \ldots, k\}$ 
	\begin{gather*}
		\qftp_{\mathcal{L}^k_{\opg}}\left(\bar{a}^1, \ldots, \bar{a}^k \right) = \qftp_{\mathcal{L}^k_{\opg}}\left(\bar{b}^1, \ldots, \bar{b}^k \right) \implies \\
		\left( \chi_{E^t_{(a^1_{\ell_1}, \ldots, a^k_{\ell_k})}} : t \in \{0,1\},  (l_1, \ldots, l_k) \in [m]^k  \right) =^{\dist} \\
		\left( \chi_{E^t_{(b^1_{\ell_1}, \ldots, b^k_{\ell_k})}} : t \in \{0,1\}, (l_1, \ldots, l_k) \in [m]^k  \right).
	\end{gather*}

	\end{enumerate}  
	 Then for any finite $Q_i \subseteq P_i, i \in [k]$, taking $Q := \prod_{i \in [k]} Q_i$, we have
	$$\mu \Bigg( \bigcap_{\bar{a} \in Q \cap R_k} E^0_{\bar{a}}  \cap  \bigcap_{\bar{a} \in Q \setminus R_k} V \setminus E^1_{\bar{a}}\Bigg) > 0.$$
\end{lemma}

\begin{proof}
Without loss of generality the domain of $G'_{k,p}$ is $\mathbb{N}$, i.e.~$(\bigcup_{i \in [k]} P_i)^k = \mathbb{N}$.
For each $\bar{a} \in \prod_{i \in [k]}P_i$ and $t \in \{0,1\}$, let $\xi^t_{\bar{a}} := \chi_{E^t_{\bar{a}}}$. For any $\bar{a} \in \mathbb{N}^k \setminus \prod_{i \in [k]}P_i$, let $\xi^t_{\bar{a}}$ be the constant zero map for $t \in \{0,1\}$. By assumption (2) it follows that $\left( \xi^t_{\bar{a}} : t \in [2], \bar{a} \in \mathbb{N}^k \right)$ is a $G'_{k,p}$-exchangeable random $\mathcal{L}'$-structure for $\mathcal{L}'$ containing two $k$-ary relational symbols.  Since the relations are partite, they may be extended to symmetric relations containing only tuples with exactly one element from each part.  Besides, $G'_{k,p}$ is ultrahomogeneous by Fact \ref{fac: Ramsey classes}(4) and satisfies $n$-DAP for all $n \in \mathbb{N}_{\geq 1}$ by Proposition \ref{prop: Gkp has nDAP}.
Moreover, for any tuple $(g_1, \ldots, g_k) \in \prod_{i \in [k]}P_i$, there only two possible isomorphism types for the induced substructure $G'_{k,p}|_{\range (g_1, \ldots, g_k)}$ (see Definition \ref{def: induced substr}) --- one for $(g_1, \ldots, g_k) \in R_k$ and one for $(g_1, \ldots, g_k) \notin R_k$.
Hence, applying Fact \ref{fac: Ald-Hoov}, there exist
 a probability space $\left( V', \B', \mu' \right)$, a collection of $\Uniform$ i.i.d.~random variables $\zeta_{\bar{a}} : V' \to [0,1]$
	indexed by the tuples $ \bar{a} \in \bigcup_{I \subseteq [k]} \prod_{ i \in I} P_i$, and
	 Borel measurable functions $f^t_{s} : [0,1]^{2^k} \to \{0,1 \}$ for $t \in \{0,1\}, s \in \{+,-\}$, such that we have
\begin{gather}\label{eq: Ald Hoov pres}
	\left( \chi_{E^t_{\bar{a}}} : t \in \{0,1\}, \bar{a} \in \prod_{i \in [k]} P_i \right) =^{\dist}\\
	 \left( f^t_{\rho(\bar{a})} \left( \left( \zeta_{\bar{a}_I} : I \subseteq [k]\right) \right) : t \in \{0,1\}, \bar{a} \in \prod_{i \in [k]} P_i  \right),\nonumber
\end{gather}
where $\rho(\bar{a}) = +$ if $\bar{a} \in R_k$ and $\rho(\bar{a}) = -$ if $\bar{a} \notin R_k$.

Let $S_{+} := (f^0_+)^{-1} \left( \left\{1 \right\} \right)$ and $S_{-} := (f^1_{-})^{-1} \left( \left\{1 \right\} \right)$, both are Borel subsets of $[0,1]^{2^k}$. Let $\bar{a} \in R_k, \bar{a}' \in \prod_{i \in [k]} P_i \setminus R_k$ be as given by assumption (1). Then, using  Remark \ref{rem: measure comparison}, we have
\begin{gather*}
\mu(E^0_{\bar{a}}) = \mu'\left( \left\{x \in V' : f^0_+ \left( \zeta_{\bar{a}_I}(x) : I \subseteq [k] \right) = 1  \right\} \right) = \\
 \mu'\left( \left\{x \in V' :  \left( \zeta_{\bar{a}_I}(x) : I \subseteq [k] \right) \in S_{+}  \right\} \right) = \nu_{2^k} \left( S_+ \right).
\end{gather*}
Similarly, $\mu(E^1_{\bar{a}'}) = \nu_{2^k} \left( S_- \right)$. As $\mu \left(E^0_{\bar{a}} \right) > \left(E^1_{\bar{a}'} \right)$ by assumption, it follows that $\nu_{2^k}\left( S_+ \setminus S_- \right) > 0$.

Fix any $\varepsilon \in \mathbb{R}_{>0}$. Then, by the basic properties of Lebesgue measure, we can choose some $\left(A_I : I \subseteq [k] \right)$ with each $A_I$ a Borel subset of $[0,1]$ with $\nu_1(A_I) > 0$, so that, taking $A := \prod_{I \subseteq [k]} A_I$, we have
\begin{gather}\label{eq: box almost cont}
	\nu_{2^k} \left( A \cap \left( S_+ \setminus S_- \right) \right) \geq (1 -\varepsilon ) \cdot \nu_{2^k} \left( A \right).
\end{gather}

Let $Q_i \subseteq P_i$ be arbitrary finite subsets. It is enough to prove the lemma assuming that for some $n \in \mathbb{N}$, $|Q_i| = n$ for all $i \in [k]$. Let $K := \sum_{l = 0}^{k} \binom{k}{l} n^l$.

We let
\begin{gather*}
	W := \Bigg\{\bigg(x_{\bar{a}} : \bar{a} \in \bigcup_{I \subseteq [k]} \prod_{i \in I} Q_i \bigg) \in [0,1]^{K} : \\
	\bigwedge_{\bar{a} \in \prod_{i \in [k]} Q_i \cap R_k} \left( x_{\bar{a}_I} :  I \subseteq [k] \right) \in S_{+} \land \bigwedge_{\bar{a} \in \prod_{i \in [k]} Q_i \setminus R_k} \left( x_{\bar{a}_I} : I \subseteq [k] \right) \notin S_{-} \Bigg\}.
	\end{gather*}
	
Let 
\begin{gather*}
	B := \prod_{\bar{a} \in \bigcup_{I \subseteq [k]} \prod_{i \in I} Q_i} A_I,
\end{gather*}
then $B$ is a box in $[0,1]^K$ with $\nu_{K}(B) > 0$. For every $\bar{b} \in \prod_{i \in [k]} Q_i$ let
\begin{gather*}
B_{\bar{b}} := \left\{\bigg(x_{\bar{a}} : \bar{a} \in \bigcup_{I \subseteq [k]} \prod_{i \in I} Q_i \bigg) \in B :  \left( x_{\bar{b}_I} :  I \subseteq [k] \right) \in A \setminus \left( S_+ \setminus S_- \right) \right\}.
\end{gather*}

 We have 
\begin{gather*}
B \setminus W \subseteq \bigcup_{\bar{b} \in \prod_{i \in [k]} Q_i} 	B_{\bar{b}},
\end{gather*}
which by \eqref{eq: box almost cont} and definition of $B_{\bar{b}}$'s implies
\begin{gather*}
	\nu_{K} \left( B \setminus W \right) \leq \sum_{\bar{b} \in \prod_{i \in [k]} Q_i} \nu_{K}\left( B_{\bar{b}} \right) \leq n^k \cdot \varepsilon \cdot \nu_K(B).
\end{gather*}
So, if we take $\varepsilon < \frac{1}{n^k}$ , we get $\nu_{K} \left(B \cap W \right) > 0$, in particular $\nu_K \left( W \right) > 0$.

Then, using \eqref{eq: Ald Hoov pres} and Remark \ref{rem: measure comparison}, we get 
\begin{gather*}
	0 < \nu_{K}(W) = \\
	\mu' \Bigg( \bigg\{ x \in V' : \bigwedge_{\bar{a} \in \prod_{i \in [k]} Q_i \cap R_k} \left( \zeta_{\bar{a}_I}(x) :  I \subseteq [k] \right) \in S_{+} \land \\
	\bigwedge_{\bar{a} \in \prod_{i \in [k]} Q_i \setminus R_k} \left( \zeta_{\bar{a}_I}(x) : I \subseteq [k] \right) \notin S_{-}  \bigg\}\Bigg) =\\
	\mu \Bigg(  \bigcap_{\bar{a} \in \prod_{i \in [k]} Q_i \cap R_k} E^0_{\bar{a}}  \cap  \bigcap_{\bar{a} \in \prod_{i \in [k]} Q_i \setminus R_k} V \setminus E^1_{\bar{a}} \Bigg).
\end{gather*}

\end{proof}

The next fact follows from  model-theoretic stability of probability algebras in continuous logic \cite[Section 16]{yaacov2008model}, or a more general \cite[Proposition 2.25]{hrushovski2012stable}. See \cite{TaoHrush} for a short elementary proof.
\begin{fac}\label{fac: HrushTao}
	For any real numbers $0 \leq p < q \leq 1$ there exists some $N = N(p,q)$ satisfying the following. If  $(V,\B,\mu)$ is a probability space, and $A_1, \ldots, A_n, B_1, \ldots, B_n \in \B$ satisfy $\mu(A_i \cap B_j) \geq q$ and $\mu(A_j \cap B_i) \leq p$ for all $1 \leq i < j \leq n$, then $n \leq N$.
\end{fac}

Using this we show that the generic $k$-partite ordered hypergraph $G_{k,p}$-exchangeability of a collection of random variables implies its exchangeability with respect to the reduct $G'_{k,p}$ \emph{without the ordering} (this can be viewed as an analog of Ryll-Nardziewski's classical result that for a sequence of random variables, spreadability implies exchangeability for our more complicated notion of exchangeability, see e.g.~\cite{kallenberg1988spreading}).

\begin{lemma}\label{lem: exchang w order implies wo}
Let $(V, \mathcal{B}, \mu)$ be a probability space, and assume that for each $\bar{a} = (a_1, \ldots, a_k) \in \prod_{i \in [k]} P_i$ we have some sets $E^0_{\bar{a}}, E^1_{\bar{a}}  \in \mathcal{B}$ such that the following holds: for any $m \in \mathbb{N}$, $\bar{a}^i = \left(a^i_1, \ldots, a^i_m \right) \in P_i$ and $\bar{b}^i = \left(b^i_1, \ldots, b^i_m \right) \in P_i$ for $i \in \{1, \ldots, k\}$ such that $\qftp_{\mathcal{L}^k_{\opg}}\left(\bar{a}^1, \ldots, \bar{a}^k \right) = \qftp_{\mathcal{L}^k_{\opg}}\left(\bar{b}^1, \ldots, \bar{b}^k \right)$, we have that 
$$\mu \left(\bigwedge_{(\bar{l},v) \in [m]^k \times \{0,1\}}E^{v,t^v_{\bar{l}}}_{(a^1_{l_1}, \ldots, a^k_{l_k})} \right) = \mu \left(\bigwedge_{(\bar{l},v) \in [m]^k \times \{0,1\}}E^{v,t^v_{\bar{l}}}_{(b^1_{l_1},  \ldots, b^k_{l_k})} \right)$$
 for every tuple $\left(t^v_{\bar{l}} \in \{0,1\} : v \in \{0,1\}, \bar{l} \in [m]^k \right)$ (where $E^{t,1}$ denotes $E^t$ and $E^{t,0}$ denotes $\neg E^t$). 
 Then the same holds for any pair of tuples satisfying the weaker assumption 
 $\qftp_{\mathcal{L}^k_{\pg}}\left(\bar{a}^1, \ldots, \bar{a}^k \right) = \qftp_{\mathcal{L}^k_{\pg}} \left( \bar{b}^1, \ldots, \bar{b}^k \right)$, i.e.~Assumption (2) in Lemma \ref{lem: random graph Bergelson} is satisfied.
\end{lemma}
\begin{proof}
 It suffices to show the following (under the given assumption of $G_{k,p}$-exchangeability). Let $a_1^i < \ldots <a_m^i$ in $P
 _i$ be arbitrary, for $i \in [k]$, let a tuple $\left(t^v_{\bar{l}} \in \{0,1\} : v \in \{0,1\}, \bar{l} \in [m]^k\right)$ be fixed, and let $\sigma$ be a permutation of $[m]$ such that $\left( a^1_{l_1}, \ldots, a^k_{l_k} \right) \in R_k \iff \left( a^1_{\sigma(l_1)}, a^2_{l_2}, \ldots, a^k_{l_k} \right) \in R_k$ for all $\bar{l} = \left( l_1, \ldots, l_k \right) \in [m]^k$ (i.e.~$\sigma$ preserves the quantifier-free $\mathcal{L}^k_{\pg}$-type  of the tuple); then
 $$\mu \left(\bigwedge_{(\bar{l},v) \in [m]^k \times \{0,1\}} E^{v,t^v_{\bar{l}}}_{(a^1_{l_1}, \ldots, a^k_{l_k})} \right) = \mu \left(\bigwedge_{(\bar{l},v) \in [m]^k \times \{0,1\}}E^{v,t^v_{\bar{l}}}_{(a^1_{\sigma(l_1)}, a^2_{l_2}, \ldots, a^k_{l_k})} \right)$$
(the case of a permutation $\sigma$ acting on the elements in $P_i$ for $i \neq 1$ is symmetric, and they can be performed separately one by one). As every permutation is a composition of transpositions of consecutive elements, it suffices to show this assuming that $\sigma$ is a transposition of two consecutive elements. That is, towards a contradiction we assume that there is some $i^* \in [m]$, $1 \leq i^* <  i^*+1 \leq m$ such that $\sigma(i^*)=i^*+1, \sigma(i^*+1) = i^*$ and $\sigma$ is constant on all  $i \in [m] \setminus \{ i^*, i^* +1 \}$, and 
 \begin{gather}
  p :=\mu \left(\bigwedge_{((i, l_2,\ldots, l_k),v) \in [m]^k \times \{0,1\}}E^{v,t^v_{(i, l_2,\ldots, l_k)}}_{(a^1_{i}, a^2_{l_2}\ldots, a^k_{l_k})} \right) \label{eq: diff measure} \\
  < q := \mu \left(\bigwedge_{((i, l_2,\ldots, l_k),v) \in [m]^k \times \{0,1\}}E^{v,t^v_{(i, l_2,\ldots, l_k)}}_{(a^1_{\sigma(i)}, a^2_{l_2}, \ldots, a^k_{l_k})} \right) \nonumber
 \end{gather}
(the case with ``$>$'' is symmetric). By the genericity of the hypergraph $G_{k,p}$ (Definition \ref{def: generic hypergraph}) we can find a \emph{strictly  $<$-increasing} infinite sequence of elements $(a'_{i} : i \in \mathbb{N})$ in $P_1$ such that:
 \begin{itemize}
 	\item for $i = 2j$ we have 
 	\begin{gather*}
 	\qftp_{\mathcal{L}^k_{\opg}}\left(a^1_1, \ldots, a^1_{i^*-1}, a'_i,  a^1_{i^*+2}, \ldots, a^1_{m} ; \bar{a}^2, \ldots, \bar{a}^k \right)\\
 	=\qftp_{\mathcal{L}^k_{\opg}} \left( a^1_1, \ldots, a^1_{i^*-1}, a^1_{i^*},  a^1_{i^*+2}, \ldots, a^1_{m} ; \bar{a}^2, \ldots, \bar{a}^k \right);
 	\end{gather*}
 	\item for $i = 2j+1$ we have 
 	\begin{gather*}
 	\qftp_{\mathcal{L}^k_{\opg}} \left( a^1_1, \ldots, a^1_{i^*-1}, a'_i,  a^1_{i^*+2}, \ldots, a^1_{m} ; \bar{a}^2, \ldots, \bar{a}^k \right)\\
 	=\qftp_{\mathcal{L}^k_{\opg}} \left( a^1_1, \ldots, a^1_{i^*-1}, a^1_{i^*+1},  a^1_{i^*+2}, \ldots, a^1_{m} ; \bar{a}^2, \ldots, \bar{a}^k \right).
 	\end{gather*}
 \end{itemize}
 In particular, for any $j < j' \in \mathbb{N}$ we then have
 \begin{gather}
 	\qftp_{\mathcal{L}^k_{\opg}}\left(a^1_1, \ldots, a^1_{i^*-1}, a'_{2j}, a'_{2j'+1},  a^1_{i^*+2}, \ldots, a^1_{m} ; \bar{a}^2, \ldots, \bar{a}^k \right)  \label{eq: even measure}\\
  = \qftp_{\mathcal{L}^k_{\opg}} \left( a^1_1, \ldots, a^1_{i^*-1}, a^1_{i^*}, a^1_{i^*+1},  a^1_{i^*+2}, \ldots, a^1_{m} ; \bar{a}^2, \ldots, \bar{a}^k \right); \nonumber\\
 \qftp_{\mathcal{L}^k_{\opg}} \left( a^1_1, \ldots, a^1_{i^*-1}, a'_{2j+1}, a'_{2j'},  a^1_{i^*+2}, \ldots, a^1_{m} ; \bar{a}^2, \ldots, \bar{a}^k \right)  \label{eq: odd measure}\\
  = \qftp_{\mathcal{L}^k_{\opg}} \left( a^1_1, \ldots, a^1_{i^*-1}, a^1_{i^*+1}, a^1_{i^*}, a^1_{i^*+2}, \ldots, a^1_{m} ; \bar{a}^2, \ldots, \bar{a}^k \right). \nonumber
 \end{gather}

 For $l \in \mathbb{N}$ we define
 \begin{align*}
 &A_l := \bigcap_{\substack{i\in [m] \setminus \{i^*, i^*+1\}, l_2, \ldots, l_k \in [m],\\ v \in \{0,1\}}}E^{v,t^v_{(i, l_2, \ldots, l_k)}}_{(a^1_i, a^2_{l_2}, \ldots, a^k_{l_k})} \cap \bigcap_{\substack{l_2, \ldots, l_k \in [m] \\ v \in \{0,1\}}} E^{v,t^v_{(i^*,l_2, \ldots, l_k)}}_{(a'_{2l}, a^2_{l_2}, \ldots, a^k_{l_k})};\\
 &B_l :=  \bigcap_{\substack{i \in [m] \setminus \{i^*, i^*+1\}, l_2, \ldots, l_k \in [m],\\ v \in \{0,1\} }}E^{v,t^v_{(i,l_2, \ldots, l_k)}}_{(a^1_i, a^2_{l_2}, \ldots, a^k_{l_k})} \cap \bigcap_{\substack{l_2, \ldots, l_k \in [m],\\ v \in \{0,1\}}} E^{v,t^v_{(i^* +1 , l_2, \ldots, l_k)}}_{(a'_{2l +1}, a^2_{l_2}, \ldots, a^k_{l_k})}.
 \end{align*}
 Then by (\ref {eq: diff measure}), (\ref{eq: even measure}), (\ref{eq: odd measure}) and the assumption of $G_{k,p}$-exchangeability, we have $\mu(A_i \cap B_j) = p$ for all $i < j$, and $\mu(A_i \cap B_j) = q$ for all $i > j$ --- contradicting Fact \ref{fac: HrushTao}.
 \end{proof}
  
\subsubsection{Proof of Theorem \ref{prop: gen fib fin VCk dim}}

Assume towards a contradiction that there exist some $k$, $\bar{d}$ and $r<s$ in $\mathbb{Q}^{[0,1]}$ such that: for every $j \in \mathbb{N}$ we have a $(k+2)$-partite graded probability space $(V^j_{[k+2]}, \B^j_{\bar{n}}, \mu^j_{\bar{n}})_{n \in \mathbb{N}^{k+2}}$ and a $(k+2)$-ary $\B^j_{\bar{1}^{k+2}}$-measurable function $f^j: (V^j)^{\bar{1}^{k+2}} \to [0,1]$ such that $\VC_k(f^j) \leq \bar{d}$, but such that the function $(f^j)': (V^j)^{\bar{1}^{k+1}} \to [0,1], (f^j)'(x_1, \ldots, x_{k+1}) := \int  f(x_1, \ldots, x_{k+2}) d \mu^j_{\bar{0}^{k+1 \frown} (1)}(x_{k+2})$
$(r,s)$-shatters some $k$-box 
$$B^j = \left\{ a^{j,1}_{1}, \ldots, a^{j,1}_{j}\right\} \times \ldots \times \left\{ a^{j,k}_{1}, \ldots, a^{j,k}_{j}\right\}.$$

As in the proof of Lemma \ref{lem: type-def of norm}, for any $r \in [0,1]$ there exist countable partial $\mathcal{L}_{\infty}$-types $\rho_{\leq r}(x_1, \ldots, x_{k+1})$ and $\rho_{\geq r}(x_1, \ldots, x_{k+1})$ satisfying the following: for any $(k+2)$-partite graded probability space $\mathfrak{P} = (V_{[k+2]}, \B_{\bar{n}}, \mu_{\bar{n}})_{n \in \mathbb{N}^{k+2}}$, $\B_{\bar{1}^{k+2}}$-measurable function $f$, an $\mathcal{L}_{\infty}$-structure $\mathcal{M}' \propto \mathcal{M}_{\mathfrak{P},f}$ and a tuple $(a_1, \ldots, a_{k+1}) \in V^{\bar{1}^{k+1}}$ we have
\begin{gather}\label{eq: int preserves VCk1}
	\mathcal{M}' \models \rho_{\leq r}(a_1, \ldots, a_{k+1}) \\
	\iff \int f(a_1, \ldots, a_{k+1}, x_{k+2}) d\mu_{\bar{0}^{k+1 \frown} (1)}(x_{k+2}) \leq r, \textrm{and similarly for ``$\geq r$''.} \nonumber
\end{gather} Consider the countable partial $\mathcal{L}_{\infty}$-type 
\begin{gather*}
	\tau\left( (x_g : g \in G_{k+1,p}) \right) :=
	\bigwedge_{(g_1, \ldots, g_{k+1}) \in R_{k+1}} \rho_{\leq r} (x_{g_1}, \ldots,  x_{g_{k+1}}) \land \\
	\bigwedge_{(g_1, \ldots, g_{k+1}) \in \prod_{i \in [k+1]}P_i \setminus R_{k+1}} \rho_{\geq s} (g_{g_1}, \ldots, x_{g_{k+1}}).
\end{gather*}
Let $\tau_0$ be a finite set of formulas from $\tau$ only involving $\ell$ variables from  $(x_g : g \in G_{k+1,p})$. As in the proof of Lemma \ref{lem: full arity indisc witness to IPn} (1)$\Rightarrow$(2), using that trivially $\mathcal{M}_{\mathfrak{P}^j, f^j} \propto \mathcal{M}_{\mathfrak{P}^j, f^j}$, by assumption and \eqref{eq: int preserves VCk1}, for every $j \geq \ell$ we have that $\tau_0$ is realized in $\mathcal{M}_{\mathfrak{P}^j, f^j}$. By \L os' theorem this implies that $\tau_0$ is also realized in $\tilde{\mathcal{M}}$. Hence, by $\aleph_1$-saturation of $\tilde{\mathcal{M}}$, we have $\tilde{\mathcal{M}} \models \tau \left( (a_g : g \in G_{k+1,p}) \right)$ for some $(a_g : g \in G_{k+1,p})$ with $g \in P_i \Rightarrow a_g \in \tilde{V}_i$ for $i \in [k+1]$.

By $\aleph_1$-saturation of $\tilde{\mathcal{M}}$ and Fact \ref{fac: existence of indiscernibles}(2), let $(a'_g)_{g \in G_{k+1,p}}$ be $G_{k+1,p}$-indiscernible over $\emptyset$ in $\tilde{\mathcal{M}}$ based on $(a_g)_{g \in G_{k+1,p}}$. Then we still have $\tilde{\mathcal{M}} \models \tau \left( (a'_g : g \in G_{k+1,p}) \right)$.

For $\bar{g} = (g_1, \ldots, g_{k+1}) \in \prod_{i \in [k+1]} P_i$, we write $\bar{a}_{\bar{g}} := \left(a'_{g_1}, \ldots, a'_{g_{k+1}} \right)$; and let $\tilde{\mu} := \tilde{\mu}_{\bar{\delta}^{k+2}}$. Then, as $\tilde{\mathcal{M}} \propto \mathcal{M}_{\tilde{\mathfrak{P}}, \tilde{f}}$, by definition of $\tau$ and \eqref{eq: int preserves VCk1} we have, 
\begin{gather}\label{eq: int preserves VCk3}
	G_{k+1,p} \models R_{k+1}(g_1, \ldots, g_{k+1}) \Rightarrow \int \tilde{f}(\bar{a}_{\bar{g}}, x_{k+2}) d\tilde{\mu}(x_{k+2}) \leq r, \\
	G_{k+1,p} \models \neg R_{k+1}(g_1, \ldots, g_{k+1}) \Rightarrow \int \tilde{f}(\bar{a}_{\bar{g}}, x_{k+2}) d\tilde{\mu}(x_{k+2}) \geq s.\nonumber
\end{gather}

Fix arbitrary $\bar{g}^0 = (g^0_1, \ldots, g^0_{k+1}) \in R_{k+1}, \bar{g}^1 = (g^1_1, \ldots, g^1_{k+1}) \in \prod_{i \in [k+1]}P_i \setminus R_{k+1}$. We let $F^{\bowtie q}_{\bar{a}_{\bar{g}}} = \left\{ \bar{x}_{k+2} \in \tilde{V}_{k+2} : \tilde{\mathcal{M}} \models F^{\bowtie q} \left(\bar{a}_{\bar{g}},x_{k+2} \right) \right\}$ for $q \in \mathbb{Q}^{[0,1]}$, $\bowtie \in \{ <, \geq \}$ and $\bar{g} \in \prod_{i \in [k+1]}P_i$.

By \eqref{eq: int preserves VCk3}, Lemma \ref{lem: comparing integrals} and $\tilde{\mathcal{M}} \propto \mathcal{M}_{\tilde{\mathfrak{P}}, \tilde{f}}$, there exist some $r' < s' \in \mathbb{Q}^{[0,1]}$ so that 
\begin{gather*}
	\tilde{\mu} \left( F^{< r'}_{\bar{a}_{\bar{g}^0}} \right) > \tilde{\mu} \left( F^{< s'}_{\bar{a}_{\bar{g}^1}} \right).
\end{gather*}

For $\bar{g} \in \prod_{i \in [k+1]}P_i$, let $E^0_{\bar{g}} := F^{< r'}_{\bar{a}_{\bar{g}}}, E^1_{\bar{g}} := F^{< s'}_{\bar{a}_{\bar{g}}}$.
As $(a'_g)_{g \in G_{k+1,p}}$ is $G_{k+1,p}$-indiscernible, this implies that the assumption of Lemma \ref{lem: exchang w order implies wo} is satisfied (using that the $F^{<q}$ and $m<q$ predicates are in $\mathcal{L}_{\infty}$ for all $q \in \mathbb{Q}^{[0,1]}$). Hence the assumption of Lemma \ref{lem: random graph Bergelson} is also satisfied, and it follows that for any finite $Q_i \subseteq P_i$ and $Q := \prod_{i \in [k+1]} Q_i$, we have
\begin{gather*}
	\tilde{\mu} \left( \bigcap_{\bar{g} \in Q \cap R_{k+1}} F^{< r'}_{\bar{a}_{\bar{g}}}   \cap  \bigcap_{\bar{g} \in Q \setminus R_{k+1}} F^{\geq s'}_{\bar{a}_{\bar{g}}} \right)> 0.
\end{gather*}

In  particular, this intersection is non-empty. Hence, by $\aleph_1$-saturation of $\tilde{\mathcal{M}}$, there exists some $b \in V_{k+2}$ so that for all $(g_1, \ldots, g_{k+1}) \in \prod_{i \in [k+1]} P_i$ we have
\begin{gather}
	G_{k+1,p} \models R_{k+1}(g_1, \ldots, g_{k+1}) \Rightarrow \tilde{\mathcal{M}} \models F^{< r'}(a'_{g_1}, \ldots, a'_{g_{k+1}}, b) \textrm{ and}\\
	G_{k+1,p} \models \neg R_{k+1}(g_1, \ldots, g_{k+1}) \Rightarrow \tilde{\mathcal{M}} \models F_{\geq s'}(a'_{g_1}, \ldots, a'_{g_{k+1}}, b).\nonumber
\end{gather}
By Lemma \ref{lem: full arity indisc witness to IPn}(2)$\Rightarrow$(1), this implies that the $(k+1)$-ary function $\tilde{f}_c$ has infinite $\VC_k$-dimension --- a contradiction to the assumption by Lemma \ref{lem: shattering is definable}.

Theorem \ref{prop: gen fib fin VCk dim} implies the following slightly more general version.

\begin{cor}\label{thm: gen fib fin VCk dim}
	For every $t \in \mathbb{N}, \bar{d} < \infty$ there exists some $\bar{D} = \bar{D}(t,\bar{d}) < \infty$ satisfying the following.
	
	Assume that $k \in \mathbb{N}$, $(V_{[k]}, \B_{\bar{n}}, \mu_{\bar{n}})_{n \in \mathbb{N}^k}$ is a $k$-partite graded probability space, $f: V^{\bar{n}} \to [0,1]$ is $ \B_{\bar{m}}$-measurable for some $\bar{m} = \bar{m}' + \bar{m}'' \in \mathbb{N}^{k}$ and $\VC_{t}(f) \leq \bar{d}$ (in the sense of Definition \ref{def: VCk dimension of functions}(4), i.e.~with respect to any partition of the variables of $f$  into $(t+1)$ groups). 
	Then the function $g: V^{\bar{m}'} \to [0,1]$ defined by
	$$g(\bar{x}') := \int f(\bar{x}' \oplus \bar{x}'') d \mu_{\bar{m}''} \left( \bar{x}'' \right)$$
	(so $g$ is $\B_{\bar{m}'}$-measurable by Fubini) satisfies $\VC_{t}(g) \leq \bar{D}$.
\end{cor}
\begin{proof}
	Since permuting the variables preserves finiteness of $\VC_{t}$-dimension by Proposition \ref{prop: perm of vars fin VCk dim}, we only have to show that if $\bar{m}' = \bar{m}_1 + \ldots + \bar{m}_{t+1}$ for some $\bar{m}_i \in \mathbb{N}^{k}$ and the $(t+2)$-ary function 
	
	$$f': (\bar{x}_1, \ldots, \bar{x}_{t+1}, \bar{x}'') \in \left(\prod_{i \in [t+1]} V^{\bar{m}_i} \right) \times V^{\bar{m}''}  \to f(\bar{x}_1 \oplus \ldots \oplus \bar{x}_{t+1} \oplus \bar{x}'')$$
	satisfies $\VC_{t} \leq \bar{d}$, then the $(t+1)$-ary function
	$$g': (\bar{x}_1, \ldots, \bar{x}_{t+1}) \in \prod_{i \in [t+1]} V^{\bar{m}_i} \to \int f(\bar{x}_1 \oplus \ldots \oplus \bar{x}_{t+1} \oplus \bar{x}'') d \mu_{\bar{m}''}(\bar{x}'')$$
	satisfies $\VC_{t} \leq \bar{D}$.

 We let $V'_i := V^{\bar{m}_i}$ for $i \in [t+1]$, $V'_{t+2} := V^{\bar{m}''}$ and for $\bar{n} = (n_1, \ldots, n_{t+2}) \in \mathbb{N}^{t+2}$, we let $\bar{n}' := n_1 \bar{m}_1 + \ldots + n_{t+1}\bar{m}_{t+1} + n_{t+2}\bar{m}''$ and $\B'_{\bar{n}} := \B_{\bar{n}'}, \mu'_{\bar{n}} := \mu_{\bar{n}'}$. By ``gluing coordinates'' (Remark \ref{rem: power graded prob space}), $\left(V'_{[t+2]}, \bar{B}'_{\bar{n}}, \mu'_{\bar{n}} \right)_{\bar{n} \in \mathbb{N}^{t+2}}$ is a $(t+2)$-partite graded probability space and the $(t+2)$-ary function
 	$$f'': (\bar{x}_1, \ldots, \bar{x}_{t+1}, \bar{x}'') \in \prod_{i \in [t+2]} V'_{i} \to f(\bar{x}_1 \oplus \ldots \oplus \bar{x}_{t+1} \oplus \bar{x}'')$$
is $\B'_{\bar{1}^{k+2}}$-measurable and satisfies $\VC_t(f'') \leq \bar{d}$. Then, applying Theorem \ref{prop: gen fib fin VCk dim},  there exists some $\bar{D} = \bar{D}(t, \bar{d})$ so that the function
$$g'': (\bar{x}_1, \ldots, \bar{x}_{t+1}) \in \prod_{i \in [t+1]} V'_i \to \int f(\bar{x}_1 \oplus \ldots \oplus \bar{x}_{t+1} \oplus \bar{x}'') d \mu'_{\bar{\delta}_{k+2}}(\bar{x}'')$$
satisfies $\VC_t(g'') \leq \bar{D}$. Unwinding, this  gives $\VC_t(g) \leq \bar{D}$.
\end{proof}

\section{Final remarks}\label{sec: final remarks}

\subsection{Directions for future work}


It would be interesting to obtain explicit bounds and investigate their optimality for the main results of the paper (Proposition \ref{prop: finite VCk-dim implies approx bounded} and Corollary \ref{thm: the very main thm}).

\begin{problem}
	It is possible to finitize our proof of Proposition \ref{prop: finite VCk-dim implies approx bounded}, replacing the use of ultraproducts and indiscernible sequences by multiple applications of Ramsey's theorem and complicated $\varepsilon-\delta$ bookkeeping. We expect that the bound on $N_0$ should be as bad as in the regularity lemma for general hypergraphs (i.e.~an exponential tower of hight depending on $\frac{1}{\varepsilon}$), while we expect $N$ to be bounded by an exponential tower of height bounded in terms of $d$. We leave the investigation of these bounds for future work.
\end{problem}   

In Proposition \ref{prop: finite VCk-dim implies approx} we show that every $k$-ary fiber of a $(k
+1)$-ary function of finite $\VC_k$-dimension can be approximated in $L^2$ in terms of a fixed finite set of its $k$-ary fibers along with smaller arity data. And in Lemma \ref{lem: pos meas approx} we strengthen its conclusion from ``there exists an approximation'' to ``there exists a positive measure set of approximations''. We ask if this can further be strengthened to ``there exists a measure $1$ set of approximations'':

\begin{problem}
	Is it possible to strengthen the conclusion of Lemma \ref{lem: pos meas approx} to ``the set of tuples $\bar{w} \in V^{\bar{m}}$ with $||f_x-f^t_{\bar{ w},x}||_{L^2}\leq\delta$ has $\mu_{\bar{m}}$-measure converging to $1$ when $l,t \to \infty$''? 
	\end{problem}
	This problem has a positive answer in the case of bounded $\VC$-dimension (i.e.~the case $k=1$) using that a sufficiently long tuple almost surely gives an $\varepsilon$-net for differences (see the discussion in the introduction), but for $k>2$, we only know that we get a good choice with positive measure.

%
%
%

\subsection{Some model-theoretic consequences}
We record a couple of model theoretic corollaries of our results.

As we already mentioned, Theorem \ref{prop: gen fib fin VCk dim} generalizes \cite[Corollary 4.2]{yaacov2009continuous} in the case $k=1$. 
Using it (and recalling that a first-order theory $T$ is $k$-dependent if every $(k+1)$-ary relation definable on tuples in a model of $T$ has finite $\VC_k$-dimension), one immediately obtains the following model-theoretic corollary generalizing the main Theorem 5.3 there.
\begin{cor}\label{cor: Keils rand is n-dep}
	Let $T$ be a $k$-dependent first-order theory (classical or continuous). Then its Keisler randomization $T^{R}$ is also $k$-dependent.
\end{cor}

We also have the following application to \emph{Keisler measures}, i.e.~finitely additive probability measures on the space of types of a first-order theory. We refer to e.g.~\cite{starchenko2016nip} for a detailed discussion.

\begin{cor}
	Assume that $T$ is $k$-dependent, $k' \geq k+1$, $\mathbb{M} \models T$ and let $\mu_1, \ldots, \mu_{k'}$ be global Keisler measures on the definable subsets of the sorts $\mathbb{M}^{x_1}, \ldots, \mathbb{M}^{x_{k'}}$ respectively, such that each $\mu_i$ is Borel-definable and all these measures commute, i.e.~$\mu_i \otimes \mu_j$ for all $i,j \in [k']$. 
Then for every formula $\varphi(x_1, \ldots, x_{k'}) \in \mathcal{L}(\mathbb{M})$  and $\varepsilon \in \mathbb{R}_{>0}$ there exist some formula $\psi(x_1, \ldots, x_{k'})$ which is a Boolean combination of finitely many $(\leq k)$-ary formulas each given by an instances of $\varphi$ with some parameters placed in all but at most $k$ variables, so that taking $\mu := \mu_1 \otimes \ldots \otimes \mu_{k'}$ we have $\mu \left(\varphi \triangle \psi \right) < \varepsilon$.
\end{cor}

Indeed,  for $\bar{n} = (n_1, \ldots, n_{k'}) \in \mathbb{N}^{k'}$ we let $\mathbb{M}_{\bar{x}_{\bar{n}}}$ be the sort corresponding to $\prod_{i \in [k']} \left( \mathbb{M}_{x_i} \right)^{n_i}$, $\B^0_{\bar{n}}$ the Boolean algebra of all definable subsets of $\mathbb{M}_{\bar{x}_{\bar{n}}}$ and 
 $\mu_{n_1, \ldots, n_{k'}} := \mu_1^{\otimes n_1} \otimes \ldots \otimes \mu_{k'}^{\otimes n_{k'}}$. Each Boolean algebra $\B^0_{\bar{n}}$ can be viewed as a Boolean algebra of the clopen subsets of the corresponding space of types  $V^{\bar{n}} := S_{\bar{x}_{\bar{n}}}(\mathbb{M})$, and $\mu_{\bar{n}}$ as a finitely additive probability measure on it. By Carath\'eodory's theorem, it extends uniquely to a regular countably additive probability measure $\mu'_{\bar{n}}$ on the $\sigma$-algebra $\B_{\bar{n}}$ of all Borel subsets of this space. Then we have that $\left(V_{[k']}, \B_{\bar{n}}, \mu'_{\bar{n}} \right)_{\bar{n} \in \mathbb{N}^{k'}}$ is a $k'$-partite graded probability space. Indeed, the assumption of pairwise commuting on the $\mu_i$'s implies
 \begin{gather*}
\left(\mu_1^{\otimes(n_1 + m_1)} \otimes \ldots \otimes  \mu_{k'}^{\otimes(n_{k'} + m_{k'})}\right) = \\
\left( \mu_1^{\otimes n_1} \otimes \ldots \otimes  \mu_{k'}^{\otimes n_{k'} } \right) \otimes \left( \mu_1^{\otimes m_1} \otimes \ldots \otimes  \mu_{k'}^{\otimes m_{k'} }  \right),
\end{gather*}
which together with Borel definability imply the Fubini property in Definition \ref{def: kGPS}, and the other conditions in the definition are clearly satisfied. Now we apply Corollary \ref{cor: main thm for hypergraphs} to $\varphi$ viewed as a clopen subset in $\B^0_{\bar{1}^{k'}}$, and approximating Borel sets in the resulting decomposition by the clopen ones from the generating set, we obtain the corollary.

\bibliographystyle{alpha}
\bibliography{refs}
\end{document}